\documentclass{article}
\usepackage[utf8]{inputenc}
\DeclareFontFamily{T1}{cbgreek}{}
\DeclareFontShape{T1}{cbgreek}{m}{n}{<-6>  grmn0500 <6-7> grmn0600 <7-8> grmn0700 <8-9> grmn0800 <9-10> grmn0900 <10-12> grmn1000 <12-17> grmn1200 <17-> grmn1728}{}
\DeclareSymbolFont{quadratics}{T1}{cbgreek}{m}{n}
\DeclareMathSymbol{\qoppa}{\mathord}{quadratics}{19}
\DeclareMathSymbol{\Qoppa}{\mathord}{quadratics}{21}

\usepackage{amsmath, amssymb, amsthm}
\usepackage{longtable}
\usepackage{stmaryrd}
\usepackage[letterpaper,margin=1in]{geometry}   
\usepackage{tocloft}
\usepackage{bookmark}
\usepackage{url}		
\usepackage{lmodern}			
\usepackage[T1]{fontenc}
\usepackage{enumerate}
\usepackage{enumitem}
\usepackage{mathrsfs}
\usepackage{graphicx}
\usepackage{soul,color}
\usepackage{tikz-cd}
\usepackage[maxbibnames=99,style=alphabetic]{biblatex}
\usepackage{csquotes}
\usepackage{chngcntr}
\usepackage[bbgreekl]{mathbbol}
\counterwithin{equation}{section}
\usepackage[disable]{todonotes}
\usepackage{xparse}

\definecolor{darkred}{RGB}{128,0,0}
\definecolor{darkgreen}{RGB}{0,128,0}
\definecolor{darkblue}{RGB}{0,0,128}

\hypersetup{linktocpage,
	pdfborder = {0 0 0},
	colorlinks,
	citecolor=darkgreen,
	filecolor=darkred,
	linkcolor=darkblue,
	urlcolor=cyan!50!black!90}

\DeclareSymbolFontAlphabet{\mathbb}{AMSb}
\DeclareSymbolFontAlphabet{\mathbbl}{bbold}

\DeclareMathOperator{\Alg}{Alg} 
\DeclareMathOperator{\Brp}{\mathfrak{br}^\mathrm{p}} 
\DeclareMathOperator{\Spec}{Spec}
\DeclareMathOperator{\CAlg}{CAlg} 
\DeclareMathOperator{\CAlgp}{CAlg^\mathrm{p}} 
\DeclareMathOperator{\Cat}{\mathcal{C}at} 
\NewDocumentCommand{\CAT}{O{}}{%
	\IfValueTF{#1}{%
		C\textsc{at}^{#1}_\infty
	}{%
		C\textsc{at}_\infty
	}%
} 
\DeclareMathOperator{\Catex}{\Cat_\infty^{ex}} 
\DeclareMathOperator{\Catperf}{\Cat_{\infty,idem}} 
\DeclareMathOperator{\Cath}{Cat^\mathrm{h}_\infty} 
\DeclareMathOperator{\Cathidem}{Cat^\mathrm{h}_{\infty,idem}} 
\DeclareMathOperator{\Catp}{Cat^\mathrm{p}_\infty} 
\DeclareMathOperator{\Catpidem}{Cat^\mathrm{p}_{\infty, idem}} 
\DeclareMathOperator{\Einfty}{\mathbb{E}_\infty} 
\DeclareMathOperator{\ex}{ex} 
\DeclareMathOperator*{\fiberproduct}{\times}
\DeclareMathOperator{\Fun}{Fun} 
\DeclareMathOperator{\gp}{gp} 
\DeclareMathOperator{\id}{id} 
\DeclareMathOperator{\Mod}{Mod} 
\DeclareMathOperator{\LMod}{Mod} 
\DeclareMathOperator{\Mon}{Mon} 
\DeclareMathOperator{\pic}{\mathfrak{pic}} 
\DeclareMathOperator{\Picp}{\mathfrak{pic}^\mathrm{p}} 
\DeclareMathOperator{\Pn}{Pn} 
\DeclareMathOperator{\Spectra}{Sp} 
\DeclareMathOperator{\Spaces}{\mathcal{S}} 
\DeclareMathOperator{\gmq}{\mathrm{gl}_1^{\mathrm{p}}}

\renewcommand{\epsilon}{\varepsilon}
\renewcommand{\rho}{\varrho}
\renewcommand{\phi}{\varphi}

\newcommand{\ZZ}{\mathbb{Z}}

\newcommand{\CC}{\mathbb{C}}

\newcommand{\EE}{\mathbb{E}}

\newcommand{\op}{\mathrm{op}} 
\newcommand{\sphere}{\mathbb{S}}

\newcommand{\pnpic}{\ensuremath{\mathfrak{pic}^\mathrm{p}}}
\newcommand{\pnbr}{\ensuremath{\mathfrak{br}^\mathrm{p}}}
\newcommand{\Pic}{\ensuremath{\mathrm{Pic}}}
\newcommand{\br}{\ensuremath{\mathrm{Br}}}

\newcommand{\perf}{\mathrm{Perf}}

\newcommand{\kr}{\mathrm{K}\mathbb{R}}
\newcommand{\Fin}{\mathrm{Fin}} 
\newcommand{\Assoc}{\mathrm{Assoc}}

\newcommand{\ev}{\mathrm{ev}}
\newcommand{\morphism}[3]{\mathop{\rotatebox{270}{$\xrightarrow{\rotatebox{90}{$\scriptstyle#2$}}$}}\limits^{#1}_{#3}}
\newcommand{\Ar}{\mathrm{Ar}}

\newcommand{\rlarrows}{\mathrel{\substack{\textstyle\longrightarrow\\[-0.6ex]
			\textstyle\longleftarrow}}}

\numberwithin{equation}{subsection}
\newtheorem{theorem}[equation]{Theorem}
\newtheorem{proposition}[equation]{Proposition}
\newtheorem{lemma}[equation]{Lemma}
\newtheorem{corollary}[equation]{Corollary}

\theoremstyle{definition}
\newtheorem{definition}[equation]{Definition}
\newtheorem{construction}[equation]{Construction}
\newtheorem{remark}[equation]{Remark}

\newtheorem{observation}[equation]{Observation}
\newtheorem{notation}[equation]{Notation}
\newtheorem{example}[equation]{Example}
\newtheorem{examples}[equation]{Examples}

\newtheorem{recollection}[equation]{Recollection}
\newtheorem{variant}[equation]{Variant}

\title{Involutive Brauer groups and Poincar\'e rings}
\author{Viktor Burghardt, Noah Riggenbach, Lucy Yang}
\date{}
\begin{document}

\maketitle
\begin{abstract}
    We use the formalism of Poincar{\'e} $ \infty $-categories, as developed by \cite{CDHHLMNNSI}, to define and study moduli stacks of line bundles with $ \lambda $-hermitian pairings and of Morita equivalence classes of Azumaya algebras equipped with an involution.
    Our moduli spaces give rise to enhancements of the ordinary Picard and Brauer groups which incorporate the data of an involution on the base; we will refer to these new invariants as the Poincar{\'e} Picard group and the Poincar{\'e} Brauer group. 
    We show that we can recover the involutive Brauer group of Parimala-Srinivas in \cite{Parimala_Srinivas} from the Poincar\'e Brauer group when the former is defined; however, they no longer agree even for closed points due to the existence of shifted perfect pairings. 
    A natural context for Poincar\'e Picard and Poincar\'e Brauer groups is a category of highly structured ring spectra which we refer to as Poincar\'e rings. We also compute these invariants for the sphere spectrum and other examples. 
    As a consequence, we deduce a derived enhancement of a classical theorem of Saltman. 
\end{abstract}
\tableofcontents
	
\section{Introduction}\label{section:introduction} 	
	
\subsection{Motivation}\label{subsection:Azumaya_motivation}
A fundamental invariant in arithmetic geometry and algebraic geometry is the Brauer group. 
This is a functor $\mathrm{Br}(-)$ which sends a scheme $X$ to the group of Morita equivalence classes of Azumaya algebras over $X$. 
By work of Grothendieck, there is an injective transformation
    \[
\mathrm{Br}(-) \to \mathrm{Br}'(-):= \mathrm{H}^2_{\acute{e}t}(-; \mathbb{G}_m)_{\mathrm{tors}};
\] 
the latter group is often referred to as the \emph{cohomological Brauer group}. 
This map is known to be an isomorphism in many cases. Unpublished work of Gabber shows that $\mathrm{Br}(X)\simeq\mathrm{Br}'(X)$, when $X$ is quasi-compact, separated, and admits an ample line bundle; see \cite{deJong_Gabber}.
However, it is also known that $\mathrm{Br}(X)\not\simeq \mathrm{Br}'(X)$ in general; see \cite{br_neq_br_prime}.
	
Given a scheme $X$ and an Azumaya algebra $\mathcal{A}$ on $X$, an involution of $\mathcal{A}$ is an isomorphism\footnote{Here we do not assume that $\sigma$ is $\mathcal{O}_X$-linear.} $\sigma:\mathcal{A}\to \mathcal{A}^\op$ such that $\sigma^\op\circ \sigma =\mathrm{id}_\mathcal{A}$. 
When $2\in \mathcal{O}_X^\times$, such objects are very closely connected to adjoint group schemes of types A, B, C, and D; see \cite{Srimathy}. 
The presence of an involution on $ \mathcal{A} $ also constrains the behavior of its Brauer class $ [\mathcal{A}] $: there exists an Azumaya algebra $ \mathcal{A}' $ in the Brauer class $ [\mathcal{A}] $ admitting an involution if and only if $ [\mathcal{A}] \in H^2_{\acute{e}t}(X;\mathbb{G}_m) $ is sent to the trivial class under the norm \cites[\S9 Theorem 19]{MR123587}[Theorem 3.1(a)]{MR495234}[Theorem 1]{Parimala_Srinivas}. 
    
Assume that $2 \in \mathcal{O}_X^\times$.  
In the cases where either $\lambda := \sigma|_{Z(\mathcal{A})} $ is the identity or $ \lambda $ is the canonical morphism $ \mathcal{O}_X \to \overline{\lambda}_*\mathcal{O}_X $ associated to a $ C_2 $-Galois action $ \overline{\lambda}: X \xrightarrow{\sim} X $, Parimala-Srinivas introduced an \emph{involutive Brauer group} consisting of Morita equivalence classes of Azumaya algebras with involution \cite{Parimala_Srinivas}. 
Such Azumaya algebras with involution are classified by \'etale cohomology with nonabelian coefficients, and the involutive Brauer group sits in a long exact sequence involving the norm map. 
Fr\"ohlich-Wall and Wall have constructed extensions of the involutive Brauer group \cite{MR1803361}.\footnote{We thank Ben Williams for bringing this to our attention.} 
First and Williams laid the groundwork for extending the involutive Brauer group under fewer restrictions on $\lambda$ in \cite{azumaya_involution}. 
However, by \cite{first2025counterexamplesinvolutionsazumayaalgebras}, the group scheme whose torsors classify isomorphism classes of Azumaya algebras with $\lambda$-involution will be subtle (also see discussion on \cite[p.531]{azumaya_involution}). 
For example, as pointed out to us by Uriya First, this group scheme will not always be reductive. 
    
This work extends the theory of involutive Brauer groups to more general involutions, albeit using a different approach than that suggested by \cite[530]{azumaya_involution}. 
Let $\mathrm{Sch}^{C_2-\textrm{aff}, \mathrm{qcqs}}$ be the category of pairs $(X,\lambda)$ where $X$ is a quasicompact quasiseparated scheme and $\lambda: X \xrightarrow{\sim} X$ is a $C_2$-action on $X$ such that all orbits lie in affine open subschemes.
	
\begin{theorem}\label{thm:intro_inv_br_and_inv_br_prime}
    There exist contravariant functors
    \[
    \mathrm{Br}(-,-), \mathrm{Br}'(-,-): \mathrm{Sch}^{C_2-\mathrm{aff},\mathrm{ qcqs}} \to \mathrm{Ab},
    \]
    called the involutive Brauer group and the cohomological Brauer group, respectively, and an injective natural transformation
    \[
    \mathrm{Br}(-,-) \to \mathrm{Br}'(-,-)
    \]
    such that
    \begin{enumerate}
        \item (Theorem~\ref{thm: PS comparison in text}) $\mathrm{Br}(X,\lambda)$ agrees with the functor constructed by Parimala and Srinivas when the latter is defined;
        \item(Definition~\ref{defn:cohomological_involutive_Brauer_group}) there are maps
        \[
        \mathrm{Br}(X,\lambda) \to \mathrm{Br}(X) \quad \text{and} \quad \mathrm{Br}'(X,\lambda) \to \mathrm{Br}'(X)
        \]
        which have the same kernel given by the Picard group of a certain category related to $X$ and the action $\lambda$ (see Theorem~\ref{thm: omnibus theorem for intro}(4) for the precise statement);
        \item (Corollary~\ref{cor: br=br' implies involutive version}) the map
        \[
        \mathrm{Br}(X,\lambda) \to \mathrm{Br}'(X,\lambda)
        \]
        is an isomorphism whenever the map
        \(
        \mathrm{Br}(X) \to \mathrm{Br}'(X)
        \)
        is.
    \end{enumerate}
\end{theorem}

We construct the group $\mathrm{Br}(X,\lambda)$ as a subgroup of a larger group $\mathrm{Br}^\mathrm{p}(X,\lambda)$, as described in the next section. 
The elements of $\mathrm{Br}^\mathrm{p}(X,\lambda)$ may be regarded as Morita equivalence classes of certain generalized Azumaya algebras with involution. 
There is a homomorphism $ g $ defined on $ \mathrm{Br}^\mathrm{p}(X,\lambda, Y,p) $; the invariant $ g $ detects whether the underlying derived Brauer class of $ \alpha $ and/or the involution associated to $ \alpha $ are exotic (Construction \ref{const: map from poincare brauer to homotopy fixed points} \& Remark \ref{rmk:pnbr_to_H1_fxpt_interpretation}). 
We define $\mathrm{Br}'(X,\lambda)$ to be the kernel of $g$, and $\mathrm{Br}(X,\lambda)$ to be the subgroup of $\mathrm{Br}'(X,\lambda)$ whose underlying cohomological Brauer classes are representable by ordinary Azumaya algebras. 

Just as ordinary Brauer classes give rise to twisted forms of K-theory, classes in $\mathrm{Br}^\mathrm{p}(X,\lambda)$ give rise to twisted forms of $ \lambda $-hermitian K-theory and Grothendieck--Witt theory. 
Understanding these twists is crucial for applications and computations, as they arise frequently in the construction of invariants in enriched enumerative geometry \cite{MR4198841}. 
\begin{remark}
    One might wonder if Fröhlich--Wall's involutive Brauer group agrees with ours whenever both are defined (compare, for example, the bottom row of \cite[Theorem 4.2]{MR1803361} with the fiber sequence of Corollary~\ref{cor:Poincare_Brauer_to_Brauer_fiber}). 
    We define a comparison map from Fr\"ohlich--Wall's equivariant Brauer group to an invariant related to our Brauer group in Proposition~\ref{prop:Frohlich_Wall_eqvt_Brauer_comparison}, but do not pursue this question further. 
\end{remark}

\subsection{Deeper results} \label{subsection:deeperresults}
To motivate our approach, recall that Toën and Antieau--Gepner \cites{MR3190610,MR2957304} posed a variant of the Br = Br' question and proved that equality holds. 
This variation replaces $ \mathrm{Br}(X) $ with invertible $ \mathcal{O}_X $-linear categories satisfying certain compact generation hypotheses. 
Any class $ \alpha \in H^2_{\mathrm{\acute{e}t}}(X; \mathbb{G}_m) $ gives rise to such a category by taking a derived version of the categories of twisted sheaves considered by C\u ald\u araru and Lieblich \cites{MR1887894}{MR2579390}. 
Toën shows that any such category $ \mathcal{C} $ arises as modules over a sheaf $ \mathcal{A} $ of \emph{derived Azumaya algebras} when $ X $ is qcqs by constructing a suitable generator $ G \in \mathcal{C} $ and taking $ \mathcal{A} $ to be the algebra of derived endomorphisms of $ G $ \cites[Theorem 0.2]{MR2957304}[Theorem 6.1]{MR3190610}. 
On the other hand, such invertible $ \mathcal{O}_X $-linear categories can be realized as the connected components $ \pi_0 \mathfrak{br}(X) $ of a spectrum $ \mathfrak{br}(X) $. 
Unlike $ \pi_0 \mathfrak{br} $, the assignment $ X \mapsto \mathfrak{br}(X) $ satisfies étale (in fact fppf) descent. 
A computation of the higher homotopy sheaves of $ \mathfrak{br} $ and geometricity of a certain stack of Morita equivalences imply that $ \pi_0 \mathfrak{br} (X) \simeq \mathrm{H}^1_{\mathrm{\acute{e}t}}(X;\ZZ) \times \mathrm{H}^2_{\mathrm{\acute{e}t}}(X; \mathbb{G}_m) $. 

This work extends To\"en and Antieau--Gepner's program to the involutive setting. 
For simplicity, let $ \lambda = \id $. 
A key technique used to study Azumaya algebras with involution is to replace the involution $\sigma:\mathcal{A}\to \mathcal{A}^{\op}$ with the associated trace bilinear form $\mathcal{A}\otimes \mathcal{A}\to \mathcal{O}_X$, $a\otimes b\mapsto \mathrm{Tr}(a\sigma(b))$. 
We work with derived categories equipped with the data of `allowable quadratic forms.' 
Such a theory of form data on higher categories has recently been developed by Lurie and Calmès--Dotto--Harpaz--Hebestreit--Land--Moi--Nardin--Nikolaus--Steimle under the name of \emph{Poincaré $ \infty $-categories}, see \cites{CDHHLMNNSI,Lurie_Ltheory_notes}. 

We define a category of algebras which we refer to as \emph{Poincar\'e rings} $\mathrm{CAlg}^\mathrm{p}$ (Definition~\ref{definition:poincare_ring_spectrum}). 
A Poincaré ring $ R $ has an underlying ring with involution $ R^e $; roughly, a lift of $ R^e $ to $ \CAlgp $ is the minimal algebraic structure needed to both specify $\lambda$-symmetric bilinear forms on objects of $ \Mod_{R^e}^\omega $ and be able to tensor such bilinear forms together (Theorem \ref{thm:calgp_to_poincare_cat}). 
In this more general setting, we prove the following results. 	
\begin{theorem}~\label{thm: omnibus theorem for intro}
    There are functors
    \[
    \mathfrak{br}^\mathrm{p}, \mathfrak{pic}^\mathrm{p}, \gmq: \mathrm{CAlg}^\mathrm{p} \to \mathrm{Sp}_{\geq 0}
    \]
    referred to as the \emph{Poincaré Brauer} spectrum, \emph{Poincaré Picard} spectrum, and Poincar\'e units such that
    \begin{enumerate}[label=(\arabic*)]
        \item(Theorem~\ref{theorem:loops_Poincare_pic_is_Gm_Qoppa}, Proposition~\ref{prop:loops_pnbr_is_pnpic}) there are natural equivalences
        \[
        \mathfrak{pic}^\mathrm{p} \simeq \Omega \mathfrak{br}^\mathrm{p} \quad \text{and} \quad \gmq \simeq \Omega \mathfrak{pic}^\mathrm{p};
        \]
        \item(Proposition~\ref{prop:gmq_underlying_calg}) $\gmq$ is corepresented by a Poincar\'e ring with underlying Borel equivariant ring spectrum
        \(
        \mathbb{S}\{x^{\pm 1}\}
        \)
        where $\mathbb{S}\{-\}$ denotes the free $\mathbb{E}_\infty$ algebra on the input and the $ C_2 $-action is given by $\lambda(x) = x^{-1}$; 
        \item \label{omnibus_thmitem:fiber_of_pnbr_to_br}(Corollary~\ref{cor:Poincare_Brauer_to_Brauer_fiber}) there is a natural forgetful map
        \[
        \mathfrak{br}^\mathrm{p} \to \mathfrak{br}
        \]
        whose fiber on $R \in \mathrm{CAlg}^\mathrm{p}$ is naturally identified with
        \(
        \mathfrak{pic}\left(\mathrm{Mod}_{R^L}\left(\mathrm{Sp}^{C_2}\right)\right);
        \) here $ \mathfrak{br} $ is the derived Brauer space of \cite{MR3190610} and $ (-)^L $ denotes the forgetful functor $ \mathrm{CAlg}^\mathrm{p} \xrightarrow{} \mathrm{CAlg}(\mathrm{Sp}^{C_2}) $ of Remark \ref{rmk:Poincare_ring_has_underlying_C2_spectrum_alg}.\footnote{The reader does not need to be familiar with the objects $ \Spectra^{C_2} $ or $ R^L $ appearing in Theorem \ref{thm: omnibus theorem for intro}\ref{omnibus_thmitem:fiber_of_pnbr_to_br} to read the rest of the introduction; the main point is that the fiber admits a tractable description.} 
        \item(Corollary~\ref{cor:pnbr_as_etale_sheaf_affine_spectral}) All three of these invariants are sheaves for an analogue of the {\'e}tale topology of Definition~\ref{defn:C2_etale_topology}.
    \end{enumerate}
\end{theorem} 
The Poincar\'e Brauer group $ \mathrm{Br}^\mathrm{p} $ is defined as $\pi_0(\mathfrak{br}^\mathrm{p})$, and similarly for the Poincar\'e Picard group.

An ordinary ring with involution $ (R, \lambda)$ can be regarded as an object of $\mathrm{CAlg}^{\mathrm{p}}$ in a canonical way (Example \ref{ex:fixpt_Mackey_functor}). 
Given a scheme with involution $ (X,\lambda) $ admitting a good quotient $ p : X \to Y $ (Definition \ref{defn:Category of good quotients}), we show that $ Y $ admits a structure sheaf $ \underline{\mathcal{O}} $ valued in $ \CAlgp $ (Construction \ref{cons:structure_sheaf_of_Green_functors} \& Lemma \ref{lemma:identify_structure_sheaf_of_Green_func}). 
We will use the notation $ (X,\lambda,Y,p) $ throughout the introduction and take all schemes to be qcqs. 
Compatibility of the étale topology on $ \CAlgp $ and the ordinary étale topology on schemes and the descent results of Theorem \ref{thm: omnibus theorem for intro} allow us to extend the Poincaré Brauer spectrum and Poincaré Picard spectrum to schemes with involution. 

The Poincaré Picard group of $ (X,\lambda,Y,p) $ is an involutive enhancement of the Picard group of the \emph{derived category}\footnote{Here we refer to the derived category of quasicoherent sheaves.} of $ X $. 
We give a classical interpretation for the Poincaré Picard group; to motivate our result, consider its non-involutive counterpart. 
The functor $ \mathcal{L} \mapsto \mathcal{L}[0] $ exhibits the classical Picard group of a scheme $X$ as a subgroup of the Picard group of $\mathcal{D}(X)$. 
The inclusion is always proper, since $ \mathcal{O}_X[n] $ defines a $\otimes$-invertible object in $\mathcal{D}(X)$ for all $n\in \mathbb{Z}$ which is not equivalent to a classical line bundle except when $n=0$.
Nevertheless, a theorem of Fausk says that these cohomological shifts are essentially the only reason that the inclusion fails to be an isomorphism \cite{MR1966659}. 
\begin{theorem}
    [{Theorem~\ref{thm: Fausk for pnpic and schemes with good quotient}}] Let $(X,\lambda, Y, p)$ be a scheme with involution and good quotient. 
    Then there is a split short exact sequence of abelian groups \[0\to \mathrm{Pic}^\mathrm{h}(X, \lambda)\to \pi_0(\pnpic(X,\lambda, Y,p))\to C_{C_2}(X,\ZZ^-)\to 0\]
    where $\mathrm{Pic}^{\mathrm{h}}$ is the \emph{hermitian Picard group of $ X $}, i.e. the group of isomorphism classes of pairs $(\mathcal{L}, q)$ where $\mathcal{L} \in \mathrm{Pic}(X)$ and
    \(
    q : \mathcal{L} \otimes_{\mathcal{O}_X} \lambda^* \mathcal{L} \to \mathcal{O}_X 
    \)
    is a nondegenerate $ \lambda $-hermitian pairing, and $C_{C_2}(-,\mathbb{Z}^-)$ denotes continuous $C_2$-equivariant maps where $\mathbb{Z}^-$ is the discrete topological space $\mathbb{Z}$ with $C_2$-action given by $n \mapsto -n$. 
    This splitting is compatible with the splitting in \cite[Proposition 4.4]{MR1966659}. 
\end{theorem}
In fact, we show that the involution $ \lambda $ and the quotient $ p $ allow us to regard $ \mathcal{D}_{\mathrm{perf}}(X) $ as a symmetric monoidal Poincaré $ \infty $-category in a canonical way. 
It has an underlying stable $ \infty $-category with perfect duality, i.e. the functor $ \lambda_*(-^\vee) $ exhibits $ \mathcal{D}_{\mathrm{perf}}(X) $ as a $ C_2 $-(homotopy) fixed point of the involution sending a small stable $ \infty $-category to its opposite (Observation \ref{obs:symmetric_structure_module_cat}). 
A Poincaré structure is more data than a perfect duality: The pair $ (X,\lambda) $ does not uniquely determine a Poincaré structure on $ \mathcal{D}_{\mathrm{perf}}(X) $. 
The required additional data reflects a geometric subtlety: a quotient of $ X $ by $ \lambda $ involves a choice of $ p \colon X \to Y $ -- replacing $ Y $ by the stacky quotient $ [X/C_2] $ in general gives a different Poincaré structure on $ \mathcal{D}_{\mathrm{perf}}(X) $ (cf. Proposition \ref{prop:symmetric_calgp_stacky_interpretation}). 
We regard $ \mathcal{D}_{\mathrm{perf}}(X) $ equipped with its Poincaré structure as an involutive noncommutative invariant of $ (X,\lambda,Y,p) $. 
By definition, classes in $ \pnbr(X,\lambda,Y,p) $ are represented by invertible Poincaré $ \infty $-categories over $ \mathcal{D}_{\mathrm{perf}}(X) $, considered with its additional structure. 
Any such class in $ \pnbr(X,\lambda,Y,p) $ has an underlying twisted derived invertible $ \mathcal{O}_X $-linear category $ \mathcal{C} $ with $ \lambda_* $-linear duality $ \mathcal{C} \simeq \mathcal{C}^\op $ (Observation \ref{obs:symmetric_relative_poincare_cat}). 

We bootstrap the relevant results of Toën and Antieau--Gepner and develop relative enhancements of descriptions of endomorphism algebras contained in \cite[\S3.1]{CDHHLMNNSI} to prove the following. 
\begin{theorem}
    [{Corollary~\ref{cor: Brauer classes represented by Azumaya algebras with genuine involution}}] \label{thm:pnbr_classes_as_deraz_with_involution}
    Let $ (X,\lambda, Y,p) $ be a scheme with involution and good quotient and assume that $ X $ is qcqs. 
    Then every class in $\pi_0(\mathfrak{br}^\mathrm{p}(X,\lambda,Y,p))$ can be represented by a derived Azumaya algebra with genuine involution. 
\end{theorem}
\begin{remark} \label{rmk:E_sigma_alg_interpretation}
    An algebra with $ \lambda $-involution is a homotopy-coherent version of an associative algebra with involution; this is an ``involutive'' analogue of the relationship between $ \mathbb{A}_\infty $-algebras and associative algebras. 
    A \emph{derived Azumaya algebra with genuine involution} is an $ \EE_\sigma $-algebra whose underlying $ \mathbb{A}_\infty $-algebra is Azumaya in the sense of \cites{MR2927172}[\S2.2]{MR2957304}{MR3190610} and some additional data; see Definition \ref{defn:alt definition of Azumaya algebra with involution}. 
\end{remark}
Recall that the isomorphism from Morita equivalence classes of generalized Azumaya algebras to the derived Brauer group is implemented by $ \left[\mathcal{A}\right] \mapsto \Mod_{\mathcal{A}}\left(\mathcal{D}_{\mathrm{perf}}(X)\right) $; the inverse takes $ \mathcal{C} $, chooses a generator $ G $, and associates to $ \mathcal{C} $ to the algebra of derived endomorphisms $ \left[\mathrm{End}_{\mathcal{C}}(G)\right] $. 
A Poincaré structure on a category $ \mathcal{C} $ gives rise to so-called Poincaré objects; for each $ d \in \ZZ $, there is a Poincaré structure on $ \mathcal{D}_{\mathrm{perf}}(\ZZ) $ whose Poincaré objects consist of $ M \in \mathcal{D}_{\mathrm{perf}}(\ZZ) $ equipped with a $ d $-shifted Poincaré duality pairing. 
A natural way in which such Poincaré objects in $ \mathcal{D}_{\mathrm{perf}}(X) $ arise is Serre--Grothendieck duality (see Observation \ref{obs:anti_involutions_shifted_pairing_Skolem_Noether} and Example \ref{ex:Grothendieck_Serre_duality_objects}). 
Given a class $ \beta $ in $\pi_0(\mathfrak{br}^\mathrm{p}(X,\lambda,Y,p))$, taking endomorphisms of a suitably chosen Poincaré object produces a derived Azumaya algebra with involution representing $ \beta $. 

Under certain hypotheses, a classical Azumaya algebra with involution determines a derived Azumaya algebra with genuine involution, which induces a homomorphism from Parimala--Srinivas's involutive Brauer group to the Poincaré Brauer group, see Proposition \ref{prop:from_involutive_Brauer_to_Poincare_Brauer}. 
Toën and Antieau--Gepner show that $ \pi_0 \mathfrak{br} $ vanishes étale-locally just like its non-derived counterpart. 
However, in our situation the comparison map fails to be an isomorphism (even étale-locally) for a very basic reason. 
We illustrate this reason with an example. 
\begin{example} [{Examples~\ref{ex:br and br' for alg closed field} \& \ref{ex:pnbr_closed_point_ramified}, Remark \ref{rmk:non_classical_PnBr_classes}}]
     Let $k$ be an algebraically closed field of characteristic $\neq 2$, equipped with the trivial involution and thought of as a Poincar\'e ring via Example~\ref{ex:fixpt_Mackey_functor}. Then 
    \[
    \mathrm{Br}^\mathrm{p}(k) := \pi_0 \pnbr(k) \cong \mathbb{Z}/4\mathbb{Z} 
    \] 
    and the comparison map from Parimala--Srinivas's involutive Brauer group is given by the inclusion $ \ZZ/2 \ZZ \xrightarrow{\cdot 2} \ZZ/4\ZZ $. 
    A generator for $ \mathrm{Br}^\mathrm{p}(k) $ is not represented by an ordinary Azumaya algebra with involution; instead it is represented by a \emph{generalized} Azumaya algebra $ A \simeq \mathrm{End}_k(P) $, where $ P $ generates $ \mathcal{D}(k) $ and is self-dual up to a shift $ q \colon P \simeq P^\vee[1] $; the involution on $ A $ is induced by duality and $ q $. 
    In particular, this example shows why the map $ \mathrm{Br}(k,\mathrm{id}) \to \mathrm{Br}^{\mathrm{p}} (k) $ fails to be an isomorphism. 
\end{example}
\begin{example}
    Suppose instead that $k$ is an algebraically closed field of characteristic $2$, with the trivial involution. Then
    \[
    \mathrm{Br}^\mathrm{p}(k,\mathrm{id}) \cong \mathbb{Z}, 
    \] 
    and the comparison map from involutive Brauer group we define in Definition~\ref{defn:cohomological_involutive_Brauer_group} is given by the inclusion $ 2 \ZZ \subseteq \ZZ $. 
\end{example}   
While we originally set out to provide an enhancement of the involutive Brauer group of Parimala-Srinivas, the previous example shows that the Poincaré Brauer group detects altogether new phenomena. 

In order to prove Theorem \ref{thm:intro_inv_br_and_inv_br_prime}, the preceding example implies we cannot simply set $ \mathrm{Br}'= \pi_0 \pnbr $. 
To\"en's and Antieau--Gepner's work shows that $\pi_0(\mathfrak{br}(X))\cong \mathrm{H}^2_{\acute{e}t}(X;\mathbb{G}_m)\oplus \mathrm{H}^1_{\acute{e}t}(X;\mathbb{Z})$ (see \cite[Corollary 7.14]{MR3190610}). 
The classes in $\mathrm{H}^1_{\acute{e}t}(X;\mathbb{Z})$ have no possibility of coming from an ordinary Azumaya algebra, and so if one wanted to build $\mathrm{Br}'(-)$ from the Brauer space, one should remove these `exotic' elements from consideration.
By analogy, we define $\mathrm{Br}'(R,\lambda)$ to be the subgroup of
\(
\mathrm{Br}^\mathrm{p}(\underline{R}) := \pi_0(\mathfrak{br}^\mathrm{p}(\underline{R}))
\)
which are in the kernel of a canonical surjective map
\(
\mathrm{Br}^\mathrm{p}(\underline{R}) \to \mathrm{H}^1(\mathrm{R}\Gamma_{\Acute{e}t}(R; \mathbb{Z})^{\mathrm{h}C_2})
\)
where the $C_2$ action is given by $(-1)\cdot\lambda$. All of the classes in $\mathrm{Br}'(R,\lambda)$ are then represented by generalized Azumaya algebras with genuine involution, and $\mathrm{Br}(R,\lambda)$ is the subgroup of elements of $\mathrm{Br}'(X,\lambda)$ whose underlying Brauer class is representable by an ordinary Azumaya algebra. 
These constructions patch together to give definitions for schemes with involutions in $\mathrm{Sch}^{C_2-\mathrm{aff},\mathrm{qcqs}}$.

We do not introduce the category of Poincaré rings for generality's sake alone -- we expect that the invariants constructed in Theorem \ref{thm: omnibus theorem for intro} will be useful in the study of the arithmetic of $ \EE_\infty $-ring spectra and of categories of twisted spectra \cite{MR4743054}. 
One application of the existence of the theories we have constructed is an extension of the work of Albert, Saltman, and Parimala--Srinivas in \cites[\S9 Theorem 19]{MR123587}[Theorem 3.1(a)]{MR495234}[Theorem 1]{Parimala_Srinivas}. 
These results concern the question of using cohomological invariants to classify when an Azumaya algebra $A$ is Morita equivalent to an Azumaya algebra $B$ which admits a $\lambda$-involution. 
\begin{theorem}[Theorem~\ref{thm: Saltman theorem in text}]
    Let $R$ be a Poincar{\'e} ring such that $R^{\phi C_2}\simeq 0$. 
    Let $A$ be a generalized Azumaya algebra over $R^e$. Then
    \begin{enumerate}[label=(\arabic*)]
        \item if the $C_2$-action on $R$ is trivial, then $A$ is Morita equivalent to one admitting an Azumaya algebra with genuine involution structure if and only if $2[A]=0$ in $\mathrm{Br}(R^e)$, and 
        \item if the $C_2$-action on $R$ is faithful Galois, then $A$ is Morita equivalent to one admitting an Azumaya with genuine $\lambda$-involution enhancement if and only if $[(A\otimes_R \lambda^*A)^{\mathrm{h}C_2}]=0$ in $\mathrm{Br}(R^{\mathrm{h}C_2})$.
    \end{enumerate}
\end{theorem}
The following computations are consequences of fiber sequences identified in this work (e.g. Theorem \ref{thm: omnibus theorem for intro}\ref{omnibus_thmitem:fiber_of_pnbr_to_br}, Theorem \ref{thm: extended fiber sequence of pnbr to br}).    
\begin{examples}   
\begin{enumerate}[label=(\arabic*)]        
    \item (Example~\ref{ex: Atiyah real K-theory}) Let $\mathrm{K}\mathbb{R}$ denote Atiyah’s real $K$-theory spectrum, which classifies complex vector bundles together with their almost complex structures. This naturally lifts to a Poincar\'e ring spectrum with
    \[
    \mathrm{Pic}^\mathrm{p}(\mathrm{K}\mathbb{R}) \cong \mathbb{Z}/2\mathbb{Z} \times \mathbb{Z}/2\mathbb{Z},
    \]
    and there exists an injective map 
    \(
    \mathbb{Z}/8\mathbb{Z} \hookrightarrow \mathrm{Br}^\mathrm{p}(\mathrm{K}\mathbb{R})
    \)
    which is an isomorphism if $\mathrm{Br}(\mathrm{KU}) \cong 0$.
    \item (Corollary~\ref{cor:pn_pic_of_the_unit}, Example~\ref{ex: pnbr of sphere}) Let $\mathbb{S}$ be the sphere spectrum with trivial $C_2$-action, and let $\mathbb{S}^u$ denote the initial Poincar\'e ring spectrum lifting $\mathbb{S}$. Then
    \[
    \mathrm{Pic}^\mathrm{p}(\mathbb{S}^u) \cong \mathbb{Z}/2\mathbb{Z} , \quad \text{and} \quad \mathrm{Br}^\mathrm{p}(\mathbb{S}^u) \cong \mathbb{Z}.
    \]
\end{enumerate}
\end{examples}

\subsection{Outline}
\S\ref{section:poincare_structures_on_compact_modules} is devoted to preliminary material on Poincaré $ \infty $-categories. 
In \S\ref{section:poincare_ring_spectra}, we introduce the $ \infty $-category of Poincaré rings.
\S\ref{subsection:Poincare_rings} collects examples of Poincaré rings and discusses the existence of a functor from Poincaré rings to structured Poincaré $ \infty $-categories, whose construction is postponed to \S\ref{appendix:calgp_to_catp}. 
In \S\ref{subsection:Poincare_structures_schemes_involution}, we show that a qcqs scheme with involution and a good quotient acquires a symmetric monoidal Poincar\'e structure on its perfect derived category. 
In \S\ref{section:relative_poincare_cats}, we discuss formal properties of relative Poincaré $ \infty $-categories and show that they satisfy a suitable notion of descent in \S\ref{subsection: etale descent results}. 
In \S\ref{section:the_poincare_picard_group}, we consider the moduli of invertible Poincaré objects in a symmetric monoidal Poincaré $ \infty $-category. 
In \S\ref{subsection:calgp_units}, we introduce units for Poincaré rings and show that the Poincaré Picard spectrum is a delooping of the units. 
\S\ref{subsection:the_discrete_case:nondegenerate_hermitian_forms} is devoted to a discussion of hermitian line bundles.  
In \S\ref{subsection:poincare_picard_group_inv_scheme}, we characterize the Poincaré Picard group of a scheme with involution $ \lambda $ and good quotient in terms of ordinary line bundles with nondegenerate $ \lambda $-linear forms and their (cohomological) shifts.  
In \S\ref{section:the_poincare_brauer_group}, we introduce the Poincaré Brauer space via invertible $ R $-linear Poincaré $ \infty $-categories. 
In \S\ref{subsection:invertible_relative_Poincare_cats}, we define the Poincaré Brauer space, show that it deloops the Poincaré Picard space, and describe a relationship between it and the non-involutive Brauer space of Toën and Antieau--Gepner.   
In \S\ref{subsection:generalized_Azumaya_alg_with_involution}, we introduce generalized Azumaya algebras with involution and show that they give rise to Poincaré Brauer classes. 
In \S\ref{section:inv_Brauer_comparison}, we show that our work refines that of Parimala-Srinivas, prove Theorem \ref{thm:intro_inv_br_and_inv_br_prime}, and exhibit a derived enhancement of the exact sequence Parimala-Srinivas develop in the case that $\lambda$ is the deck transformation of a $C_2$-Galois cover in \cite{Parimala_Srinivas}.

The reader may wish to refer to the beginning of each section (and certain subsections) for a synopsis and motivation of the contents therein. 
For the reader who is primarily interested in schemes with involutions and what this paper has to say about them, the most relevant portions of this work are \S\ref{subsection:Poincare_structures_schemes_involution}, \S\ref{subsection:poincare_picard_group_inv_scheme}, \S\ref{subsection:generalized_Azumaya_alg_with_involution}, and \S\ref{section:inv_Brauer_comparison}. 
We caution the reader that notation and terminology in the introduction may differ from in the body of the paper; any changes are recorded in \S\ref{subsection:conventions}. 

\subsection{Related work}
A key technical component of this work is the construction of suitably structured, functorial Poincaré structures on the perfect derived categories of schemes.  
Similar constructions have been developed by Calmès--Harpaz--Nardin in \cite[\S3.3 \& \S3.4]{CHN2024}, which specifies the data needed to define a functor from schemes to symmetric monoidal Poincaré $ \infty $-categories, and Hoyois--Land in \cite[\S2]{Hoyois_Land_GW_derived}, which defines various assignments from spectral algebraic spaces to Poincaré $ \infty $-categories. 
In the language of this paper, the aforementioned works endow a commutative ring (resp. $ \EE_\infty $-ring) $ R $ with trivial $ C_2 $-action, then construct Poincaré structures on $ \Mod^\omega_R $ whose duality functor is $ R $-linear. 
By considering rings with possibly nontrivial involution, we allow our duality functors to be linear with respect to a given involution on $ R $. 
In the language of Poincaré objects, our construction allows us to speak of perfect complexes with nondegenerate \emph{hermitian} forms (or variants such as skew-hermitian) in addition to symmetric forms and their variants, while the former only gives rise to (anti)-symmetric forms (cf. \cite[Remark 2.2]{Hoyois_Land_GW_derived}).  
On the other hand, we do not consider forms valued in line bundles which are not trivial or explore how hermitian or Poincaré structures on $ \Mod_R^\omega $ might vary according to a given t-structure on $ \Mod_R $, so in these respects our construction is less general than those of Calmès--Harpaz--Nardin and Hoyois--Land. 
Marcus Nicolas has ongoing work constructing Poincaré structures on perfect derived categories with not necessarily trivial involution; we are under the impression that our work differs both in approach and end goal and as such our works should be regarded as complementary. 

Vezzosi has studied quadratic structures, including shifted quadratic forms, in derived algebraic geometry \cite{MR3539372}. 
While we generally confine our analysis to schemes with involution in this paper, we expect that a suitable extension of our theory would allow a direct comparison between notions in Vezzosi's work and in ours. 
In particular, we wonder (but do not take up the question here) whether our work on derived Azumaya algebras with involution could be relevant to questions posed about a derived version of the Brauer--Wall group \cite[paragraph entitled `Leftovers']{MR3539372}. 

\subsection{Acknowledgements} 
	The authors wish to thank the Institute for Advanced Study and the organizers of the 2024 Park City Mathematics Institute on motivic homotopy theory. The authors would also like to thank the organizers of the BIRS CDM 2024 program on Algebraic K-theory and Brauer groups, where NR and LY made progress on this project. 
	The authors would like to thank Columbia University for its hospitality. 
    During the writing of this paper, NR was supported by the Simons Collaboration on Perfection in Algebra, Geometry, and Topology.
	
    We are immensely grateful to Ben Antieau, who originally suggested to VB and NR that looking into the Brauer group in the setting of Poincar\'e $\infty$-categories might be interesting and connected to \cite{azumaya_involution}. 
    LY would like to thank James Hotchkiss and Emanuele Dotto for helpful discussions related to this work. 
    The authors would like to thank Ben Antieau, Uriya First, Ben Williams, and Elden Elmanto for valuable feedback on an early draft. 
	
	\subsection{Conventions}
	\label{subsection:conventions} 	
	As indicated in the introduction, key objects of study for us will be derived categories of algebras admitting certain extra structure, and remembering only the tensor triangulated structure of these categories will be insufficient for our purposes. We therefore need to work in a higher categorical setup, of which we have three main choices: differential graded (DG) categories, stable model categories, and stable $\infty$-categories. 
	
	We chose to work in stable $\infty$-categories, and will frequently cite Lurie's \cite{HTT} and \cite{LurHA} for foundational results in this framework. This is partly because the extra structure we will be equipping our derived categories with is that of a Poincar\'e $ \infty $-category, whose basic properties were studied in \cite{CDHHLMNNSI} in the context of stable $\infty$-categories. 
    Our choice facilitates comparisons to the theory introduced in \cite{MR3190610}.
	
	For the reader who is more comfortable with DG categories, we note here that DG categories embed fully faithfully into stable $\infty$-categories. While many of the details would be different in the DG categorical setup, the results presented in the introduction would not be. We would recommend such a reader on a first pass to read Section~\ref{section:poincare_structures_on_compact_modules} where we review some of the key results of \cite{CDHHLMNNSI}, then skip to  Section~\ref{section:the_poincare_picard_group} and black-box the results of the previous Sections. 
	
	Going forward we will use the term Azumaya algebra over a ring spectrum $R$ to be a \textbf{derived} Azumaya algebra over $R$ by default. 
    If we wish to talk about Azumaya algebras over a discrete ring which are not derived, we will emphasize this by calling them \textbf{ordinary} Azumaya algebras. We note that this is the opposite convention that was used in the introduction. 
	
	Let $A$ be an associative algebra in a symmetric monoidal $\infty$-category $\mathcal{C}$. We will use $\mathrm{Mod}_A(\mathcal{C})$ to denote the category of \textbf{left} $A$ modules, and a module over $A$ is by default assumed to be a left module unless stated otherwise. 

    Our indexing conventions are \textbf{homological}. To be consistent with our notation for rings and schemes, we will use $\mathrm{Mod}_{\mathcal{O}_X}^\omega$ (resp. $ \mathrm{Mod}_{\mathcal{O}_X} $) to denote the stable $\infty$-category of perfect $\mathcal{O}_X$-modules (resp. all $ \mathcal{O}_X $-modules); in particular, we switch from $ \mathcal{D}(k) $ in the introduction to $ \Mod_k $ in the body of the paper. 
	
	Throughout this article we will need to work with $ \infty $-categories whose objects themselves consist of (stable) $ \infty $-categories. 
    To handle size issues, we will assume the existence of strongly inaccessible cardinals, and fix once and for all three strongly inaccessible cardinals $\kappa_1<\kappa_2<\kappa_3$ which we call small, large, and huge, respectively. 
    Let \[\CAT[\mathrm{ex}]\subseteq \CAT\] denote the huge $\infty$-categories of large \textbf{stable} $\infty$-categories with exact and all functors between them, respectively. 

    The Poincaré categories we will consider have small underlying $ \infty $-categories, i.e. they will belong to $\Catex_{, \mathrm{idem}}$. On the other hand, we will frequently make use of results from \cites{MR2957304,MR3190610}, which concern compactly-generated stable $ \infty $-categories in $\mathrm{Pr}^{\mathrm{L}}$. 
    These two contexts are equivalent via taking compact objects and taking the $\mathrm{Ind}^\omega$ construction (cf. \cite[\S5.3.5]{HTT}), and we will suppress this equivalence in the sequel.
	
	We include a table of some notation below which we use frequently throughout the document: 
	\begin{longtable}{lll}
        $\Spaces$ & $\infty$-category of spaces\\
        $\Spectra$ & $\infty$-category of spectra\\
        $\Spectra^{C_2}$ & $\infty$-category of genuine $C_2$-spectra \\
		$\Brp$ & Poincaré Brauer spectrum\\ 
		$\pnpic$ & Poincar{\'e} Picard spectrum\\
		$\mathrm{Br}^\mathrm{p}$ & Poincar{\'e} Brauer group $\pi_0(\pnbr)$\\
		$\mathrm{Pic}^\mathrm{p}$ & Poincar{\'e} Picard group $\pi_0(\pnpic)$\\
		$\mathfrak{br}$ & Brauer spectrum\\
		$\mathfrak{pic}$ & Picard spectrum\\
		$\mathrm{Br}$ & Brauer group\\
		$\mathrm{Br}'$ & cohomological Brauer group $\mathrm{H}^2(-;\mathbb{G}_m)_{\mathrm{tors}}$\\
		$\CAlg$ & $\infty$-category of $\Einfty$-ring spectra\\
		$\CAlg(\Spaces)$ & $\infty$-category of $\Einfty$-spaces\\
		$\CAlg^{\gp}(\Spaces)$ & $\infty$-category of grouplike $\Einfty$-spaces\\
		$\CAlgp$ & $\infty$-category of Poincaré ring spectra\\
		$\Catex$ & $\infty$-category of small stable $\infty$-categories and exact functors\\
		$\Catp$ & $\infty$-category of Poincaré $\infty$-categories\\
		$\Catpidem$ & $\infty$-category of idempotent complete Poincaré $\infty$-categories 
	\end{longtable}
	
\section{Poincaré structures on compact modules}
\label{section:poincare_structures_on_compact_modules}
This section recalls notions and results about Poincaré $\infty$-categories which we require in the sections to follow. 
We refer the reader to \cite{CDHHLMNNSI,CDHHLMNNSII} for details and proofs.
    
\begin{notation}\label{notation:quadratic_functor}
    Let $\mathcal{C}$ be a stable $\infty$-category. We call a reduced functor $\Qoppa:\mathcal{C}^\op\to \Spectra$ \emph{quadratic} if it is $2$-excisive. We will write $\mathrm{Fun}^\mathrm{q}(\mathcal{C})$ for the full subcategory of $\mathrm{Fun}(\mathcal{C}^{\op},\mathrm{Sp})$ spanned by quadratic functors. 
\end{notation}

\begin{definition}[{\cite[Definition 1.2.1]{CDHHLMNNSI}}]
    A hermitian $\infty$-category is a pair $(\mathcal{C},\Qoppa)$ where $\mathcal{C}$ is a stable $\infty$-category and $\Qoppa:\mathcal{C}^\op\to \mathrm{Sp}$ is a quadratic functor. A map of hermitian $\infty$-categories $(f,\eta):(\mathcal{C},\Qoppa)\to (\mathcal{D},\Psi)$ consists of an exact functor $f:\mathcal{C}\to \mathcal{D}$ and a natural transformation $\eta:\Qoppa\implies \Psi \circ f^\op$.
\end{definition}

The collection of hermitian $\infty$-categories and maps between them can be assembled into a large $\infty$-category $\Cath$. Let $\mathcal{C}$ be a stable $\infty$-category and let $\mathrm{Fun}^\mathrm{q}(\mathcal{C})\subset \mathrm{Fun}(\mathcal{C}^{\op},\mathrm{Sp})$ be the full subcategory spanned by quadratic functors. The assignment $\mathcal{C} \mapsto \mathrm{Fun}^\mathrm{q}(\mathcal{C})$ defines a functor $\left(\Catex\right)^\op \to \CAT[]$. Following \cite[Definition 1.2.1]{CDHHLMNNSI}, the $\infty$-category of hermitian $\infty$-categories is the Grothendieck construction \[\mathrm{Cat}^\mathrm{h}_\infty :=\int_{\Catex}\mathrm{Fun}^\mathrm{q}(-).\]

\begin{notation}~\label{notation: shorthand for quadratic and symmetric functors}
    We will use the notation $$\mathrm{Fun}^\mathrm{s}(\mathcal{C}):=\mathrm{Fun}^{\ex}(\mathcal{C}^\op\times \mathcal{C}^\op,\mathrm{Sp})^{\mathrm{h}C_2}$$ for the \emph{$\infty$-category of symmetric bilinear functors on $\mathcal{C}$}.
\end{notation}

Given a hermitian $\infty$-category $(\mathcal{C},\Qoppa)$, we have an associated bilinear functor $\mathrm{B}_\Qoppa:\mathcal{C}^\mathrm{op}\times \mathcal{C}^\mathrm{op}\rightarrow \Spectra$ given by the assignment $$\mathrm{B}_\Qoppa(x,y):= \Qoppa(x\oplus y)^{\text{red}},$$ where $\Qoppa((-)\oplus (-))^{\text{red}}$ denotes the universal bi-reduced functor associated to $\Qoppa((-)\oplus (-)):\mathcal{C}^\mathrm{op}\times \mathcal{C}^\mathrm{op}\rightarrow \Spectra$. The functor $\mathrm{B}_\Qoppa$ is called the \emph{cross effect of $\Qoppa$}.

\begin{definition}[{\cite[Definitions 1.2.2 \& 1.2.8]{CDHHLMNNSI}}]
    Let $(\mathcal{C},\Qoppa)$ be a hermitian $\infty$-category. We say that the cross effect $\mathrm{B}_\Qoppa$ is \emph{non-degenerate} if the functor \[\mathcal{C}^\op\to \mathrm{Fun}^\mathrm{ex}(\mathcal{C}^\op,\mathrm{Sp})\hspace{1cm} y\mapsto \mathrm{B}_\Qoppa(-,y)\] takes values in the essential image of the stable Yoneda embedding. In this case, this functor factors as a \emph{duality functor} $$\mathrm{D}_\Qoppa:\mathcal{C}^\op \to \mathcal{C}$$ and we say that the hermitian $\infty$-category $(\mathcal{C},\Qoppa)$ is a \emph{Poincar{\'e} $\infty$-category} if $\mathrm{B}_\Qoppa$ is non-degenerate and $\mathrm{D}_\Qoppa$ is an equivalence. A \emph{map of Poincar{\'e} $\infty$-categories} is a map of hermitian $\infty$-categories which further preserves the corresponding duality functors.
\end{definition}

We will denote by $\Catp \subset \Cath$ the subcategory of Poincar{\'e} $\infty$-categories. 
	
\begin{example}\label{example:symmetric_poincare_structure}
    Let $R$ be a commutative ring, and let $\Qoppa_R^\mathrm{s}:\mathrm{Mod}_{R}^{\omega,\op}\to \mathrm{Sp}$ denote the functor given by \[\Qoppa_R^\mathrm{s}(M)=\mathrm{hom}_{R\otimes R}(M\otimes M, R)^{\mathrm{h}C_2},\] where the action on $R$ is trivial and the action on $M\otimes M$ is via the swap isomorphism. The associated duality functor is given by $$\mathrm{D}_{\Qoppa_R^{\mathrm{s}}}(M)=\mathrm{hom}_{R}(M, R)=:M^\vee.$$ This is a Poincar\'e $\infty$-category.
\end{example}		
\begin{example}[{\cite[\S3.5]{CDHHLMNNSI}}]\label{ex:shifted_Poincare_structure} 
    Let $n$ be an integer and suppose given a Poincaré $ \infty $-category $ \left(\mathcal{C},\Qoppa\right) $. 
    Then $ \Sigma^{n} \circ \Qoppa $ defines a reduced quadratic functor which we denote by $ \Qoppa^{[n]} $, and $ \left(\mathcal{C},\Qoppa^{[n]}\right) $ is a Poincaré $ \infty $-category. 
    In particular, if we apply this to Example \ref{example:symmetric_poincare_structure}, we obtain a Poincar\'e $\infty$-category $(\Qoppa_R^\mathrm{s})^{[n]}:\mathrm{Mod}_{R}^\omega\to \mathrm{Sp}$ given by \[(\Qoppa_R^\mathrm{s})^{[n]}(M)=\mathrm{hom}_{R\otimes R}(M\otimes M, R[n])^{\mathrm{h}C_2},\] with duality functor $$\mathrm{D}_{(\Qoppa_R^{\mathrm{s}})^{[n]}}(M)=\mathrm{hom}_{R}(M, R[n])=M^\vee[n].$$ 
\end{example}
	
The collections of hermitian and Poincar\'e $\infty$-categories have several desirable properties. We collect some of the key properties below.

\begin{theorem}[{\cite[Thm. 5.2.7, Prop. 6.1.2, Prop. 6.1.4, Cor. 6.2.9]{CDHHLMNNSI}}]\label{thm:catp_cath_properties_summary}
    \begin{enumerate}
        \item The $\infty$-categories $\Catp$ and $\Cath$ have small limits and small colimits. 
        \item The forgetful functors $\Catp\to \Cath\to \mathrm{Cat}_\infty^{\mathrm{ex}}$ are conservative and preserve small limits and small colimits.
        \item There is a symmetric monoidal $\infty$-operad $(\Cath)^\otimes$ lifting $\Cath$.
        \item The forgetful functor $\Cath\to \mathrm{Cat}_\infty^{\mathrm{ex}}$ lifts to a symmetric monoidal map $(\Cath)^\otimes \to (\mathrm{Cat}_\infty^{\mathrm{ex}})^\otimes$.
        \item The symmetric monoidal structure on $\Cath$ restricts to a symmetric monoidal structure on $\Catp$.
        \item The symmetric monoidal structures on $\Catp$ and $\Cath$ are closed. 
    \end{enumerate}
\end{theorem}
\begin{recollection}\label{rec:symmetric_Poincare_cat}
    Recall that there is an $ \infty $-category $ \Cat^{\mathrm{ps}}_\infty $ whose objects consist of small stable $ \infty $-categories $ \mathcal{C} $ equipped with a perfect symmetric bilinear functor $ B \colon \mathcal{C}^\op \times \mathcal{C}^\op \to \Spectra $; these are referred to as \emph{perfect symmetric} or \emph{perfect symmetric bilinear} $ \infty $-categories \cite[\S7.2]{CDHHLMNNSI}. 
    Taking the adjoint $ \mathcal{C}^\op \xrightarrow{\sim} \mathcal{C} $ of $ B $, there is an equivalence $ \Cat^{\mathrm{ps}}_\infty \simeq \left(\Catex\right)^{hC_2} $ where we regard $ \Catex $ as equipped with the involution sending a stable $ \infty $-category to its opposite \cite[Corollary 7.2.16]{CDHHLMNNSI}. 
    Furthermore, $ \Cat^{\mathrm{ps}}_\infty $ has a symmetric monoidal structure, the assignment $ \left(\mathcal{C},\Qoppa\right) \mapsto \left(\mathcal{C},B_\Qoppa\right) $ induces a forgetful functor $ \Catp \to \Cat^{\mathrm{ps}}_\infty $, and the forgetful functor is symmetric monoidal \cite[\S3.1, esp. Remark 3.1.2 through Corollary 3.1.3]{CHN2024}. 
    In fact, the assignment $ (\mathcal{C},B) \mapsto \mathcal{C} $ induces a functor $ \Cat^{\mathrm{ps}}_\infty \to \Catex $ so that the composite $ \Catp \to \Cat^{\mathrm{ps}}_\infty \to \Catex $ agrees with the symmetric monoidal forgetful functor of Theorem \ref{thm:catp_cath_properties_summary}.
\end{recollection}
To study Poincar{\'e} $\infty$-categories we will make use of the following constructions: the $\infty$-category of hermitian objects, the space of forms and the space of Poincar{\'e} objects.
	
\begin{definition}[{\cite[Definitions 2.1.1, 2.1.3]{CDHHLMNNSI}}]
    Let $(\mathcal{C},\Qoppa)$ be a Poincar{\'e} $ \infty $-category. Then the \emph{$\infty$-category of hermitian objects} in $(\mathcal{C},\Qoppa)$ is given as the Grothendieck construction \[\mathrm{He}(\mathcal{C},\Qoppa)=\int_{\mathcal{C}}\Omega^\infty \Qoppa.\] The \emph{space of hermitian forms} is the groupoid core $\mathrm{Fm}(\mathcal{C},\Qoppa)=\mathrm{He}(\mathcal{C},\Qoppa)^\simeq$. The \emph{space of Poincar\'e objects} $\mathrm{Pn}(\mathcal{C},\Qoppa)$ is the full subgroupoid of $\mathrm{Fm}(\mathcal{C},\Qoppa)$ spanned by objects $(x,q)$ such that the induced map $q_\#:x\to \mathrm{D}_\Qoppa (x)$ is an equivalence. 
\end{definition}
The next example motivates some of the naming conventions thus far. 
\begin{example}[{\cite[Example 2.1.8]{CDHHLMNNSI}}]\label{ex:poincare_duality_objects}
    Consider the Poincar{\'e} $\infty$-category $\left(\mathrm{Mod}_{\mathbb{Z}}^\omega, (\Qoppa_\mathbb{Z}^{\mathrm{s}})^{[-n]}\right)$ from Example \ref{ex:shifted_Poincare_structure}. Then the Poincar{\'e} objects are exactly complexes $M_\bullet$ together with a perfect symmetric pairing $M_\bullet \otimes M_\bullet \to \mathbb{Z}[-n]$. 
    If $ M $ is a closed compact oriented $n$-manifold, then chains on $ M $ with the pairing induced by Poincar\'e duality lifts to a Poincaré object of $\left(\mathrm{Mod}_{\mathbb{Z}}^\omega, (\Qoppa_\mathbb{Z}^{\mathrm{s}})^{[-n]}\right)$. 
\end{example}
	
\begin{proposition}[{\cite[Corollary 5.2.8]{CDHHLMNNSI}}]~\label{prop: pn is lax monoidal}
    The functor $\mathrm{Pn}:\Catp\to \mathcal{S}$ is lax symmetric monoidal. 
\end{proposition}
\begin{remark}\label{rmk:he_lax_monoidal}
    There appears to be a typo in \cite[Corollary 5.2.8]{CDHHLMNNSI}: The functor $ \mathrm{He} \colon \Cath \to \Cat_\infty $ does not take values in spaces. The proof suggests that the statement is supposed to be about $\mathrm{Fm}:\Catp\to \mathcal{S}$ instead.
    One may show that $ \mathrm{He} \colon \Cath \to \Cat_\infty $ is lax symmetric monoidal using the enrichment of $ \Cath $ over $ \Cat_\infty $ \cite[Remark 6.2.21]{CDHHLMNNSI}, the symmetric monoidal structure on $ \Cath $, the symmetric monoidality of the enriched Yoneda embedding \cite[Proposition 8.4.3]{MR4567127}, and the fact that $ \mathrm{He} $ is corepresented by $ \left(\Spectra^{\mathrm{fin}},\Qoppa^u \right) $, by applying $ \mathrm{He} $ to the equivalence of hermitian $ \infty $-categories in \cite[Example 6.2.5]{CDHHLMNNSI}.  
\end{remark}

Let us write $ \Catpidem $ for the full subcategory of $ \Catp $ spanned by Poincar\'e $ \infty $-categories whose underlying $ \infty $-category is idempotent-complete. For any stable $ \infty $-category $ \mathcal{C} $, write $ \mathcal{C}^\natural $ for the idempotent-completion of $ \mathcal{C} $ and write $ i \colon \mathcal{C} \to \mathcal{C}^{\natural} $ for the canonical inclusion.

\begin{proposition}[{\cite[Proposition 1.3.4]{CDHHLMNNSII}}] \label{rec:catpidem}
    The inclusion $ \Catpidem \subseteq \Catp $ admits a left adjoint exhibiting $ \Catpidem $ as a reflective localization of $ \Catp $. The map $(i,\mathrm{id_\Qoppa}):(\mathcal{C},\Qoppa) \to (\mathcal{C}^\natural,i_! \Qoppa) $, where $i_! \Qoppa$ denotes the left Kan extension of $\Qoppa$ along $i^\mathrm{op}$, exhibits $ (\mathcal{C}^\natural,i_! \Qoppa) $ as a $\Catpidem$-reflection of $ (\mathcal{C},\Qoppa) $.
\end{proposition} 

\begin{notation}
    We will denote the localization functor of Proposition \ref{rec:catpidem} by $ (-)^\natural: \Catp\rightarrow \Catpidem$. We will call $ \left(\mathcal{C},\Qoppa \right)^\natural = \left(\mathcal{C}^\natural,i_! \Qoppa \right) $ the \emph{idempotent-completion of $ \left(\mathcal{C},\Qoppa \right) $}.
\end{notation}

\begin{proposition}\label{proposition:idempotent_completion_of_sym_mon_catp}
    The subcategory $ \Catpidem \subset \Catp$ inherits a symmetric monoidal structure from $ \Catp $ so that the idempotent-completion functor $ (-)^\natural:\Catp\rightarrow \Catpidem $ promotes canonically to a symmetric monoidal functor. 
    In particular, idempotent-completion induces a functor $ (-)^\natural :\EE_\infty\Mon\left(\Catp\right) \to \EE_\infty\Mon\left(\Catpidem\right) $.
\end{proposition}

\begin{proof}
    The idempotent-completion functor $(-)^\natural :\Catp \rightarrow \Catpidem$ is a localization with respect to Karoubi equivalences of Poincar\'e $\infty$-categories, i.e. maps of Poincar\'e $\infty$-categories $(f,\eta):(\mathcal{C},\Qoppa)\rightarrow (\mathcal{D},\Phi)$ such that $f$ is a Karoubi equivalence and $\eta$ is an equivalence; see \cite[Corollary 1.3.4]{CDHHLMNNSII}. By \cite[Remark 2.2.1.7]{LurHA} and \cite[Proposition 2.2.1.9]{LurHA}, it suffices to show that for any Poincaré $ \infty $-categories $ \left(\mathcal{C},\Qoppa\right), \left(\mathcal{D},\Phi\right) $, the map $ (i,\mathrm{id}_\Qoppa)\otimes \id_{\left(\mathcal{D},\Phi\right)} \colon \left(\mathcal{C},\Qoppa \right) \otimes \left(\mathcal{D},\Phi\right) \to \left(\mathcal{C}^\natural,i_! \Qoppa \right) \otimes \left(\mathcal{D},\Phi\right) $ is a Karoubi equivalence of Poincar\'e $\infty$-categories. This follows from \cite[Proposition 1.5.5]{CDHHLMNNSII}.
\end{proof}

We now recall the connection to genuine-equivariant homotopy theory. Let $\Qoppa: \mathcal{C}^\op \to \Spectra$ be a quadratic functor. The diagonal and collapse natural transformations $$\Delta:\mathrm{id}\Rightarrow \mathrm{id}\oplus \mathrm{id} \text{ and } \nabla: \mathrm{id}\oplus \mathrm{id}\Rightarrow \mathrm{id}$$ induce natural transformations $\mathrm{B}_\Qoppa\circ \Delta \Rightarrow \Qoppa \Rightarrow \mathrm{B}_\Qoppa\circ \Delta$. The swap action on components induces a $C_2$-action on $\mathrm{B}_\Qoppa\circ \Delta$ with respect to which these natural transformations naturally refine to $C_2$-equivariant maps and induce natural transformations $(\mathrm{B}_\Qoppa\circ \Delta)_{\mathrm{h} C_2} \Rightarrow \Qoppa \Rightarrow (\mathrm{B}_\Qoppa\circ \Delta)^{\mathrm{h} C_2}$. The cofiber of the first natural transformation, denoted $\Lambda_\Qoppa:\mathcal{C}^\op \to \Spectra$, is exact by \cite[Proposition 1.1.13]{CDHHLMNNSI} and $\Lambda_\Qoppa$ is called the \emph{linear part of $\Qoppa$}. One can show, see \cite[Corollary 1.3.10]{CDHHLMNNSI}, that the linear part of $(\mathrm{B}\circ\Delta)^{\mathrm{h}C_2}$ is $(\mathrm{B}\circ\Delta)^{\mathrm{t}C_2}$. This induces a functor $$\Fun^\mathrm{q}(\mathcal{C})\rightarrow \Ar(\Fun^{\ex}(\mathcal{C}^\op,\Spectra))$$ which sends $\Qoppa$ to the natural transformation $\Lambda_\Qoppa \Rightarrow (\mathrm{B}_\Qoppa\circ \Delta)^{\mathrm{t}C_2}$, i.e. the linear part of $\Qoppa \Rightarrow (\mathrm{B}_\Qoppa\circ \Delta)^{\mathrm{h} C_2}$. 

\begin{theorem}[Corollary 1.3.12 \cite{CDHHLMNNSI}]~\label{thm: characterization of quadratic functors on a cat}
    Let $\mathcal{C}$ be a stable $\infty$-category. Then the commutative diagram 
    \[
    \begin{tikzcd}
        \mathrm{Fun}^\mathrm{q}(\mathcal{C}) \ar[r] \ar[d, "\mathrm{B}"] & \mathrm{Ar}(\mathrm{Fun}^\mathrm{ex}(\mathcal{C}^\op, \mathrm{Sp})) \ar[d, "\mathrm{t}"]\\
        \mathrm{Fun}^\mathrm{s}(\mathcal{C}) \ar[r] & \mathrm{Fun}^\mathrm{ex}(\mathcal{C}^\op,\mathrm{Sp})
    \end{tikzcd}
    \]
    is cartesian, where the bottom horizontal map sends a symmetric bilinear functor $\mathrm{B}$ to $(\mathrm{B}\circ \Delta)^{\mathrm{t}C_2}$.
\end{theorem}
\begin{notation}\label{notation:genuine_C2_spectra_and_underlying_spectra}
    We will write $\Spectra^{C_2}$ for the $\infty$-category of genuine $C_2$-spectra. Recall that isotropy separation gives a pullback square of $\infty$-categories
    \[
    \begin{tikzcd}
        \Spectra^{C_2} \ar[rr, "(-)^{\phi C_2}\to(-)^{\mathrm{t}C_2}"] \ar[d] & &\mathrm{Ar}(\Spectra) \ar[d, "\mathrm{t}"] \\
        \Spectra^{\mathrm{B}C_2} \ar[rr, "(-)^{\mathrm{t}C_2}"] & & \Spectra.
    \end{tikzcd}
    \]
    The left vertical map forgets the genuine structure and only retains the underlying spectrum and its $C_2$-action. The functor $(-)^{\phi C_2}:\Spectra^{C_2}\to\Spectra$ denotes geometric fixed points. We will also write $(-)^\mathrm{e}:\Spectra^{\mathrm{B}C_2}\to \Spectra$ for the map that forgets the action. Informally, the above square says that a genuine $C_2$-spectrum $X$ is determined by its underlying spectrum $X^\mathrm{e}$ with a $C_2$-action, its geometric fixed points $X^{\phi C_2}$, and a reference map $X^{\phi C_2}\to (X^\mathrm{e})^{\mathrm{t}C_2}$.
\end{notation}

Let $A$ be an $\mathbb{E}_1$-ring spectrum. The norm $\mathrm{N}^{C_2}A$ of $A$ is the genuine $C_2$-spectrum with underlying spectrum $A\otimes A$ equipped with the swap $C_2$-action, geometric fixed ponts $A$, and reference map $A\rightarrow (A\otimes A)^{\mathrm{t}C_2}$ given by the Tate-diagonal. More generally, the norm forms a symmetric monoidal functor $\mathrm{N}^{C_2}:\Spectra\rightarrow \Spectra^{C_2}$, see \cite{hillhopkinsravenel}. In particular, $\mathrm{N}^{C_2}A$ is an algebra object in $\Spectra^{C_2}$ and we denote by $\mathrm{Mod}_{\mathrm{N}^{C_2}A}$ the $\infty$-category of left modules over it. Moreover, for any left $A$-module $X$, the norm $\mathrm{N}^{C_2}X$ acquires a left $\mathrm{N}^{C_2} A$-module structure. The following theorem relates quadratic functors to genuine $C_2$-spectra and modules over norms.

\begin{theorem}[{\cite[Theorem 3.3.1, Remark 3.3.4, Lemma 5.4.6]{CDHHLMNNSI}}]~\label{thm: genuine equivarient modules are hermitian structures}
    Let $A$ be an $\mathbb{E}_1$-ring spectrum. Then we have an equivalence of $\infty$-categories \[\mathrm{Mod}_{\mathrm{N}^{C_2}A}\xrightarrow{\simeq} \mathrm{Fun}^\mathrm{q}(\mathrm{Mod}_A^{\omega})\] given by sending a module $M$ to the functor $\Qoppa_M:=\underline{\mathrm{Hom}}_{\mathrm{N}^{C_2}A}(\mathrm{N}^{C_2}(-),M)^{C_2}$. This hermitian structure is Poincar{\'e} if and only if the underlying $A$-module of $M$ is invertible. In the case where $A$ is an $\mathbb{E}_\infty$-ring spectrum this equivalence is symmetric monoidal.
\end{theorem}
Owing to the connection between equivariant homotopy theory and Poincaré structures, we will often draw upon existing work in equivariant homotopy theory. 
Here we record a fiber sequence which will be useful to us later on. 
\begin{remark}\label{rmk:C2_spectra_twists}
    Let $X\in \CAlg^\mathrm{gp}(\Spaces)^{\mathrm{B}C_2}$ be a grouplike $\EE_\infty $-space with a $C_2$-action $\lambda:X\rightarrow X$. Let $X^{\mathrm{h}(-\lambda)}$ denote the homotopy fixed points of $X$ with respect to the $C_2$-action induced by $
    \iota\circ \lambda$ where $\iota$ is the inverse map $\iota:X\to X$. 
    Then we have a fiber sequence $X\xrightarrow{(id_X,\iota\circ \lambda)} X\times X\xrightarrow{N_\lambda} X,$ where $N_\lambda:X\times X\rightarrow X$ is the map given by $N_\lambda(x,y)=x\lambda(y)$. We endow the space $X$ on the left with the action $\iota\circ \lambda$, the $X$ on the right with the action $\lambda$, and $X\times X$ with the flip action. Note that this sequence of maps is $C_2$-equivariant now. 
    Taking homotopy fixed points, we obtain a fiber sequence $$X^{\mathrm{h}(-\lambda)}\rightarrow X\xrightarrow{N_\lambda} X^{\mathrm{h}C_2}.$$ Alternatively, we can reverse the role of $\lambda$ and $ \iota\circ \lambda$ and apply homotopy fixed points to the fiber sequence $X\xrightarrow{(id_X,\lambda)} X\times X\xrightarrow{N_{i\circ\lambda}} X$ to obtain the fiber sequence $$X^{\mathrm{h}C_2}\rightarrow X\xrightarrow{N_{i\circ\lambda}} X^{\mathrm{h}(-\lambda)}.$$ 
    
    If $2$ acts invertibly on $ X $, then the fiber sequence splits and we have an equivalence $$X[1/2]\simeq X^{\mathrm{h}(-\lambda)}[1/2]\times X^{\mathrm{h}C_2}[1/2].$$ 

    Alternatively, we may regard $ X $ as a Borel $ C_2 $-spectrum whose underlying spectrum-with-$ C_2 $-action is connective. 
    Then smashing with the exact sequences of $ C_2 $-spectra $ S^{\sigma-1}_+ \to C_{2+} \to (C_2/C_2)_+ $ and $ (C_2/C_2)_+ \to C_{2+} \to S^{1-\sigma}_+ $ and taking $ \Omega^\infty (-^{\mathrm{h} C_2}) $ gives the aforementioned fiber sequences.  
\end{remark}
	
	We will primarily be interested in Poincar\'e $\infty$-categories of the form $\left(\Mod^\omega_R, \Qoppa_M\right)$, where $R$ is an $\mathbb{E}_\infty$-ring spectrum and $M$ is a left $\mathrm{N}^{C_2} R$-module \cite[\S3.3]{CDHHLMNNSI}. 
    We will use Remark~\ref{rmk:C2_spectra_twists} in the context of certain $C_2$-actions induced by the duality $\mathrm{D}_{\Qoppa_M}:(\mathrm{Mod}_{R}^\omega)^\op\to \mathrm{Mod}_{R}^\omega$. We conclude this section by describing the duality functor more concretely in this case.

    Let $M$ be a left $\mathrm{N}^{C_2}R$-module. Then $M^\mathrm{e}$ is naturally a left $(\mathrm{N}^{C_2}R)^{e}\simeq R\otimes R$-module. We may view $M^\mathrm{e}$ as a left $R$-module via the map $R\simeq R\otimes \mathbb{S}\to R\otimes R$. Additionally, the map $R\simeq \mathbb{S}\otimes R\to R\otimes R$ induces a $C_2$-action $\lambda:M^\mathrm{e}\to M^\mathrm{e}$ on $M^\mathrm{e}$ which is compatible with the $R$-module structure.
	
	\begin{lemma}~\label{lemma: duality identification}
		Let $R$ be an $\mathbb{E}_{\infty}$-ring spectrum, and let $M$ be a left $\mathrm{N}^{C_2}R$-module. Assume $M^\mathrm{e}\simeq R$ as an $R$-module with the module structure as described above, and let $\lambda:R\to R$ denote the induced $C_2$-action. Then we have natural equivalences of functors $(\mathrm{Mod}_R^\omega)^\op\to \mathrm{Mod}_R^{\omega}$ \[\mathrm{D}_{\Qoppa_M}(-)\simeq \lambda^*\mathrm{hom}_R(-,R)\simeq \mathrm{hom}_R(\lambda^* (-), R).\] 
	\end{lemma}
	\begin{proof}
		By \cite[Proposition 3.1.6(ii)]{CDHHLMNNSI}, the duality functor in this case is given by \[\mathrm{D}_{\Qoppa_M}(-)\simeq \lambda^*\mathrm{hom}_R(-,R).\] Since $\lambda$ is an involution, $(\lambda^*)\circ (\lambda^*)\simeq (\lambda \circ\lambda)^*=\mathrm{id}$ and $\lambda^*\simeq (\lambda^*)^{-1}\simeq \lambda_*$. The second equivalence now follows from the $\otimes_R-\mathrm{hom}_R$-adjunction, the fact that $\lambda^* \simeq\lambda_*:\mathrm{Mod}^\omega_R\to \mathrm{Mod}^\omega_R$ is symmetric monoidal and the equivalence $\lambda^*(R)\simeq R$ induced by $\lambda$.
	\end{proof}
\begin{remark}~\label{remark:different_c2_actions}
    We frequently will be interested in the groupoid of objects $ \Mod_{R}^{\omega,\simeq} $ for some $ \EE_\infty $-ring $R$ with involution $\lambda:R\to R$. 
    In this situation, there are multiple \textbf{distinct} actions of $C_2$:  
\begin{enumerate}[label=(\arabic*)]
    \item \label{C2action:duality_associated_to_Poincare_structure} an action induced by the duality $ D_{\Qoppa_{R^s}} $ associated to $ \Qoppa $ as described in Remark~\ref{rmk:pnpic_to_pic_equivariant} 
    \item \label{C2action:basechange} induced by the functor $\lambda^* \simeq \lambda_* :\mathrm{Mod}_{R^e}^{\omega} \to \mathrm{Mod}_{R^e}^\omega$ 
    \item \label{C2action:untwisted_canonical_duality} induced by an `untwisted' duality equivalence $\mathrm{Mod}_{R^e}^{\omega,\op}\xrightarrow{\simeq}\mathrm{Mod}_{R^e}^\omega $, $ M\mapsto \mathrm{hom}_{R^e}(M,R^e) $, which we typically denote by $ M^\vee $. 
\end{enumerate}
    By Lemma~\ref{lemma: duality identification}, the action \ref{C2action:duality_associated_to_Poincare_structure} is obtained by pulling the $ C_2 \times C_2 $-action back along the diagonal homomorphism $ C_2 \xrightarrow{\Delta} C_2 \times C_2 $. 
\end{remark}
	
\section{\texorpdfstring{Poincar\'e}{Poincare} rings and schemes with involution}
\label{section:poincare_ring_spectra}
As discussed in \S\ref{subsection:deeperresults}, we are interested in stable $ \infty $-categories with duality which are \emph{linear} over a fixed stable $ \infty $-category with duality. 
Our starting point is \cite[discussion immediately preceding Examples 5.4.10]{CDHHLMNNSI}, where the 9 authors describe the most general algebraic structures which can be expected to give rise to symmetric monoidal Poincaré structures. 
Referred to as \emph{Poincaré rings} in this paper, in \S\ref{subsection:Poincare_rings}, we introduce the $ \infty $-category of such algebras and show (via a lengthy construction which is postponed to Appendix \ref{appendix:calgp_to_catp}) that there is a functorial assignment from Poincaré rings to symmetric monoidal Poincaré $ \infty $-categories. 
In \S\ref{subsection:Poincare_structures_schemes_involution}, we show that a scheme with involution $ (X, \lambda: X \xrightarrow{\sim} X) $ and good quotient gives rise to a canonical symmetric monoidal Poincaré structure on $ \mathrm{Mod}_{\mathcal{O}_X}^\omega $. 
The functoriality established here is of independent interest; invariants such as hermitian K-theory/Grothendieck--Witt theory, algebraic L-theory, and real trace theories are defined using \cites[\S4]{CDHHLMNNSII}{Harpaz_Nikolaus_Shah}. 
	
\subsection{Poincar{\'e} rings}\label{subsection:Poincare_rings}
Let $ R $ be an $ \EE_\infty $-ring spectrum. 
In \cite[discussion immediately preceding Examples 5.4.10]{CDHHLMNNSI}, the 9 authors describe the additional data needed to endow $ \Mod^\omega_R $ with a symmetric monoidal Poincaré structure. 
In the following, we introduce an $ \infty $-category which we refer to as \emph{Poincaré ring spectra} whose objects are the $\mathbb{E}_\infty$-ring spectra with genuine involution of Calmès--Dotto--Harpaz--Hebestreit--Land--Moi--Nardin--Nikolaus--Steimle. 
We construct a natural symmetric monoidal functor from Poincaré ring spectra to Poincaré $ \infty $-categories (Theorem \ref{thm:calgp_to_poincare_cat}). 
We include examples throughout and discuss how ordinary commutative rings with involution can be regarded as Poincaré ring spectra (Example \ref{example:genuine_symmetric_poincare_structure}). 

We expect the theory developed here to extend to a well-behaved theory of spectral schemes with involution and of symmetric monoidal Poincaré structures on their module categories, but we do not take up such questions in this work. 
\begin{recollection}\label{recollection:Tate_valued_norm}
	Write $ \mathcal{O}_{C_2} $ for the orbit category of $ C_2 $.  
	The \emph{Tate-valued norm} is a functor $ \CAlg\left(\Spectra^{BC_2}\right) \to \operatorname{Fun}\left(\mathcal{O}_{C_2}, \CAlg(\Spectra)\right) $ (see Definition 3.8 and Lemma 3.10 of \cite{LYang_normedrings}). 
	Note that there is a canonical functor $ \iota \colon \Delta^1 = [0<1] \to \mathcal{O}_{C_2} $ which sends $ 0 $ to $ C_2/e $ and $ 1 $ to $ C_2/C_2 $. 
	Via precomposition with $ \iota $, we will regard the Tate-valued norm as a functor $ \CAlg\left(\Spectra^{BC_2}\right) \to \operatorname{Fun}\left(\Delta^1, \CAlg(\Spectra)\right) $. 
\end{recollection}
\begin{definition}
	\label{definition:poincare_ring_spectrum}
	Let $\CAlgp$ be the $\infty $-category defined by the pullback 
	\begin{equation}\label{diagram:poincare_ring_defn}    
		\begin{tikzcd}
			\CAlgp \arrow[rr]\arrow[d] & & \operatorname{Fun}(\Delta^2, \CAlg(\Spectra))\arrow[d,"d_1^*"]\\
			\CAlg\left(\Spectra^{BC_2}\right)\arrow[rr,"U(-)\to (-)^{\mathrm{t}C_2}"] & & \operatorname{Fun}\left(\Delta^1, \CAlg(\Spectra)\right)
		\end{tikzcd}
	\end{equation}
	where $U$ is induced by the forgetful functor $\Spectra^{BC_2}\to \Spectra$ and the lower horizontal arrow is the Tate-valued norm of Recollection \ref{recollection:Tate_valued_norm}. 
	An object of $ \CAlgp $ will be called a \emph{Poincaré ring spectrum}.\footnote{For an explanation behind the name, we will show later that these objects are closely connected with Poincar{\'e} $\infty$-categories, which are objects that encode Poincar{\'e} duality objects in a very general context, see Section~\ref{section:poincare_structures_on_compact_modules} for more details on this.}
	
	Let $R$ be an $ \EE_\infty $ ring spectrum. 
	A \emph{Poincaré structure} on $R$ is the data of a lift of $ R $ to the $ \infty $-category $ \CAlgp $. 
	By definition, it comprises the following data:
	\begin{itemize}
		\item A $C_2$-action on $R$ via maps of ring spectra, i.e. a functor $\lambda: BC_2\rightarrow \CAlg$.
		\item An $ \EE_\infty $-$R$-algebra $ C$.
		\item An $ \EE_\infty $-$R$-algebra map $C\rightarrow R^{\mathrm{t}C_2}$. 
	\end{itemize}
	Here $R^{\mathrm{t}C_2}$ is the Tate construction with respect to the given action. 
	Since the Tate construction is lax symmetric monoidal, $R^{\mathrm{t}C_2}$ is naturally an $R$-algebra via the Tate-valued norm. 
\end{definition}

We will denote objects of $\CAlgp$ by $ R=(R^e,s:R^{\phi C_2}\to R^{\mathrm{t}C_2})$. 
Here $ s:R^{\phi C_2}\to R^{\mathrm{t}C_2}$ is the image of $R$ under the top horizontal map above. 
The use of the notation $R^{\phi C_2}$ is justified by Lemma \ref{lemma:Poincare_ring_geom_fixpt}.

\begin{remark}\label{rmk:Poincare_ring_has_underlying_C2_spectrum_alg}
	Recall that there is a symmetric monoidal récollement $ \Spectra^{C_2} \simeq \Spectra^{BC_2} \times_{(-)^{\mathrm{t}C_2},\Spectra, \mathrm{ev}_1} \Spectra^{\Delta^1} $ (cf. \cite[Theorem 6.24]{MNN}). 
	There is a commutative diagram of $ \infty $-categories: 
	\[
	\begin{tikzcd}[column sep=small,row sep=small]
		& & \operatorname{Fun}(\Delta^2, \CAlg(\Spectra)) \arrow[rd, "d_0^*"] \arrow[dd, "d_1^*"] & \\
		& & & \operatorname{Fun}(\Delta^1,\CAlg(\Spectra)) \arrow[dd,"ev_{1}"]\\
		\CAlg(\Spectra^{BC_2}) \arrow[rr, "U(-)\to (-)^{\mathrm{t}C_2}"] \arrow[rd, equals] & & \operatorname{Fun}(\Delta^1, \CAlg(\Spectra)) \arrow[rd, "ev_{1}"] & \\
		& \CAlg(\Spectra^{BC_2}) \arrow[rr, "(-)^{\mathrm{t}C_2}"] & & \CAlg(\Spectra)
	\end{tikzcd}\,.
	\] 
	The diagram induces a functor from the pullback of the upper left cospan to the pullback of the lower right cospan $ (-)^L \colon \CAlgp\to \CAlg(\Spectra^{C_2}) $. 
	
	Now observe that there is a commutative diagram
	\begin{equation}\label{diagram:Poincare_ring_alternate_diagram}
		\begin{tikzcd}[column sep=large,row sep=small]
			&& \Fun\left(\Delta^1 \times \Delta^1,\CAlg(\Spectra) \right) \ar[r,"{\mathrm{ev}_{00 \to 01 \to 11}}"] \ar[d,"{\mathrm{ev}_{00 \to 10 \to 11}}"] \ar[rd,phantom,"\lrcorner",very near start] & \Fun\left(\Delta^2, \CAlg(\Spectra)\right) \ar[d,"{\mathrm{ev}_{0\to 2}}"] \\
			\CAlg(\Spectra)^{BC_2} \ar[d,"{m_{C_2}}"] \ar[rr,"{A \mapsto (A \to (A^{\otimes 2})^{\mathrm{t}C_2} \to A^{\mathrm{t}C_2})}"] & &\Fun\left(\Delta^2,\CAlg(\Spectra) \right) \ar[r,"{\mathrm{ev}_{0 \to 2}}"] \ar[d,"{\mathrm{ev}_{1\to 2}}"] & \Fun\left(\Delta^1,\CAlg(\Spectra) \right) \\
			\CAlg(\Spectra)^{BC_2 \times \Delta^1} \ar[rr,"{(-)^{\mathrm{t}C_2}}"] && \Fun\left(\Delta^1 ,\CAlg(\Spectra) \right) &
		\end{tikzcd}
	\end{equation}
	where the upper right square is a pullback by definition of $ \Delta^1 \times \Delta^1 $ and $ m_{C_2} $ is the multiplication map functor of \cite[Construction 3.1]{LYang_normedrings} precomposed with the inclusion $ \iota $ of Recollection \ref{recollection:Tate_valued_norm}. 
	Since the upper right is a pullback, the pullback of the outer span in the top row (which is $ \CAlgp $) agrees with the pullback of the upper left span. 
	On the other hand, the pullback of the upper left span maps to the pullback of the `tall' span on the left. 
	By the récollement of $ \Spectra^{C_2} $, the pullback of the `tall' span on the left may be identified with $ \Einfty \Alg\left(\Spectra^{C_2}\right)^{\Delta^1} $. 
	Furthermore, by definition of the Hill--Hopkins--Ravenel norm, we have described a functor
	\begin{equation}\label{eq:Poincare_ring_functorial_norm}
		\CAlgp \to \Fun\left(\Delta^1,\Einfty \Alg\left(\Spectra^{C_2}\right)\right)
	\end{equation}
	which sends a Poincaré ring $ R $ to a map of $ \Einfty $-algebras $ N^{C_2}(R^e) \to R^L $. 
\end{remark}
\begin{lemma}\label{lemma:Poincare_ring_geom_fixpt}
	There is a commutative diagram
	\begin{equation*}
		\begin{tikzcd}
			\CAlgp \ar[r] \ar[d,"(-)^L"] & \operatorname{Fun}(\Delta^2,\CAlg(\Spectra)) \ar[d,"{\mathrm{ev}_1}"] \\
			\CAlg(\Spectra^{C_2}) \ar[r,"{(-)^{\phi C_2}}"] & \CAlg(\Spectra)
		\end{tikzcd}
	\end{equation*}
	where $ (-)^L $ is the functor of Remark \ref{rmk:Poincare_ring_has_underlying_C2_spectrum_alg} and the upper horizontal arrow is the canonical functor in Definition \ref{definition:poincare_ring_spectrum}. 
\end{lemma}
\begin{proof}
	Follows from the récollement of $ \Spectra^{C_2} $ and Remark \ref{rmk:Poincare_ring_has_underlying_C2_spectrum_alg}. 
\end{proof}

\begin{example}[Symmetric and Tate Poincaré rings]\label{example:tate_poincare_structure}
	We note that the functor $\CAlg^\mathrm{p}\to \CAlg\left(\mathrm{Sp}^{BC_2}\right)$ admits both left and right adjoints. 
    The left adjoint is given by the \textit{Tate} Poincar{\'e} ring \[R\mapsto R^t:=\left(R, N:R\to R^{\mathrm{t}C_2}\right)\] where the structure map is the Tate-valued norm of Recollection \ref{recollection:Tate_valued_norm}. 
    The right adjoint is given by the \textit{symmetric} Poincar{\'e} ring \[R\mapsto R^s:=\left(R, R^{\mathrm{t}C_2}=R^{\mathrm{t}C_2}\right)\,.\] To explain the naming convention, the first functor is named the Tate Poincar{\'e} ring after the map from the geometric fixed points to the Tate construction, i.e. the Tate-valued norm. 
    To explain why $R^s$ is called the symmetric Poincar{\'e} ring, by Theorem~\ref{thm:calgp_to_poincare_cat} the Poincar{\'e} $\infty$-category we associate to this Poincar{\'e} ring is $\mathrm{Mod}_{R}^\omega$ together with the hermitian functor given by $\Qoppa_{R^s}(M)=\mathrm{Hom}_{N^{C_2}R}\left(N^{C_2}-, (R^s)^L\right)^{C_2}$. Since the map from the geometric fixed points of $R^s$ to the Tate construction is the identity, it follows that the map \[\Qoppa_{R^s}(-)\to \mathrm{Hom}_{R\otimes R}(-\otimes \lambda^*-, R)^{\mathrm{h}C_2}\] is an equivalence, i.e. $ \Qoppa_{R^s} $ is equivalent to the symmetric Poincaré structure $ \Qoppa^s $ associated to the bilinear functor $ \mathrm{Hom}_{R\otimes R}(-\otimes \lambda^*-, R) $ (cf. Recollection \ref{rec:symmetric_Poincare_cat} and the dual to Lemma 1.3.1 of \cite{CDHHLMNNSI}). 
    In particular, taking $M$ a discrete module we get that $\pi_0(\Qoppa_{R^s}(M))$ is given by bilinear maps $\langle -,-\rangle:M\otimes_R M\to R$ which are invariant under $\langle x,y\rangle=\langle y,\lambda(x)\rangle$, in other words these are $\lambda$-symmetric bilinear forms. See \cite[Example 1.2.12]{CDHHLMNNSI} for a discussion of this Poincar{\'e} $\infty$-category from the bilinear form perspective. 

    Note that the functor $(-)^{\mathrm{t}}:\CAlg\left(\mathrm{Sp}^{\mathrm{B}C_2}\right)\to \CAlgp$ is fully faithful since the unit of the adjunction is an equivalence. 
    Similarly, $(-)^\mathrm{s}$ is fully faithful since the counit of the adjunction is an equivalence. 
\end{example}

\begin{observation}\label{observation:normed_C2_ring_forget_to_Poincare_ring}
	In view of \cite[Theorem 1.3 \& Definition 3.13]{LYang_normedrings}, there is a forgetful functor from $ C_2 $-$ \EE_\infty $-algebras in $ \Spectra^{C_2} $ to Poincaré rings which forgets the $ C_2 $-equivariance of the map $ R \to R^{\varphi C_2 } $. 
	Their underlying $ \Einfty $-algebras in $ \Spectra^{C_2} $ agree. 
\end{observation}

\begin{example}~\label{example:classification_of_poincare_structures_when_tate_vanishes}
	Let $R$ be a ring spectrum with a $ C_2 $-action. If $2\in \pi_0(R)$ is invertible, we have $R^{\mathrm{t}C_2}\simeq 0 $ by \cite[Lemma I.2.8]{NS}. 
	A Poincaré structure on $R$ is then equivalent to the data of an $ \mathbb{E}_\infty $-$R$-algebra $R\rightarrow C$, in which case $\Qoppa_{(R, C)}(-)\simeq \Qoppa_{R^s}(-)\oplus \mathrm{Hom}_R(-,C)$.
\end{example}

\begin{example}
	\label{example:universal_poincare_ring_spectrum}
	The sphere spectrum $\mathbb{S}$ together with the Tate Poincaré structure will be called the \emph{universal Poincaré ring spectrum} (see \cite[\S4.1]{CDHHLMNNSI}). 
    We will denote it by $\mathbb{S}^u$ and the associated quadratic functor by $ \Qoppa^u $. 
\end{example}

\begin{remark}
	\label{remark:factorizations_of_tate_frobenius_through_invariants_induce_splittings_of_forms}
	Let $(R,\Qoppa)$ be a ring spectrum associated to a factorization $R\rightarrow C\rightarrow R^{\mathrm{t}C_2}$. A factorization of the map $C\rightarrow R^{\mathrm{t}C_2}$ through $R^{\mathrm{h}C_2}$ induces a section of the canonical map $\Qoppa(R)\rightarrow \hom_R(R,C)\simeq C$. In that case, we have a splitting $\Qoppa(R)\simeq R_{\mathrm{h}C_2}\oplus C$ by Theorem~\ref{thm: characterization of quadratic functors on a cat}.
\end{remark}

\begin{example}
	\label{example:universal_tate_poincare_splits_at_unit}
	The Tate Frobenius for the sphere spectrum factors through $\mathbb{S}^{\mathrm{h}C_2}$. Therefore, Remark \ref{remark:factorizations_of_tate_frobenius_through_invariants_induce_splittings_of_forms} implies $\Qoppa^u(\mathbb{S})\simeq \mathbb{S}_{\mathrm{h}C_2}\oplus \mathbb{S}\simeq \Sigma^\infty(\mathbb{P}_\mathbb{R}^\infty) \oplus \mathbb{S}$.
\end{example}

\begin{example}
	\label{example:genuine_symmetric_poincare_structure}
	Let $R$ be a connective ring spectrum equipped with a $C_2$-action via maps of ring spectra. The connective cover $\tau_{\geq 0}(R^{\mathrm{t}C_2})\rightarrow R^{\mathrm{t}C_2}$ of $R^{\mathrm{t}C_2}$ induces a Poincaré structure on $R$ given by the factorization $R\rightarrow \tau_{\geq 0}(R^{\mathrm{t}C_2})\rightarrow R^{\mathrm{t}C_2}$. We will call this Poincaré structure the \emph{genuine symmetric Poincaré structure on $R$}.
\end{example}

\begin{example}
	\label{ex:fixpt_Mackey_functor}   
	Let $ R $ be a commutative ring endowed with an involution $ \lambda \colon R \xrightarrow{\sim} R $. 
	Write $ \underline{R}^\lambda $ for the $ C_2 $-Green functor with $ C_2 $-fixed points $ R^{C_2} $, where $ R^{C_2} $ denotes the strict fixed points of the $ C_2 $-action on $ R $, and underlying object $ R $. 
	The Mackey functor $ \underline{R}^\lambda $ is a $ C_2 $-$ \EE_\infty $ ring, therefore in particular we may regard it as a Poincaré ring by Observation \ref{observation:normed_C2_ring_forget_to_Poincare_ring}. 
	This is a special case of Example \ref{example:genuine_symmetric_poincare_structure}. 
    The assignment $ (R,\lambda) \mapsto \underline{R}^\lambda $ is functorial in equivariant maps of commutative rings. 
\end{example}

\begin{theorem}\label{thm:poincare_rings_cat_formal_properties}
	The following statements about $\CAlgp$ hold:
	\begin{enumerate}[label=(\arabic*)]
		\item \label{thmitem:defining_diagram_homotopy_pullback} the pullback diagram (\ref{diagram:poincare_ring_defn}) is homotopy Cartesian; 
		\item \label{thmitem:poincare_ring_has_colimits} The category $\CAlgp$ has all small colimits;
		\item \label{thmitem:poincare_ring_to_naive_ring_preserves_colims} the functor $ \CAlgp \to \CAlg(\Spectra^{BC_2})$ preserves all small colimits;
		\item \label{thmitem:poincare_ring_has_limits} The category $\CAlgp$ has all small limits;
		\item \label{thmitem:poincare_ring_to_naive_ring_preserves_lims} the functor $ \CAlgp \to \CAlg(\Spectra^{BC_2})$ preserves all small limits;
		\item \label{thmitem:poincare_ring_to_ring_preserves_lims} the functor $ \CAlgp \to \CAlg\left(\Spectra^{C_2}\right) $ of Remark \ref{rmk:Poincare_ring_has_underlying_C2_spectrum_alg} preserves all small limits;
		\item \label{thmitem:poincare_ring_to_ring_preserves_sift_colims} the functor $ \CAlgp \to \CAlg\left(\Spectra^{C_2}\right) $ of Remark \ref{rmk:Poincare_ring_has_underlying_C2_spectrum_alg} preserves sifted colimits;
		\item \label{thmitem:poincare_rings_presentable_accessible} The $ \infty $-category $ \CAlgp $ is presentable and accessible. 
	\end{enumerate}
\end{theorem}
\begin{proof}
	To prove \ref{thmitem:defining_diagram_homotopy_pullback}, it is enough to show that $d_1^*$ is a categorical fibration. 
	In fact, $ d_1^* $ is a cocartesian and cartesian fibration; this follows from the existence of colimits and limits, resp., in $ \CAlg(\Spectra) $ and Lemma \ref{lemma:restriction_functor_cat_as_fibrations}. 
	
	To prove \ref{thmitem:poincare_ring_has_colimits}, let $p:K\to \CAlgp$ be a map of simplicial sets, where $K$ is a small simplicial set. 
	Write $ f' $ for the functor $\CAlgp \to \CAlg(\Spectra^{BC_2})$ and $ g'\colon \CAlgp \to  \CAlg(\Spectra)^{\Delta^2} $ and $ f \colon \CAlg(\Spectra)^{\Delta^2} \to  \CAlg(\Spectra)^{\Delta^1} $ and $ g \colon \CAlg(\Spectra^{BC_2}) \to  \CAlg(\Spectra)^{\Delta^1} $. 
	Choose an extension $ \overline{f'p} \colon K^\vartriangleright\to \CAlg(\Spectra^{BC_2}) $ of $ f'\circ p $ which is a colimit diagram. 
	By \cite[Proposition 4.3.1.5(2)]{HTT}, to show that $ p $ admits a colimit, it suffices to exhibit a lift $ K^\vartriangleright \to \CAlgp $ of $ \overline{f' \circ p} $ which is an $ f' $-colimit diagram. 
	By \cite[Proposition 4.3.1.5(4)]{HTT} and \ref{thmitem:defining_diagram_homotopy_pullback}, it suffices to show that there exists an extension $ \overline{g'p} \colon K^\vartriangleright \to \CAlg(\Spectra)^{\Delta^2} $ of $ g \circ \overline{f'p} $ which is an $ f=d_1^* $-colimit. 
	Observe that $ f= d_1^* $ is a cocartesian fibration, the fibers of $ f $ admit small colimits, and $ f $-cocartesian transport preserves all small colimits by \ref{thmitem:defining_diagram_homotopy_pullback}, Lemma \ref{lemma:restriction_functor_cat_as_fibrations}, and the fact that $ \CAlg $ admits all small colimits. 
	The existence of such an extension $ \overline{g'p} $ follows from \cite[Corollary 4.3.1.11]{HTT}. 
	
	Part \ref{thmitem:poincare_ring_to_naive_ring_preserves_colims} follows from \ref{thmitem:poincare_ring_has_colimits}. 
	
	The proofs of parts \ref{thmitem:poincare_ring_has_limits} and \ref{thmitem:poincare_ring_to_naive_ring_preserves_lims} are analogous to those of \ref{thmitem:poincare_ring_has_colimits} and \ref{thmitem:poincare_ring_to_naive_ring_preserves_colims} and have been omitted. 
	
	To prove part \ref{thmitem:poincare_ring_to_ring_preserves_lims}, observe that there is a commutative square
	\begin{equation*}
		\begin{tikzcd}
			\CAlg^{\Delta^2} \ar[d,"{d_0}"] \ar[r,"{d_1=f}"] & \CAlg^{\Delta^1} \ar[d,"{\delta_0}"] \\
			\CAlg^{\Delta^1} \ar[r,"{\partial_0}"] & \CAlg \,.
		\end{tikzcd}
	\end{equation*}
	In particular, $ d_0 $ takes $ d_1 $-cartesian morphisms to $ \partial_0 $-cartesian morphisms and for any $ \alpha \in \CAlg^{\Delta^1} $, $ d_0 $ takes limits in $ \CAlg^{\Delta^2}_{\alpha} $ to limits in $ \CAlg^{\Delta^1}_{\delta_0(\alpha)} $. 
	It follows from the dual of \cite[Proposition 4.3.1.9]{HTT} that if $ \overline{g'p} $ is an $ f $-limit, then $ d_0 \circ \overline{g'p} $ is a $ \partial_0 $-limit. 
	Therefore, the functor $ \CAlgp \to \EE_\infty \Alg \left(\Spectra^{BC_2}\right) \times_{\EE_\infty\Alg(\Spectra)} \EE_\infty\Alg(\Spectra)^{\Delta^1} $ preserves all small limits, and the codomain is equivalent to $ \EE_\infty \Alg \left(\Spectra^{C_2}\right) $ by the symmetric monoidal récollement decomposition of $ \Spectra^{C_2} $. 
	
	The proof of part \ref{thmitem:poincare_ring_to_ring_preserves_sift_colims} is dual to the proof of part \ref{thmitem:poincare_ring_to_ring_preserves_lims}, noting that sifted colimits of algebras are computed in spectra.  
	
	Now let us consider part \ref{thmitem:poincare_rings_presentable_accessible}. 
	Accessibility of $ \CAlgp $ follows from closure of accessible $ \infty $-categories under fiber products \cite[Proposition 5.4.6.6]{HTT} and accessibility of the $ \infty $-categories in the pullback diagram defining $ \CAlgp $, which itself follows from accessibility of $ \CAlg $ and \cite[Proposition 5.4.4.3]{HTT}. 
	Now presentability of $ \CAlgp $ follows from accessibility and \ref{thmitem:poincare_ring_has_colimits}. 
\end{proof}
\begin{lemma}\label{lemma:restriction_functor_cat_as_fibrations}
	Let $ \mathcal{C} $ be an $ \infty $-category, and let $ \mathcal{K} $ be a collection of simplicial sets. 
	\begin{enumerate}[label=(\alph*)]
		\item \label{lemma_item:restriction_functor_cocartesian} If $ \mathcal{C} $ has finite colimits, then the functor $ d_1^* \colon \Fun\left(\Delta^2, \mathcal{C}\right) \to \Fun\left(\Delta^1, \mathcal{C}\right) $ exhibits the source as a cocartesian fibration over the target. 
		If $ \mathcal{C} $ admits all $ \mathcal{K} $-indexed colimits, then the fibers of $ d_1^* $ admit all $ \mathcal{K} $-indexed colimits and for each morphism $ f \colon \alpha \to \beta $ in $ \mathcal{C}^{\Delta^1} $, the associated functor $ f_* \colon \mathcal{C}^{\Delta^2}_{\alpha} \to \mathcal{C}^{\Delta^2}_{\beta} $ preserves all $ \mathcal{K} $-indexed colimits. 
		\item \label{lemma_item:restriction_functor_cartesian} The functor $ d_1^* $ exhibits the source as a cartesian fibration over the target. 
		If $ \mathcal{C} $ has finite limits, then $ d_1^* $ exhibits the source as a cartesian fibration over the target. 
		If $ \mathcal{C} $ admits all $ \mathcal{K} $-indexed limits, then the fibers of $ d_1^* $ admit all $ \mathcal{K} $-indexed limits and for each morphism $ f \colon \alpha \to \beta $ in $ \mathcal{C}^{\Delta^1} $, the associated functor $ f^* \colon \mathcal{C}^{\Delta^2}_{\beta} \to \mathcal{C}^{\Delta^2}_{\alpha} $ preserves all $ \mathcal{K} $-indexed limits. 
	\end{enumerate} 
\end{lemma}
\begin{proof} 
	Consider the maps 
	\begin{equation*}
		\Delta^1 \simeq \Delta^1 \times \{0\} \xrightarrow{i_0} \Delta^1 \times \Delta^1 \xrightarrow{j_0} \Delta^2
	\end{equation*}
	where $ i_0 $ classifies the morphism $ 00 \to 10 $ and $ j_0 $ sends the edge $ 10 \to 11 $ (i.e. the unique nonidentity in $ \{1\} \times \Delta^1 $) to the identity at $ 2 $ in $ \Delta^2 $ and satisfies $ j_0(00) = 0 $, $ j_0(01) = 1 $. 
	Note that their composite is $ j_0 \circ i_0 = d_1 $. 
	Precomposition induces maps 
	\begin{equation*}
		\mathcal{C}^{\Delta^2} \xrightarrow{j^*_0} \mathcal{C}^{\Delta^1 \times \Delta^1} \xrightarrow{i^*_0} \mathcal{C}^{\Delta^1}
	\end{equation*}  
	whose composite is $ d_1^* $. 
	Note that $ i^*_0 $ corresponds to taking the `source' $ \Fun\left(\Delta^1,\mathcal{C}^{\Delta^1}\right) \to \mathcal{C}^{\Delta^1} $. 
	
	That $ d_1^*$ is an inner fibration follows from \cite[Corollary 2.3.2.5]{HTT}. 
	It remains to show that morphisms in $ \mathcal{C}^{\Delta^1} $ have $ d_1^*$-cocartesian lifts. 
	By our assumption on $ \mathcal{C} $, the functor category $ \mathcal{C}^{\Delta^1} $ also has finite colimits. 
	Because of our previous observation that $ i^*_0 $ corresponds to the source functor on the arrow category of $ \mathcal{C}^{\Delta^1} $ and $ \mathcal{C}^{\Delta^1} $ has finite colimits, $ i^*_0 $ is a cocartesian fibration. 
	In particular, $ i^*_0 $-cocartesian transport along a morphism in $ \mathcal{C}^{\Delta^1} $ corresponds to taking a pushout in $ \mathcal{C}^{\Delta^1} $. 
	Furthermore, $ j^*_0 $ is fully faithful and identifies $ \mathcal{C}^{\Delta^2} $ with its essential image, the full subcategory of $ \mathcal{C}^{\Delta^1 \times \Delta^1} $ on those functors which send the edge $ 10 \to 11 $ to an equivalence in $ \mathcal{C} $. 
	To prove part \ref{lemma_item:restriction_functor_cocartesian}, it suffices to show that the essential image of $ j^*_0 $ is closed under $ i^*_0 $-cocartesian transport. 
	This is true because colimits in functor categories are computed pointwise by \cite[Corollary 5.1.2.3]{HTT} and equivalences are stable under pushouts.  
	
	To prove \ref{lemma_item:restriction_functor_cartesian}, it suffices to show that morphisms in $ \mathcal{C}^{\Delta^1} $ have $ d_1^*$-cartesian lifts. 
	Consider the maps 
	\begin{equation*}
		\Delta^1 \simeq \{1\} \times \Delta^1 \xrightarrow{i_1} \Delta^1 \times \Delta^1 \xrightarrow{j_1} \Delta^2
	\end{equation*}
	where $ i_1 $ classifies the morphism $ 10 \to 11 $ and $ j_1 $ sends the edge $ 00 \to 10 $ (i.e. the unique nonidentity in $ \Delta^1 \times \{0\} $) to the identity at $ 0 $ in $ \Delta^2 $ and satisfies $ j_1(01) = 1 $, $ j_1(11) = 2 $. 
	Note that their composite is $ j_1 \circ i_1 = d_1 $. 
	Precomposition induces maps 
	\begin{equation*}
		\mathcal{C}^{\Delta^2} \xrightarrow{j^*} \mathcal{C}^{\Delta^1 \times \Delta^1} \xrightarrow{i^*} \mathcal{C}^{\Delta^1}
	\end{equation*}  
	whose composite is $ d_1^* $. 
	Note that $ i^*_1 $ corresponds to taking the `target' $ \Fun\left(\Delta^1,\mathcal{C}^{\Delta^1}\right) \to \mathcal{C}^{\Delta^1} $. 
	Because $ \mathcal{C}^{\Delta^1} $ has finite limits, $ i^*_1 $ is a cartesian fibration. 
	In particular, $ i^*_1 $-cartesian transport along a morphism in $ \mathcal{C}^{\Delta^1} $ corresponds to taking a pullback in $ \mathcal{C}^{\Delta^1} $. 
	Furthermore, $ j^*_1 $ is fully faithful and identifies $ \mathcal{C}^{\Delta^2} $ with its essential image,\footnote{Note that in this proof we have produced two different fully faithful embeddings of $ \mathcal{C}^{\Delta^2} $ into $ \mathcal{C}^{\Delta^1 \times \Delta^1} $, but their essential images are different. While $ i_0^* $ \emph{is} a cartesian fibration, the essential image of $ j_0^* $ is not closed under $ i_0^* $-cartesian transport. Thus, part \ref{lemma_item:restriction_functor_cartesian} does not follow from the same construction as \ref{lemma_item:restriction_functor_cocartesian}. } the full subcategory of $ \mathcal{C}^{\Delta^1 \times \Delta^1} $ on those functors which send the edge $ 00 \to 10 $ to an equivalence in $ \mathcal{C} $. 
	To prove part \ref{lemma_item:restriction_functor_cartesian}, it suffices to show that the essential image of $ j^*_1 $ is closed under $ i^*_1 $-cartesian transport. 
	This is true because limits in functor categories are computed pointwise by the dual to \cite[Corollary 5.1.2.3]{HTT} and equivalences are stable under pullbacks.  
\end{proof}
\begin{remark}\label{rmk:limits_of_Poincare_rings}
    Suppose given a diagram $ p \colon K \to \CAlgp $ where $ K $ is some simplicial set. 
	Let us use the notation of \ref{thmitem:poincare_ring_has_colimits}, with the modifications discussed in the proof of \ref{thmitem:poincare_ring_to_ring_preserves_lims}. 
	Then the limit of $ p $ is computed as follows: Take $ f' \circ p \colon K \to \CAlg^{BC_2} $ and extend it to a limit diagram $ \overline{f'p} \colon K^{\vartriangleleft} \to \CAlg^{BC_2} $. 
	Write $ \infty $ for the cone point in $ K^{\vartriangleleft} $ and $ R_\infty:= \overline{f'p}(\infty) $. 
	Now regard $ g \circ \overline{f'p} $ as a natural transformation $ \overline{\eta} $ of functors $ K \to \CAlg^{\Delta^1} $ from the constant functor at $ R_\infty \to R_\infty^{\mathrm{t}C_2} $ to $ g \circ f' \circ p $. 
	We may lift this to a pointwise (in $ K $) $ f $-cartesian natural transformation $ \eta $ of functors $ K \to \CAlg^{\Delta^2} $ with target $ g'\circ p $. 
	Let us write $ q_\infty \colon K \to \CAlg^{\Delta^2} $ for the domain of $ \eta $; it takes values in the fiber of $ f $ over $ R_\infty \to R_\infty^{\mathrm{t}C_2} $. 
	Then the image of $ \lim p $ under $ g' $ can be identified with the limit of $ q_\infty $. 
	
	An analogous description holds for colimits in $ \CAlgp $; we leave the details to the reader. 
\end{remark}

\begin{remark}
	\label{remark:poincare_ring_spectra_to_modules_with_Poincare_structure}
	Let $R$ be an $\mathbb{E}_\infty$ ring spectrum. 
	By \cite[Corollary 5.4.8]{CDHHLMNNSI} and (\ref{eq:Poincare_ring_functorial_norm}), a Poincaré ring $(R, q:C\to R^{\mathrm{t}C_2})$ lifting $R$ gives rise to a symmetric monoidal lift of $ \Mod_R^\omega $ to the symmetric monoidal $ \infty $-category of Poincaré $\infty$-categories $ \Qoppa_{(R,q:C\to R^{\mathrm{t}C_2})}: (\Mod_R^\omega)^{\op}\rightarrow \Spectra $. 
	Furthermore, the structure map $ R \to C $ gives a canonical lift of $ R \in \Mod_R^\omega $ to a Poincaré object $ (R, u) \in \mathrm{Pn}\left(\Mod_R^\omega,\Qoppa_R \right) $. 
\end{remark}

We end this subsection by noting that the assignment of a Poincar\'e ring $R$ to the Poincar\'e category $\left(\Mod_{R^e}^\omega,\Qoppa_R\right)$ of Remark \ref{remark:poincare_ring_spectra_to_modules_with_Poincare_structure} is functorial. 
While exhibiting a Poincaré structure on $ \Mod_{R^e}^\omega $ (even one which is symmetric monoidal) has been done in \cite[Theorem 3.3.1 \& Corollary 5.4.8]{CDHHLMNNSI}, more care is required to specify how $ \left(\Mod_{R^e}^\omega,\Qoppa_R\right) $ varies as both the underlying $ \EE_\infty $-ring $ R^e $ and an associated module with genuine involution vary. 
We record the result here and refer the interested reader to the appendix for the proof.
\begin{theorem}\label{thm:calgp_to_poincare_cat}
	\begin{enumerate}[label=(\arabic*)]
		\item There is a canonical functor $ \CAlgp \to \Catp $ whose composite with the forgetful functor $ \Catp \to \Catex $ sends a Poincaré ring $ R $ to $ \Mod^\omega_{R^e} $. 
		\item The functor $ \CAlgp \to \Catp $ sends Borel/symmetric Poincaré rings to Poincaré $ \infty $-categories which are symmetric in the sense of \cite[Definition 1.2.11]{CDHHLMNNSI}.
		\item Recall that $ \CAlgp $ has an initial object given by the sphere spectrum of Example \ref{example:universal_poincare_ring_spectrum}. 
		Then there is an equivalence $ \Mod^\mathrm{p}_{\mathbb{S}^u} \simeq \left(\Spectra^\omega,\Qoppa^u\right) $ where the latter is the universal Poincaré $ \infty $-category of \cite[\S4.1]{CDHHLMNNSI}. 
		\item \label{thmitem:calgp_to_poincare_cat_with_tensor} Recall that $ \Catpidem $ has a symmetric monoidal structure. 
		Then the functor $ \Mod^\mathrm{p} \colon \CAlgp \to \Catpidem $ admits a canonical lift to a functor $ \CAlgp \to \EE_\infty \Alg\left(\Catpidem\right) $, where the latter denotes symmetric monoidal Poincaré $ \infty $-categories. 
		\item The functor $ \Mod^\mathrm{p} \colon \CAlgp \to \EE_\infty \Alg\left(\Catpidem\right) $ is symmetric monoidal. 
	\end{enumerate}
\end{theorem}
\begin{proof}
	Follows from Theorems \ref{thm:Ep_alg_to_poincare_cat} and \ref{thm:operadic_diagrammatic_calgp_agree}. 
\end{proof}
The following corollary is an immediate consequence of Corollary \ref{cor:unit_poincare_object} and Theorem \ref{thm:operadic_diagrammatic_calgp_agree}. 
\begin{corollary}\label{cor:tensor_unit_as_poincare_object}
    There are canonical lifts 
    \begin{equation}
        \EE_\sigma\Alg \to \left(\Catp\right) _{\left(\Spectra^\omega,\Qoppa^u\right)/-}\qquad \qquad \CAlgp \to \EE_\infty \Mon\left(\Catp\right)_{\left(\Spectra^\omega,\Qoppa^u\right)/-}
    \end{equation} 
    of the functors of Theorem \ref{thm:calgp_to_poincare_cat}. 
    In particular, if $ A $ is an $ \EE_\sigma $-algebra, then $ A^e $ canonically promotes to a Poincaré object $ (A^e, q_A) $ of $ \Mod^\mathrm{p}_A $. 
\end{corollary}

\subsection{From schemes with involution to Poincaré structures on module categories}\label{subsection:Poincare_structures_schemes_involution} 
In this section, we show that a scheme with an involution acquires a symmetric monoidal Poincaré structure on its perfect derived category. 
Our basic geometric object of interest consists of a scheme $ X $ endowed with an involution $ \lambda \colon X \to X $ and a good quotient $ X \to Y $.  
We show that, given such data, we may define a Zariski sheaf of commutative rings with $ C_2 $-action (Construction \ref{cons:structure_sheaf_of_Green_functors}) on the quotient $ Y $. 
The values of this new sheaf may be regarded as instances of the Poincaré rings of \S\ref{subsection:Poincare_rings}. 
Then we use the functor of Theorem \ref{thm:calgp_to_poincare_cat} to define a symmetric monoidal Poincaré $ \infty $-category associated to this sheaf. 
We briefly discuss how the ringed topoi with involutions (and their exact quotients) of \cite{azumaya_involution} give rise to symmetric monoidal Poincaré $ \infty $-categories in Remark \ref{rmk:ringed_topoi_with_involution}, and give a `stacky' interpretation of symmetric Poincaré rings in Proposition \ref{prop:symmetric_calgp_stacky_interpretation}. 
	
\begin{recollection} [{\cite[Remark 4.20]{azumaya_involution}}]\label{rec:good_quotient}
    Let $ X $ be a scheme with an involution $ \lambda \colon X \to X $. 
    A map $ p \colon X \to Y $ is called a \emph{good quotient of $ X $ relative to $ \lambda $} if $ p $ is $ C_2 $-invariant and affine and $ p_{\#} \colon \mathcal{O}_Y \to p_{*} \mathcal{O}_X $ induces an isomorphism $ \mathcal{O}_Y \simeq \left(p_{*} \mathcal{O}_X\right)^{C_2} $. 
    A good quotient of $ X $ exists if and only if every $ C_2 $-orbit is contained in an affine open subscheme, in which case it is unique up to isomorphism. 
    Equivalently, a good quotient is a moduli space for the stacky quotient $ [X/C_2] $; \cite[Theorem 5.4]{existence_of_moduli} gives general conditions for existence of good moduli spaces. 
    The assumption on affine orbits is sufficient to apply \cite[Theorem 5.4]{existence_of_moduli} since it reduces to checking $\Theta$-reductivity and $\mathrm{Spec}(\mathbb{Z})$-completeness for quotients of affine schemes, which follow from \cite[Propositions 3.21(2) and 4.44(2)]{existence_of_moduli}.
\end{recollection}
\begin{definition}~\label{defn:Category of good quotients}
    Define a category $ \mathrm{qSch}^{C_2} $ so that
    \begin{itemize}
        \item an object of $ \mathrm{qSch}^{C_2} $ consists of the data of qcqs schemes $ X $ and $ Y $, an involution $ \lambda \colon X \to X $, and a morphism $ p \colon X \to Y $ which exhibits $ Y $ as a \emph{good quotient} of the involution on $ X $ in the sense of Recollection \ref{rec:good_quotient}. 
        \item a morphism from $ (X,\lambda, Y, p) $ to $ (Z,\nu, W, q) $ consists of a $ C_2 $-equivariant morphism $ X \to Z $ and a morphism $ Y \to W $ so that the diagram
        \begin{equation*}
            \begin{tikzcd}
                X \ar[d,"p"] \ar[r] & Z \ar[d,"q"] \\
                Y \ar[r] & W
            \end{tikzcd}
        \end{equation*}
        commutes. 
    \end{itemize} 
    We will consider $ \mathrm{qSch}^{C_2} $ as an $\infty$-category via the $\infty$-categorical nerve, see \cite[Example 1.1.2.6]{HTT}. 
    Write $ \mathrm{Aff}^{C_2} $ for the full subcategory of $ \mathrm{qSch}^{C_2} $ on those objects $ (X,\lambda,Y,p) $ so that $ Y $ is affine. 
    By our assumption on $ p $, $ X $ is also affine whenever $Y$ is. 
\end{definition}
\begin{observation}\label{obs:fixpt_Mackey_functor_as_affine_C2_scheme}
    Suppose $ R $ is a discrete commutative ring with a $ C_2 $-action, and regard $ R $ as a $ C_2 $-Mackey functor via Example \ref{ex:fixpt_Mackey_functor}. 
    Then $ \Spec R \to \Spec (R^{C_2}) $ may be regarded as an object of $ \mathrm{Aff}^{C_2} $. 
    In fact, any object of $ \mathrm{Aff}^{C_2} $ is of this form. 
\end{observation}
\begin{examples}\label{ex:schemes_and_involutions}
    Let $ \mathrm{qSch} $ denote the category of qcqs schemes. 
\begin{enumerate}[label=(\arabic*)]
    \item \label{ex_item:scheme_with_trivial_involution} There is a canonical functor $ \mathrm{qSch} \to \mathrm{qSch}^{C_2} $ which sends $ X \mapsto (X,\id_X,X, \id_X) $. 
    In other words, we may regard an ordinary scheme as being equipped with the trivial involution. 
    \item Write $ \mathrm{qSch}_{\ZZ\left[\frac{1}{2}\right]} $ for the category of qcqs schemes over $ \Spec \ZZ \left[\frac{1}{2}\right] $. 
    Consider the object \[ \left(\Spec \ZZ \left[\frac{1}{2},\sqrt{-1}\right] , \lambda, \Spec \ZZ \left[\frac{1}{2}\right],p \right) \] of $ \mathrm{qSch}^{C_2} $ where $ \Spec \ZZ \left[\frac{1}{2},\sqrt{-1}\right] $ is endowed with the involution $ \lambda: i \mapsto -i $ and $ p $ is induced by the canonical inclusion $ \ZZ \left[\frac{1}{2}\right] \to \ZZ \left[\frac{1}{2},\sqrt{-1}\right]$. 
    For any $ X $, write $ X[\sqrt{-1}]:= X \times_{\Spec \ZZ \left[\frac{1}{2}\right]}\Spec \ZZ \left[\frac{1}{2},\sqrt{-1}\right] $; it inherits an involution $ \lambda_X $ and a projection $ p_X \colon X[\sqrt{-1}] \to X $ via pullback. 
    Then there is a canonical functor $ \mathrm{qSch}_{\ZZ \left[\frac{1}{2}\right]} \to \mathrm{qSch}^{C_2} $ which sends $ X \mapsto \left(X[\sqrt{-1}],\lambda_X,X,p_X\right) $. 
\end{enumerate}
\end{examples}
\begin{remark}\label{remark:restriction_of_schemes_with_involution}
    Suppose $ (X,\lambda, Y, p) $ is an object of $ \mathrm{qSch}^{C_2} $ and $ j \colon U \to Y $ is a flat map. 
    Then $ (X_U, \lambda|_U, U, p|_U) $ is an object of $ \mathrm{qSch}^{C_2} $. 
    Affineness and invariance under the $ C_2 $-action are stable under pullback, so it suffices to show that $ p|_U $ satisfies $ \mathcal{O}_U \simeq \left(p|_U\right)_*(\mathcal{O}_{X_U})^{C_2} $. 
    This follows from the proof of \cite[Theorem 4.35(i)]{azumaya_involution}.  
\end{remark}
\begin{proposition}\label{prop:products_of_C2_schemes}
    Write $ U \colon \mathrm{qSch}^{C_2} \to \mathrm{qSch} $ for the functor so that $ U(X, \lambda, Y, p) = X $.  
    \begin{enumerate}[label=(\arabic*)]
        \item \label{propitem:C2_schemes_monoidal} The category $ \mathrm{qSch}^{C_2} $ has a symmetric monoidal structure $ \boxtimes $ so that $ U $ is symmetric monoidal, where $ \mathrm{qSch} $ is endowed with the product symmetric monoidal structure. 
        \item \label{propitem:C2_schemes_monoidal_cartesian} The symmetric monoidal structure $ \boxtimes $ is cartesian. 
    \end{enumerate}
\end{proposition}
\begin{proof}
    If $ X $, $ Z $ are schemes with involutions $ \lambda_X $, $ \lambda_Z $, then $ \lambda_X \times \lambda_Z $ endows $ X \times Z $ with an involution. 
    To prove \ref{propitem:C2_schemes_monoidal}, it suffices to show that $ X \times Z $ admits a good quotient, as a good quotient is a categorical quotient and is therefore unique up to isomorphism. 
    By \cite[Remark 4.20]{azumaya_involution}, a good quotient exists if and only if every $ C_2 $-orbit is contained in an affine open subscheme. 
    Consider a $ C_2 $-orbit in $ X \times Z $. 
    Its image under the projection $ \pi_1 \colon X \times Z \to X $ (resp. $ \pi_2 \colon X \times Z \to Z $) is contained in an affine open subscheme $ U \subseteq X $ (resp. $ V \subseteq Z $). 
    Thus the orbit under consideration is contained in $ U \times V $, which is affine. 
    
    Now \ref{propitem:C2_schemes_monoidal_cartesian} follows from noting that if $ Y $ is any other scheme with an involution admitting a good quotient, then the projections $ X \times Z \to X $, $ X \times Z \to Z $ induce a bijection between $ C_2 $-equivariant maps $ Y \to X \times Z $ and pairs of $ C_2 $-equivariant maps $ (Y\to X, Y \to Z) $. 
\end{proof}
\begin{construction}\label{cons:structure_sheaf_of_Green_functors}
    Let $ X $ be a scheme with involution $ \lambda : X \xrightarrow{\sim} X $.
    Assume that $ X $ has a \emph{good quotient} $ Y $ in the sense of \cite[Remark 4.20]{azumaya_involution}. 
    We write $ p \colon X \to Y $ for the quotient map.
    Let $ j \colon \Spec A  \simeq U \subseteq Y $ be an affine open subscheme of $ Y $. 
    Because $ p $ is an affine map, the fiber product $ \Spec B := \Spec A \times_{Y} X $ is an affine open of $ X $ which is invariant under the $ C_2 $-action. 
    In particular $ \Spec B $ inherits a $ C_2 $-action from $ X $ (hence so does its ring of functions $ B$).  
    Now $ A \to B $ acquires the structure of a $C_2$-Green functor which we will denote by $ \underline{\mathcal{O}}_{X,\lambda}(U) $. 
    Regarding $ \underline{\mathcal{O}}_{X,\lambda}(U) $ as a $ C_2 $-spectrum, by the isotropy separation sequence, we have an equivalence of $ A $-modules $ \underline{\mathcal{O}}_{X,\lambda}(U)^{\varphi C_2} \simeq \mathrm{cofib} (\mathrm{tr} \colon B_{\mathrm{h}C_2} \to A) $. 
\end{construction}
\begin{lemma} \label{lemma:identify_structure_sheaf_of_Green_func}
    Let $ X $ be a scheme with an involution. 
    Assume that $ X $ has a \emph{good quotient} $ Y $ in the sense of Recollection \ref{rec:good_quotient}, and write $ p \colon X \to Y $ for the quotient map. 
    \begin{enumerate}[label=(\roman*)]
        \item \label{lem_item:structure_sheaf_of_Green_func_is_functor} The assignment of Construction \ref{cons:structure_sheaf_of_Green_functors} lifts to a contravariant functor from (the nerve of) the category of affine opens of $ Y $ to the $\infty $-category of $ C_2 $-$ \EE_\infty $-algebras in $ C_2 $-spectra. 
        Composing with the forgetful functor of Observation \ref{observation:normed_C2_ring_forget_to_Poincare_ring}, $ \underline{\mathcal{O}}_{X,\lambda} $ may also be regarded as a functor valued in Poincaré rings. 
        \item \label{lem_item:structure_sheaf_of_Green_func_is_sheaf} The presheaves $ \underline{\mathcal{O}}_{X,\lambda} $ of \ref{lem_item:structure_sheaf_of_Green_func_is_functor} define Zariski sheaves valued in $ C_2 \EE_\infty\Alg$ and $ \CAlgp $.  
        \item \label{lem_item:structured_pushforward_is_equiv} Write $ p_* \mathcal{O}_X $ for the sheaf of $ \EE_\infty $-$ \mathcal{O}_Y $-algebras (all functors are derived). 
        Then the pushforward $ p_* $ induces an equivalence $ \mathrm{Mod}_{\mathcal{O}_X} \xrightarrow{\sim} \Mod_{p_*\mathcal{O}_X} $. 
        \item \label{lem_item:structure_pushforward_equiv_is_monoidal} The equivalence of \ref{lem_item:structured_pushforward_is_equiv} is symmetric monoidal. 
    \end{enumerate}
\end{lemma} 
\begin{remark}
    $ C_2 $-$ \EE_\infty $-rings are a homotopy coherent version of $ C_2 $-Tambara functors. 
\end{remark}
\begin{proof}[Proof of Lemma \ref{lemma:identify_structure_sheaf_of_Green_func}]
    Part \ref{lem_item:structure_sheaf_of_Green_func_is_functor} follows from a similar argument to \cite[Theorem 5.1]{LYang_normedrings}; functoriality follows from noting that $ \tau_{\geq 0} $ is a functor.  
    Part \ref{lem_item:structure_sheaf_of_Green_func_is_sheaf} follows from Lemma \ref{lemma:limits_of_param_alg_detected_orbitwise} and the fact that there is a limit-preserving functor $ C_2 \EE_\infty\Alg \to \CAlgp $. 
    To prove \ref{lem_item:structured_pushforward_is_equiv} and \ref{lem_item:structure_pushforward_equiv_is_monoidal}, consider a Zariski cover $ \left\{j_i \colon U_i \to Y \right\} $ of $ Y $ by affine opens. 
    Then $ \left\{p^*(j_i) = p \times_Y j_i \colon U_i \times_Y X \to X\right\} $ is a Zariski cover of $ X $ by affine opens. 
    By Zariski descent, there are symmetric monoidal equivalences $ \displaystyle\mathrm{Mod}_{\mathcal{O}_X} \simeq \lim_{p^*(j_i) = p \times_Y j_i \colon U_i \times_Y X \to X} \Mod_{\mathcal{O}_X(U_i \times_Y X)} $ and $ \displaystyle\Mod_{p_*(\mathcal{O}_X)} \simeq \lim_{j_i \colon U_i \to Y} \Mod_{p_*\mathcal{O}_X(U_i)} $, hence the result follows. 
\end{proof}
\begin{lemma}\label{lemma:limits_of_param_alg_detected_orbitwise}
    Let $ K $ be a simplicial set, and let $ f \colon K^{\triangleleft} \to C_2 \EE_\infty\mathrm{Alg}\left(\Spectra^{C_2}\right) $ be a diagram. 
    Then $ f $ is a limit diagram if and only if $ f^e \colon K^{\triangleleft} \to \EE_\infty\mathrm{Alg}(\Spectra) $ and $ f^{C_2} \colon K^{\triangleleft} \to \EE_\infty\mathrm{Alg}(\Spectra) $ are both limit diagrams. 
\end{lemma}
\begin{proof}
    The result follows from the observation that limits in $ \EE_\infty \mathrm{Alg} \left(\Spectra^{C_2}\right) $ are computed in $ \Spectra^{C_2} $. 
\end{proof}
\begin{construction}\label{cons:C2_mod_over_sheaf_of_Green_func}
    Let $(X,\lambda, Y, p)$ be a scheme with good quotient. 
    Consider the composites
    \begin{equation}\label{eq:functor_classifying_poincare_mod_over_sheaf_of_Green_func}
        \begin{split}
            \Mod_{\underline{\mathcal{O}}_{X,\lambda}}^\mathrm{p} \colon & \mathrm{Op}(Y)^\op \xrightarrow{\underline{\mathcal{O}}_{X,\lambda}} \CAlgp\xrightarrow{\Mod^\mathrm{p}} \Catp  \\
            \left(\Mod_{\underline{\mathcal{O}}_{X,\lambda}}^\mathrm{p}\right)^\otimes \colon & \mathrm{Op}(Y)^\op \xrightarrow{\underline{\mathcal{O}}_{X,\lambda}} \CAlgp \xrightarrow{\Mod^\mathrm{p}} \EE_\infty \Mon\left(\Catp\right)  \,,    
        \end{split}    
    \end{equation}
    where $ \Mod^\mathrm{p} $ is the functor of Definition \ref{defn:E_sigma_alg_to_hermitian_cat} (see Theorem \ref{thm:calgp_to_poincare_cat}) and $ \underline{\mathcal{O}}_{X,\lambda} $ is the structure sheaf of Lemma \ref{lemma:identify_structure_sheaf_of_Green_func}. 
    In the notation of Construction \ref{cons:structure_sheaf_of_Green_functors}, this functor sends the affine open $ \Spec A \subseteq Y $ to $ \Mod_B^\omega $ equipped with the Poincaré structure associated to the $ C_2 $-Green functor $ A \to B $, regarded as an invertible module with genuine involution over $ B $ (see \cite[\S3.2, 3.3, esp. Theorem 3.3.1]{CDHHLMNNSI}). 
    Define $ \Mod_{\underline{\mathcal{O}}_{X,\lambda}}^\mathrm{p} $, $ \left(\Mod_{\underline{\mathcal{O}}_{X,\lambda}}^\mathrm{p}\right)^\otimes $ to be the limits in $ \Catp $, $ \EE_\infty\Mon\left(\Catp\right) $, resp. of the functors in (\ref{eq:functor_classifying_poincare_mod_over_sheaf_of_Green_func}); said limits exist by \cite[Proposition 6.1.1]{CDHHLMNNSI}. 

    Similarly, the functor $ \EE_\infty \Alg\left(\Spectra^{C_2}\right) \to \CAT $ given by $ A \mapsto \Mod_A\left(\Spectra^{C_2}\right) $ with the functor of Remark \ref{rmk:Poincare_ring_has_underlying_C2_spectrum_alg} gives rise to
     \begin{equation}\label{eq:functor_classifying_C2_mod_over_sheaf_of_Green_func}
        \Mod_{\underline{\mathcal{O}}_{X,\lambda}}^{C_2} \colon \mathrm{Op}(Y)^\op \xrightarrow{\underline{\mathcal{O}}_{X,\lambda}} \CAlgp\xrightarrow{\Mod_{(-)^L}\left(\Spectra^{C_2}\right)} \CAT \,.
    \end{equation}
    In the notation of Construction \ref{cons:structure_sheaf_of_Green_functors}, this functor sends the affine open $ \Spec A \subseteq Y $ to the category of modules in $ C_2 $-spectra over the $ C_2 $-$ \EE_\infty $-algebra which has underlying $ C_2 $-Green functor $ A \to B $. 
    Define $ \Mod_{\underline{\mathcal{O}}_{X,\lambda}}^{C_2} $ to be the limit in $ \CAT $ of the functor in (\ref{eq:functor_classifying_C2_mod_over_sheaf_of_Green_func}). 
\end{construction} 
\begin{notation}
    Let $ (X,\lambda,Y,p) \in \mathrm{qSch}^{C_2} $. 
    We observe that $ \underline{\mathcal{O}}_{X,\lambda} $ of Lemma \ref{lemma:identify_structure_sheaf_of_Green_func} does in fact depend on the scheme $ Y $ and the quotient map $ p $, but they have been suppressed from notation. 
    This is simply because we found a quadruple subscript to be too unwieldy. 
    Similarly, we will write $ \left(\mathrm{Mod}_{\mathcal{O}_X}^\omega,\Qoppa_{X,\lambda}\right) $ or $ \left(\mathrm{Mod}_{\mathcal{O}_X}^\omega,\Qoppa_{\lambda}\right) $ for the limit of \eqref{eq:functor_classifying_poincare_mod_over_sheaf_of_Green_func}, suppressing $ (Y,p) $ from notation. 
    If the involution $ \lambda $ and $ p $ are both the identity, then we may drop $ \lambda $ entirely (as in Proposition \ref{prop:CHN_comparison}). 
\end{notation}
\begin{observation}\label{obs:symmetric_structure_module_cat}
    Let $ (X,\lambda,Y,p) \in \mathrm{qSch}^{C_2} $ and write $ \left(\mathrm{Mod}_{\mathcal{O}_X}^\omega,\Qoppa_{(\lambda,p)}\right) $ for the limit defined in Construction \ref{cons:C2_mod_over_sheaf_of_Green_func}. 
    Taking the bilinear part of $ \Qoppa_{(\lambda, p)} $, we find that $ \left(\mathrm{Mod}_{\mathcal{O}_X}^\omega,B_{\Qoppa_{(\lambda,p)}}\right) $ is an $ \Einfty $-algebra in $ \Cat^{\mathrm{ps}} $ (see Recollection \ref{rec:symmetric_Poincare_cat}). 
    Now the collection of small stable idempotent-complete symmetric monoidal $ \infty $-categories has a $ C_2 $-action where the generator acts by sending $ \mathcal{D} $ to $ \mathcal{D}^\op $. 
    By \cite[Corollary 7.2.16]{CDHHLMNNSI}, $ B_{\Qoppa_{(\lambda,p)}} $ induces an equivalence $ D_\Qoppa \colon \mathrm{Mod}_{\mathcal{O}_X}^\omega \simeq \mathrm{Mod}_{\mathcal{O}_X}^{\omega, \op} $ which exhibits $ \mathrm{Mod}_{\mathcal{O}_X}^\omega $ as a $ C_2 $-homotopy fixed point of this action. 
\end{observation}
\begin{observation}\label{obs:globalize_relative_norm}
    Let $ X $ be a scheme with an involution, and let $ p \colon X \to Y $ exhibit $ Y $ as a good quotient of $ X $. 
    Assume that $ p $ is affine. 
    The norm functors (resp. relative norm functors) $ N^{C_2}_e $ assemble under Construction \ref{cons:structure_sheaf_of_Green_functors} to a `global' norm functor $ N^{C_2}_Y \colon \Mod_{p_{*}\mathcal{O}_X} \to \Mod_{N^{C_2} p_{\#} \mathcal{O}_X} $ (resp. relative norm functor $ \underline{N}^{C_2}_Y \colon  \Mod_{p_{*}\mathcal{O}_X} \to \Mod_{\underline{\mathcal{O}}_{X,\lambda}}^{C_2} $). 
    Moreover, the functors $ N^{C_2}_Y $ and $ \underline{N}^{C_2}_Y $ are quadratic. 
\end{observation}
\begin{remark}
    Fix $ (X,\lambda,Y,p) $ as before and write $ s \colon \int \mathrm{Pn} \left(\Mod_{\underline{\mathcal{O}}_{X,\lambda}}^\mathrm{p}\right) \to \mathrm{Op}(Y)^\op $ for the cocartesian fibration obtained by taking the Grothendieck construction on the functor $ \mathrm{Pn} \left(\Mod_{\underline{\mathcal{O}}_{X,\lambda}}^\mathrm{p}\right) $ of (\ref{eq:functor_classifying_poincare_mod_over_sheaf_of_Green_func}). 
    A Poincaré object of $ \Mod_{\underline{\mathcal{O}}}^\mathrm{p} $ is a cocartesian section of $ s $. 
    In other words, it is a choice, for each affine open $ \Spec A $ of $ Y $ (same notation as before--write $ p^*\left(\Spec A\right) \simeq \Spec B $), of a pair $ (M,q_M) $ where $ M $ is a perfect $ B $-module and $ q_M $ is a $ N^{C_2}B $-linear map $ \underline{N}^{C_2} M \to (A \to B) $, where $ \underline{N}^{C_2} $ denotes the relative norm $ N^{C_2}(-) \otimes_{N^{C_2}B} (A \to B) $.  
    The bilinear part of $ q_M $ induces a perfect pairing $ M \otimes_B \lambda^*M \to B $, where $ \lambda \colon B \simeq B $ is the involution on $ B $ induced by the involution on $ X $. 
    These are required to glue compatibly. 
    
    Alternatively, a Poincaré object of $ \Mod_{\underline{\mathcal{O}}_{X,\lambda}}^\mathrm{p} $ is a pair $ (M, q_M) $ consisting of an $ \mathcal{O}_X $-module $ M $ and an $ \underline{\mathcal{O}}_{X,\lambda} $-linear map $ q_M \colon \underline{N}^{C_2}_Y (M) \to \underline{\mathcal{O}}_{X,\lambda} $, where $ \underline{N}^{C_2}_Y $ is the relative norm of Observation \ref{obs:globalize_relative_norm}. 
    The bilinear part of $ q_M $ induces a perfect pairing $ M \otimes_{\mathcal{O}_X} \lambda^*M \to \mathcal{O}_X $, where $ \lambda \colon X \simeq X $ is the involution on $ X $.
\end{remark}
\begin{definition}\label{defn:sym_mon_catp_from_good_quotient}
    By Observation \ref{obs:fixpt_Mackey_functor_as_affine_C2_scheme}, Theorem \ref{thm:calgp_to_poincare_cat}, and the functoriality described in Example \ref{ex:fixpt_Mackey_functor}, there are contravariant functors from $ \mathrm{Aff}^{C_2} $ to (symmetric monoidal) Poincaré $ \infty $-categories. 
    Right Kan extension along the inclusion $ \mathrm{Aff}^{C_2,\op} \to \mathrm{qSch}^{C_2,\op} $ defines functors 
    \begin{align*}
        \left(\mathrm{qSch}^{C_2}\right)^{\op} &\to \Catp & & & \left(\mathrm{qSch}^{C_2}\right)^{\op} &\to \EE_\infty\Mon\left(\Catp\right) \\
        \left(X, \lambda, Y, p \right) &\mapsto \Mod^\mathrm{p}_{\underline{\mathcal{O}}_{X,\lambda}} & & &\left(X, \lambda, Y, p \right) &\mapsto \left(\Mod_{\underline{\mathcal{O}}_{X,\lambda}}^\mathrm{p}\right)^\otimes \,.
    \end{align*}
\end{definition} 
\begin{remark}
    The values of the functors of Definition \ref{defn:sym_mon_catp_from_good_quotient} on a fixed $ (X,\lambda,Y,p) $ agree with those of Construction \ref{cons:C2_mod_over_sheaf_of_Green_func}. 
    Say that a family $ \left\{(U_\alpha,\ell_\alpha,V_\alpha,\pi_\alpha) \to (U,\ell,V,\pi) \right\}_{\alpha} $ in $ \mathrm{Aff}^{C_2} $ is a $ C_2 $-Zariski cover if $ \left\{V_\alpha \to V\right\} $ are a Zariski cover in the usual sense and the $ U_\alpha $ are base changed from $ V_\alpha \to V $. 
    Observe that the functor $ \hom_{\mathrm{qSch}^{C_2}}(-, (X,\lambda,Y,p)) $ sends $ C_2 $-Zariski covers in $ \mathrm{Aff}^{C_2} $ to limit diagrams. 
    The result follows from $ C_2 $-Zariski descent for $  \mathrm{Aff}^{C_2,\op} \to \Catp $ (and likewise for $ \Einfty\Mon \left(\Catp\right) $, cf. Lemma \ref{lemma:identify_structure_sheaf_of_Green_func} and Proposition \ref{prop:Poincare_modules_as_etale_sheaf}\ref{prop:Poincare_modules_as_etale_sheaf_affine_spectral}) and a standard cofinality argument. 
\end{remark}
\begin{observation}\label{obs:sheafy_unit_poincare_object}
    It follows from Corollary \ref{cor:unit_poincare_object} (also see Corollary \ref{cor:tensor_unit_as_poincare_object}) that $ \mathcal{O}_X $ admits a canonical lift to $ \int \mathrm{Pn} \left(\Mod_{\underline{\mathcal{O}}_{X,\lambda}}^\mathrm{p}\right) $. 
\end{observation}
\begin{remark}\label{rmk:sym_mon_catp_from_good_quotient_underlying}
    By Lemma \ref{lemma:identify_structure_sheaf_of_Green_func}\ref{lem_item:structured_pushforward_is_equiv}, the functors of Definition \ref{defn:sym_mon_catp_from_good_quotient} participate in commutative diagrams
    \begin{equation*}
        \begin{tikzcd}[column sep=huge]
            \left(\mathrm{qSch}^{C_2}\right)^{\op} \ar[r,"{\Mod^\mathrm{p}_{\underline{\mathcal{O}}_{(-)}}}"] \ar[d,"{\left(X, \lambda, Y, p \right) \mapsto X}"'] & \Catp \ar[d,"{U}"] & \left(\mathrm{qSch}^{C_2}\right)^{\op} \ar[r,"{\left(\Mod_{\underline{\mathcal{O}}}^\mathrm{p}\right)^\otimes}"] \ar[d,"{\left(X, \lambda, Y, p \right) \mapsto X}"'] & \EE_\infty\Mon\left(\Catp\right) \ar[d,"U"] \\
            \mathrm{qSch}^{\op} \ar[r,"{\mathrm{Mod}_{\mathcal{O}}^\omega}"] & \Catex & \mathrm{qSch}^{\op} \ar[r,"{\mathrm{Mod}_{\mathcal{O}}^{\omega,{\otimes}}}"] & \EE_\infty\Mon\left(\Catex\right) \,.		
        \end{tikzcd}
    \end{equation*}
\end{remark}
\begin{remark}\label{rmk:ringed_topoi_with_involution}
    Suppose given $ \left(X, \mathcal{O}_X\right) $ a ringed topos with involution in the sense of \cite[Definition 4.2]{azumaya_involution}. 
    Let $ \left(Y, \mathcal{O}_Y \right) $ be a ringed topos with trivial involution and a map $ p \colon \left(X, \mathcal{O}_X\right) \to \left(Y, \mathcal{O}_Y\right) $ exhibiting $ \left(Y, \mathcal{O}_Y\right) $ as an \emph{exact quotient} of $ \left(X, \mathcal{O}_X\right) $ by the given $ C_2 $-action in the sense of \cite[Definition 4.18]{azumaya_involution}. 
    Then by Definition 4.18(E1) \emph{loc. cit.}, the diagram $ p_\# \colon \mathcal{O}_Y \to p _* \mathcal{O}_X $ is a sheaf of $ C_2 $-Green functors on $ Y $ (compare Construction \ref{cons:structure_sheaf_of_Green_functors}), where $ p _* \mathcal{O}_X $ has the $ C_2 $-action induced by $ p _* \lambda $. 
    By the same argument as Lemma \ref{lemma:identify_structure_sheaf_of_Green_func}, $ \mathcal{O}_Y \to p _* \mathcal{O}_X $ in fact refines to a sheaf of Poincaré rings $ \underline{\mathcal{O}}_{X,\lambda} $ on $ Y $. 
    In particular, using Example \ref{ex:fixpt_Mackey_functor} and the functors $ \Mod^\mathrm{p} $ and $ \left(\Mod^\mathrm{p}\right)^\otimes $ of Theorem \ref{thm:calgp_to_poincare_cat}, we have sheaves of (symmetric monoidal) Poincaré $ \infty $-categories on $ Y $. 

    More generally, we observe that this construction does not use the full power of the definition of an exact quotient, just the first part \cite[Definition 4.18(E1)]{azumaya_involution}. 
\end{remark}
\begin{remark}\label{rmk:global_pncat_shift_twist}
    Many results of \cite[\S3]{CDHHLMNNSI} admit globalizations. 
    For instance, observe that the limit defining $ \Mod_{\underline{\mathcal{O}}}^{C_2} $ (see (\ref{eq:functor_classifying_C2_mod_over_sheaf_of_Green_func}) and following discussion) can be computed in $ \Spectra^{C_2} $-linear stable $ \infty $-categories. 
    In particular, we may smash any (sheaf of) $ \underline{\mathcal{O}}_{X,\lambda} $-modules in $ C_2 $-spectra with an object in $ \Spectra^{C_2} $ (alternatively, see \cite[\S7.4]{CDHHLMNNSI}). 
    Naturality of the equivalence \cite[p.63, esp. Example 3.5.14i)]{CDHHLMNNSI} implies an equivalence $ \sphere^{1-\sigma} \otimes \underline{\mathcal{O}}^e_{X,\lambda} \simeq \sphere^{\sigma-1} \otimes \underline{\mathcal{O}}^{e}_{X,\lambda} $ of $ \underline{\mathcal{O}}^e_{X,\lambda} $-modules in $ \Spectra^{BC_2} $. 
    
    Suppose $ (X,\lambda,Y,p) $ is such that the map to the terminal object factors through $ \Spec \ZZ \left[\frac{1}{2}\right] $ (regarded as an object of $ \mathrm{qSch}^{C_2} $ with the trivial involution). 
    Then it follows from \cite[Example 3.2.5]{CDHHLMNNSI} that the canonical maps $ \left(\Mod^\mathrm{p}_{(X,\lambda,Y,p)}\right)^q \to \Mod^\mathrm{p}_{(X,\lambda,Y,p)} \to \left(\Mod^\mathrm{p}_{(X,\lambda,Y,p)}\right)^s $ are equivalences. 
    Now naturality of \cite[Corollary 3.5.16]{CDHHLMNNSI} implies that the loop functor $ \Omega $ refines to an equivalence of Poincaré $ \infty $-categories $ \left(\mathrm{Mod}_{\mathcal{O}_X}^\omega,\Qoppa_{\lambda}^{[2]}\right) \simeq \left(\mathrm{Mod}_{\mathcal{O}_X}^\omega,\Qoppa_{-\lambda}\right) $. 
\end{remark}
The results of this section show that for a suitably general category of schemes with $ C_2 $-action, there is a contravariant functor sending such a scheme to its perfect module category equipped with a Poincaré structure; in particular, by \cites[\S4]{CDHHLMNNSII} we obtain a functorial definition of hermitian K-theory or Grothendieck--Witt spectra for schemes with involution extending the results of \cite{CHN2024} which only considers schemes with trivial involution. 
Algebraic K-theory has pushforwards along proper lci maps; Grothendieck--Witt theory has pushforwards associated to proper lci maps equipped with the data of an \emph{orientation}. 
Calmès--Harpaz--Nardin defined oriented pushforwards in \cite[\S5.1]{CHN2024}; we expect their construction to generalize readily to our setting but omit the details here. 
In following example, we explain how Calmès--Harpaz--Nardin's work on oriented pushforwards gives rise to oriented fundamental classes in Grothendieck--Witt theory. 
It may be regarded as an algebraic counterpart to Example \ref{ex:poincare_duality_objects}. 
\begin{example}\label{ex:Grothendieck_Serre_duality_objects}
    Let $ f \colon X \to Y $ be a smooth proper morphism of schemes over $ \ZZ\left[\frac{1}{2}\right]$ and assume that the cotangent complex $ \mathbb{L}_{X/Y} \simeq \Omega^1_{X/Y} $ is equivalent to a vector bundle of rank $ 2r $. 
    Suppose $ f $ is equipped with the data of an \emph{orientation}, i.e. a line bundle $ \mathfrak{o} $ on $ X $ and an equivalence $ \mathfrak{o}^{\otimes 2} \simeq \Omega^{2r}_{X/Y} $. 
    (Here we write $ \Omega^{2r}_{X/Y} $ for the invertible sheaf on $ X $ in the heart of the standard t-structure, i.e. $ f^{!}\mathcal{O}_Y \simeq \Omega^{2r}_{X/Y}[2r] $.) 
    Let us write $ \left(\mathrm{Mod}_{\mathcal{O}_Y}^\omega,\Qoppa_{\mathcal{O}}\right) $ for the image of $ Y $ under the functor of Remark \ref{rmk:modp_of_scheme_trivial_involution} (also see Proposition \ref{prop:CHN_comparison}). 
    Then $ f_*(\mathfrak{o}) $ admits a canonical lift to a Poincaré object of $ \left(\mathrm{Mod}_{\mathcal{O}_Y}^\omega,\Qoppa_{\mathcal{O}}\right) $ if $ r $ is even and $ \left(\mathrm{Mod}_{\mathcal{O}_Y}^\omega,\Qoppa_{-\mathcal{O}}\right) $ if $ r $ is odd (see Remark \ref{rmk:global_pncat_shift_twist} for definition of $ \Qoppa_{-\mathcal{O}} $). 
    
    To see this, take $ N = \mathcal{O}_X $, $ M = \mathfrak{o}  $, $ \mathcal{L} = \mathcal{O}_Y[-2r] $, and $ \tau $ to be induced by the given orientation in \cite[Construction 5.1.3]{CHN2024}; it follows from Lemma 5.1.5 \emph{ibid.} that this promotes the pushforward $ f_*(- \otimes \mathfrak{o}) $ to a Poincaré functor.  
    Now by Observation \ref{obs:sheafy_unit_poincare_object}, $ \mathcal{O}_X $ can be regarded as a Poincaré object of $ \left(\mathrm{Mod}_{\mathcal{O}_X}^\omega,\Qoppa_{\mathcal{O}}\right) $, and taking its image under the pushforward gives a Poincaré object of $ \left(\mathrm{Mod}_{\mathcal{O}_Y}^\omega,\Qoppa_{\mathcal{O}}^{[2r]}\right) $. 
    We may regard this as a Poincaré object of $\left(\mathrm{Mod}_{\mathcal{O}_Y}^\omega,\Qoppa_{\pm\mathcal{O}}\right) $ via Remark \ref{rmk:global_pncat_shift_twist}. 
    Now such a Poincaré object defines a class in the (anti-)symmetric Grothendieck--Witt theory of $ Y $ by \cite[\S4]{CDHHLMNNSII}. 

    Suppose the relative dualizing sheaf is trivial, i.e. $ \Omega^{2r}_{X/Y} \simeq \mathcal{O}_{X} $. 
    Unraveling definitions, the $2r$-shifted perfect pairing on $ f_* \mathcal{O}_X $ recovers the pairing from Grothendieck duality (Serre duality when $ Y $ is the spectrum of a field). 
\end{example}
\begin{remark}\label{rmk:modp_of_scheme_trivial_involution}
    Composing Example \ref{ex:schemes_and_involutions}\ref{ex_item:scheme_with_trivial_involution} with the functors of Definition \ref{defn:sym_mon_catp_from_good_quotient}, we obtain a functor $ \mathrm{qSch}^{\op} \to \EE_\infty\Mon\left(\Catp\right) $. 
\end{remark} 
\begin{proposition}\label{prop:sym_mon_catp_good_quotient_lax_monoidal}
    The functor $ \left(\Mod^\mathrm{p}\right)^\otimes $ of Definition \ref{defn:sym_mon_catp_from_good_quotient} admits a canonical lax symmetric monoidal structure, where $ \mathrm{qSch}^{C_2,\mathrm{op}} $ is endowed with the symmetric monoidal structure of Proposition \ref{prop:products_of_C2_schemes}. 
\end{proposition}
\begin{corollary}\label{cor:from_sch_good_quotient_to_Z_linear_Poincare_cat}
    The functor $ \left(\Mod^\mathrm{p}\right)^\otimes $ of Definition \ref{defn:sym_mon_catp_from_good_quotient} admits a canonical lift 
    \begin{equation*}
        \left(\mathrm{qSch}^{C_2}\right)^{\op} \to \EE_\infty\Mon\left(\Mod_{\left(\Mod_{\ZZ}^\omega,\Qoppa^{\mathrm{gs}}\right)}\left(\Catp\right)\right) \,.
    \end{equation*}
\end{corollary}
\begin{proof}
    The scheme $ \Spec \ZZ $ endowed with the trivial action is the initial object in $ \left(\mathrm{qSch}^{C_2}\right)^{\op} $. 
    In view of Proposition \ref{prop:sym_mon_catp_good_quotient_lax_monoidal}, it suffices to observe that under the functor $ \left(\Mod^\mathrm{p}\right)^\otimes $ of Definition \ref{defn:sym_mon_catp_from_good_quotient}, $ \left(\Spec \ZZ, \id\right) $ is sent to $\left(\Mod_{\ZZ}^\omega,\Qoppa^{\mathrm{gs}} \right) $. 
\end{proof}
\begin{proof}[Proof of Proposition \ref{prop:sym_mon_catp_good_quotient_lax_monoidal}]
    Note that the symmetric monoidal structures on $ \mathrm{qSch}^{C_2,\mathrm{op}} $ and $ \EE_\infty\Mon\left(\Catp\right) $ are cocartesian, the former by Proposition \ref{prop:products_of_C2_schemes}\ref{propitem:C2_schemes_monoidal_cartesian} and the latter by \cite[Proposition 3.2.4.7]{LurHA}. 
    The result now follows from \cite[Proposition 2.4.3.8]{LurHA}. 
\end{proof}
\begin{recollection}\label{recollection:CHN_Poincare_module_cats}
    Letting $ S = \Spec \ZZ $ and $ \mathcal{L} = \mathcal{O}_S $ in \cite[(3.4.2)]{CHN2024} defines a functor 
    \begin{equation}
        \begin{split}
            \mathrm{qSch}^{\op} &\to \EE_\infty\Mon\left(\Catp\right) \\
            X & \mapsto \left(\mathrm{Mod}_{\mathcal{O}_X}^\omega,\Qoppa^{\geq 0}_{\mathcal{O}_X}\right) 
        \end{split}
    \end{equation} 
    where the functor lifts to symmetric monoidal Poincaré $ \infty $-categories by the discussion immediately following Proposition 3.4.3 \emph{loc. cit.} 
\end{recollection}
\begin{proposition}\label{prop:CHN_comparison}
    The functor $ \left(\mathcal{D}^\mathrm{p}(-),\Qoppa^{\geq 0}_{\mathcal{O}}\right) $ defined by Calmès--Harpaz--Nardin of Recollection \ref{recollection:CHN_Poincare_module_cats} agrees with the functor of Remark \ref{rmk:modp_of_scheme_trivial_involution}. 
\end{proposition}
\begin{proof}
    We will write $ \Mod^{pg} $ for the restriction of $ \Mod^\mathrm{p} $ to $ \mathrm{qSch}^\op $ of Remark \ref{rmk:modp_of_scheme_trivial_involution}. 
    Recall that there is a forgetful functor $ F \colon \Catp \to \Cat^{\mathrm{ps}} $ where the latter $ \infty $-category is \cite[Definition 7.2.13]{CDHHLMNNSI}. 
    Note that $ F $ has a right adjoint $ G $ \cite[Proposition 7.2.18]{CDHHLMNNSI}, and by \cite[Corollary 3.1.3]{CHN2024} it participates in a symmetric monoidal adjunction, i.e. $ F $ is symmetric monoidal and $ G $ is lax symmetric monoidal. 
    Therefore, there is an induced adjunction $ F^\otimes \colon \EE_\infty\Mon\left(\Catp\right) \rlarrows \EE_\infty\Mon\left(\Cat^{\mathrm{ps}}\right) \colon G^\otimes $. 
    By Remark \ref{rmk:sym_mon_catp_from_good_quotient_underlying}, the functor $ U \circ F^\otimes \circ \Mod^{pg} $ is valued in $ \EE_\infty \Mon\left(\mathrm{Rig}^{\mathrm{ex}}\right) \subseteq \EE_\infty \Mon\left(\Catex\right) $ and we have an equivalence $ U \circ F^\otimes \circ \Mod^{pg} \simeq U \circ F^\otimes \circ \left(\mathcal{D}^\mathrm{p}(-),\Qoppa^{\geq 0}_{\mathcal{O}}\right) $ of functors $ \mathrm{qSch}^\op \to \EE_\infty\Mon\left(\mathrm{Rig}^{\mathrm{ex}}\right) $. 
    By definition of $ \Mod^\mathrm{p} $, the bilinear part of the Poincaré structure on $ \Mod^{pg}_X $ is given by $ (E,E') \mapsto \hom_{\mathcal{O}_X}\left(E \otimes_{\mathcal{O}_X} E', \mathcal{O}_X\right) $. 
    It follows that $ F^\otimes \circ \Mod^{\mathrm{pg}} $ factors through $ \EE_\infty \Mon\left(\mathrm{Rig}^{\mathrm{ex}}\right) \to \EE_\infty\Mon\left(\Cat^{\mathrm{ps}}\right) $ of \cite[18]{CHN2024}, and we have an equivalence $ F^\otimes \circ \Mod^{pg} \simeq F^\otimes \circ \left(\mathcal{D}^\mathrm{p}(-),\Qoppa^{\geq 0}_{\mathcal{O}}\right) $ of functors $ \mathrm{qSch}^\op \to \EE_\infty\Mon\left(\Cat^{\mathrm{ps}}\right) $.  
    By adjunction, we obtain a natural transformation $ \Mod^{pg} \to G^\otimes \circ F^\otimes \circ \left(\mathcal{D}^\mathrm{p}(-),\Qoppa^{\geq 0}_{\mathcal{O}}\right) $ of functors $ \mathrm{qSch}^{\op} \to \EE_\infty\Mon\left(\Catp\right) $. 
    It remains to show that for each qcqs scheme $ X $, taking linear parts induces an equivalence $ \Qoppa^{pg}_X \xrightarrow{\sim} \left(\Qoppa_{\mathcal{O}_X}^s\right)^{\geq 0} $, where $ \left(-\right)^{\geq m} $ denotes the right adjoint of \cite[Lemma 3.4.2]{CHN2024}, using the standard t-structure on $ \mathcal{D}^\mathrm{p}(X) $ (see Remark A.3.3 \emph{loc. cit.}). 
    This follows from the fact that for any fixed point $ C_2 $-Mackey functor $ \underline{M} $ associated to an abelian group $ M $ with $ C_2 $-action, the map $ \underline{M}^{\varphi C_2} \to \underline{M}^{\mathrm{t}C_2} $ exhibits $ \underline{M}^{\varphi C_2} $ as a connective cover of $ \underline{M}^{t C_2} $. 
\end{proof}

Given a scheme $ X $ with involution $ \lambda $ and a good quotient $ p: X \to Y $, we defined a sheaf of Poincaré rings on $ Y $ in Construction \ref{cons:structure_sheaf_of_Green_functors} (also see Lemma \ref{lemma:identify_structure_sheaf_of_Green_func}), which we denoted by $ \underline{\mathcal{O}}_{X,\lambda} $. 
In the notation of $ C_2 $-genuine equivariant homotopy theory, we have $ \underline{\mathcal{O}}_{X,\lambda}^{C_2} = \mathcal{O}_Y $ as a sheaf of $ \EE_\infty $ rings and $ \underline{\mathcal{O}}^e_{X,\lambda} \simeq p_* \mathcal{O}_X $ as a sheaf valued in $ \EE_\infty \Alg^{BC_2} $. 
The functor $ (-)^{C_2} : \Spectra^{C_2} \to \Spectra $ fails to be monoidal for essentially the same reason that the $ C_2 $-fixed points of the tensor product of two abelian groups with $ C_2 $-action is not equal to the tensor product of their fixed points. 
There is a functor $ (-)^{\varphi C_2}: \Spectra^{C_2} \to \Spectra $ which is in some sense a universal approximation to $ (-)^{C_2} $ which is monoidal; this is referred to as the \emph{geometric fixed points}. 
By the symmetric monoidal récollement of $ C_2 $-spectra, there is a pullback square of sheaves of $ \EE_\infty $-rings 
\begin{equation*}
\begin{tikzcd}
    \underline{\mathcal{O}}^{C_2}_{X,\lambda} \ar[d] \ar[r] & \underline{\mathcal{O}}^{\varphi C_2}_{X,\lambda} \ar[d] \\
    \left(\underline{\mathcal{O}}^{e}_{X,\lambda}\right)^{hC_2} \ar[r] & \underline{\mathcal{O}}^{t C_2}_{X,\lambda}  \,.
\end{tikzcd}
\end{equation*}
The following proposition suggests a geometric interpretation of $ \underline{\mathcal{O}}^{\varphi C_2}_{X,\lambda} $. 
\begin{notation}\label{ntn:good_quotient_branch_locus}
    Let $ X $ be a scheme with a given involution and suppose that $ X $ has a \emph{good quotient} $ Y $ in the sense of \cite[Remark 4.20]{azumaya_involution}; write $ p \colon X \to Y $ for the quotient map. 
    Let $ i \colon U \subseteq Y $ be the largest open subscheme on which $ p|_{X_U} $ is quadratic étale \cite[Proposition 4.45]{azumaya_involution}. 
    Write $ Z(p) $ for the closed complement of $ U $ regarded as a topological space, and let $ j \colon Z(p) \to Y $ denote the inclusion. 
    The subspace $ Z(p) $ is referred to as the \emph{branch locus} in \cite{azumaya_involution}. 
    In particular, this notion of branch locus includes loci where the map $p$ is an isomorphism.
\end{notation}
\begin{proposition}\label{prop:geomfixpt_C2_structure_sheaf}
    Let $ X $ be a scheme with a given involution and suppose that $ X $ has a \emph{good quotient} $ Y $ in the sense of \cite[Remark 4.20]{azumaya_involution}; write $ p \colon X \to Y $ for the quotient map. 
    Recall the notation $ \underline{\mathcal{O}}_{X,\lambda} $ for the sheaf of Poincaré rings on $ Y $ of Construction \ref{cons:structure_sheaf_of_Green_functors} (also see Lemma \ref{lemma:identify_structure_sheaf_of_Green_func}). 
    Assume further that $ Y $ is regular Noetherian and qcqs of finite Krull dimension. 
    Then
    \begin{enumerate}[label=(\arabic*)]
        \item \label{propitem:geomfixpt_C2_structure_sheaf_ramification_locus} $ \underline{\mathcal{O}}^{\varphi C_2}_{X,\lambda} $ is in the essential image of $ j_* \colon \mathrm{Shv}_{\mathrm{Zar}}(Z(p); \Spectra) \to \mathrm{Shv}_{\mathrm{Zar}}(Y;\Spectra) $ (Notation \ref{ntn:good_quotient_branch_locus}). 
        In other words, there exists a sheaf $ \mathcal{Q} $ of $ \EE_\infty $-rings on the branch locus $ Z(p) $ so that $ j_* \mathcal{Q} \simeq \underline{\mathcal{O}}^{\varphi C_2}_{X,\lambda} $. 
        \item \label{propitem:geomfixpt_C2_structure_sheaf_away_from_2} The support of $ \mathcal{Q} $ intersects trivially with the open subscheme $ Y \times_{\Spec \ZZ} \Spec \ZZ \left[\frac{1}{2}\right] $. 
    \end{enumerate}
\end{proposition}
\begin{remark}
    The proof of Proposition \ref{prop:geomfixpt_C2_structure_sheaf} in fact shows that the \emph{hypercompletion} of $ \underline{\mathcal{O}}^{\varphi C_2} $ is supported on $ Z(p) $ even without the assumption that $ Y $ be of finite Krull dimension. 
\end{remark}
\begin{remark}
    The sheaf $ \mathcal{Q} $ identified in Proposition \ref{prop:geomfixpt_C2_structure_sheaf} does \emph{not} agree with the structure sheaf on the branch locus discussed immediately before \cite[Proposition 4.47]{azumaya_involution}. 
    Indeed, when $ X = \Spec \mathbb{Z} $ with the trivial action, $ \mathcal{Q}(\Spec \mathbb{Z}) $ has unbounded above homotopy groups. 
    When $ X = \Spec R $ for $ R $ a commutative ring with trivial $ C_2 $-action, we may regard $ \mathcal{Q} $ as a measure of the discrepancy between the homology of the classifying space $ BC_2 $ with coefficients in $ R $ and the point.  
    Thus, \ref{propitem:geomfixpt_C2_structure_sheaf_away_from_2} reflects the fact that $ \widetilde{H}_*\left(BC_2; \mathbb{Z} \left[\frac{1}{2}\right]\right) \simeq 0 $. 
\end{remark}
\begin{proof} [Proof of Proposition \ref{prop:geomfixpt_C2_structure_sheaf}]
    The proof of \ref{propitem:geomfixpt_C2_structure_sheaf_ramification_locus} proceeds via a series of reductions followed by a computation. 
    In the following, we will omit the subscript from $ \underline{\mathcal{O}}_{X,\lambda} $. 
    Recall that the open-closed decomposition of $ Y $ induces a symmetric monoidal récollement
    \begin{equation*}
        \mathrm{Shv}_{\mathrm{Zar}}(U;\Spectra) \xleftarrow{i^*} \mathrm{Shv}_{\mathrm{Zar}}(Y;\Spectra) \xrightarrow{j^*} \mathrm{Shv}_{\mathrm{Zar}}(Z(p);\Spectra) 
    \end{equation*}
    of sheaves of spectra (cf. \cite[Exposé VIII Corollaire 6.4]{SGA4}). 
    Therefore, to show that $ \underline{\mathcal{O}}^{\varphi C_2} $ is in the essential image of $ j_* $, it suffices to show that $ i^*\left(\underline{\mathcal{O}}^{\varphi C_2}\right) \simeq 0 $ as a sheaf on $ U $. 
    
    By \cite[Theorem 1.7]{MR4296353}, the sheaf $ \underline{\mathcal{O}}^{\varphi C_2} $ is hypercomplete. 
    In view of \cite[Proposition 4.45]{azumaya_involution}, it suffices to show that if $ y $ is a point in $ U $, then $ \underline{\mathcal{O}}^{\varphi C_2}_{y} = 0 $. 
    By definition of the $ C_2 $-structure sheaf, we have that $ \underline{\mathcal{O}}_{y} $ is given by the $ C_2 $-Mackey functor corresponding to the pullback of $ p $ along the inclusion $ \Spec \mathcal{O}_{Y,y} \to Y $. 
    This is given by a map of rings $ \Gamma(p,y) \colon A = \Gamma\mathcal{O}_{Y,y} \to B $ where $ B $ has an involution $ \lambda $ and $ A = B^{\lambda} $ is the fixed subring. 
    By definition of geometric fixed points, $ \underline{\mathcal{O}}^{\varphi C_2}_{y} = \tau_{\geq 0} \left(B^{\mathrm{t}C_2}\right) $. 
    Since $ y \in U $, the map $ \Gamma(p,y) $ is quadratic étale. 
    By our assumption on $ Y $, $ A $ is a regular Noetherian local ring with maximal ideal $ \mathfrak{m}_A $ (therefore $ B $ is semilocal by \cite[Proposition 3.15]{azumaya_involution}). 
    Now to prove the result, it suffices to show that $ B^{\mathrm{t}C_2} = 0 $. 
    
    Since $ A $ is regular and Noetherian, $ \mathfrak{m}_A $ is finitely generated and $ A/\mathfrak{m}_A $ is a perfect $ A $-module. 
    By \cite[Proposition 4.13]{MR1879003}, there is a recollement of $ A $-modules in terms of (derived) $ \mathfrak{m}_A $-complete and $ A[\mathfrak{m}_A^{-1}] $-modules. 
    It suffices to show that $ \left(B_{\mathfrak{m}_A}^{L\wedge}\right)^{\mathrm{t}C_2} = 0 $ and $ \left(B[\mathfrak{m}_A^{-1}]\right)^{\mathrm{t}C_2} = 0 $, where we write $ (-)_{\mathfrak{m}_A}^{L\wedge} $ for the derived completion. 
    Notice that there is a canonical map $ (-)_{\mathfrak{m}_A}^{L\wedge}  \to (-)_{\mathfrak{m}_A}^{\wedge} = \pi_0 (-)_{\mathfrak{m}_A}^{L\wedge} $ from the derived completion to the ordinary completion. 
    
    We claim that $ \left(B^\wedge_{\mathfrak{m}_A}\right)^{\mathrm{t}C_2} \simeq 0 $ implies that $ \left(B^{L\wedge}_{\mathfrak{m}_A}\right)^{\mathrm{t}C_2} \simeq 0 $, in other words that it suffices to show that the Tate construction vanishes on the underived completion. 
    To see this, note that by \cite[\href{https://stacks.math.columbia.edu/tag/091N}{Tag 091N} Lemma 15.92.21]{stacks}, the derived completion has finite cohomological dimension. 
    In particular, we can write $ B_{\mathfrak{m}_A}^{L\wedge} $ as a finite extension of $ \pi_0 B_{\mathfrak{m}_A}^{L\wedge} \simeq B_{\mathfrak{m}_A}^{\wedge} $-modules. 
    Since $ (-)^{\mathrm{t}C_2} $ is exact and lax symmetric monoidal, we can write $ \left(B_{\mathfrak{m}_A}^{L\wedge}\right)^{\mathrm{t}C_2} $ as a finite extension of $ \left(B_{\mathfrak{m}_A}^{\wedge}\right)^{\mathrm{t}C_2} $-modules. 
    Thus if $ \left(B_{\mathfrak{m}_A}^{\wedge}\right)^{\mathrm{t}C_2} $ vanishes, then so must any module over it, and we conclude that $ \left(B_{\mathfrak{m}_A}^{\wedge}\right)^{\mathrm{t}C_2} \simeq 0 $ implies that $ \left(B_{\mathfrak{m}_A}^{L\wedge}\right)^{\mathrm{t}C_2} \simeq 0 $. 
    
    By \cite[Propositions 3.4 \& 3.15]{azumaya_involution}, $ B \mathfrak{m}_A = J \subseteq B $, where $ J $ denotes the Jacobson radical of $ B $. 
    Writing $ B/J^i $ for the \emph{underived} quotient, note that the tower $ \left\{B/J^i\right\}_{i \geq 1} $ is Mittag--Leffler. 
    It follows that the 1-categorical limit defining the underived completion in fact exhibits $ B^\wedge_{\mathfrak{m}_A} $ as the homotopy limit $ \lim_i B/J^i $. 		
    Now observe that $ \left(B/J^i\right)^{\mathrm{t}C_2} \simeq 0 $ for all $ i $ implies that $ \left(\lim_i B/J^i\right)^{\mathrm{t}C_2} \simeq 0 $. 
    Since homotopy fixed points commute with arbitrary limits, it suffices to show that $ B_{\mathfrak{m}_A}^{\wedge} \simeq \lim_i B/J^i $ induces an equivalence $ \left(B_{\mathfrak{m}_A}^{\wedge}\right)_{\mathrm{h}C_2} \to \lim_i \left(B/J^i\right)_{\mathrm{h}C_2} $. 
    This is true because the $ B/J^i $ are uniformly bounded below.  
    
    Thus it suffices to show that $ \left(B[\mathfrak{m}_A^{-1}]\right)^{\mathrm{t}C_2} = 0 $ and $ \left(B/J^i\right)^{\mathrm{t}C_2} $ is zero for each $ i \geq 1 $. 
    Since $ (-)^{\mathrm{t}C_2} $ is exact and lax symmetric monoidal and each $ B/J^i $ can be written as an extension of finitely many $ B/J $-modules, it suffices to show that $ \left(B[\mathfrak{m}_A^{-1}]\right)^{\mathrm{t}C_2} $ and $ \left(B/J\right)^{\mathrm{t}C_2} $ are zero. 
    
    Now observe that $ A/\mathfrak{m}_A $ (resp. $ A[\mathfrak{m}_A^{-1}] $-algebra) is a field and $ B/J $ (resp. $ B[\mathfrak{m}_A^{-1}] $) is a quadratic étale $ A/\mathfrak{m}_A $-algebra (resp. $ A[\mathfrak{m}_A^{-1}] $-algebra). 
    By \cite[Proposition 3.4(ii)]{azumaya_involution}, $ B/J $ (resp. $ B[\mathfrak{m}_A^{-1}] $) is either a separable quadratic field extension of $ A/\mathfrak{m}_A $-algebra (resp. $ A[\mathfrak{m}_A^{-1}] $-algebra), or it is isomorphic to $ \prod_{C_2} A/\mathfrak{m}_A $ (resp. $ \prod_{C_2} A[\mathfrak{m}_A^{-1}] $). 
    In the latter case, the action of $ C_2 $ on $ B/J $ (resp. $ B[\mathfrak{m}_A^{-1}] $) is manifestly free, hence $ (B/J)^{\mathrm{t}C_2} = 0 $ (resp. $ B[\mathfrak{m}_A^{-1}]^{\mathrm{t}C_2} = 0 $). 
    Suppose instead that $ B/J $ (resp. $ B[\mathfrak{m}_A^{-1}] $) is a separable quadratic field extension of $ A/\mathfrak{m}_A $-algebra (resp. $ A[\mathfrak{m}_A^{-1}] $-algebra). 
    By \cite[Proposition 3.4(ii)]{azumaya_involution}, $ \lambda \otimes_{B} B/J $ (resp. $ \lambda \otimes_B B[\mathfrak{m}_A^{-1}] $) is nontrivial, hence by \cite[Lemma 9.21.2, Tag 09DU]{stacks} the extension $ A/\mathfrak{m}_A \to B/J $ (resp. $ A[\mathfrak{m}_A^{-1}] \to B[\mathfrak{m}_A^{-1}] $) is Galois. 
    Since $ C_2 $ acts freely on $ B/J $ as an $ A/\mathfrak{m}_A $-module by the normal basis theorem, $ (B/J)^{\mathrm{t}C_2} = 0 $ (resp. $ B[\mathfrak{m}_A^{-1}]^{\mathrm{t}C_2} =0$). 
    
    Part \ref{propitem:geomfixpt_C2_structure_sheaf_away_from_2} follows immediately from \cite[Lemma I.2.9]{NS}. 
\end{proof}  
The following result gives a `stacky' interpretation of symmetric Poincaré rings. 
\begin{construction}\label{cons:stacky_quotient_calgp_sheaf}
    Let $ X $ be a scheme with an involution $ \lambda \colon X\xrightarrow{\sim} X $, and write $ \pi \colon X \to [X/C_2] $ where $ [X/C_2] $ is the stacky quotient (cf. \cite[Construction 4.41]{azumaya_involution}). 
    Write $ [X/C_2]_{\mathrm{\acute{e}t}} $ for the étale topos of $ [X/C_2] $. 
    Applying Remark \ref{rmk:ringed_topoi_with_involution} (see Proposition 4.42(v) \emph{loc. cit.}), there is a sheaf $ \underline{\mathcal{O}}_{[\lambda]} := \left(\mathcal{O}_{[X/C_2]} \to \pi_*\mathcal{O}_X\right) $ of Poincaré rings on the quotient stack $ [X/C_2] $. 
    Similarly to Construction \ref{cons:structure_sheaf_of_Green_functors}, we have functors $ [X/C_2]_{\mathrm{\acute{e}t}} \to \Catp $ and $ [X/C_2]_{\mathrm{\acute{e}t}} \to \EE_\infty\Mon\left(\Catp\right) $. 
    Since $ \Catp $ admits all small limits, we may take limits to obtain a symmetric monoidal Poincaré $ \infty $-category $ \left(\perf_{[X/C_2]},\Qoppa_\lambda\right)$. 
\end{construction}
\begin{proposition}\label{prop:symmetric_calgp_stacky_interpretation}	
    Let $ X $ be a scheme with an involution $ \lambda $, and keep the notation of Construction \ref{cons:stacky_quotient_calgp_sheaf}. 
\begin{enumerate}[label=(\arabic*)]
    \item \label{propitem:stacky_quotient_symmetric_structure_sheaf} Suppose $ X = \Spec R $ is affine. 
    Then there is an equivalence of Poincaré rings 
    \begin{equation*}
        R\Gamma\left([X/C_2]_{\mathrm{\acute{e}t}};\underline{\mathcal{O}}_{[\lambda]}\right) \simeq R^s \,,
    \end{equation*}
    where $ R^s $ is the symmetric Poincaré structure of Example \ref{example:tate_poincare_structure}. 
    \item \label{propitem:stacky_quotient_symmetric_poincare_modules} The Poincaré $ \infty $-category $ \left(\perf_{[X/C_2]},\Qoppa_\lambda\right)$ is symmetric in the sense of Recollection \ref{rec:symmetric_Poincare_cat}. 
\end{enumerate}
\end{proposition}
\begin{proof}
    Observe that \ref{propitem:stacky_quotient_symmetric_poincare_modules} follows from \ref{propitem:stacky_quotient_symmetric_structure_sheaf}, the fact that symmetric Poincaré rings give rise to symmetric Poincaré $ \infty $-categories (Example \ref{example:symmetric_poincare_structure}), and the fact that the subcategory of Poincaré $ \infty $-categories which are symmetric is closed under limits by \cite[Proposition 7.2.18]{CDHHLMNNSI}. 
    We turn to the proof of \ref{propitem:stacky_quotient_symmetric_structure_sheaf}. 
    For each object $ U \in [X/C_2]_{\mathrm{\acute{e}t}} $, the morphism $ \mathcal{O}_{[X/C_2]}(U) \to \pi_*\mathcal{O}_X(U) $ exhibits the former as the (strict) fixed points of the $ C_2 $-action on the latter. 
    Since the $ C_2 $-action on $ \pi_* \mathcal{O}_X(U) $ is free, strict fixed points agree with homotopy fixed points and we have an equivalence $ \mathcal{O}_{[X/C_2]}(U) \xrightarrow{\sim} \pi_*\mathcal{O}_X(U)^{\mathrm{h}C_2} $ of $ \EE_\infty $ rings. 
    Since homotopy limits commute with each other, we deduce that $ R\Gamma\left([X/C_2];\mathcal{O}_{[X/C_2]}\right) \simeq R\Gamma\left([X/C_2];\pi_*\mathcal{O}_X\right)^{\mathrm{h}C_2} \simeq R\Gamma\left(X; \mathcal{O}_X\right)^{\mathrm{h}C_2} \simeq R^{\mathrm{h}C_2} $. 
    The result follows from noting that a symmetric Poincaré ring is characterized by the property that the canonical map $ R^{C_2} \to \left(R^e\right)^{\mathrm{h}C_2} $ is an equivalence. 
\end{proof}	
	
\section{\texorpdfstring{$R$-linear Poincar\'e $\infty$-categories}{R-linear Poincare ∞-categories}}\label{section:relative_poincare_cats}
Now that we have the objects which will replace rings and schemes in our setting, we need to define and study linear Poincar\'e categories over them. For $R$ a Poincar\'e ring this is done in \S~\ref{subsection: R-linear poincare categories}. We then in \S~\ref{subsection: etale descent results} show that with these definitions, the $\infty$-category of dualizable $R$-linear Poincar\'e categories over $R$ satisfy descent for a certain natural analogue of the {\'e}tale topology defined in Definition~\ref{defn:C2_etale_topology}. This will lay the groundwork for showing that the Poincar\'e Picard and Poincar\'e Brauer spaces we construct are sheaves for this topology as well.
\subsection{Generalities on \texorpdfstring{$R$}{R}-linear Poincar\'e \texorpdfstring{$\infty$}{∞}-categories}\label{subsection:relative_Poincare_cats_generalities}~\label{subsection: R-linear poincare categories}

Let $R$ be a Poincaré ring spectrum. 
By Theorem \ref{thm:calgp_to_poincare_cat}, $ \left(\Mod_{R^e}^\omega,\Qoppa_R\right) $ promotes to a commutative algebra object in the $\infty$-category of Poincaré $\infty$-categories $ \Catp $, and we may thus consider modules over it, or \emph{$R$-linear Poincaré $ \infty $-categories}. This subsection is dedicated to generalizing many of the results from Section~\ref{section:poincare_structures_on_compact_modules} to the setting of $R$-linear Poincar\'e $\infty$-categories. Many of the proofs in this subsection follow the general outline of the analogous arguments in \cite{CDHHLMNNSI}, but keeping track of the $R$-linear structure in several arguments is subtle and so we include these details where necessary. Additionally, some of the preliminaries that go into the arguments of \cite{CDHHLMNNSI} must be proven in the $R$-linear context.

\begin{notation}~\label{notation: R-linear poincare cats}
    Let $R$ be a Poincar\'e ring with associated Poincar\'e category from Theorem~\ref{thm:calgp_to_poincare_cat} $\mathrm{Mod}^{\mathrm{p}}_R = (\mathrm{Mod}_{R^e}^\omega,\Qoppa_R)$. If $\mathcal{C}$ is one of $\Cathidem$, $\Cath$, $\Catpidem$, or $\Catp$, we will use \[\mathcal{C}_R\] to denote the $\infty$-category of modules $\mathrm{Mod}_{\mathrm{Mod}_R^\mathrm{p}}(\mathcal{C})$.
\end{notation}
\begin{lemma}\label{lemma:hermitian_structures_relative_linearity}
    Let $ R $ be a Poincaré ring. 
    \begin{enumerate}[label=(\arabic*)]
        \item     If $ \mathcal{C} $ is an $ R $-linear small stable $ \infty $-category, then $ \Fun^q(\mathcal{C}^{\op},\Spectra) $ is left-tensored over $ \Fun^q(\Mod_R^{\omega}) $ in the sense of \cite[Definition 4.2.1.19]{LurHA}, where the symmetric monoidal structure on $ \Fun^q(\Mod_R^{\omega}) $ is from \cite[Corollary 5.3.7]{CDHHLMNNSI}. 
        \item \label{lemmaitem:precompose_relative_hermitian_structure} Let $ \mathcal{C} $, $ \mathcal{D} $ be $ R $-linear small stable $ \infty $-categories, and let $ f \colon \mathcal{C} \to \mathcal{D} $ be an $ R $-linear functor. 
        Suppose we are given a $ \Qoppa_R $-module structure on $ \Qoppa \colon \mathcal{D}^\op \to \Spectra $ in the sense of (1). 
        Then $ f^*\Qoppa $ inherits a canonical $ \Qoppa_R $-module structure. 
    \end{enumerate}
\end{lemma}
\begin{proof}
    Let $ \mathcal{LM}^\otimes $ denote the $ \infty $-operad of \cite[Definition 4.2.1.7]{LurHA}. 
    Recall that the action of $ \Mod_R^{\omega} $ on $ \mathcal{C} $ is given by a functor $ \mathcal{LM}^\otimes \to \Cat_\infty^\times $. 
    Modifying \cite[Construction 5.3.15 \& Lemma 5.3.15]{CDHHLMNNSI} slightly (note that Corollary 5.1.4 did not assume the tensor factors to be equivalent), define $ \infty $-operads by requiring all squares in the diagram 
    \begin{equation}\label{diagram:module_structure_on_pb_functors}
        \begin{tikzcd}
            \Fun_{\Mod_R^{\omega,\op}}^\mathrm{p}(\mathcal{C}^{\op},\Spectra)^\otimes \ar[d] \ar[r] & \Fun_{\Mod_R^{\omega,\op}}^q(\mathcal{C}^{\op},\Spectra)^\otimes \ar[r]\ar[d] & \Fun_{\Mod_R^{\omega,\op}}(\mathcal{C}^{\op},\Spectra)^\otimes \ar[r,"{p}"]\ar[d] & \mathcal{LM}^\otimes \ar[d,"{\Mod_R^\omega,\mathcal{C}}"] \ar[d] \\
            \Catp^\otimes \ar[r] &\Cath^\otimes \ar[r] \ar[rr, bend right=10,"{p^\otimes}"'] & \left(\Cat_\infty\right)_{\op/-/\Spectra}^\otimes \ar[r] & \Cat_\infty^\otimes 
        \end{tikzcd}
    \end{equation}
    to be pullbacks, where $ p^\otimes $ is from \cite[Theorem 5.2.7]{CDHHLMNNSI}. 
    Since $ p^\otimes $ is a cocartesian fibration, it follows that $ \Fun_{\Mod_R^{\omega,\op}}^q(\mathcal{C}^{\op},\Spectra)^\otimes \to \mathcal{LM}^\otimes  $ is a cocartesian fibration exhibiting $ \Fun^q(\mathcal{C}) $ as being left-tensored over $ \Fun^q\left(\Mod_R^\omega\right) $. 
    Informally, an object $ F \in \Fun_{\Mod_R^{\omega,\op}}(\mathcal{C},\Spectra)^\otimes_{\mathfrak{a}} $ is a functor $ F \colon \Mod_R^{\omega,\op} \to \Spectra $ and an object $ G $ over the fiber of $ \mathfrak{m} $ is a functor $ G \colon \mathcal{C}^{\op} \to \Spectra $. 
    The $p $-cocartesian edge over the canonical map $ (\mathfrak{a},\mathfrak{m}) \to \mathfrak{m} $ in $ \mathcal{LM}^\otimes $ sends $ (F,G) $ to the lower arrow in the diagram
    \begin{equation*}
        \begin{tikzcd}[column sep=4.5em]
            \Mod_R^{\omega,\op}\times \mathcal{C}^{\op} \ar[rr,"{F\times G}"] \ar[d,"{-\otimes_R -}"'] & & \Spectra \times \Spectra \ar[d,"{\otimes_{\Spectra}}"] \\
            \mathcal{C}^{\op} \ar[rr,"{F \otimes G := \mathrm{LKE}_{\otimes_R}( \otimes_{\Spectra} \circ (F \times G))}"] & & \Spectra
        \end{tikzcd}\,.
    \end{equation*}
    
    We turn to the proof of \ref{lemmaitem:precompose_relative_hermitian_structure}. 
    By \cite[Corollary 3.2.2.3(1)]{LurHA}, the functor $ p: \Alg_{\mathcal{LM}}\left(\Cath\right) \to \Alg_{\mathcal{LM}} \left(\Catex\right)$ is a cartesian fibration and a map in $ \Alg_{\mathcal{LM}}\left(\Cath\right) $ is $ p $-cartesian if its image in $ \Cath$ is cartesian with respect to $ q: \Cath \to \Catex $. 
    By assumption, $ \mathcal{C} $ and $\mathcal{D}$ are objects of $ \Alg_{\mathcal{LM}} \left(\Catex\right) $, and $ f $ being $ \Mod_R^\omega $-linear means that $ f $ is a morphism in $ \Alg_{\mathcal{LM}} \left(\Catex\right)$. 
    Thus we can take $ p $-cartesian transport of the $ R $-linear hermitian structure on $ \mathcal{D} $. 
\end{proof}

We are now ready to prove several of the key properties of $R$-linear Poincar\'e $ \infty $-categories. We find it helpful to collect the next several results into one Proposition recording all of the key facts we will need about the structure of $R$-linear Poincar\'e $\infty$-categories. 
\begin{proposition}\label{prop:relative_poincare_cats_basic_properties}
    Let $R=(R^{e}, C \to R^{\mathrm{t}C_2})$ be a Poincaré ring spectrum and write $ (\Mod_{R^e}^\omega, \Qoppa_R) $ for the Poincaré $ \infty $-category of Theorem \ref{thm:calgp_to_poincare_cat}. 
    \begin{enumerate}[label=(\arabic*)]
        \item \label{propitem:Rlin_Poincare_cats_is_symm_mon} The $ \infty $-categories $ \Mod_{\left(\Mod_{R^e}^\omega, \Qoppa_R \right)}(\Catp) $ and $ \Mod_{\left(\Mod_{R^e}^\omega, \Qoppa_R \right)}(\Cath) $ admit all small limits and colimits, and it inherits a canonical symmetric monoidal structure, and there is a symmetric monoidal forgetful functor $ \Mod_{\left(\Mod_{R^e}^\omega, \Qoppa_R \right)}(\Catpidem) \to \Mod_{\Mod_{R^e}^\omega}\left(\Catperf\right) $. 
        For every morphism $ R=\left(R^e, R^{\varphi C_2} \to R^{\mathrm{t}C_2}\right) \to S=(S^e, S^{\varphi C_2} \to S^{\mathrm{t}C_2}) $, the functor \[ \Mod_{\left(\Mod_{R^e}^\omega, \Qoppa_R \right)}(\Catpidem) \to \Mod_{\left(\Mod_{S^e}^\omega, \Qoppa_S \right)}(\Catpidem) \] is a symmetric monoidal left adjoint. (Proposition~\ref{prop: structure of R-linear poincare cats part 1})
        \item \label{propitem:classify_R_lin_hermitian_struct_gen} Let $ \mathcal{C} $ be an $ R^e $-linear small stable $ \infty $-category. Then the data of a lift of $ \mathcal{C} $ to $ \Mod_{\left(\Mod_{R^e}^\omega, \Qoppa_R \right)}(\Catp) $ (resp. $ \Mod_{\left(\Mod_{R^e}^\omega, \Qoppa_R \right)}(\Cath) $) is equivalent to the data of a left $ \Qoppa_R $-module in $ \Fun^q(\mathcal{C}) $, resp. $ \Fun^\mathrm{p}(\mathcal{C}) $ (see Lemma \ref{lemma:hermitian_structures_relative_linearity}). (Proposition~\ref{prop: structure of R-linear poincare categories part 2})
        \item \label{propitem:classify_R_lin_hermitian_struct_mod_cat} Let $ A $ be an $ \EE_1 $-$ R^e $-algebra in spectra, and regard the category of compact right $ A $-modules $ \Mod_A^\omega $ as left-tensored over $ \Mod_{R^e}^\omega $ in the canonical way. 
        Then the pullback
        \begin{equation}
            \begin{tikzcd}
                & \Mod_{(\Mod_{R^e}^\omega, \Qoppa_R)}\left(\Cath\right) \ar[d] \\
                \{\Mod_A^\omega\} \ar[r] & \Catex_{, R^e}
            \end{tikzcd}
        \end{equation}
        is canonically equivalent to $ \Mod_{N_{R^e} A \otimes_{N_{R^e} R^e} R^L }\left(\Spectra^{C_2}\right) $ where $ R^L $ is the $ \EE_\infty $-$ N_{R^e} R^e $-algebra with $ (R^L)^e \simeq R^e $ and $ (R^L)^{\varphi C_2}  \simeq C $. 
        
        A $ N_{R^e} A \otimes_{N_{R^e} R^e} R^L $-module classifies a $ (\Mod_{R^e}^\omega, \Qoppa_R) $-module in Poincaré $ \infty $-categories if its underlying $ A $-module is invertible in the sense of \cite[Definition 3.1.4]{CDHHLMNNSI}. (Proposition~\ref{prop: structure of R-linear poincare categories part 3})
        \item \label{propitem:Rlin_Poincare_cats_tensor_mod_gen_inv} Let $ A,B $ be $ R^e $-algebras with associated ($R$-linear) modules with genuine involution $ (M_A, N_A, N_A \to M_A^{\mathrm{t}C_2}) $ and $ (M_B, N_B, N_B \to M_B^{\mathrm{t}C_2}) $, respectively so that (under item \ref{propitem:classify_R_lin_hermitian_struct_mod_cat}) $ \left(\Mod_A^\omega, \Qoppa_A \right) $ and $ \left(\Mod_B^\omega, \Qoppa_B \right) $ are objects of $ \Mod_{\left(\Mod_{R^e}^\omega, \Qoppa_R \right)}(\Catpidem) $. 
        Then the symmetric monoidal structure of item \ref{propitem:Rlin_Poincare_cats_is_symm_mon} is so that the underlying $ R^e$-linear $ \infty $-category with perfect duality $ \left(\Mod_A^\omega, \Qoppa_A \right) \otimes_{\left(\Mod_{R^e}^\omega, \Qoppa_R \right)} \left(\Mod_B^\omega, \Qoppa_B \right) $ is $ \Mod_A^\omega \otimes_{\Mod_{R^e}^\omega} \Mod_B^\omega \simeq \Mod_{A \otimes_{R^e} B}^\omega $, and the associated module with genuine involution is given by $ M_A \otimes_{R^e} M_B $, $ N_A \otimes_{R^{\varphi C_2}} N_B $, and the structure map is $ N_A \otimes_{R^{\varphi C_2}} N_B \to M_A^{\mathrm{t}C_2} \otimes_{R^{\mathrm{t}C_2}} M_B^{\mathrm{t}C_2} \to (M_A \otimes_{R^e} M_B)^{\mathrm{t}C_2} $ where the latter map arises canonically from lax monoidality of the Tate construction. (Proposition~\ref{prop: structure of R-linear poincare categories part 4})
        
        \item \label{prop_item:R_linear_poincare_cats_maps} Let $ \left(\mathcal{C}, \Qoppa_{\mathcal{C}}\right),  \left(\mathcal{D}, \Qoppa_{\mathcal{D}}\right) $ be objects of $ \Mod_{\left(\Mod_{R^e}^\omega, \Qoppa_R \right)}(\Cath) $. 
        Then, using Notation~\ref{notation: R-linear poincare cats}, the forgetful functor induces the map $ \hom_{\Cath_{, R}} \left(\left(\mathcal{C}, \Qoppa_{\mathcal{C}}\right),  \left(\mathcal{D}, \Qoppa_{\mathcal{D}}\right)\right) \to \hom_{\Catex_R}\left(\mathcal{C}, \mathcal{D}\right) $ on mapping spaces so that the fiber over an $ R^e $-linear functor $ F \colon \mathcal{C} \to \mathcal{D} $ is the mapping space $ \mathrm{map}_{\Qoppa_R}\left(F_! \Qoppa_\mathcal{C}, \Qoppa_{\mathcal{D}}\right) \simeq \mathrm{map}_{\Qoppa_R}(\Qoppa_{\mathcal{C}}, \Qoppa_{\mathcal{D}} \circ F^\op) $, where the mapping space is taken in $ \Fun_{\Qoppa_R}^q(\mathcal{D}^{\op},\Spectra) $ and $ \Fun_{\Qoppa_R}^q(\mathcal{C}^{\op},\Spectra) $, respectively (this makes sense by Lemma \ref{lemma:hermitian_structures_relative_linearity}). (Proposition~\ref{prop: structure of R-linear poincare categories part 5})
        
        \item The symmetric monoidal forgetful functor $ \theta \colon \Mod_{\left(\Mod_{R^e}^\omega, \Qoppa_R \right)}(\Cath) \to \Mod_{\Mod_{R^e}^\omega}(\Catex) $ is a (co)cartesian fibration. 
        It is classified by the assignment sending a small stable $ R^e $-linear $ \infty $-category $ \mathcal{C} $ to $ \Mod_{\Qoppa_R}\left(\Fun^q\left(\mathcal{C}\right)\right) $, where $ \Fun^q\left(\mathcal{C}\right) $ is regarded as left-tensored over $ \Fun^q\left(\Mod_{R^e}^\omega\right) $ by Lemma \ref{lemma:hermitian_structures_relative_linearity}. (Proposition~\ref{prop: structure of R-linear poincare categories part 6})
        \item~\label{prop_item:internal hom in R-linear poincare cats} The symmetric monoidal structure on $\mathrm{Mod}_{(\mathrm{Mod}_{R^e}^\omega,\Qoppa_R)}(\Cath)$ and $\mathrm{Mod}_{(\mathrm{Mod}_{R^e}^\omega,\Qoppa_R)}(\Catp)$ are closed, and the functors $\mathrm{Mod}_{(\mathrm{Mod}_{R^e}^\omega,\Qoppa_R)}(\Catp)\to \mathrm{Mod}_{(\mathrm{Mod}_{R^e}^\omega,\Qoppa_R)}(\Cath)\to \Catex_{, R^e}$ preserve the internal mapping objects. (Proposition~\ref{prop: structure of R-linear poincare categories part 7})
    \end{enumerate}
\end{proposition}
\begin{remark}
    A special case of part \ref{propitem:classify_R_lin_hermitian_struct_mod_cat} is \cite[Example 5.4.13]{CDHHLMNNSI}.
\end{remark}

\begin{proposition}~\label{prop: structure of R-linear poincare cats part 1}
    The $ \infty $-categories $ \Mod_{\left(\Mod_{R^e}^\omega, \Qoppa_R \right)}(\Catp) $ and $ \Mod_{\left(\Mod_{R^e}^\omega, \Qoppa_R \right)}(\Cath) $ admit all small limits and colimits, and it inherits a canonical symmetric monoidal structure, and there is a symmetric monoidal forgetful functor $ \Mod_{\left(\Mod_{R^e}^\omega, \Qoppa_R \right)}(\Catpidem) \to \Mod_{\Mod_{R^e}^\omega}\left(\Catperf\right) $. 
    For every morphism $ R=\left(R^e, R^{\varphi C_2} \to R^{\mathrm{t}C_2}\right) \to S=(S^e, S^{\varphi C_2} \to S^{\mathrm{t}C_2}) $, the functor \[ \Mod_{\left(\Mod_{R^e}^\omega, \Qoppa_R \right)}(\Catpidem) \to \Mod_{\left(\Mod_{S^e}^\omega, \Qoppa_S \right)}(\Catpidem) \] is a symmetric monoidal left adjoint.
\end{proposition}
\begin{proof}
    The first part of the statement follows from \cite[\S6.1]{CDHHLMNNSI} and \cite[\S4.2.3]{LurHA}. 
    The second part of the statement follows from \cite[Corollary 6.2.17]{CDHHLMNNSI} and \cite[Corollary 4.8.5.21]{LurHA}. 
    
\end{proof}
\begin{proposition}~\label{prop: structure of R-linear poincare categories part 2}
    Let $ \mathcal{C} $ be an $ R^e $-linear small stable $ \infty $-category. Then the data of a lift of $ \mathcal{C} $ to an element of $ \Mod_{\left(\Mod_{R^e}^\omega, \Qoppa_R \right)}(\Catp) $ (resp. $ \Mod_{\left(\Mod_{R^e}^\omega, \Qoppa_R \right)}(\Cath) $) is equivalent to the data of a left $ \Qoppa_R $-module in $ \Fun^q(\mathcal{C}) $, resp. $ \Fun^\mathrm{p}(\mathcal{C}) $ (see Lemma \ref{lemma:hermitian_structures_relative_linearity}).
\end{proposition}
\begin{proof}
    Follows from the pullback squares in (\ref{diagram:module_structure_on_pb_functors}).
\end{proof}
\begin{proposition}~\label{prop: structure of R-linear poincare categories part 3}
    Let $ A $ be an $ \EE_1 $-$ R^e $-algebra in spectra, and regard the category of compact right $ A $-modules $ \Mod_A^\omega $ as left-tensored over $ \Mod_R^\omega $ in the canonical way. 
    Then the pullback
    \begin{equation}
        \begin{tikzcd}
            & \Mod_{(\Mod_{R^e}^\omega, \Qoppa_R)}\left(\Cath\right) \ar[d] \\
            \{\Mod_A^\omega\} \ar[r] & \Catex_{, R^e}
        \end{tikzcd}
    \end{equation}
    is canonically equivalent to $ \Mod_{N_{R^e} A \otimes_{N_{R^e} R^e} R^L }\left(\Spectra^{C_2}\right) $ where $ R^L $ is the $ \EE_\infty $-$ N_{R^e} R^e $-algebra with $ (R^L)^e \simeq R^e $ and $ (R^L)^{\varphi C_2}  \simeq C $. 
    
    A $ N_{R^e} A \otimes_{N_{R^e} R^e} R^L $-module classifies a $ (\Mod_{R^e}^\omega, \Qoppa_R) $-module in Poincaré $ \infty $-categories if its underlying $ A $-module is invertible in the sense of \cite[Definition 3.1.4]{CDHHLMNNSI}
\end{proposition}
\begin{proof}
    Let $ \mathcal{LM}^\otimes $ denote the $ \infty $-operad of \cite[Definition 4.2.1.7]{LurHA}. 
    We will use a (suitably coherent version of) the classification of hermitian structures on module categories as categories of modules over the Hill--Hopkins--Ravenel norm \cite[Theorem 3.3.1]{CDHHLMNNSI}. 				By \ref{propitem:classify_R_lin_hermitian_struct_gen}, to lift $ \Mod_A^\omega $ to a module over $ \left(\Mod_{R^e}^\omega,\Qoppa_R\right) $ compatibly with the $ \Mod_{R^e}^\omega $-module structure on $ \Mod_A^\omega $ is to give a map of $ \infty $-operads \[ \mathcal{LM}^\otimes \to \Fun^q\left(\Mod_{A}^{\omega}\right)^\otimes \to \Cath^\otimes \] so that the restriction along the canonical inclusion $ \mathrm{Assoc}^\otimes \to \mathcal{LM}^\otimes $ gives the algebra object $ \left(\Mod_{R^e}^\omega,\Qoppa_R\right) $ and postcomposing with the canonical projection to $ \Catex^\times $ recovers the given $ \Mod_{R^e}^\omega $-module structure on $ \Mod_A^\omega $. 
    By the pullback square (\ref{diagram:module_structure_on_pb_functors}), this is equivalent to giving an object of \[ \mathrm{Alg}_{\mathcal{LM}/\mathcal{LM}}\left(\Fun_{\Mod_{R^e}^{\omega,\op}}^q(\Mod_A^{\omega,\op},\Spectra)^\otimes\right)\,. \] 
    Now let us identify the bilinear functor $ \Mod_{R^e}^\omega \times \Mod_A^\omega \xrightarrow{ - \otimes_{R^e} -} \Mod_A^\omega $ with the exact functor \[ \Mod_{R^e}^\omega \otimes \Mod_A^\omega \simeq \Mod_{R^e \otimes A}^\omega \to \Mod_A^\omega \] which is induction along the action map $ R^e \otimes A \to A $. 
    Using \cite[Corollary 3.4.1]{CDHHLMNNSI} and unravelling definitions gives the claim for $ R$-linear hermitian structures. 
    The proof for $ R $-linear Poincaré structures considers the $ \infty $-operad $ \Fun^{p}\left(\Mod_A^\omega\right) $ instead but otherwise proceeds in an identical fashion.   
\end{proof}
\begin{proposition}~\label{prop: structure of R-linear poincare categories part 4}
    Let $ A,B $ be $ R^e $-algebras with associated ($R$-linear) modules with genuine involution $ (M_A, N_A, N_A \to M_A^{\mathrm{t}C_2}) $ and $ (M_B, N_B, N_B \to M_B^{\mathrm{t}C_2}) $, respectively so that (under item \ref{propitem:classify_R_lin_hermitian_struct_mod_cat}) $ \left(\Mod_A^\omega, \Qoppa_A \right) $ and $ \left(\Mod_B^\omega, \Qoppa_B \right) $ are objects of $ \Mod_{\left(\Mod_{R^e}^\omega, \Qoppa_R \right)}(\Catpidem) $. 
    Then the symmetric monoidal structure of item \ref{propitem:Rlin_Poincare_cats_is_symm_mon} is so that the underlying $ R^e$-linear $ \infty $-category with perfect duality $ \left(\Mod_A^\omega, \Qoppa_A \right) \otimes_{\left(\Mod_{R^e}^\omega, \Qoppa_R \right)} \left(\Mod_B^\omega, \Qoppa_B \right) $ is $ \Mod_A^\omega \otimes_{\Mod_{R^e}^\omega} \Mod_B^\omega \simeq \Mod_{A \otimes_{R^e} B}^\omega $, and the associated module with genuine involution is given by \[ (M_A \otimes_{R^e} M_B,  N_A \otimes_{R^{\varphi C_2}} N_B, N_A \otimes_{R^{\varphi C_2}} N_B \to M_A^{\mathrm{t}C_2} \otimes_{R^{\mathrm{t}C_2}} M_B^{\mathrm{t}C_2} \to (M_A \otimes_{R^e} M_B)^{\mathrm{t}C_2}) \]  where the latter map arises canonically from lax monoidality of the Tate construction.
\end{proposition}
\begin{proof}
    By \cite[Theorem 4.4.2.8]{LurHA}, the relative tensor product $ \left(\Mod_A^\omega, \Qoppa_A \right) \otimes_{\left(\Mod_{R^e}^\omega, \Qoppa_R \right)} \left(\Mod_B^\omega, \Qoppa_B \right) $ is computed as the geometric realization of the bar construction 
    \begin{align*}
        p \colon \Delta^\op & \to \Cath \\
        [n] &\mapsto \left(\Mod_A^\omega, \Qoppa_A \right) \otimes \left(\Mod_{R^e}^\omega, \Qoppa_R \right)^{\otimes n} \otimes \left(\Mod_B^\omega, \Qoppa_B \right)
    \end{align*} 
    Write $ f \colon \Cath \to \Catex $ for the forgetful functor. 
    Then $ f \circ p $ has a colimit with value $ \Mod_A^\omega \otimes_{\Mod_{R^e}^\omega} \Mod_B^\omega \simeq \Mod_{A \otimes_{R^e} B}^\omega $. 
    Writing $ g \colon \Catex \to \{*\} $, by Example 4.3.1.3 of \cite{HTT} we see that $ f \circ p $ is a $ g $-colimit. 
    By Proposition 4.3.1.5(2) and Example 4.3.1.3 of \cite{HTT}, $ p $ admits a colimit in $ \Cath $ if and only if it admits an $ f $-colimit. 
    Now recall that $ f $ is a cocartesian fibration with pushforward given by left Kan extension \cite[Corollary 1.4.2]{CDHHLMNNSI}. 
    We show that $ f $ satisfies the conditions of \cite[Corollary 4.3.1.11]{HTT}. 
    \begin{itemize}
        \item Condition (1) follows from Theorem 6.1.1.10 of \cite{LurHA} applied to $ \Spectra^\op $ (see the end of \cite[Construction 1.1.26]{CDHHLMNNSI}). 
        \item Condition (2) follows from \cite[Corollary 1.4.2]{CDHHLMNNSI}, the adjoint functor theorem, and presentability of $ \Fun^q(\mathcal{C}) $, which is discussed in the proof of \cite[Lemma 5.3.3]{CDHHLMNNSI} (also see \cite[Remark 6.1.1.11]{LurHA}). 
    \end{itemize}
    Thus the preceding discussion shows that there exists a map of simplicial sets $ p' $ making the diagram commute
    \begin{equation*}
        \begin{tikzcd}
            \Delta^\op \ar[d] \ar[r,"p"] & \Cath \ar[d,"f"] \\
            \left(\Delta^\op\right)^{\triangleright} \ar[r]\ar[ru,"{p'}"] & \Catex 
        \end{tikzcd}\,.
    \end{equation*}
    Since $ \{0\} \to \Delta^1 $ is left anodyne, by \cite[Corollary 2.1.2.7]{HTT} the inclusions
    \begin{align*}
        \{0\} \times \Delta^\op & \to \Delta^1 \times \Delta^\op \\
        \iota \colon \left(\{0\} \times (\Delta^\op)^\triangleright\right) \sqcup_{\{0\} \times \Delta^\op}\left(\Delta^1 \times \Delta^\op \right) & \to \Delta^1 \times \left(\Delta^\op\right)^\triangleright 
    \end{align*}
    are left anodyne. 
    The former implies that there exists a map $ p'' $ making the diagram 
    \begin{equation*}
        \begin{tikzcd}
            \{0\} \times \Delta^\op \ar[r,"p"]\ar[d] &\Cath \ar[d,"f"] \\
            \Delta^1 \times \Delta^\op \ar[r] \ar[ru,"{p''}"] & \Catex
        \end{tikzcd}
    \end{equation*}
    commute. 
    The maps $ p' $ and $ p'' $ assemble to give a map $ p''' := p' \sqcup_p p'' $ making the diagram 
    \begin{equation*}
        \begin{tikzcd}[arrows={crossing over},row sep=large]
            \{0\} \times \Delta^\op \ar[rr,"p"]\ar[d] &  &\Cath \ar[d,"f"] \\
            \left(\{0\} \times (\Delta^\op)^\triangleright\right) \sqcup_{\{0\} \times \Delta^\op}\left(\Delta^1 \times \Delta^\op \right) \ar[r,"\iota"] \ar[rru,"{p'''}", near start] & \Delta^1 \times \left(\Delta^\op\right)^\triangleright \ar[r] \ar[ru,"{\overline{p}}"', bend right=10] & \Catex
        \end{tikzcd}
    \end{equation*}
    commute, and likewise $ \overline{p} $ exists making the diagram commute since $ \iota $ is left anodyne. 
    Now we show that $ \overline{p} $ satisfies the conditions of \cite[Proposition 4.3.1.9]{HTT}. 
    By (the cocartesian version of) \cite[Remark 3.1.1.10]{HTT} and Proposition 3.1.1.5(2'') \emph{ibid.} and the fact that $ f $ is a cocartesian fibration, we can choose $ \overline{p} $ so that for all $ k \in (\Delta^\op)^\triangleright $, $ \overline{p}|_{\Delta^1 \times \{k\}} $ is $f$-cocartesian. 
    Furthermore, since we can choose $ \Delta^\op, \,\left(\Delta^\op\right)^\triangleright $ to have the markings $ (-)^\flat $ in \cite[Remark 3.1.1.10]{HTT}, $ f \circ \overline{p}|_{\Delta^1 \times \{\infty\}} $ is a degenerate edge in $ \Catex $. 
    
    Now \cite[Proposition 4.3.1.9]{HTT} implies that $ \overline{p}_0 $ is an $ f $-colimit diagram if and only if $ \overline{p}_1 $ is an $ f $-colimit diagram. 
    Notice that $ \overline{p}|_{\{1\} \times \left(\Delta^\op\right)^{\triangleright}} $ has image contained in the fiber of $ f $ over $ \Mod_{A \otimes_{R^e} B}^\omega $. 
    By \cite[Proposition 4.3.1.10]{HTT}, it suffices to show that $ \overline{p}_1 $ is a colimit diagram in $ \Fun^q\left(\Mod_{A \otimes_{R^e} B}^\omega\right) $. 
    Write $ \overline{M}_A \in \Mod_{N^{C_2}A} $ and $ \overline{M}_B \in \Mod_{N^{C_2}B} $ for the corresponding modules (see introduction to \S3.3 of \cite{CDHHLMNNSI}). 
    Unraveling definitions and using \cite[Theorem 3.3.1 \& Corollary 3.4.1 \& Lemma 5.4.6]{CDHHLMNNSI}, it follows that the diagram $ \overline{p}_1|_{\{1\} \times \Delta^\op} $ is the bar construction 
    \begin{align*}
        [n] & \mapsto \overline{M}_A \otimes_{N^{C_2}R^e} R^{L\otimes_{N^{C_2}R^e} n} \otimes_{N^{C_2} R^e} \overline{M}_B \,.
    \end{align*}
    This proves the result.
\end{proof}

\begin{proposition}~\label{prop: structure of R-linear poincare categories part 5}
    Let $ \left(\mathcal{C}, \Qoppa_{\mathcal{C}}\right),  \left(\mathcal{D}, \Qoppa_{\mathcal{D}}\right) $ be objects of $ \Mod_{\left(\Mod_{R^e}^\omega, \Qoppa_R \right)}(\Cath) $. 
    Then, using notation~\ref{notation: R-linear poincare cats}, the forgetful functor induces the map $ \hom_{\Cath_{, R}} \left(\left(\mathcal{C}, \Qoppa_{\mathcal{C}}\right),  \left(\mathcal{D}, \Qoppa_{\mathcal{D}}\right)\right) \to \hom_{\Catex_R}\left(\mathcal{C}, \mathcal{D}\right) $ on mapping spaces so that the fiber over an $ R^e $-linear functor $ F \colon \mathcal{C} \to \mathcal{D} $ is the mapping space $ \mathrm{map}_{\Qoppa_R}\left(F_! \Qoppa_\mathcal{C}, \Qoppa_{\mathcal{D}}\right) \simeq \mathrm{map}_{\Qoppa_R}(\Qoppa_{\mathcal{C}}, \Qoppa_{\mathcal{D}} \circ F^\op) $, where the mapping space is taken in $ \Fun_{\Qoppa_R}^q(\mathcal{D}^{\op},\Spectra) $ and $ \Fun_{\Qoppa_R}^q(\mathcal{C}^{\op},\Spectra) $, respectively (this makes sense by Lemma \ref{lemma:hermitian_structures_relative_linearity}).
\end{proposition}
\begin{proof}
    Let $ \left(\mathcal{C}, \Qoppa_{\mathcal{C}}\right) $ be an object of $ \Mod_{\left(\Mod_R^\omega, \Qoppa_R \right)}(\Cath) $ and let $ F \colon \mathcal{C} = \theta \left(\mathcal{C}, \Qoppa_{\mathcal{C}}\right)\to \mathcal{D} $ be an $ R $-linear functor. 
    Now define $ \Qoppa_{\mathcal{D}} \colon \mathcal{D}^\op \to \Spectra $ to be the left Kan extension of $ \Qoppa_{\mathcal{C}} $ along $ F^\op $. 
    We have that $ \left(\mathcal{D}, \Qoppa_{\mathcal{D}}\right) \in \Cath $ and there is a canonical map $ (f, \eta )\colon \left(\mathcal{C}, \Qoppa_{\mathcal{C}}\right) \to \left(\mathcal{D}, \Qoppa_{\mathcal{D}}\right)$. 
    Now $ F $ is classified by a functor $ \Delta^1 \times \mathcal{LM}^\otimes \to \Catex^\otimes $, and we may form the pullback
    \begin{equation}
        \begin{tikzcd}
            \mathcal{N} \ar[r] \ar[d] & \Delta^1 \times \mathcal{LM}^\otimes \ar[d]  \\
            \Cath^\otimes \ar[r,"{p}"] & \Catex^\otimes
        \end{tikzcd}\,.
    \end{equation}
    Since $ p $ is a cocartesian fibration by \cite[Theorem 5.2.7]{CDHHLMNNSI}, $ \mathcal{N} \to \Delta^1 \times \mathcal{LM}^\otimes $ is a cocartesian fibration, and the nontrivial morphism in $ \Delta^1 $ classifies a map $ F_! \colon \Fun_{\Mod_{R^e}^{\omega,\op}}^q(\mathcal{C}^{\op},\Spectra)^\otimes \to \Fun_{\Mod_{R^e}^{\omega,\op}}^q(\mathcal{D}^{\op},\Spectra)^\otimes $ of $ \infty $-operads over $ \mathcal{LM}^\otimes $. 
    Passing to algebra objects, we obtain the desired result on mapping spaces. 
\end{proof}
\begin{proposition}~\label{prop: structure of R-linear poincare categories part 6}
    The symmetric monoidal forgetful functor \[ \theta \colon \Mod_{\left(\Mod_{R^e}^\omega, \Qoppa_R \right)}(\Cath) \to \Mod_{\Mod_{R^e}^\omega}(\Catex) \] is a (co)cartesian fibration. 
    It is classified by the assignment sending a small stable $ R^e $-linear $ \infty $-category $ \mathcal{C} $ to $ \Mod_{\Qoppa_R}\left(\Fun^q\left(\mathcal{C}\right)\right) $, where $ \Fun^q\left(\mathcal{C}\right) $ is regarded as left-tensored over $ \Fun^q\left(\Mod_{R^e}^\omega\right) $ by Lemma \ref{lemma:hermitian_structures_relative_linearity}.
\end{proposition}
\begin{proof}
    By \cite[Proposition 2.4.2.8]{HTT}, it suffices to show that $ \theta $ is a locally (co)cartesian fibration, and that locally (co)cartesian edges are closed under composition. 
    We give the proof that $ \theta $ is a cocartesian fibration; the proof that $ \theta $ is a cartesian fibration is formally dual and will be left to the reader. 
    
    Let $ \left(\mathcal{C}, \Qoppa_{\mathcal{C}}\right) $ be an object of $ \Mod_{\left(\Mod_{R^e}^\omega, \Qoppa_R \right)}(\Cath) $ and let $ F \colon \mathcal{C} = \theta \left(\mathcal{C}, \Qoppa_{\mathcal{C}}\right)\to \mathcal{D} $ be an $ R $-linear functor. 
    Now define $ \Qoppa_{\mathcal{D}} \colon \mathcal{D}^\op \to \Spectra $ to be the left Kan extension of $ \Qoppa_{\mathcal{C}} $ along $ F^\op $. 
    By the proof of \ref{prop_item:R_linear_poincare_cats_maps}, we see that the image of $ \Qoppa_{\mathcal{C}} $ under $ F_! $ is a lift of $ \left(\mathcal{D}, \Qoppa_{\mathcal{D}}\right) $ to an object of $ \Mod_{\left(\Mod_{R^e}^\omega, \Qoppa_R \right)}(\Cath) $ and $ (f, \eta) $ to a morphism in $ \Mod_{\left(\Mod_{R^e}^\omega, \Qoppa_R \right)}(\Cath) $. 
    
    Now by Lemma 2.4.4.1 and the locally cocartesian version of Proposition 2.4.1.10 of \cite{HTT}, we must show that for all choices $ \Qoppa_{\mathcal{D}}' $ of an $ R $-linear Hermitian structure on $ \mathcal{D} $, precomposition with $ F_! $ induces a pullback square
    \begin{equation}
        \begin{tikzcd}
            \hom_{\Cath_{, R}}\left(\left(\mathcal{D}, \Qoppa_{\mathcal{D}}\right), \left(\mathcal{D}, \Qoppa_{\mathcal{D}}'\right)\right) \ar[d] \ar[r] & \hom_{\Cath_{, R}}\left(\left(\mathcal{C}, \Qoppa_{\mathcal{C}}\right), \left(\mathcal{D}, \Qoppa_{\mathcal{D}}'\right)\right) \ar[d] \\
            \hom_{\Catex_{, R^e}}\left(\mathcal{D}, \mathcal{D}\right) \ar[r] & \hom_{\Catex_{, R^e}}\left(\mathcal{C}, \mathcal{D} \right)
        \end{tikzcd}\,.
    \end{equation}
    By \ref{prop_item:R_linear_poincare_cats_maps}, $ F_! $ induces equivalences on the fibers of the vertical maps, hence $ (f, \eta) $ is locally $ \theta $-cocartesian. 
    The locally $ \theta $-cocartesian maps are manifestly closed under composition, hence we are done. 
\end{proof}
	
\begin{proposition}~\label{prop: structure of R-linear poincare categories part 7}
    The symmetric monoidal structure on $\mathrm{Mod}_{(\mathrm{Mod}_{R^e}^\omega,\Qoppa_R)}(\Cath)$ and $\mathrm{Mod}_{(\mathrm{Mod}_{R^e}^\omega,\Qoppa_R)}(\Catp)$ are closed, and the functors $\mathrm{Mod}_{(\mathrm{Mod}_{R^e}^\omega,\Qoppa_R)}(\Catp)\to \mathrm{Mod}_{(\mathrm{Mod}_{R^e}^\omega,\Qoppa_R)}(\Cath)\to \Catex_{, R^e}$ preserve the internal mapping objects.
\end{proposition}
\begin{proof}
    We will first prove the following more general claim: Let $\mathcal{C}$ be a symmetric monoidal closed $\infty$-category which is complete and cocomplete. Let $R\in \CAlg(\mathcal{C})$. Then $\mathrm{Mod}_R(\mathcal{C})$ is also a closed symmetric monoidal category.
    We will prove this claim following the proof of \cite[Lemma I.24]{ramzi2023separabilityhomotopicalalgebra}. 
    For this, note that for any $Y\in \mathrm{Mod}_R(\mathcal{C})$ and any simplicial set $I$, the existence of an internal mapping object $\underline{\mathrm{Hom}}_{\mathrm{Mod}_R(\mathcal{C})}(X,Y)$ is closed under $I$-shaped colimits in $X$. To see this, note that the existence of such an object is by definition equivalent to the functor $Z\mapsto \mathrm{Maps}_{\mathrm{Mod}_{A}(\mathcal{C})}(X\otimes_R Z,Y)$ being representable, and taking the relative tensor product commutes with colimits since it is equivalent to a colimit of iterated tensor products in $\mathcal{C}$, all of which commute with colimits because $\mathcal{C}$ is closed. Thus the functor in question is the $I^\op$-shaped limit of representable functors which is again representable.
    
    We also note that for any $X\in \mathcal{C}$, the internal mapping object $\underline{\mathrm{Hom}}_{\mathrm{Mod}_R(\mathcal{C})}(R\otimes X, Y)$ exists since for any $Z\in \mathrm{Mod}_R(\mathcal{C})$ we need that \[\mathrm{Maps}_{\mathrm{Mod}_R(\mathcal{C})}(Z\otimes_R R\otimes X,Y)\] is represented by an object of $\mathrm{Mod}_R(\mathcal{C})$. We will show that this functor is representable by the object $\underline{\mathrm{Hom}}_{\mathcal{C}}(X,Y)$ with the canonical $R$-module structure coming from the (adjoint of) the evaluation map. The evaluation map induces a natural transformation of these functors, and so by \cite[Proposition 4.7.3.14]{LurHA} we may reduce to the case of $Z=X'\otimes R$. Then \[\mathrm{Maps}_{\mathrm{Mod}_R(\mathcal{C})}(R\otimes X'\otimes X, Y)\simeq \mathrm{Maps}_{\mathcal{C}}(X'\otimes X,Y)\simeq \mathrm{Maps}_{\mathcal{C}}(X',\underline{\mathrm{Hom}}_{\mathcal{C}}(X,Y))\simeq \mathrm{Maps}_{\mathrm{Mod}_R}(Z,\underline{\mathrm{Hom}}_{\mathcal{C}}(X,Y))\] and this equivalence comes from the evaluation map as desired. Consequently the internal mapping object exists and is functorially identified with $\underline{\mathrm{Hom}}_{\mathcal{C}}(X,Y)$.
    
    The result now follows by using \cite[Proposition 4.7.3.14]{LurHA} again to reduce to the case of the previous paragraph.
    
    Applying this result to $\mathcal{C}=\Cath$ or $\Catp$ and $R=(\mathrm{Mod}_{R^e}^\omega,\Qoppa_R)$ gives the existence of the internal mapping object in $\Cath_{, R}$ and $\Catp_{, R}$. Furthermore,  the above proof actually gives a formula for the internal mapping object in terms of a totalization \[(M, N)\mapsto \mathrm{Tot}(\bullet\mapsto \mathrm{hom}_{\mathcal{C}}(M\otimes R^{\otimes \bullet}, N))\] and so if $U:\mathcal{C}\to \mathcal{D}$ is a symmetric monoidal functor which commutes with limits and preserves the internal mapping objects, it follows that the induced functor \[\mathrm{Mod}_R(\mathcal{C})\to \mathrm{Mod}_{U(R)}(\mathcal{D})\] also preserves internal mapping objects.
\end{proof}

\begin{notation}~\label{notation: internal mapping objects}
    Let $R$ be a Poincar\'e ring. By Proposition~\ref{prop:relative_poincare_cats_basic_properties}\ref{prop_item:R_linear_poincare_cats_maps}, both $\Cath_{, R}$ and $\Catp_{, R}$ admit internal mapping objects. For $(\mathcal{C},\Qoppa), (\mathcal{D},\Phi)\in \Cath_{, R}$, let \[\mathrm{Fun}_{R}((\mathcal{C},\Qoppa),(\mathcal{D},\Phi))\] denote this internal mapping object. Note that the underlying $R^e$-linear stable category is $\mathrm{Fun}_{R^e}^{\ex}(\mathcal{C},\mathcal{D})$, and so if both $(\mathcal{C},\Qoppa)$ and $(\mathcal{D},\Phi)$ are idempotent complete so is $\mathrm{Fun}_R((\mathcal{C},\Qoppa),(\mathcal{D},\Phi))$.
\end{notation}
	
This internal mapping object allows us to lift the classical Morita theory to the setting of $R$-linear Poincar{\'e} $\infty$-categories. This will be helpful in connecting the Poincar\'e Brauer spectrum we will define in Section~\ref{section:the_poincare_brauer_group} with the Poincar{\'e} Picard spectrum we define in Section~\ref{section:the_poincare_picard_group}. 
For any $R$-algebras $A$ and $B$, denote by $\mathrm{Mod}_{A^{\op}\otimes_R B}^{L\omega}$ the full subcategory of $\mathrm{Mod}_{A^{\op}\otimes_R B}$ spanned by modules $P$ which are compact when regarded as a $ R $-module.

\begin{corollary}\label{cor:poincare_fourier_mukai}
    Let $ R $ be a Poincaré ring, and let $ A, B $ be $ \mathbb{E}_1 $-$ R $-algebras with genuine involution in the sense of Definition~\ref{defn:alt definition of Azumaya algebra with involution}.
    Then there is a Poincaré structure $ \Qoppa_{A \boxtimes B^\op} $ on $ \Mod_{A \otimes_R B^\op}^{L\omega} $ with equivalences 
    \begin{equation*}
        \begin{split}
            \hom_{\Cathidem_{, R}}\left(\left(\Mod_A^\omega, \Qoppa_A\right), \left(\Mod_B^\omega, \Qoppa_B\right)\right) & \simeq \mathrm{He} \left(\Mod_{A^\op \otimes_R B}^{L\omega}, \Qoppa_{A^\op \boxtimes B} \right) \\
            \hom_{\Catpidem_{, R}}\left(\left(\Mod_A^\omega, \Qoppa_A\right), \left(\Mod_B^\omega, \Qoppa_B\right)\right) &\simeq \mathrm{Pn} \left(\Mod_{A^\op \otimes_R B}^{L\omega}, \Qoppa_{A^\op \boxtimes B} \right) 
        \end{split}
    \end{equation*}  
    which are compatible with ordinary Morita theory, i.e. under the forgetful functors $ \Cathidem_{, R}, \Catpidem_{, R} \to \Catex_R $, the hermitian (resp. Poincaré) functor represented by the pair $ (P \in \Mod_{A^\op \otimes_R B}, q) $ is the $ R $-linear exact functor which is given by tensoring with the $ A^e $-$ B^e $-bimodule $ P $ (compare, for instance, \cite[p.738]{LurHA}). 
\end{corollary}
\begin{proof}	
    By Morita theory, there is a $\mathrm{Mod}_R^\omega$-linear equivalence $\mathrm{Fun}_R^{\ex}(\mathrm{Mod}_{A^e}^\omega, \mathrm{Mod}_{B^e}^\omega)\simeq\mathrm{Mod}_{A^{e,\op}\otimes_{R^e}B^e}^{L\omega} $. 
    Transporting the Poincar{\'e} structure from Proposition~\ref{prop_item:internal hom in R-linear poincare cats} on $\mathrm{Fun}_{(\mathrm{Mod}_R^\omega, \Qoppa_R)}((\mathrm{Mod}_{A^e}^\omega, \Qoppa_A), (\mathrm{Mod}_{B^e}^\omega,\Qoppa_B))$ along this equivalence then produces the desired Poincar{\'e} structure $\Qoppa_{A^{\op}\boxtimes_R B}$ on $\mathrm{Mod}_{A^{e,\op}\otimes_{R^e}B^e}$. 
    The result about hermitian and Poincar\'e objects then follows from Lemma \ref{lemma:hermitian_poincare_objects_in_relative_functor_cats}.
\end{proof}

\begin{lemma}\label{lemma:hermitian_poincare_objects_in_relative_functor_cats}
    Let $ R $ be a Poincaré ring. 
    Let $(\mathcal{C},\Qoppa)$ and $(\mathcal{D},\Phi)$ be $(\mathrm{Mod}_{R^e}^\omega,\Qoppa_R )$-modules in $\Catpidem$. Then, using Notation~\ref{notation: internal mapping objects}, there are natural equivalences \[\mathrm{He}(\mathrm{Fun}_{(\mathrm{Mod}_{R^e}^\omega,\Qoppa_R )}((\mathcal{C},\Qoppa), (\mathcal{D},\Phi)))\simeq \mathrm{Maps}_{\Cathidem_{, R}}((\mathcal{C},\Qoppa), (\mathcal{D},\Phi))\] and \[\mathrm{Pn}(\mathrm{Fun}_{(\mathrm{Mod}_{R^e}^\omega,\Qoppa_R )}((\mathcal{C},\Qoppa), (\mathcal{D},\Phi)))\simeq \mathrm{Maps}_{\Catpidem_{, R}}((\mathcal{C},\Qoppa), (\mathcal{D},\Phi))\] compatible with the canonical inclusions into $\mathrm{Fun}_{R^e}^{\ex} (\mathcal{C},\mathcal{D})=U((\mathrm{Fun}_{(\mathrm{Mod}_{R^e}^\omega,\Qoppa_R )}((\mathcal{C},\Qoppa), (\mathcal{D},\Phi)))$.
\end{lemma}
\begin{proof}
    We will prove the statement identifying Poincar\'e objects with $ R $-linear Poincaré functors; the proof that hermitian objects may be identified with $ R $-linear hermitian functors is similar. 
    \begin{align*}
        \mathrm{Maps}_{\Catpidem_{, R}}((\mathcal{C},\Qoppa),(\mathcal{D},\Phi)) &\simeq \mathrm{Maps}_{\Catpidem_{, R}}((\mathrm{Mod}_{R^e}^\omega,\Qoppa_R)\otimes_{(\mathrm{Mod}_{R^e}^\omega, \Qoppa_R)}(\mathcal{C},\Qoppa),(\mathcal{D},\Phi))\\
        &\simeq \mathrm{Maps}_{\Catpidem_{, R}}((\mathrm{Mod}_{R^e}^\omega,\Qoppa_R),\mathrm{Fun}_{R}((\mathcal{C},\Qoppa),(\mathcal{D},\Phi)))\\
        &\simeq \mathrm{Maps}_{\Catpidem}(\mathbb{S}^u,\mathrm{Fun}_{R}((\mathcal{C},\Qoppa),(\mathcal{D},\Phi)))\\
        &\simeq \mathrm{Pn}(\mathrm{Fun}_{R}((\mathcal{C},\Qoppa),(\mathcal{D},\Phi)))
    \end{align*}
    where the last equivalence is \cite[Proposition 4.1.3]{CDHHLMNNSI}
\end{proof}
	
\begin{observation}\label{obs:conjugate_opposite_action_on_assoc_alg}
    Let $ R $ be an object of $ \EE_\infty\Alg^{BC_2} $, i.e. an $ \EE_\infty $-algebra with $ C_2 $-action. 
    There is a $ C_2 $-action on $ \EE_1 \Alg_R $ given by $ A \mapsto \lambda_* A^\op $. 
    Suppose $ (X,\lambda) $ is a scheme with an involution. 
    Consider $ \EE_1\Alg_{\mathcal{O}_X} := \EE_1 \Alg \left(\Mod_{\mathcal{O}_X}\right) $. 
    Then there is a $ C_2 $-action on $ \EE_1 \Alg_{\mathcal{O}_X} $ given by $ A \mapsto \lambda_* A^\op $. 
\end{observation}
\begin{definition}\label{defn:E1_alg_with_lambda_involution}
    Let $ R $ be an object of $ \EE_\infty\Alg^{BC_2} $, i.e. an $ \EE_\infty $-algebra with $ C_2 $-action. 
    An \emph{$ \EE_1 $-$ R $-algebra with $ \lambda $-involution} is an object of $ \left(\EE_1\Alg_R\right)^{hC_2} $ where $ \EE_1 \Alg_R $ is regarded as an $ \infty $-category with $ C_2 $-action via Observation \ref{obs:conjugate_opposite_action_on_assoc_alg}. 

    Suppose $ (X,\lambda) $ is a scheme with an involution. 
    An \emph{$ \EE_1 $-$ \mathcal{O}_X $-algebra with $ \lambda $-involution} is an object of $ \left(\EE_1\Alg_{\mathcal{O}_X}\right)^{hC_2} $ where $ \EE_1 \Alg_{\mathcal{O}_X} $ is regarded as an $ \infty $-category with $ C_2 $-action via Observation \ref{obs:conjugate_opposite_action_on_assoc_alg}. 
\end{definition}
\begin{remark}
    We note here that the notion in Definition~\ref{defn:E1_alg_with_lambda_involution} is called an $\EE_1$-algebra with $\lambda$\textit{-anti-}involution in \cite[Example 3.1.9]{CDHHLMNNSI}. 
    We chose to follow the convention set in \cite{book_of_involutions}. 
\end{remark}
\begin{definition}
    Let $ R $ be a Poincar\'e ring, and let $ A $ be an $ \EE_1 $-$ R $-algebra with a $ \lambda $-involution in the sense of Definition \ref{defn:E1_alg_with_lambda_involution}. 
    We will refer to the data of $ (M_A, N_A, N_A \to M_A^{\mathrm{t}C_2}) $ of Proposition~\ref{prop:relative_poincare_cats_basic_properties}\ref{propitem:classify_R_lin_hermitian_struct_mod_cat} as an \emph{$ R $-linear $ A $-module with genuine involution}. 
\end{definition}
\begin{remark}
    When $ R = \mathbb{S}^u $ is the initial Poincar\'e ring of Example~\ref{example:universal_poincare_ring_spectrum}, then an $ \mathbb{S}^u $-linear $ A $-module with genuine involution is simply an $ A $-module with genuine involution in the sense of \cite[Defintiion 3.2.3]{CDHHLMNNSI}. 
\end{remark}
The following result is an $ R $-linear version of \cite[Corollary 3.4.2]{CDHHLMNNSI}. 
\begin{proposition}\label{prop:classify_R_linear_Poincare_functors}
    Let $ R $ be a Poincar\'e ring, and let $ A, B $ be $ \EE_1 $-$ R^e $-algebras with $ \lambda $-involution in the sense of Definition \ref{defn:E1_alg_with_lambda_involution} and let $ (M_A, N_A, \alpha \colon N_A \to M_A^{\mathrm{t}C_2} ) $, $ (M_B, N_B, \beta \colon N_B \to M_B^{\mathrm{t}C_2}) $ be $ R^e $-linear modules with genuine involution over $ A $ and $ B $, respectively. 
    Suppose given a map $ f \colon A \to B $ of $ \EE_1 $-$ R^e $-algebras with $ \lambda $-involution, and consider the functor $ B \otimes_A - \colon \Mod_A^\omega \to \LMod_B^\omega $. 
    Then 
    \begin{enumerate}[label=(\arabic*)]
        \item the data of a $ R $-linear hermitian functor $ \left(\Mod_A^\omega, \Qoppa_{M_A}^\alpha \right) \to \left(\Mod_B^\omega, \Qoppa_{M_B}^\beta \right) $ covering the base change functor $ B \otimes_A - $ can be encoded by a triple $ (\delta, \gamma, \sigma ) $ where $ \delta \colon M_A \to M_B $ is a morphism in $ \LMod_{A \otimes_{R^e} A}^{\mathrm{h}C_2} $, $ \gamma \colon N_A \to N_B $ is a morphism in $ \Mod_{A \otimes_{R^e} R^{\varphi C_2}} $, and $ \sigma $ is a homotopy making the square
        \begin{equation*}
            \begin{tikzcd}
                N_A \ar[d,"\alpha"] \ar[r,"\gamma"] & N_B \ar[d,"\beta"] \\
                M_A^{\mathrm{t}C_2} \ar[r,"{\delta^{\mathrm{t}C_2}}"] & M_B^{\mathrm{t}C_2}
            \end{tikzcd}
        \end{equation*}
        commute. 
        
        \item $ (\delta, \gamma, \sigma) $ defines an $ R $-linear Poincar\'e functor if the maps
        \begin{equation*}
            \begin{split}
                B \otimes_A M_A \to (B \otimes_{R^e} B) \otimes_{A \otimes_{R^e} A} M_A \to M_B \\
                B \otimes_A N_A \to N_B    
            \end{split}    
        \end{equation*}
        are equivalences, where the map $ B \to B \otimes_{R^e} B $ is either induced by the right unit or the left unit.\footnote{The first condition is equivalent to the requirement that the map $ \lambda^* B \to B \otimes_{R^e} \lambda^* B $ induced by the left unit is an equivalence.} 
    \end{enumerate}
\end{proposition}
\begin{proof}
    Follows from Corollary \ref{cor:poincare_fourier_mukai}. 
\end{proof}
	Write $ \mathrm{Fm} $, resp. $ \mathrm{Pn} $ for the composite $ \Cath_{, R} \xrightarrow{U} \Catp \xrightarrow{\mathrm{Fm}} \Spaces $, resp. $ \Catp_{, R} \xrightarrow{U} \Catp \xrightarrow{\mathrm{Pn}} \Spaces $ where $ \mathrm{Fm} $ and $ \mathrm{Pn} $ are defined in \cite[Definitions 2.1.1, 2.1.3]{CDHHLMNNSI}. 
	\begin{proposition}\label{prop:forms_are_corepresented}
		Let $ (R, R^{\varphi C_2} \to R^{\mathrm{t}C_2}) $ be a Poincaré ring. 
		Then $ \left(\Mod_R^\omega, \Qoppa_R \right) $ corepresents the functors $ \mathrm{Fm} \colon \Cath_{, R} \to \Spaces $ and $ \mathrm{Pn} \colon \Catp_{, R} \to \Spaces $.
	\end{proposition}
	Since the forgetful functor $ \Cath_{, R} \to \Cath $ and its Poincaré counterpart both preserve filtered colimits, an immediate consequence is that the unit is compact (cf. \cite[Proposition 6.1.8]{CDHHLMNNSI}).
	\begin{corollary}\label{cor:unit_is_compact}
		Let $ (R, R^{\varphi C_2} \to R^{\mathrm{t}C_2}) $ be a Poincaré ring. 
		Then $ \left(\Mod_R^\omega, \Qoppa_R \right) $ is a compact object of both $ \Cath_{, R} $ and $ \Catp_{, R} $.  
	\end{corollary}
	\begin{proof}[Proof of Proposition \ref{prop:forms_are_corepresented}]
		We prove the statement for $ \mathrm{Pn} $; the proof for $ \mathrm{Fm} $ is similar and is left to the reader. 
		Recall that Proposition \ref{prop:relative_poincare_cats_basic_properties}.\ref{propitem:Rlin_Poincare_cats_is_symm_mon} furnishes an adjoint pair $ \Catp_{, R} \rlarrows \Catp $ of functors. 
		Write $ \overline{\mathcal{C}} = (\mathcal{C},\Qoppa_{\mathcal{C}}) \in \Catpidem_{, R} $. 
		Then
		\begin{align*}
			\mathrm{Pn}(\mathcal{C}) = \mathrm{Maps}_{\Catp}\left((\Spectra^f,\Qoppa^u), U(\overline{\mathcal{C}})\right) &\simeq \mathrm{Maps}_{\Catp_{, R}}\left(\left(\Mod_R^\omega,\Qoppa_R\right)\otimes(\Spectra^f,\Qoppa^u), \overline{\mathcal{C}}\right)\\
            &\simeq \mathrm{Maps}_{\Catp_{, R}}\left((\mathrm{Mod}_{R^e}^\omega, \Qoppa_R),\overline{C}\right) \,,
		\end{align*}
		where the first equivalence is \cite[Proposition 4.1.3]{CDHHLMNNSI}. 
	\end{proof}
	
\subsection{\'Etale descent for Poincaré \texorpdfstring{$\infty$}{∞}-categories}~\label{subsection: etale descent results}

In this subsection we will define the $C_2$-\'etale topology (Definition~\ref{defn:C2_etale_topology}) on Poincar\'e schemes. We will then establish that some key constructions related to the $R$-linear Poincar\'e categories considered in the previous subsection are sheaves in this topology. We begin by defining these constructions.
\begin{notation}\label{notation:poincare_ring_basechange}
    Let $ (R, R\to R^{\varphi C_2} \to R^{\mathrm{t}C_2}) $ be a Poincaré ring. 
    There is a functor 
    \begin{equation*}
        \begin{split}
            \EE_\infty\Alg^{BC_2}_{R/} &\to \CAlgp_{R/-} \\
            S & \mapsto (S, S \to R^{\varphi C_2} \otimes_{R} S \to S^{t C_2}) =: (R, R \to R^{\varphi C_2} \to R^{\mathrm{t}C_2}) \otimes S \,,
        \end{split}    
    \end{equation*}
    where the map $ R^{\varphi C_2} \otimes_R S \to R^{\mathrm{t}C_2} \otimes_{R^{\mathrm{t}C_2}} S^{\mathrm{t}C_2} \simeq S^{\mathrm{t}C_2} $ is given by base change along the Tate-valued norm composed with the structure map $ R^{\varphi C_2} \to R^{\mathrm{t}C_2} $. 
    Composing the aforementioned functor with the functor that sends a Poincaré ring to its category of compact modules equipped with the canonical Poincaré structure defines a functor
    \begin{equation*}
        \Mod^\mathrm{p} \colon \EE_\infty\Alg^{BC_2}_{R/-} \to \EE_\infty\Alg\left(\Catpidem_{, R}\right) \,.
    \end{equation*}
\end{notation}
\begin{notation}\label{notation:scheme_involution_basechange}
    Let $ X $ be a scheme with an involution $ \lambda $ and let $ p \colon X \to Y $ exhibit $ Y $ as a good quotient of $ X $. 
    If $ j \colon U \to Y $ is flat, let us write $ p ^*U $ for the tuple $ (X \times_Y U, j^*(\sigma), U, j^*(p)) $ of Remark \ref{remark:restriction_of_schemes_with_involution}. 
    Then the assignment $ (j \colon U \to Y ) \mapsto \left(\Mod^\omega_{j^*X}, \Qoppa_{j^*\underline{\mathcal{O}}}\right) $ defines a functor 
    \begin{equation*}
        \Mod^\mathrm{p} \colon \mathrm{\acute{E}t}_Y^\op \to \EE_\infty \Alg\left(\Mod_{\left(\Mod^\omega_X, \Qoppa_{\underline{\mathcal{O}}}\right)}(\Catpidem)\right) \,.
    \end{equation*}
\end{notation}
\begin{observation}
    Let $ R $ be a discrete commutative ring with a $ C_2 $-action, and recall that $ X:= \Spec R \to \Spec(R^{C_2}) = Y $ may be regarded as a $ C_2 $-scheme (Observation \ref{obs:fixpt_Mackey_functor_as_affine_C2_scheme}). 
    Then base change along $ R^{C_2} \to R $ defines a map from étale covers of $ Y $ of Notation \ref{notation:scheme_involution_basechange} to $ C_2 $-equivariant étale covers of $ X $ of Notation \ref{notation:poincare_ring_basechange}. 
    However, not all $ C_2 $-equivariant étale covers of $ X $ arise from this construction. 
    For instance, consider $ R $ with the trivial involution regarded as a Poincaré ring via Example \ref{ex:fixpt_Mackey_functor}. 
    Consider the map $ R \xrightarrow{r \mapsto (r,r)} R \times R $, where $ R \times R $ is endowed with the flip action $ (r, s) \mapsto (s,r) $. 
    This is a $ C_2 $-equivariant étale cover which is not base changed from an étale cover of $ R $. 
\end{observation}

We will now define analogues of the \'etale topology we will be considering.

\begin{definition}\label{defn:C2_etale_topology}
    Let $f:R\to S$ be a map of Poincar{\'e} rings. Then we say that $f$ is \textit{$C_2$-\'etale} if the map $R^{C_2}\to S^{C_2}$ is an \'etale map in the sense of \cite[Definition 7.5.0.4]{LurHA}, and that the induced maps $R^{\phi C_2}\otimes_{R^{C_2}}S^{C_2}\to S^{\phi C_2}$ and $R^e \otimes_{R^{C_2}}S^{C_2}\to S^e $ are equivalences. 
    Let $ (X,\lambda,Y,p) $ be a scheme with involution admitting a good quotient (Definition \ref{defn:Category of good quotients}). 
    A map $ (U, \ell, V, \pi) \to (X,\lambda,Y,p) $ in $ \mathrm{qSch}^{C_2} $ is said to be \emph{$ C_2 $-étale} if $ V \to Y $ is étale and the canonical map $ U \to V \times_Y X $ is an equivalence of schemes with $ C_2 $-action. 
\end{definition}

\begin{lemma}
    The $C_2$-\'etale topology of Definition \ref{defn:C2_etale_topology} defines a Grothendieck pretopology. 
\end{lemma}
\begin{proof}
    The only thing to check is that if $f:R\to S$ is $C_2$-\'etale and $g:R\to T$ is any other map of Poincar{\'e} rings, then the map $T\to S\otimes_R T$ is also $C_2$-\'etale. Note first that, up to localization at a subset of $\pi_0(R^{C_2})$, $S^{C_2}$ is a finite colimit of $R^{C_2}$ as a module. 
    Recall that taking geometric fixed points with respect to any subgroup is symmetric monoidal; $ C_2 $-étaleness of $ R \to S $ gives equivalences $(S\otimes_R T)^{e}\simeq S^e \otimes_{R^e} T^e \simeq S^{C_2}\otimes_{R^{C_2}}T^e$ and similarly $(S\otimes_R T)^{\phi C_2} \simeq S^{C_2} \otimes_{R^{C_2}} T^{\phi C_2} $. 
    It follows that $(S\otimes_R T)_{\mathrm{h}C_2} \simeq S^{C_2}\otimes_{R^{C_2}}\left(T^e_{\mathrm{h}C_2}\right)$ , whence by the isotropy separation sequence we deduce \[(S\otimes_R T)^{C_2}\simeq S^{C_2}\otimes_{R^{C_2}}T^{C_2}\,;\] the latter is \'etale over $T^{C_2}$ by \cite[Remark 7.5.0.5]{LurHA}. The rest of the conditions for $T\to S\otimes_R T$ being $C_2$-\'etale were proven in the course of showing that the map on fixed points is \'etale, whence the result.
\end{proof}

In fact this Grothendieck topology for schemes with good quotients is already a familiar one.

\begin{lemma}~\label{lem: c2 etale topology=etale topology on the base}
    Let $(X,\lambda, Y ,p)$ be a scheme with good quotient, and let $C_2-\Acute{e}t_{(X,\lambda,Y,p)}$ be the category with objects $C_2$-\'etale maps $(U,\sigma, V,q)\to (X,\lambda, Y, p)$ and morphisms $C_2$-\'etale maps between them. Then the functor \[C_2-\Acute{e}t_{(X,\lambda,Y,p)}\to \Acute{e}t_Y\] from $C_2$-\'etale $(X,\lambda,Y,p)$-schemes to \'etale $Y$ schemes given by sending $(U,\sigma, V, p)\mapsto V$ is an equivalence.
\end{lemma}
\begin{proof}
    The inverse of this functor is given by the construction of Notation~\ref{notation:scheme_involution_basechange}.
\end{proof}

\begin{remark}\label{rmk:C2_etale_isovariant_etale_comparison}
    Fix a scheme with involution admitting a good quotient $ (X,\lambda,Y,p) $. Suppose further that $p:X\to Y$ is a scheme-theoretic geometric quotient in the sense of \cite[Definition 2.6]{isovariant}.
    Then the functor $ U $ (see Proposition \ref{prop:products_of_C2_schemes}) sends $C_2$-\'etale covers of $ (X,\lambda,Y,p) $ to \textit{isovariant \'etale covers} of $ X $ in the sense of Thomason \cite[Definition 2.6]{isovariant}. 
    The proof that these do indeed give rise to the same covers is the content of \cite[Proposition 2.17]{isovariant} and Lemma~\ref{lem: c2 etale topology=etale topology on the base}. 

    Note that, when they both exist, a good quotient of $Y$ must agree with the scheme-theoretic quotient of Thomason (\cite[Definition 2.3]{isovariant}). This is because both are categorical quotients, the former by \cite[Proposition V.1.3, Proposition V.1.8]{Grothendieck2003RevtementsE} and the latter by definition.
\end{remark}

We will now show that the prestacks of Notation~\ref{notation:poincare_ring_basechange} and Notation~\ref{notation:scheme_involution_basechange} are in fact stacks for this topology.

\begin{proposition}\label{prop:Poincare_modules_as_etale_sheaf}
    Let $ R $ be a Poincaré ring. Let $ (X, \lambda, Y, p) $ be a scheme with involution $ X $ and a good quotient $ Y $. Then:
    \begin{enumerate}[label=(\arabic*)]
        \item~\label{prop:Poincare_modules_as_etale_sheaf_affine_spectral} The assignment of Notation~\ref{notation:poincare_ring_basechange} is a hypersheaf on the small $C_2$-equivariant étale site of $ R $. 
        \item~\label{prop:Poincare_modules_as_etale_sheaf_inv_scheme} The assignment of Notation~\ref{notation:scheme_involution_basechange} is a hypersheaf on the small étale site of $ Y $. 
    \end{enumerate}
\end{proposition}
\begin{proof}
    Since limits in categories of algebras and modules are computed at the level of underlying objects, it suffices to show that the functor sends an étale hypercovering $ j_\bullet \colon S \to T^\bullet $ to a limit diagram in $ \Catpidem $.  
    By Proposition 6.1.4 of \cite{CDHHLMNNSI}, it suffices to show that the relevant diagram is a limit diagram in $ \Cathidem $.  
    The proof of Lemma 5.4 in \cite{MR3190610} implies that the diagram defines a limit diagram on underlying $ \infty $-categories. 
    Thus by Remark 6.1.3 of \cite{CDHHLMNNSI}, it suffices to show that $ j_\bullet^* \colon \Mod^\omega_{S} \to \Mod^\omega_{ T^\bullet} $ induces an equivalence $ \Qoppa_{R^\mathrm{p} \otimes S} \xrightarrow{\sim} \lim_{\Delta} \Qoppa_{R^\mathrm{p} \otimes T^\bullet} \circ \left(j_\bullet^*\right)^{\mathrm{op}} $ of quadratic functors $ \Mod^{\omega,\op}_{R^e \otimes_{R^{C_2}} S} \to \Spectra $. 
    This follows from our assumption on $ S \to T^\bullet $ and \cite[Theorem 3.3.1]{CDHHLMNNSI}. 
    
    The proof of the second point is similar.
\end{proof}	
	
\section{The Poincaré Picard space}\label{section:the_poincare_picard_group}
A symmetric monoidal $ \infty $-category $ \mathcal{C} $ has an associated Picard spectrum; a point in $ \mathfrak{pic}\,\mathcal{C} $ is a tensor-invertible object in $ \mathcal{C} $. 
Just as a \emph{Poincaré object} in a Poincaré $ \infty $-category $ \left(\mathcal{C},\Qoppa\right) $ can be regarded as an enhancement of an object of $ \mathcal{C} $, given a symmetric monoidal Poincaré $ \infty $-category $ \left(\mathcal{C},\Qoppa\right) $, it is natural to consider enhancements of invertible objects which take into account the Poincaré structure. 
Combined with the constructions of \S\ref{section:poincare_ring_spectra}, we may consider the Poincaré Picard space of Poincaré rings and of schemes with involution. 

These variants of the Picard group are interesting invariants in their own right. 
The Poincaré Picard group of a scheme with trivial involution recovers the group of isometry classes of discriminant bundles (see \S\ref{subsection:the_discrete_case:nondegenerate_hermitian_forms} and \S\ref{subsection:poincare_picard_group_inv_scheme}). 
We also briefly discuss the relationship between the Poincaré Picard group and Grothendieck--Witt theory in Observation \ref{obs:pnpic_to_GW}. 
When a symmetric monoidal $ \infty $-category is of the form $ \Mod^\omega_R $ (i.e. a derived category), a point of the Picard space of $ \Mod^\omega_R $ is a generalization of the notion of a line bundle on $ \Spec R $. 
A result of Fausk \cite{MR1966659} characterizes $ \pi_0 \mathfrak{pic} $ of the derived category of a scheme in terms of line bundles and their shifts.    
We show that the Poincaré Picard group of a scheme with involution admits an analogous characterization.
	
In \S\ref{subsection:pnpic_general_and_units}, we introduce the Poincaré Picard spectrum $ \Picp $ of a symmetric monoidal Poincaré $ \infty $-category and establish its basic properties. 
In \S\ref{subsection:calgp_units}, we introduce a notion of units $ \gmq $ for Poincaré rings and show that there is an equivalence $ \Omega \Picp \circ \Mod^\mathrm{p} \simeq \gmq $. 
In \S\ref{subsection:the_discrete_case:nondegenerate_hermitian_forms}, we discuss the classical counterpart of $ \Picp $ consisting of line bundles with nondegenerate $ \lambda $-hermitian pairings.
In \S\ref{subsection:poincare_picard_group_inv_scheme}, we prove an involutive/hermitian version of Fausk's result. 
	
	\subsection{Invertible Poincaré objects}\label{subsection:pnpic_general_and_units}  
	In this section, we show that the collection of invertible Poincaré objects in a symmetric monoidal Poincaré $ \infty $-category $ \left(\mathcal{C},\Qoppa\right) $ assembles into a connective spectrum $ \Picp\left(\mathcal{C},\Qoppa\right) $ which is functorial in $ \left(\mathcal{C},\Qoppa\right) $ (Definition~\ref{definition:poincare_picard_space}).  
	We show that the forgetful map $ \mathrm{Pn}\left(\mathcal{C},\Qoppa\right) \to \mathcal{C} $ induces a map $ \Picp\left(\mathcal{C},\Qoppa\right) \to  \mathfrak{pic}\,\mathcal{C} $ of connective spectra (Definition~\ref{defn:pnpic_to_pic_underlying}), and identify its fiber (Proposition \ref{prop:Poincare_Picard_to_Picard_fiber}). 
	
\begin{definition}
    \label{definition:poincare_picard_space} 
    The \emph{Poincaré Picard spectrum} is the functor $$ \Picp \colon \EE_\infty\Mon\left(\Catp\right) \xrightarrow{\Pn} \EE_\infty\Mon\left(\Spaces\right) \xrightarrow{\mathfrak{pic}} \EE_\infty\Mon^{\mathrm{gp}}\left(\Spaces\right) \simeq \Spectra_{\geq 0} $$ induced by the lax symmetric monoidal structure on $ \Pn $ as recounted in Proposition~\ref{prop: pn is lax monoidal}. 
    The \emph{Poincaré Picard space} is the functor $$ \mathcal{P}\mathrm{ic}^\mathrm{p} := \Omega^\infty \Picp \,. $$
    The \emph{Poincaré Picard group} is the functor $$ \mathrm{Pic}^\mathrm{p} \colon \EE_\infty\Mon\left(\Catp\right) \xrightarrow{\Picp} \Spectra_{\geq 0} \xrightarrow{\pi_0} \mathrm{AbGrp} \,. $$
    
    Let $R$ be a Poincaré ring spectrum. 
    We define the \emph{Poincaré Picard spectrum of $R$} to be $$\Picp(R):=\Picp\left(\Mod^\mathrm{p}_R \right) \,, $$ where $ \Mod^\mathrm{p} $ is the functor of Theorem~\ref{thm:calgp_to_poincare_cat}, and likewise for the \emph{Poincaré Picard space} and \emph{group of $ R $}. 
    For $(X,\lambda, Y, p)\in \mathrm{qSch}^{C_2}$ we similarly define \[\Picp(X,\lambda,Y,p):= \mathfrak{pic}(\Pn(\mathrm{Mod}^\mathrm{p}_{\mathcal{O}(X,\lambda, Y, p)},\Qoppa_{\underline{\mathcal{O}}}))\] where the Poincar\'e category $(\mathrm{Mod}_{\mathcal{O}},\Qoppa_{\underline{\mathcal{O}}})$ is that of Construction~\ref{cons:structure_sheaf_of_Green_functors}.
\end{definition}
\begin{remark}
    By Remark \ref{rmk:he_lax_monoidal}, the functor $ \Picp \colon \EE_\infty\Mon\left(\Catp\right) \to \EE_\infty\Mon^{\mathrm{gp}}\left(\Spaces\right) $ is equivalent to the composite $ \EE_\infty\Mon\left(\Catp\right) \subseteq \EE_\infty\Mon\left(\Cath\right) \xrightarrow{\mathrm{He}} \EE_\infty\Mon\left(\Cat\right) \xrightarrow{\mathfrak{pic}} \EE_\infty\Mon^{\mathrm{gp}}\left(\Spaces\right) $. This amounts to the fact that a tensor-invertible hermitian structure on an object in a Poincar\'e $\infty$-category is automatically Poincar\'e.
\end{remark}
	
\begin{remark}
    \label{remark:poincare_picard_points_desc}
    Let $R$ be a Poincar\'e ring with $\mathrm{Mod}^\mathrm{p}_R =  \left(\Mod_{R^e}^\omega, \Qoppa_R \right)$, where $\mathrm{Mod}^\mathrm{p}$ is the functor of Theorem~\ref{thm:calgp_to_poincare_cat}. Let $(M_R=R, N_R= R^{\varphi C_2}, R^{\varphi C_2}\to R^{\mathrm{t}C_2})$ be the module with genuine involution associated to $ \Qoppa_R $. 
    Then a point in the Poincaré Picard space $\mathcal{P}\mathrm{ic}(R)$ is the data of a pair $ (\mathcal{L}, q ) $, where $ \mathcal{L} $ is an invertible module in $ \Mod_{R^e}^\omega $ and $ q $ is a point in $ \Omega^\infty\Qoppa_R(\mathcal{L}) $. 
    By \cite[Proposition 1.3.11]{CDHHLMNNSI}, the data of $ q $ is equivalent to the data of points in the lower left and upper right corner of the square
    \begin{equation}
        \begin{tikzcd}
            \Qoppa(\mathcal{L}) \ar[r] \ar[d] & \hom_{R^e}(\mathcal{L}, R^{\varphi C_2}) \ni \ell(q) \ar[d] \\
            b(q) \in \hom_{R^e \otimes R^e}\left(\mathcal{L} \otimes \mathcal{L}, R^e\right)^{\mathrm{h}C_2} \ar[r] & \hom_{R^e}(\mathcal{L}, R^{\mathrm{t}C_2})
        \end{tikzcd}
    \end{equation} 
    and a path between their images in the lower right corner. 
    In particular, the adjoint $ b(q)^\dagger $ of $ b(q) $ must define a nondegenerate hermitian form on $ \mathcal{L} $, that is, an equivalence $ b(q)^\dagger \colon \mathcal{L} \simeq \lambda^*\hom_{R^e}(\mathcal{L}, R^e) = D_{\Qoppa_R}(\mathcal{L}) $. 
    Symmetry of $ b(q) $ implies that there is a commutative diagram
    \begin{equation*}
        \begin{tikzcd}[column sep=tiny]
            \mathcal{L} \ar[rr,"{b(q)^\dagger}","{\sim}"'] \ar[rd,"{\eta_{\mathcal{L}}}"'] && \lambda^*\hom_{R^e}(\mathcal{L}, R^e)=: D_{\Qoppa_R}(\mathcal{L})  \\
            & D_{\Qoppa_R}^\op D_{\Qoppa_R}(\mathcal{L}) \ar[ru,"{D(b(q)^\dagger)}"',"{\sim}"]&
        \end{tikzcd}
    \end{equation*}
    by \cite[Remark 2.1.6]{CDHHLMNNSI}. 
    
    Write $ (\mathcal{L}^\vee,q^\vee) $ for the inverse of $ (\mathcal{L},q) $ in $\Pic(R^e)$. 
    The linear part of $ \Qoppa_R $ and $ q $ induce an $ R^e $-linear map $ \ell(q^\vee) \colon \mathcal{L}^\vee \to R^{\varphi C_2} $ so that the following diagram commutes
    \begin{equation}\label{diagram:pnpic_linear_part_condition}
        \begin{tikzcd}[column sep=huge]
            \mathcal{L} \otimes_{R^e} \mathcal{L}^\vee \ar[d,"\mathrm{ev}", "\sim"']  \ar[r,"{\ell(q) \otimes \ell(q^\vee)}"] & R^{\varphi C_2} \otimes_{R^e} R^{\varphi C_2} \ar[d,"\mathrm{multiplication}"] \\
            R^e \ar[r,"\mathrm{given}"] & R^{\varphi C_2}   \,.
        \end{tikzcd}
    \end{equation} 
\end{remark}
\begin{lemma}\label{lemma:lax_monoidal_nat_transf}
    Write $ U $ for the forgetful functor $ \Catp \to \Catex $. 
    Then there is a canonical lax symmetric monoidal transformation $ \mathrm{Pn} \Rightarrow (-)^{\simeq} \circ U $ of functors $ \Catp \to \Spaces $, where the lax symmetric monoidal structure on $ \mathrm{Pn} $ is as described in Proposition~\ref{prop: pn is lax monoidal} and $ (-)^\simeq $ is the functor sending an $ \infty $-category to its interior core $ \infty $-groupoid. 
    Its value on a given Poincaré $ \infty $-category $ \left(\mathcal{C},\Qoppa\right) $ is given by the map $ \mathrm{Pn}\left(\mathcal{C},\Qoppa\right) \to \mathcal{C}^\simeq $ sending a Poincaré object $ (x,q) $ to the object $ x $ \cite[(43)]{CDHHLMNNSI}. 
\end{lemma}
\begin{corollary}\label{cor:pnpic_to_pic_underlying}
    There is a canonical natural transformation $ \Picp \Rightarrow \mathfrak{pic} \circ U $ of functors $ \EE_\infty\Mon\left(\Catp\right) \to \EE_\infty\Mon^{\mathrm{gp}}\left(\Spaces\right) \simeq \Spectra_{\geq 0} $ whose value on a given Poincaré $ \infty $-category $ \left(\mathcal{C},\Qoppa\right) $ is given by the map $ \mathrm{Pn}\left(\mathcal{C},\Qoppa\right) \to \mathcal{C}^\simeq $ sending a Poincaré object $ (x,q) $ to the object $ x $ \cite[(43)]{CDHHLMNNSI}. 
\end{corollary}
\begin{proof} [Proof of Lemma \ref{lemma:lax_monoidal_nat_transf}]
    Observe that $ (-)^\simeq \circ U $ is corepresented by the object $ \mathrm{Hyp}\left(\Spectra^{\mathrm{fin}}\right) \in \Catp $ (where $ \mathrm{Hyp} $ is the left adjoint to $ U $ of \cite[Corollary 7.2.21]{CDHHLMNNSI}). 
    Now recall that there is a map \[ \mathrm{hyp}^u \colon \mathrm{Hyp}\left(\Spectra^{\mathrm{fin}}\right) \to \left(\Spectra^{\mathrm{fin}},\Qoppa^u\right) \] by \cite[Remark 7.2.22, (42)]{CDHHLMNNSI}. 
    For any Poincaré $ \infty $-category $ \left(\mathcal{C},\Qoppa\right) $, precomposition with $ \mathrm{hyp}^u $ induces a hermitian functor  
    \begin{equation*}
        \Fun^{\mathrm{ex}}\left(\left(\Spectra^{\mathrm{fin}},\Qoppa^u \right),\left(\mathcal{C},\Qoppa\right)\right) \to \Fun^{\mathrm{ex}}\left(\mathrm{Hyp}\left(\Spectra^{\mathrm{fin}}\right),\left(\mathcal{C},\Qoppa\right)\right) \,.
    \end{equation*} 
    There is a canonical equivalence $ \Fun^{\mathrm{ex}}\left(\mathrm{Hyp}\left(\Spectra^{\mathrm{fin}}\right),\left(\mathcal{C},\Qoppa\right)\right) \simeq \mathrm{Hyp}\left(\mathcal{C},\Qoppa\right) $. To see this, note that for any closed symmetric monoidal category $\mathcal{D}$ and any object $d\in \mathcal{D}$, the internal mapping functors $\underline{\mathrm{Hom}}_{\mathcal{D}}(-,d):\mathcal{D}^{\op}\to \mathcal{D}$ is right adjoint to $\underline{\mathrm{Hom}}_\mathcal{D}(-,d)^{\op}:\mathcal{D}\to \mathcal{D}^{\op}$. By definition of the internal mapping objects in $\Catp$ we have a commutative diagram \[
    \begin{tikzcd}
        \mathrm{Cat}^{p}_\infty \arrow[r, "U"] \arrow[d, "  {\mathrm{Fun(-,(\mathcal{C},\Qoppa))}}"] & \mathrm{Cat}^{\mathrm{ex}}_\infty \arrow[d, "\mathrm{Fun}^{\mathrm{ex}}(-\mathcal{C})"] \\
        {\mathrm{Cat}^{p,\mathrm{op}}_\infty} \arrow[r, "U^{\mathrm{op}}"]                             & {\mathrm{Cat}^{\mathrm{ex,op}}_\infty}           
    \end{tikzcd}
    \] and we note that $\mathrm{Hyp}(-)$ is both left and right adjoint to $U$ by \cite[Corollary 7.2.21]{CDHHLMNNSI}. Thus $\mathrm{Hyp}(-)^{\op}$ is also left adjoint to $U^{\op}$, and passing to left adjoints in the above diagram we get a natural isomorphism $\mathrm{Fun}(\mathrm{Hyp}(\mathcal{D}),(\mathcal{C},\Qoppa))\simeq \mathrm{Hyp}(\mathrm{Fun}^{\ex}(\mathcal{D},\mathcal{C}))$.
    Taking Poincaré objects, we have
    \begin{equation}\label{diagram:restrict_along_hyperbolization}
        \begin{tikzcd}[row sep=small]
            \mathrm{Pn}\left(\Fun^{\mathrm{ex}}\left(\left(\Spectra^{\mathrm{fin}},\Qoppa^u \right),\left(\mathcal{C},\Qoppa\right)\right)\right) \ar[r] \ar[d,phantom,"{\rotatebox{90}{$\simeq$}}"] & \mathrm{Pn}\Fun^{\mathrm{ex}}\left(\mathrm{Hyp}\left(\Spectra^{\mathrm{fin}}\right),\left(\mathcal{C},\Qoppa\right)\right) \ar[d,phantom,"{\rotatebox{90}{$\simeq$}}"]         \\ 
            \mathrm{Pn}\left(\mathcal{C},\Qoppa\right) & \mathrm{Pn}\mathrm{Hyp}\left(\mathcal{C},\Qoppa\right) \simeq \mathcal{C}^\simeq
        \end{tikzcd}
    \end{equation}
    where the vertical equivalences are \cite[Example 6.2.5 \& Proposition 2.2.5]{CDHHLMNNSI}. 
    Unraveling the aforementioned equivalences, we see that the total composite in (\ref{diagram:restrict_along_hyperbolization}) from the lower left to the lower right has the effect of sending a Poincaré object $ (X,q \in \Omega^\infty\Qoppa(X)) \mapsto X $. 
    The desired result now follows from symmetric monoidality of the Yoneda embedding \cite[Corollary 3.10]{Nik_16}. 
\end{proof}

\begin{definition}\label{defn:pnpic_to_pic_underlying}
    The transformation of Corollary \ref{cor:pnpic_to_pic_underlying} induces a natural transformation $ \Picp  \Rightarrow \mathfrak{pic} $ of functors $ \CAlgp \to \Spectra_{\geq 0} $. 
    Given a point $(\mathcal{L},q)\in \pi_0(\Picp(R)) $, we will refer to its image $ \mathcal{L}:=U(\mathcal{L},q) $ in $ \mathfrak{pic}(R^e) $ as its \textit{underlying invertible module}. 
\end{definition}
\begin{remark}\label{rmk:pnpic_trivial_action_maps_to_2_torsion}
    Suppose $ R $ is a Poincaré ring with \emph{trivial} $ C_2 $-action.  
    Then for any point $ (\mathcal{L},q) $ in $ \Picp (R) $ (see Remark \ref{remark:poincare_picard_points_desc}), the form $ q $ induces an isomorphism $\mathcal{L}\simeq \mathrm{hom}_{R^e}(\mathcal{L},R^e) =: \mathcal{L}^\vee $. 
    Therefore, the image $\mathcal{L}$ of $ (\mathcal{L},q) $ in $ \Pic(R^e) $ under the map $ U $ of Definition \ref{defn:pnpic_to_pic_underlying} is $2$-torsion.
    In particular, the value of the natural transformation of Definition \ref{defn:pnpic_to_pic_underlying} on $ R $ refines to a map \[U:\Picp(R)\to \mathfrak{pic}(R^e)[2]:= \pi_0\mathfrak{pic}(R^e)[2] \fiberproduct_{\pi_0 \mathfrak{pic}(R^e)} \mathfrak{pic}(R^e)\] which factors the underlying invertible module map, where we use the notation $-[2]$ to denote the subgroup of $2$-torsion elements. 
\end{remark}
\begin{remark}\label{rmk:pnpic_to_pic_equivariant}
    If $ \left(\mathcal{C},\Qoppa\right) $ is a symmetric monoidal Poincaré $ \infty $-category, then by \cite[Corollary 5.3.18(ii)]{CDHHLMNNSI} the duality $ D_\Qoppa $ induces a $ C_2 $-action on $ \mathfrak{pic} \left(\mathcal{C}\right) = \mathfrak{pic} \left(\mathcal{C}^\op\right) $. 
    It follows from \cite[Corollary 2.2.10]{CDHHLMNNSI} that the canonical map $ \Picp\left(\mathcal{C},\Qoppa\right) \to \mathfrak{pic}(\mathcal{C}) $ is equivariant with respect to the trivial action on $ \Picp\left(\mathcal{C},\Qoppa\right) $ and the $ C_2 $-action on $ \mathfrak{pic}(\mathcal{C}) $ induced by the duality. 
    Equivalently, there is a factorization of the map $\pnpic(\mathcal{C},\Qoppa)\to \mathfrak{pic}(\mathcal{C})$ through a map $ \Picp\left(\mathcal{C},\Qoppa\right) \to \mathfrak{pic}(\mathcal{C})^{\mathrm{h}C_2} $ of connective spectra. 
\end{remark}
\begin{proposition}\label{prop:symmetric_pnpic}
    Let $ \left(\mathcal{C},\Qoppa\right) $ be a symmetric monoidal Poincaré $ \infty $-category and assume that $ \Qoppa $ is the symmetric Poincaré structure associated to some symmetric bilinear functor $ B $ \cite[Definition 1.2.11]{CDHHLMNNSI}. 
    Then the map $ \Picp\left(\mathcal{C},\Qoppa\right) \to \mathfrak{pic}(\mathcal{C})^{\mathrm{h}C_2} $ of Remark \ref{rmk:pnpic_to_pic_equivariant} is an equivalence of connective spectra. 
\end{proposition}
\begin{proof}
    By Lemma \ref{lemma:lax_monoidal_nat_transf} and \cite[Proposition 2.2.11]{CDHHLMNNSI}, the map $ \mathrm{Pn}\left(\mathcal{C},\Qoppa\right) \to \left(\mathcal{C}^\simeq\right)^{\mathrm{h}C_2} $ refines to an equivalence of $ \EE_\infty $-monoids, hence they have equivalent Picard spaces. 
    The result follows from the canonical equivalence $ \mathfrak{pic}\left(\left(\mathcal{C}^\simeq\right)^{\mathrm{h}C_2}\right) \simeq \left(\mathfrak{pic}\mathcal{C}^\simeq\right)^{\mathrm{h}C_2} $. 
\end{proof}
\begin{remark}~\label{remark: identification of the c2 action on pic}
    We frequently will be interested in applying Proposition~\ref{prop:symmetric_pnpic} to a Poincaré $ \infty $-category of the form  $\mathrm{Mod}_{R^s}^{\mathrm{p}}$ for some symmetric Poincar\'e ring $R^s$ with involution $\lambda:R^e\to R^e$. 
    Recall that there are multiple \textbf{distinct} actions of $C_2$ on $\mathfrak{pic}\left(\mathrm{Mod}_{R^e}^\omega\right)$ (Remark \ref{remark:different_c2_actions}).   
    Combining \ref{C2action:basechange} and \ref{C2action:untwisted_canonical_duality} allows us to regard $\mathfrak{pic}\left(\mathrm{Mod}_{R^e}^\omega\right)$ as a connective spectrum with $ C_2 \times C_2 $-action. 
    By Lemma~\ref{lemma: duality identification}, the action induced by \ref{C2action:duality_associated_to_Poincare_structure} on $\mathfrak{pic}(R^e)$ is obtained by pulling the $ C_2 \times C_2 $-action back along the diagonal homomorphism $ C_2 \xrightarrow{\Delta} C_2 \times C_2 $. 
    Moreover, observe that the untwisted duality corresponds to inversion with respect to the grouplike $ \EE_\infty $ structure on $\mathfrak{pic}(R^e)$ essentially by definition. 
    It follows that the actions induced by \ref{C2action:untwisted_canonical_duality} and \ref{C2action:duality_associated_to_Poincare_structure} on $\mathfrak{pic}(R^e)$ are related by the group-inversion map, and we can relate the homotopy fixed points with respect to these actions via Remark~\ref{rmk:C2_spectra_twists}.
\end{remark}
\begin{example}
    \label{example:poincare_Picard_group_and_topological_field_theories}
    Let $R\in \CAlg$ be a ring spectrum with the trivial involution. We remark here that by \cite[Example 2.4.28]{lurie2009classificationtopologicalfieldtheories} we can identify elements of $\mathrm{Pic}^\mathrm{p}(R^s)$ with invertible unoriented $1$-dimensional field theories.
\end{example}
\begin{proposition}\label{prop:Poincare_Picard_to_Picard_fiber}
    Let $ R $ be a Poincar{\'e} ring with underlying genuine $C_2$ spectrum $R^L$ (Remark \ref{rmk:Poincare_ring_has_underlying_C2_spectrum_alg}). 
    Write $ \lambda \colon R^e \simeq R^e $ for the $ C_2 $-action on the underlying $ \EE_\infty $-ring associated to $ R $. 
    Then the fiber of the map \[U \colon \pnpic(R)\to \mathfrak{pic}(R^e)\] in connective spectra can be naturally identified with $ \mathrm{gl}_1\left(\left(R^L\right)^{C_2}\right) $.
    Moreover, the connecting map $ \Omega \mathfrak{pic}(R^e) \simeq \mathrm{gl}_1(R^e) \to \mathrm{fib}(U) $ is induced by the norm $ \mathrm{Aut}_{R^{e}}(R^e) \to \mathrm{Aut}_{N^{C_2}R^e}\left(N^{C_2}R^e\right) \xrightarrow{ - \otimes_{N^{C_2}R^e}R^L} \mathrm{Aut}_{R^L}(R^L) $ (see (\ref{eq:Poincare_ring_functorial_norm})). 
\end{proposition}
\begin{proof}
    Consider the fiber sequence of Corollary \ref{cor:Poincare_Brauer_to_Brauer_fiber}. 
    Looping the fiber sequence gives the fiber sequence \[ \Omega \mathfrak{pic} \left(\Mod_{R^L}\left(\Spectra^{C_2}\right) \right) \to \Omega \Brp(R)\to \Omega \mathfrak{br}(R^e). \]
    We identify $ \Omega \Brp \simeq \Picp $ using Proposition \ref{prop:loops_pnbr_is_pnpic} and $ \Omega \mathfrak{br} \simeq \mathfrak{pic} $ by \cite[Proposition 7.6]{MR3190610}. 
    It follows from unwinding the definitions that, under the previous identifications, the map $ \Omega \Brp \to \Omega \mathfrak{br} $ agrees with the forgetful map $ \Picp \to \mathfrak{pic} $. 
    Furthermore, $ \Omega \mathfrak{pic} \left(\Mod_{R^L}\left(\Spectra^{C_2}\right) \right) \simeq \mathrm{Aut}_{R^L}(R^L)$. By adjunction we have that $\mathrm{End}_{R^L}(R^L)\simeq \mathrm{maps}_{\mathrm{Sp}^{C_2}}(\mathbb{S},R^L)\simeq \Omega^\infty (R^L)^{C_2}$ and under this equivalence we have that $\mathrm{Aut}_{R^L}(R^L) \simeq \mathrm{gl}_1\left(\left(R^L\right)^{C_2}\right) $. 
    The description of the connecting map follows from a similar argument to that of Corollary \ref{cor:Poincare_Brauer_to_Brauer_fiber}. 
\end{proof}
\begin{remark}
    Proposition \ref{prop:Poincare_Picard_to_Picard_fiber} could alternatively be proved by directly identifying the fiber, analogously to Corollary \ref{cor:Poincare_Brauer_to_Brauer_fiber}. 
    Note that there is no circularity even though Proposition \ref{prop:Poincare_Picard_to_Picard_fiber} cites results appearing in a later section. 
\end{remark}
	
\begin{corollary}
    \label{cor:pn_pic_of_the_unit}
    Let $\mathbb{S}^u$ be the universal Poincaré ring spectrum from Example \ref{example:universal_poincare_ring_spectrum}. 
    Then there is an isomorphism $$ \mathrm{Pic}^\mathrm{p}(\mathbb{S}^u)\simeq \mathbb{Z}/2 \, , $$ where a generator is given by the form $(-1,-1)\in \Qoppa_{\mathbb{S}^u}(\mathbb{S})\simeq \mathbb{S}\times_{\mathbb{S}^{tC_2}}\mathbb{S}$ over $\mathbb{S}$.
\end{corollary}
\begin{proof}
    By Proposition \ref{prop:Poincare_Picard_to_Picard_fiber}, there is a long exact sequence
    \begin{equation*}
        \cdots \to \pi_1 \mathfrak{pic}(\sphere) \to \pi_0 \mathrm{gl}_1\left(\underline{\sphere}^{C_2}\right) \to \mathrm{Pic}^\mathrm{p}(\mathbb{S}^u) \xrightarrow{\pi_0 U} \pi_0 \mathfrak{pic}(\sphere)\,,
    \end{equation*}
    where we write $ \underline{\sphere} $ for the sphere in $ C_2 $-spectra. 
    The only $2$-torsion element of $\Pic(\mathbb{S})\simeq \mathbb{Z}$ is $\mathbb{S}$. 
    By Remark \ref{rmk:pnpic_trivial_action_maps_to_2_torsion}, any element in $\mathrm{Pic}^\mathrm{p}(\mathbb{S}^u)$ must have underlying line bundle $2$-torsion, and so the map $ \pi_0 U $ is zero. 
    It follows that $ \mathrm{Pic}^\mathrm{p}(\mathbb{S}^u) $ is isomorphic to the cokernel of the map $ \mathbb{Z}/2 \simeq \pi_0(\sphere)^\times \to \pi_0(\mathbb{S}_{\mathrm{h}C_2}\oplus \mathbb{S})^\times $ induced by the norm. 
    Now $ \pi_0(\mathbb{S}_{\mathrm{h}C_2}\oplus \mathbb{S})^\times \simeq \mathbb{Z}/2\times \mathbb{Z}/2 $ by \cites[Examples 4.5 \& 5.6]{MR722522}[\S1.5]{MR692332}. 
    By \cite[Example 1.4.6]{Mazur_Thesis}, the norm map is nontrivial, hence the result follows. 
\end{proof}	
\begin{recollection}\label{rec:spherical_Witt}
    Let $k$ be a finite field of characteristic $2$, and let $\mathbb{S}_{W(k)}$ be the spherical Witt vectors on $k$ in the sense of \cite[Example 5.2.7]{lurie-elliptic-2}; it is an $ \EE_\infty $ ring spectrum with a canonical identification $ \pi_0 \mathbb{S}_{W(k)} \simeq W(k) $.  
    Endowing $\mathbb{S}_{W(k)}$ with the trivial $ C_2 $-action, it follows from \cite[Example 3.4]{Nikolaus-Frob} that the Tate-valued Frobenius $\phi_2:\mathbb{S}_{W(k)}\to \mathbb{S}_{W(k)}^{\mathrm{t}C_2}$ is an equivalence. 
    Consider $ \mathbb{S}_{W(k)}^{\mathrm{t}} $ equipped with the Tate Poincar{\'e} structure of Example~\ref{example:tate_poincare_structure}.    
\end{recollection}
\begin{proposition}\label{prop:pnpic_spherical_Witt_char2}
    Let $ k $ be a finite field of characteristic 2, and let $ \mathbb{S}_{W(k)}^{\mathrm{t}} $ denote the spherical Witt vectors on $ k $, endowed with the Tate Poincaré structure of Example~\ref{example:tate_poincare_structure}.  
    Then there is an isomorphism
    \begin{equation*}
        \mathrm{Pic}^\mathrm{p}(\mathbb{S}_{W(k)}^{\mathrm{t}})\cong (W(k) \times W(k)^\times) /\mathrm{Gal}(k/\mathbb{F}_2) \,,
    \end{equation*}
    where $g\in \mathrm{Gal}(k/\mathbb{F}_2)$ acts on $W(k) \times W(k)^\times$ via $W(g)\times W(g)^{\times}$.
\end{proposition}
\begin{remark}
    Taking $k=\mathbb{F}_2$ in Proposition \ref{prop:pnpic_spherical_Witt_char2}, we deduce that $ \mathrm{Pic}^\mathrm{p}(\mathbb{S}_{W(\mathbb{F}_2)}) \simeq \mathbb{Z}_2\times (\mathbb{Z}_2\rtimes \mathbb{Z}/2)$ is nonzero -- in fact not even $2^\infty$-torsion.  
\end{remark}
\begin{proof}[Proof of Proposition \ref{prop:pnpic_spherical_Witt_char2}]
    By Proposition \ref{prop:Poincare_Picard_to_Picard_fiber} and Remark \ref{rmk:pnpic_trivial_action_maps_to_2_torsion}, we have a long exact sequence 
    \begin{equation*}
        \cdots \to \pi_1 \mathfrak{pic}\left(\mathbb{S}_{W(k)}^e\right) \xrightarrow{} \pi_0 \mathrm{gl}_1\left(\mathbb{S}_{W(k)}^{C_2}\right) \to \mathrm{Pic}^\mathrm{p}\left(\mathbb{S}_{W(k)}\right) \to \pi_0 \mathfrak{pic}\left(\mathbb{S}_{W(k)}^e\right)[2] \,.
    \end{equation*}
    First note that we have that $\Pic(\mathbb{S}_{W(k)})\cong \mathbb{Z}$, generated by $\Sigma \mathbb{S}_{W(k)}$, by \cite[Theorem 7.9]{MR3190610}. Thus $\pi_0\left(\mathfrak{pic}(\mathbb{S}_{W(k)})[2]\right)=0$. 
    On the other hand, we have that the unit map $\mathbb{S}_{W(k)}\to \Qoppa_{\mathbb{S}_{W(k)}^{\mathrm{t}}}(\mathbb{S}_{W(k)}) = \mathbb{S}_{W(k)}^{C_2} $ is split by the map $\Qoppa_{\mathbb{S}_{W(k)}^{\mathrm{t}}}(\mathbb{S}_{W(k)})\to \mathbb{S}_{W(k)}^{\phi C_2}=\mathbb{S}_{W(k)}$. 
    Consequently \[\pi_0(\Qoppa_{\mathbb{S}_{W(k)}^{\mathrm{t}}}(\mathbb{S}_{W(k)}))\cong \pi_0(\mathbb{S}_{W(k)}\oplus (\mathbb{S}_{W(k)})_{\mathrm{h}C_2})\cong W(k)\times W(k).\] 
    As a ring this is $W_2(W(k))$. Note that the $\delta$-ring structure on $W(k)$ induces a ring map $s:W(k)\to W_2(W(k))$ splitting the canonical projection map $W_2(W(k))\to W(k)$, and therefore $W_2(W(k))^\times$ splits as $W(k)^\times$ and the kernel of $W_2(W(k))^\times\to W(k)^\times$. This kernel is given by $W(k)$ via the map $w\mapsto 1+s(w)V(1)$. Hence $ \pi_0 \mathrm{gl}_1\left(\mathbb{S}_{W(k)}^{C_2}\right) \cong W(k) \times W(k)^\times $. 
    
    We then have that $\pi_0(\Picp(\mathbb{S}_{W(k)}^{\mathrm{t}}))\cong W(k)\times W(k)^\times/H$ where $H$ is the subgroup of Poincar{\'e} structures $q$ on $\mathbb{S}_{W(k)}$ which are identified by some automorphism $f:\mathbb{S}_{W(k)} \to \mathbb{S}_{W(k)} $. 
    By the defining property of spherical Witt vectors as a $ \sphere $-thickening of the canonical map $ \mathbb{F}_2 \to k $ (see \cite[Definition 5.2.1(c)]{lurie-elliptic-2}), there are equivalences 
    $$\mathrm{Maps}_{\mathrm{CAlg}}\left(\mathbb{S}_{W(k)}, \mathbb{S}_{W(k)}\right)\simeq \mathrm{Maps}_{\CAlg^{\heartsuit}_{\mathbb{F}_2}}(k,k)=\mathrm{Gal}\left(k/\mathbb{F}_2\right)$$ and $g\in \mathrm{Gal}(k/\mathbb{F}_2)$ acts on $W(k) \times W(k)^\times$ via $W(g)\times W(g)^{\times}$. 
    Consequently \[\mathrm{Pic}^\mathrm{p}(\mathbb{S}_{W(k)})\cong (W(k) \times W(k)^\times) /\mathrm{Gal}(k/\mathbb{F}_2)\,. \qedhere \] 
\end{proof}
\begin{proposition}
    \label{proposition:Poincare_Picard_group_of_free_Poincare_rings}
    Let $R\in \left(\CAlg^{\mathrm{cn}}\right)^{\mathrm{B}C_2}$ be a connective ring spectrum with a free $C_2$-action in the sense of \cite[Definition B.7.5.1]{Lurie_SAG}. 
    Then the canonical map $$\mathrm{Pic}^\mathrm{p}(R^\mathrm{s})\rightarrow \mathrm{Pic}^\mathrm{p}(\pi_0(R)^\mathrm{s})$$ induced by the projection $R\to \pi_0(R)$ is an isomorphism, where $ (-)^s $ denotes the symmetric Poincaré structure of Example \ref{example:tate_poincare_structure}.
\end{proposition}

\begin{proof}
    \label{proof:Poincare_Picard_group_of_free_Poincare_rings}
    By \cite[Corollary 7.2.2.19]{LurHA}, we have an equivalence $\tau_{\leq 1}\mathcal{P}\mathrm{ic}(R)\simeq \mathcal{P}\mathrm{ic}(\pi_0(R))$. We note that a $C_2$-action on $R$ is free if it is free on $\pi_0(R)$. By Proposition~\ref{prop:symmetric_pnpic} and Remark~\ref{remark: identification of the c2 action on pic} we have a fiber sequence $\mathfrak{pic}^\mathrm{p}(R^s)\simeq \mathfrak{pic}(R)^{\mathrm{h}(-\lambda)}$ (see Remark~\ref{rmk:C2_spectra_twists} for an explanation of the notation). By Remark~\ref{rmk:C2_spectra_twists} there is then a fiber sequence \[\mathfrak{pic}^\mathrm{p}(R^\mathrm{s})\to \mathfrak{pic}(R)\to \mathfrak{pic}(R)^{\mathrm{h} C_2}\] and we have a similar fiber sequence for $\pi_0(R)^s$. The projection map $R\to \pi_0(R)$ furthermore induces a map of fiber sequences 
    \[
    \begin{tikzcd}
        \mathfrak{pic}^\mathrm{p}(R^\mathrm{s})\ar[r] \ar[d] & \mathfrak{pic}(R) \ar[r] \ar[d] & \mathfrak{pic}(R)^{\mathrm{h}C_2}\ar[d]\\
        \mathfrak{pic}^\mathrm{p}(\pi_0(R)^\mathrm{s}) \ar[r] & \mathfrak{pic}(\pi_0(R))\ar[r] &\mathfrak{pic}(\pi_0(R))^{\mathrm{h}C_2}
    \end{tikzcd}\,.
    \] The middle map is an equivalence on $\tau_{\leq 1}$ by \cite[B.7.5.5]{Lurie_SAG}, so to show that the map $\pi_0(\mathfrak{pic}^\mathrm{p}(R^s))\to \pi_0(\mathfrak{pic}^\mathrm{p}(\pi_0(R)^s))$ is an isomorphism, it is enough to show that the map on the right hand terms is an equivalence on $\tau_{\leq 1}$.
    We then have equivalences
    $$  \tau_{\leq 1}(\mathcal{P}\mathrm{ic}(R)^{\mathrm{h}C_2} ) \simeq \tau_{\leq 1}(\mathcal{P}\mathrm{ic}(R^{\mathrm{h}C_2})) \simeq \tau_{\leq 1}(\mathcal{P}\mathrm{ic}(\pi_0(R^{\mathrm{h}C_2}))) \simeq \tau_{\leq 1}(\mathcal{P}\mathrm{ic}(\pi_0(R)^{\mathrm{h}C_2})) \simeq \tau_{\leq 1}(\mathcal{P}\mathrm{ic}(\pi_0(R))^{\mathrm{h}C_2}), $$
    where the first and the last equivalence follow from Galois descent of the Picard space \cite[Theorem 3.3.1, also see (3.3) \& (3.4)]{mathew2016picard}, and the third equivalence follows from the fact that $\pi_n\left(R^{\mathrm{h}C_2}\right)\simeq \pi_n(R)^{C_2}$ for a free $C_2$-action by \cite[B.7.5.5]{Lurie_SAG} again. The second equivalence is \cite[B.7.5.5]{Lurie_SAG} to get that $R^{\mathrm{h}C_2}$ is connective with $\pi_0(R^{\mathrm{h}C_2})\cong \pi_0(R)^{C_2}\cong \pi_0(R)^{\mathrm{h}C_2}$ together with \cite[Corollary 7.2.2.19]{LurHA}.
\end{proof}
\begin{lemma}\label{lemma:pnpic_of_symmetric_coconnective_alg} 
    Let $R\in \CAlg^{\mathrm{B}C_2}$ be a ring spectrum with a $C_2$-action $\lambda:R\rightarrow R$. 
    We endow $R$ with the symmetric Poincaré structure of Example \ref{example:tate_poincare_structure}. 
    \begin{enumerate}[label=(\arabic*)]
        \item If $R$ is coconnective, then $\mathcal{P}\mathrm{ic}^\mathrm{p}(R^s)$ is $1$-truncated. 
        \item If $R$ is discrete, there is a short exact sequence $$0\rightarrow (R^\times)^{C_2}/N_\lambda(R^\times)\rightarrow \mathrm{Pic}^\mathrm{p}(R^s)\rightarrow \mathrm{Pic}(R)^{-\lambda}\rightarrow 0\,,$$ where $ (-)^{-\lambda} $ denotes fixed points with respect to the twisted action of Remark \ref{rmk:C2_spectra_twists}. 
    \end{enumerate}
\end{lemma}
\begin{remark}
    Compare Lemma \ref{lemma:pnpic_of_symmetric_coconnective_alg} with Proposition \ref{proposition:exact_sequence_for_hermitian_picard_group}. 
\end{remark}
\begin{proof} [Proof of Lemma \ref{lemma:pnpic_of_symmetric_coconnective_alg}]
    By Theorem \ref{thm:calgp_to_poincare_cat} (cf. Theorem \ref{thm:Ep_alg_to_poincare_cat}\ref{thmitem:borel_to_symmetric}), $ \Mod^\mathrm{p}_R $ is symmetric. 
    By Proposition \ref{prop:symmetric_pnpic} and Remark~\ref{remark: identification of the c2 action on pic}, there is an equivalence $$\mathcal{P}\mathrm{ic}^\mathrm{p}(R^s)\simeq \mathcal{P}\mathrm{ic}(R)^{\mathrm{h}(-\lambda)}$$ where the notation is that of Remark~\ref{rmk:C2_spectra_twists}. 
    It follows $$\Omega\mathcal{P}\mathrm{ic}^\mathrm{p}(R^s)\simeq \mathrm{gl}_1(R)^{\mathrm{h}(-\lambda)}.$$ 
    When $R$ is coconnective, we find that $$\Omega\mathcal{P}\mathrm{ic}^\mathrm{p}(R^s)\simeq \left(\Omega \mathrm{gl}_1(R)\right)^{\mathrm{h}(-\lambda)} \simeq (\pi_0(R)^\times)^{1=-\lambda}\,;$$
    in particular, we deduce that $ \Omega\mathcal{P}\mathrm{ic}^\mathrm{p}(R^s) $ is discrete.  
    When $R$ is discrete, the latter result follows from the homotopy fixed point spectral sequence and the fact that $ \mathcal{P}\mathrm{ic}^\mathrm{p}(R^s) $ is $1$-truncated. 
\end{proof}

\begin{remark}
    \label{remark:Poincare_Picard_group_fiber_sequence}
    Let $X=\mathcal{P}\mathrm{ic}(R)$ be the Picard space of a ring spectrum $R$ with a $C_2$-action $\lambda:R\rightarrow R$. 
    By Proposition~\ref{prop:symmetric_pnpic}, Remark~\ref{remark: identification of the c2 action on pic}, and Remark \ref{rmk:C2_spectra_twists}, we have a fiber sequence $$\mathcal{P}\mathrm{ic}^\mathrm{p}(R^s)\rightarrow \mathcal{P}\mathrm{ic}(R)\xrightarrow{N_\lambda} \mathcal{P}\mathrm{ic}(R)^{\mathrm{h}C_2}\,.$$ Taking homotopy groups, this recovers Remark \ref{remark:5-term-exact-sequence-for-hermitian_picard_group}. 
    If the $C_2$-action on $R$ is Galois, then by Galois descent of the Picard space \cite[Theorem 3.3.1]{mathew2016picard}, we have a fiber sequence $$\mathcal{P}\mathrm{ic}^\mathrm{p}(R^s)\rightarrow \mathcal{P}\mathrm{ic}(R)\xrightarrow{N_\lambda} \mathcal{P}\mathrm{ic}(R^{\mathrm{h}C_2}).$$ 
    If $ 2 $ acts invertibly on $ \mathcal{P}\mathrm{ic}(R) $, then the exact sequence splits. 
\end{remark}
\begin{example}
    \label{example:Poincare_Picard_of_KO}
    The action of complex conjugation on $\mathrm{KU}$ is $C_2$-Galois with homotopy fixed points $\mathrm{KO}$. 
    Considered together, $ \mathrm{KO} \to \mathrm{KU} $ is Borel $ \EE_\infty $-ring spectrum in $ \Spectra^{C_2} $ and as such admits a canonical lift to $ \CAlgp $. 
    We know $\mathrm{Pic}(\mathrm{KO})\simeq \mathbb{Z}/8$ and $\mathrm{Pic}(\mathrm{KU})\simeq \mathbb{Z}/2$ \cite{mathew2016picard}. 
    Moreover, $\pi_0(\mathrm{KO})\simeq \mathbb{Z}$ and thus $\pi_0(\mathrm{KO})^\times \simeq \mathbb{Z}/2$, and similarly for $\mathrm{KU}$. 
    By Remark \ref{remark:Poincare_Picard_group_fiber_sequence}, we have an exact sequence $$ \mathbb{Z}/2\xrightarrow{\times 2}\mathbb{Z}/2\rightarrow \mathrm{Pic}^\mathrm{p}(\mathrm{KU}^s)\rightarrow \mathbb{Z}/2\xrightarrow{N_{\lambda}} \mathbb{Z}/8.$$ 
    The map $N_\lambda:\mathbb{Z}/2\to \mathbb{Z}/8$ is zero, see Example~\ref{ex: Atiyah real K-theory} for a proof. The map $\mathbb{Z}/2\rightarrow \mathrm{Pic}^\mathrm{p}(\mathrm{KU}^s)$ has a section given by evaluation of the quadratic form at the unit of $\mathrm{KU}$. 
    We conclude $$\mathrm{Pic}^\mathrm{p}(\mathrm{KU}^s)\simeq \mathbb{Z}/2\times \mathbb{Z}/2,$$ generated by $(\Sigma \mathrm{KU},\pm 1)$. This example shows that the decomposition of Remark \ref{remark:Poincare_Picard_group_fiber_sequence} does not necessarily hold without the $2$-divisibility assumption. We will show in Example~\ref{ex: Atiyah real K-theory} another perspective and an extension of this computation.
\end{example}
\begin{observation}\label{obs:pnpic_to_GW}
    Let $ \left(\mathcal{C},\Qoppa\right) $ be a symmetric monoidal Poincaré $ \infty $-category. 
    The canonical inclusion $ \Picp\left(\mathcal{C},\Qoppa\right) \subseteq \mathrm{Pn}\left(\mathcal{C},\Qoppa\right) $ and \cite[Corollary 4.2.2]{CDHHLMNNSII} induce a canonical map 
    \begin{equation}\label{eq:pnpic_to_GW_as_spaces}
        \mathcal{P}\mathrm{ic}^\mathrm{p}\left(\mathcal{C},\Qoppa\right) \to \Omega^\infty\mathrm{GW}\left(\mathcal{C},\Qoppa\right) \qquad \text{ in } \Spaces \,,
    \end{equation}
    where $ \mathrm{GW} $ denotes the Grothendieck--Witt spectrum of \cite[Definition 4.2.1]{CDHHLMNNSII} (also see Definition 4.1.1 and Corollary 4.2.3 \emph{loc. cit.}). 
    The map (\ref{eq:pnpic_to_GW_as_spaces}) is adjoint to a map
    \begin{equation}\label{eq:pnpic_to_GW_as_spectra}
        \sphere\left[\mathcal{P}\mathrm{ic}^{\mathrm{p}}\left(\mathcal{C},\Qoppa\right)\right] \to \mathrm{GW}\left(\mathcal{C},\Qoppa\right) \qquad \text{ in } \Spectra \,. 
    \end{equation}
    On the other hand, if we take $ \pi_0 $ of (\ref{eq:pnpic_to_GW_as_spaces}), then by \cite[Proposition 4.1.8]{CDHHLMNNSII}, the map $ \pi_0 \Picp\left(\mathcal{C},\Qoppa\right) \to \mathrm{GW}_0\left(\mathcal{C},\Qoppa\right) $ of sets admits a canonical lift to a map 
    \begin{equation}\label{eq:pnpic_to_GW_as_discrete_monoids}
        \pi_0 \Picp\left(\mathcal{C},\Qoppa\right) \to \mathrm{GW}_0\left(\mathcal{C},\Qoppa\right)^{\times} 
    \end{equation}
    of discrete commutative monoids, where the monoidal structure on $ \mathrm{GW}_0\left(\mathcal{C},\Qoppa\right)^{\times} $ is induced by tensor product of Poincaré objects. 
    Equivalently, we have a map 
    \begin{equation}\label{eq:pnpic_to_GW_as_comm_rings}
        \ZZ\left[\pi_0 \Picp\left(\mathcal{C},\Qoppa\right)\right] \to \mathrm{GW}_0\left(\mathcal{C},\Qoppa\right)
    \end{equation}
    of commutative rings.\footnote{In fact, using forthcoming results in \cite{CDHHLMNNSIV}, we should be able to exhibit lifts of (\ref{eq:pnpic_to_GW_as_discrete_monoids}), resp. (\ref{eq:pnpic_to_GW_as_comm_rings}) to a map $ \Picp\left(\mathcal{C},\Qoppa\right) \to \Omega^\infty \mathrm{GW}\left(\mathcal{C},\Qoppa\right)^{\times} $ of $ \EE_\infty $-monoids in spaces, resp. $ \sphere\left[\Picp\left(\mathcal{C},\Qoppa\right)\right] \to \mathrm{GW}\left(\mathcal{C},\Qoppa\right)^{\times} $ a map of $ \EE_\infty $-ring spectra.} 
    
    These maps are compatible with the non-equivariant ones in the sense that there is a commutative diagram\footnote{An astute reader may guess that the aforementioned diagram admits a genuine refinement (e.g. of Green functors). This is outside the scope of the current work.} 
    \begin{equation*}
        \begin{tikzcd}
            \ZZ\left[\pi_0 \Picp\left(\mathcal{C},\Qoppa\right)\right] \ar[d] \ar[r] & \mathrm{GW}_0\left(\mathcal{C},\Qoppa\right) \ar[d] \\
            \ZZ\left[\pi_0\Pic(\mathcal{C})\right] \ar[r] & K_0(\mathcal{C}) \,.
        \end{tikzcd}
    \end{equation*}		
    Now suppose that $ \left(\mathcal{C},\Qoppa\right) = \Mod^\mathrm{p}_{\underline{k}} $ where $ k $ is a field with the trivial $ C_2 $-action regarded as a Poincaré ring via Example \ref{ex:fixpt_Mackey_functor}; this is symmetric monoidal by Theorem \ref{thm:calgp_to_poincare_cat}.  
    Then by Theorem \ref{theorem:Poincare_Pic_of_fixpt_Mackey_functor}, the map (\ref{eq:pnpic_to_GW_as_comm_rings})
    \begin{equation}
        \ZZ\left[k^\times/k^{\times 2}\right] \simeq \ZZ\left[\pi_0 \Picp\left(\underline{k}\right)\right] \to \mathrm{GW}_0\left(\underline{k}\right) 
    \end{equation}
    is given by sending a unit $ a \in k^\times $ to the form $ \langle a \rangle $. 
\end{observation}

\begin{remark}
    In fact, for a discrete ring with involution $R$, thought of as a Poincar\'e ring via Example~\ref{example:genuine_symmetric_poincare_structure} one can show that there is a determinant map $\mathrm{GW}(R)\to \mathfrak{pic}^\mathrm{p}(R)$ which on $\pi_0$ splits the above map. In particular, the map $\mathrm{Pic}^\mathrm{p}(R)\in \mathrm{GW}_0(R)$ is always injective. This should, informally, identify $\mathfrak{pic}^\mathrm{p}(R)$ as ``$\mathrm{gr}^1_{\mathrm{mot}}\mathrm{GW}(R)$'' for some notion of the motivic filtration on $\mathrm{GW}(R)$.
\end{remark}

\subsection{Units in Poincaré rings}\label{subsection:calgp_units}

In the usual Picard spectrum one has the relationship $\mathfrak{pic} = \ZZ \times B\mathbb{G}_m$, where $\mathbb{G}_m$ is the spectral algebraic group scheme sending a ring spectrum $E$ to the spectrum of $E$-linear equivalences of $E$ $\mathrm{gl}_1E:=\mathrm{Aut}_E(E)$.\footnote{Normally the automorphism space of an object is only $\mathbb{A}_\infty$, but as the unit in a symmetric monoidal category, the automorphisms of $E$inherit a canonical and in fact functorial $\mathbb{E}_\infty$ structure and this construction makes sense.} 
This relationship between $\mathfrak{pic}$ and $\mathbb{G}_m$ has many important applications, for example relating the higher homotopy groups of $\mathfrak{pic}(A)$ with those of $A$.  
In this section, we introduce an analogue of units for Poincaré rings. 
We show that this Poincar{\'e} ring corepresents the functor which sends a Poincar{\'e} ring $ R $ to the space of automorphisms of the Poincaré object $ (R,u) $ (Corollary \ref{cor:tensor_unit_as_poincare_object}). 
In other words, the Poincaré Picard spectrum introduced in \S\ref{subsection:pnpic_general_and_units} is a delooping of the Poincaré units. 

Recall that the ordinary multiplicative group $\mathbb{G}_m$ is the functor $ \EE_\infty \Alg \to \Spectra_{\geq 0} $ corepresented by $\mathbb{S}\{x^{\pm 1}\} $, where $\mathbb{S}\{x^{\pm 1}\}$ is the free $\mathbb{E}_\infty$ ring on the $\mathbb{E}_\infty$ space $ \Omega^\infty \sphere $. 
\begin{notation}
    We let $\mathbb{S}\{-\}$ denote the left adjoint to the forgetful functor $\mathbb{E}_\infty\Alg\left(\Spectra^{BC_2}\right) \to \Spectra^{BC_2} $.  
    When $X=\mathbb{S}[C_2]$, we will denote $\mathbb{S}\{X\} $  by $\mathbb{S}\{x,y\}$ for brevity. 
\end{notation}

We will construct $\gmq$ directly as a corepresented functor. As a Borel ring spectrum, the corepresenting $\mathbb{E}_\infty$ ring will again be $\mathbb{S}\{x^{\pm 1}\}$ where the $C_2$-action is, informally, given by $\lambda(x)=x^{-1}$. The description as a Poincar\'e ring, however, will be as a pushout of certain Poincar\'e ring structures on $\mathbb{S}\{x,y\}$, $\mathbb{S}\{z\}$, and $\mathbb{S}$. To form this pushout we will need to construct maps to push out  along. This is accomplished via the following natural transformation:
\begin{construction}~\label{cons: norm as nat transformation}
    We will construct a natural transformation between the two functors \[\Omega^{\infty}(-)^e,\Omega^{\infty}((-)^L)^{C_2}:\CAlgp\to \Spaces\] which will encode the norm map. 
    We do this by using the functor $\CAlgp\to \mathbb{E}_\infty\mathrm{Mod}(\Catp)_{(\mathrm{Sp}^\omega,\Qoppa^\mathrm{u})/-}$ of Corollary~\ref{cor:unit_poincare_object} to write $\Omega^{\infty}(-)^e$ as the composite \[R\mapsto (F:(\mathrm{Sp}^{\omega},\Qoppa^\mathrm{u})\to \mathrm{Mod}^\mathrm{p}_R)\mapsto (F^e:\mathrm{Sp}^\omega\to \mathrm{Mod}^\omega_{R^e})\to \mathrm{End}_{R^e}(F^e(\mathbb{S}))\] and write $\Omega((-)^L)^{C_2}$ as the composite \[R\mapsto (F:(\mathrm{Sp}^\omega,\Qoppa^{\mathrm{u}})\to \mathrm{Mod}^\mathrm{p}_{R^e})\mapsto \Omega^\infty\Qoppa_{R}(F(\mathbb{S}))\,.\] We then have a natural transformation $\mathrm{End}_{R^e}(F^e(\mathbb{S}))\to \Omega^\infty \Qoppa_{R}(F(\mathbb{S}))$ given by the fact that $\Qoppa_R$ is a functor. Let \[f:\Omega^\infty (-)^e\implies \Omega^{\infty} ((-)^L)^{C_{2}}\] be this natural transformation. Note that the composite map \[\Omega^\infty (\mathbb{S}\{x,y\}^{\mathrm{t}})^e\to \Omega^\infty ((\mathbb{S}\{x,y\}^\mathrm{t})^L)^{C_2}\to \Omega^\infty \mathrm{hom}_{\mathbb{S}\{x,y\}^{\otimes 2}}(\mathbb{S}\{x,y\}\otimes \lambda^*\mathbb{S}\{x,y\},\mathbb{S}\{x,y\})\cong \mathbb{S}\{x,y\}\] sends $x$ to $xy$ by definition.
\end{construction}

We are now ready to construct the Poincaré ring corepresenting the functor $\gmq$.
\begin{construction}~\label{const: gmq} 
    Consider the Tate Poincar{\'e} ring $\mathbb{S}\{x,y\}^t$ in the sense of Example \ref{example:tate_poincare_structure}. 
    Let $R_\Qoppa$ denote the Poincar{\'e} ring which corepresents the functor sending a Poincar{\'e} ring $(S,S^{\phi C_2}\to S^{\mathrm{t}C_2})$ to the space $\Omega^\infty S^{C_2}$. This Poincar{\'e} ring exists by Theorem~\ref{thm:free_operadic_calgp_formula}, and in fact this gives a formula $R_\Qoppa^e =\mathbb{S}\{z\}$ with trivial $C_2$ action and $R_\Qoppa^{\phi C_2}\simeq \mathbb{S}\{z\}^{\otimes 2}$. We also have the chain of equivalences 
    \begin{align*}
    \mathrm{maps}_{\CAlgp}(\mathbb{S}\{x,y\}^\mathrm{t},-)\simeq \mathrm{maps}_{\CAlg(\mathrm{Sp}^{\mathrm{B}C_2})}(\mathbb{S}\{x,y\}, (-)^e)&\simeq \mathrm{maps}_{\mathrm{Sp}^{\mathrm{B}C_2}}(\mathbb{S}[C_2],(-)^e)\\
    &\simeq \mathrm{maps}_{\mathrm{Sp}}(\mathbb{S},(-)^e)\simeq \Omega^\infty(-)^e
    \end{align*}
    where the first equivalence comes from the fact that $(-)^\mathrm{t}$ is left-adjoint to $(-)^e$ (Example~\ref{example:tate_poincare_structure}), and the rest of the equivalences are by definition. 
    Thus $R_\Qoppa$ and $\mathbb{S}\{x,y\}^{\mathrm{t}}$ corepresent the functors in Construction~\ref{cons: norm as nat transformation}, and we therefore get a map $f(\mathrm{id}_{\mathbb{S}\{x,y\}^{\mathrm{t}}}): R_\Qoppa\to \mathbb{S}\{x,y\}^{\mathrm{t}}$ in $ \CAlgp $. 
    By definition of $ R_\Qoppa $, the class $1\in \pi_0((\mathbb{S}^\mathrm{u})^{C_2})$ corresponds to a uniquely determined map of Poincar{\'e} rings $R_\Qoppa\to \mathbb{S}^u$. 
    We then define \[\mathbb{S}\{x^{\pm 1}\}^\mathrm{u} := \mathbb{S}\{x,y\}^\mathrm{t} \otimes_{R_\Qoppa}\mathbb{S}^u\] where the tensor product denotes a pushout of Poincar{\'e} rings. 
\end{construction}
\begin{definition}\label{defn:hermitian_units}
    Let $\gmq$ denote the functor which is corepresented by the Poincar\'e ring $\mathbb{S}\{x^{\pm 1}\}^{\mathrm{u}}$ of Construction \ref{const: gmq}. 
    In other words, 
    \begin{equation*}
    \begin{split}
        \gmq \colon \CAlgp &\to \Spaces \\
        R &\mapsto \hom_{\CAlgp}\left(\mathbb{S}\{x^{\pm 1}\}^\mathrm{u}, R \right)\,.
    \end{split}
    \end{equation*}
\end{definition}
The following result justifies our notation. 
\begin{proposition}\label{prop:gmq_underlying_calg}
    The forgetful functor $ \CAlgp \to \EE_\infty\Alg^{BC_2} $ sends the Poincar\'e ring $ \mathbb{S}\{x^{\pm 1}\}^\mathrm{u} $ of Construction \ref{const: gmq} to the $ \EE_\infty $-ring $ \mathbb{S}\left\{x^{\pm 1} \right\} $, where the generator of $ C_2 $ sends $ x \mapsto x^{-1} $. 
    As an object of $ \EE_\infty\Alg $, this agrees with taking the free $\mathbb{E}_\infty$ ring on a degree zero generator $x$ and then formally adding a multiplicative inverse to $x$ in an $\mathbb{E}_\infty$ way.   
\end{proposition}
\begin{proof}
    By Theorem \ref{thm:poincare_rings_cat_formal_properties}\ref{thmitem:poincare_ring_to_naive_ring_preserves_colims}, the underlying $ \EE_\infty $-ring of $ \mathbb{S}\left\{x^{\pm 1} \right\}^u $ is computed by $\mathbb{S}\{x,y\}\otimes_{\mathbb{S}\{z\}}\mathbb{S}$. 
    It follows from the last sentence of Construction~\ref{cons: norm as nat transformation} that the map $\mathbb{S}\{z\}\to \mathbb{S}\{x,y\}$ sends $z\mapsto xy$. 
    The desired result follows from noting that $\mathbb{S} \left\{z\right\} \to \mathbb{S} $ sends $ z \mapsto 1 $. 
\end{proof}

\begin{remark}\label{remark:Poincare_units_fiber_sequence}
    The pushout description of $\mathbb{S}\{x^{\pm 1}\}^\mathrm{u}$ induces a pullback of mapping spaces 
    \[
    \begin{tikzcd}
        \gmq(R) \ar[d] \ar[r] & \mathrm{Maps}_{\CAlgp}(\mathbb{S}\{x, y\}^\mathrm{t}, R)\simeq \Omega^\infty R^e\ar[d]\\
        1 \ar[r] & \mathrm{Maps}_{\CAlgp}(R_\Qoppa,R)\simeq \Omega^\infty\Qoppa_R(R) 
    \end{tikzcd}
    \] where the right vertical map is given by taking the norm. 
\end{remark}
\begin{remark}
    While we occasionally refer to $ \gmq $ as units, $ \gmq(R) $ is \emph{not} given by $ \left((R^{C_2})^\times \to (R^e)^\times \right) $ considered as some kind of structured object. 
    A more apt interpretation is as a hermitian analogue of $ \mathrm{gl}_1 $. 
\end{remark}
	
\begin{theorem}\label{theorem:loops_Poincare_pic_is_Gm_Qoppa}
    There is a natural equivalence \[\Omega^{\infty +1} \Picp(-)\simeq \gmq\] of $ \Spaces $-valued functors on Poincar{\'e} rings, where $\gmq$ is the functor corepresented by the Poincar\'e ring of Construction~\ref{const: gmq}. 
    In particular, $ \gmq $ admits a canonical lift to a functor $ \CAlgp \to \Spectra_{\geq 0} $, and there is a fiber sequence 
    \begin{equation*}
        \gmq \to \mathrm{gl}_1\left((-)^e\right) \to \mathrm{gl}_1\left((-)^L\right) 
    \end{equation*}
    of infinite loop spaces.  
\end{theorem}
\begin{proof}
    This amounts to identifying the space $\mathrm{Aut}_{\mathrm{Pn}\left(\mathrm{Mod}_{R^e}^\omega,\Qoppa_R\right)}(R^e,u)$ functorially, where $(R^e,u)$ is the Poincar{\'e} object $R^e$ with bilinear form given by the unit map $\mathbb{S}\to \Qoppa_R(R^e)$. 
    Since $ \mathrm{Pn}\left(\mathrm{Mod}_{R^e}^\omega,\Qoppa_R\right) $ is the maximal subgroupoid of $ \mathrm{He}\left(\mathrm{Mod}_{R^e}^\omega,\Qoppa_R\right) $ on Poincaré objects, it suffices to describe automorphisms of $ (R^e,u) $ in the latter category.
    Recall that $ \mathrm{He}(\mathrm{Mod}_{R^e}^\omega,\Qoppa_R)\to \mathrm{Mod}^\omega_{R^e} $ is a right fibration classified by the functor which takes a module $M$ to the $ \infty $-groupoid $ \Omega^\infty \Qoppa_R(M)$ \cite[Definition 2.1.1]{CDHHLMNNSI}. 
    By \cite[Lemma 2.4.4.1]{HTT} we get an induced fiber sequence on mapping spaces which restricts to a fiber sequence of the form
    \begin{equation}\label{eq:loops_pnpic_fiber_sequence}
        \Omega\Picp(R) \simeq \mathrm{Aut}_{\mathrm{He}(\mathrm{Mod}_{R^e}^\omega,\Qoppa_R)}((R^e,u)) \to \mathrm{Aut}_{\mathrm{Mod}_{R^e}^\omega}(R)\xrightarrow{f \mapsto f^*(u)} \Omega^\infty \Qoppa_R(R^e)
    \end{equation} 
    where the fiber is taken over the point $ u \in \Omega^\infty \Qoppa_R(R^e) $. 
    In other words, a point in $ \mathrm{Aut}_{\mathrm{He}(\mathrm{Mod}_{R^e}^\omega,\Qoppa_R)}((R^e,u)) $ is the data of an automorphism $a\in \mathrm{Aut}(R^e)$ together with a path $q $ from $ u $ to $ a^*u$ in $\Omega^{\infty}\Qoppa_R(R^e)$.  
    
    We now define a natural transformation $\Omega \Picp(-)\to \gmq(-)$. Let $i:\mathrm{gl}_1(-^e)\to \Omega^\infty (-)^e$ denote the natural inclusion map. There is then a commutative diagram
    \[
    \begin{tikzcd}[column sep=small, row sep=small]
        \Omega \Picp(-) \ar[rr] \ar[dd] \ar[rd] & & \mathrm{gl}_1(-^e) \ar[dd] \ar[rd, "i"] & &  \\
        & \gmq(-) \ar[rr,crossing over] & & \Omega^{\infty}(-)^e  \ar[dd] & \\
        1 \ar[rd, equals] \ar[rr]  & & \Omega^\infty (-)^{C_2} \ar[rd, equals] &  \\
        & 1 \ar[rr] & & \Omega^\infty (-)^{C_2} 
        \ar[from=2-2,to=4-2,crossing over]
    \end{tikzcd}
    \] and hence a natural transformation $\Omega \Picp(-)\to \gmq(-)$. To show that this map is an equivalence it is then enough to show that any $x\in \pi_0(R^e)$ such that $N^{C_2}x=x\sigma(x)=1$ in $\pi_0(R^{C_2})$ is a unit. Note that we then have that $x\sigma(x)=1$ in $\pi_0(R^e)$ after applying the inclusion map $R^{C_2}\to R^e$, hence $x$ is invertible as desired. 
\end{proof}
\begin{example}\label{ex:hermitian_units_fixpt_Mackey}
    Let $ R $ be a commutative ring with involution and regard $ R $ as a Poincaré ring $ \underline{R}^\lambda $ via Example \ref{ex:fixpt_Mackey_functor}. 
    Since the underlying ring and fixed points of $ R $ are discrete, by the exact sequence of Theorem \ref{theorem:loops_Poincare_pic_is_Gm_Qoppa}, $ \gmq\left(\underline{R}^\lambda \right) $ is discrete and it is isomorphic to the kernel of the map $ R^\times \to \left(R^{C_2}\right)^\times $ sending $ r \mapsto r \cdot \lambda(r) $. 
\end{example}
\begin{example}
    Let $ G $ be a finite group and consider $ \prod_{G} k $ for some field $ k $. 
    We may endow $ \prod_{G} k $ with the involution $ (a_g)_{g \in G} \mapsto \left(a_{g^{-1}}\right)_{g \in G} $, which we regard as a Poincaré ring via Example \ref{ex:fixpt_Mackey_functor}. 
    By Example \ref{ex:hermitian_units_fixpt_Mackey}, $ \gmq\left(\prod_{G} k\right) $ is isomorphic to the collection of functions $ a \colon G \to k^\times $ satisfying $ a_g^{-1} = a_{g^{-1}} $, considered as a group under pointwise multiplication. 
\end{example}
		
\subsection{The discrete case: nondegenerate hermitian line bundles}
\label{subsection:the_discrete_case:nondegenerate_hermitian_forms}
In this section we define and study the \textit{hermitian Picard group} of a ring $R$ with involution $\lambda:R\to R$ (defined in Definition~\ref{definition:hermitian_picard_group}). 
We will show in Corollary~\ref{cor:fausk_for_Poincare_rings} and Theorem~\ref{thm: Fausk for pnpic and schemes with good quotient} that this hermitian Picard group is closely connected to the Poincar\'e Picard group of Definition~\ref{definition: involutive picard space} studied in the previous subsection. 
When the involution on $R$ or $X$ is trivial, this invariant agrees with the group of isometry classes of \emph{discriminant bundles} \cite[p. 470]{Knus_book_Q_and_H_forms}; in general our construction agrees with the (absolute version of the) \emph{hermitian Picard group} of Reyes Sanchez--Verhaeghe--Verschoren\footnote{Reyes Sanchez--Verhaeghe--Verschoren define their invariant for associative algebras $ A $ over $ R $; our construction agrees with theirs for $R = \ZZ $ and $ A $ commutative.} \cite{MR1348271}. 
We show additionally that in the case that the involution on $X$ is trivial that there is a canonical isomorphism \[\mathrm{Pic}^\mathrm{h}(X,\mathrm{id}_X)\cong \mathrm{H}^1_{\mathrm{fppf}}(X;\mu_2)\] in Remark~\ref{remark:hermitian_picard_group_with_trivial_action_is_the_first_fppf_mu_2_cohomology group}. 

\begin{notation}~\label{notation: pic classical}
    Let $X$ be a scheme. We will use \[\mathrm{Pic}^{\mathrm{cl}}(X)\] to denote the Picard group of $\mathrm{Mod}_{\mathcal{O}_X}^\heartsuit$. 
\end{notation}

\begin{definition}
    \label{definition:sigma-dual_of_module}
    Let $ R $ be a discrete commutative ring with a $ C_2 $-action $ \lambda \colon R \to R $, and let $M$ be an $R$-module. We define the $\lambda$-\emph{dual} of $ M $ to be the $ R $-module  
    $$ M^\dag := \mathrm{Hom}_R \left(\lambda^* M, R\right)=(\lambda^*M)^\vee.$$ 
\end{definition}

\begin{remark}
    \label{remark:restriction_of_scalars_commutes_with_dual}
    Let $R$ be a discrete commutative ring and let $\lambda:R\rightarrow R$ be an isomorphism of commutative rings such that $\lambda\circ\lambda=\operatorname{id}_R$.  Then we have a canonical isomorphism
    \begin{align*}
        M^\dag =(\lambda^* M)^\vee & \simeq \lambda^* (M^\vee),\\ 
        f &\mapsto \lambda\circ f 
    \end{align*}
    of $R$-modules, see Lemma~\ref{lemma: duality identification} for a proof of this for $R$ a commutative ring spectrum. If $M$ is finitely generated projective, we thus have a canonical isomorphism of $R$-modules $$c:M \simeq (M^\vee)^\vee\simeq \left(M^\dag\right)^\dag$$ $$m\mapsto \lambda\circ\operatorname{eval}_m.$$
    If both $M$ and $N$ are finitely generated projective $R$-modules, then, using symmetric monoidality of $\lambda^*:\mathrm{Mod}_R\rightarrow \mathrm{Mod}_R$, the $\lambda$-dual satisfies $$(M \otimes_R N)^\dag\simeq\lambda^*(M\otimes_R N)^\vee\simeq\lambda^*(M^\vee \otimes_R N^\vee)\simeq M^\dag \otimes_R N^\dag.$$
\end{remark}

\begin{definition}
    \label{definition:sigma-hermitian_form}
    Let $ R $ be a discrete commutative ring with a $ C_2 $-action $ \lambda \colon R \to R $, and let $ I $ be a finitely generated projective $ R $-module. A \emph{nondegenerate $ \lambda $-hermitian form on $ I $} is an $ R $-linear isomorphism $ \varphi \colon I \xrightarrow{\sim} I^\dag $ such that the following diagram commutes 
    \[\begin{tikzcd}
        I \ar[r,"\varphi"] \ar[d,"c", "\simeq"'] & I^\dag \\
        (I^\dag)^\dag \ar[ur,"\varphi^\dag"'] & .
    \end{tikzcd}\]
    We will denote a $\lambda$-hermitian form on $I$ by a pair $(I,\varphi)$.
    
    Let $I$ and $J$ be two finitely generated projective $R$-modules with $\lambda$-hermitian forms $(I,\varphi)$ and $(J,\psi)$. A \emph{map of $\lambda$-hermitian forms}, $f:(I,\varphi)\rightarrow(J,\psi)$, is a map of $R$-modules $f:I\rightarrow J$ such that the following diagram commutes 
    \[\begin{tikzcd}
        I \ar[r,"\varphi"] \ar[d,"f"'] & I^\dag \\
        J \ar[r,"\psi"'] & J^\dag \ar[u, "f^\dag"']. 
    \end{tikzcd}\] We say $f$ is an isomorphism of $\lambda$-hermitian forms if its underlying $R$-module map is an isomorphism of $R$-modules.
\end{definition}

\begin{remark}
    \label{remark:hermitian_forms_as_fixed_points}
    Let $ R $ be a discrete commutative ring with a $ C_2 $-action $ \lambda \colon R \to R $, and let $I$ be a finitely generated projective $R$-module. We have an induced $C_2$-action on the $R$-module $\operatorname{Hom}_R(I,I^\dag)$ given by the assignment $\varphi\mapsto \varphi^\dag\circ c$. A $\lambda$-hermitian form on $I$ is a fixed point of this action which happens to be an isomorphism. Via the following isomorphisms of $R$-modules $$\mathrm{Hom}_R(I,I^\dag)\simeq \mathrm{Hom}_R(I,(\lambda^*I)^\vee)\simeq \mathrm{Hom}_R(I\otimes_R \lambda^*I,R)$$ this turns into the action $f(x,y)\mapsto \lambda(f(y,x))$. Thus, an alternative definition of a $\lambda$-hermitian form on $I$ is: a function $f:I\times I\rightarrow R$ which is nondegenerate and $R$-linear in the first variable, and satisfies $f(x,y)=\lambda(f(y,x))$ for all $x,y\in R$.
\end{remark}

The pairing 
\begin{align*}
    \left\langle -,-\right\rangle_\lambda: & R\times  R\rightarrow R\\
    & (r,s)\mapsto \lambda(s)r
\end{align*} 
provides a canonical $\lambda$-hermitian form $u:R\rightarrow R^\dagger$, using Remark \ref{remark:hermitian_forms_as_fixed_points}. Moreover, given two $\lambda$-hermitian forms over $R$, $(I,\varphi)$ and $(J,\psi)$, the tensor product $\otimes_R$ provides a new $\lambda$-hermitian form $(I\otimes_R J, \varphi \otimes_R \psi)$
\begin{align*}
    \varphi\otimes_R \psi: & I\otimes_R J \rightarrow (I\otimes_R J)^\dagger\\
    & i\otimes j\mapsto \left[x\otimes y \mapsto \varphi(i)(x)\cdot \psi(j)(y)\right].
\end{align*}  
Hence, the set of invertible $\lambda$-hermitian forms over $R$ forms a group under $\otimes_R$ with unit $(R,u)$.

\begin{definition}
    \label{definition:hermitian_picard_group}
    Let $R$ be a discrete commutative ring with a $C_2$-action $\lambda: R \rightarrow R$. The \emph{hermitian Picard group of} $(R,\lambda)$ is the group of isomorphism classes of invertible $\lambda$-hermitian forms over $R$ under tensor product, with unit $(R,u)$. We will denote the hermitian Picard group of $(R,\lambda)$ by $\mathrm{Pic}^\mathrm{h}(R,\lambda)$.
    
    If $X$ is a scheme with involution $\lambda$, we can similarly define the hermitian Picard group of $X$ to be the group of line bundles $\mathcal{L}$ on $X$ together with a $\lambda$-hermitian form $\phi:\mathcal{L}\to \mathcal{L}^\dagger:=\mathcal{H}om_{\mathcal{O}_X}(\lambda^*\mathcal{L},\mathcal{O}_X)$.
\end{definition}

\begin{example}
    \label{example:hermitian_picard_group_is_not_2-torsion-swap_action_on_fields}
    Let $R=\mathbb{C} \times \mathbb{C}$ with $\lambda:R\rightarrow R$ given by the swap action, i.e. $(x,y)\mapsto (y,x)$. Since $R$ is a semi-local ring, its Picard group is trivial. Thus any element $(I,\varphi)\in \mathrm{Pic}^\mathrm{h}(R,\lambda)$ is of the form $(R,\varphi)$, up to isomorphism. An $R$-module isomorphism $\varphi:R\rightarrow R^\dagger$ is of the form $(r\mapsto (s\mapsto ra\lambda(s))$, for some $a\in R^\times$, and the condition $\varphi=\varphi^\dagger \circ c$ translates to $a=\lambda(a)$. Thus $a\in (R^\times)^{C_2}$ is a unit and a fixed point of the $C_2$-action on $R$, i.e. $a=(v,v)$ with $v\in \mathbb{C}^\times$. Given two $\lambda$-hermitian forms $(I,\varphi)$ and $(J,\psi)$ which correspond to elements $a$ and $b$ in $(R^\times)^{C_2}$ respectively, the tensor product $(I\otimes_R J,\varphi\otimes_R \psi)$ corresponds to the product $ab$. An isomorphism between  $(I,\varphi)$ and $(J,\psi)$ is an $R$-module isomorphism $f:I\xrightarrow{\simeq} J$ such that $$f^\dagger\circ\psi\circ f=\varphi.$$ Since $f$ is given by multiplication of an element $x\in R^\times$, this condition translates to the equation $$x\lambda(x)b=a.$$ Let $a=(v,v)$, $b=(w,w)$ and $x=(r,s)$ with $v,w,r,s\in \mathbb{C}^\times$. Then this condition breaks down to the two equations $rsw=v$ and $rsv=w$, in particular $(rs)^2=1$, and thus $rs=\pm 1$. Therefore, $(R,\varphi)$ and $(R,\psi)$ are isomorphic if and only if $\varphi=\pm\psi$. We conclude $$\mathrm{Pic}^\mathrm{h}(R,\lambda)\simeq \mathbb{C}^\times/{\pm 1}\simeq \mathbb{C}^\times.$$
\end{example}

The above Example \ref{example:hermitian_picard_group_is_not_2-torsion-swap_action_on_fields} shows that the hermitian Picard group of a discrete commutative ring with a nontrivial $C_2$-action is not necessarily $2$-torsion. We will see now that the situation simplifies in case of a trivial $C_2$-action.

\begin{remark}
    \label{remark:id-hermitian_forms_are_symmetric_billinear_forms}
    Let $R$ be a discrete commutative ring equipped with the trivial $C_2$-action, i.e. $\lambda=\mathrm{id}_R$. In that case, for any $R$-module $M$, we have $M^\dagger=M^\vee$. Thus, by Remark \ref{remark:hermitian_forms_as_fixed_points}, hermitian forms over $(R,\mathrm{id}_R)$ are in bijection with nondegenerate symmetric bilinear forms over $R$. Let $(I,\varphi)$ be a hermitian form over $(R,\mathrm{id}_R)$ corresponding to a nondegenerate symmetric bilinear form $A$ on $I$. Then $(I,\varphi)$ is invertible if $I$ is invertible as a module over $R$ and $A\otimes B=\operatorname{id}_R$, for a nondegenerate symmetric biliniear form $B$ on $I^\vee$.
\end{remark}
		
\begin{remark}\label{remark:invertible_hermitian_forms_are_locally_fixed_points_of_units}
    Let $R$ be a discrete commutative ring with trivial $C_2$-action, and let $I$ be an invertible $R$-module. 
    Suppose given a (not necessarily symmetric) nondegenerate pairing $\tau:I\otimes_R I\xrightarrow{\sim} R$. 
    Let $\mathfrak{p}\subset R$ be a prime ideal. 
    Since $I$ is locally free of rank $1$, there exists a trivialization $t:I_\mathfrak{p}\xrightarrow{\sim} R_\mathfrak{p}$. 
    Then in these coordinates/locally at $ \mathfrak{p} $, $$ R_\mathfrak{p}\otimes_{R_\mathfrak{p}} R_\mathfrak{p} \xrightarrow{\sim } I_\mathfrak{p}\otimes_{R_\mathfrak{p}} I_\mathfrak{p}\xrightarrow{\tau_\mathfrak{p}} R_\mathfrak{p} $$ and the composite is given by $r\otimes s\mapsto s\cdot x_\mathfrak{p} \cdot r$ for some unit $x_\mathfrak{p}\in R_\mathfrak{p}^\times$. 
    Given a different choice of trivialization $ t'\colon I_\mathfrak{p}\xrightarrow{\sim} R_\mathfrak{p} $, $ t' \circ t^{-1} $ is multiplication by some $ y \in R_\mathfrak{p}^\times $, and the resulting form is represented by $ y^2 x_{\mathfrak{p}} $. 
    So the symmetry condition in Definition \ref{definition:sigma-hermitian_form} is vacuous and the assignment $ \tau \mapsto x_\mathfrak{p} $ is well-defined in $ R_{\mathfrak{p}}^\times/R_{\mathfrak{p}}^{\times 2} $. 
\end{remark}
		
\begin{proposition}
    \label{proposition:hermitian_picard_group_with_trivial_action_is_2-torsion}
    Let $R$ be a discrete commutative ring with trivial action. 
    Then $\mathrm{Pic}^\mathrm{h}(R)$ is $2$-torsion. 
\end{proposition}

\begin{proof}
    \label{proof:hermitian_picard_group_with_trivial_action_is_2-torsion}
    Let $(I,\varphi)$ be an invertible hermitian form over $(R,\mathrm{id}_R)$. Then the trivialization $I\otimes_R I\xrightarrow{\mathrm{id}_R\otimes \varphi} I\otimes_R I^\vee\xrightarrow{\left\langle \cdot, \cdot\right\rangle} R$ is an isomorphism of forms $(I\otimes_R I,\varphi\otimes_R \varphi)\simeq (R,u)$, which can be checked locally as in Remark \ref{remark:invertible_hermitian_forms_are_locally_fixed_points_of_units}. 
\end{proof}

Given a discrete commutative ring $R$ with trivial action, we have that the forgetful map $\mathrm{Pic}^\mathrm{h}(R,\mathrm{id})\to \mathrm{Pic}^{\mathrm{cl}}(R)$ factors through a forgetful map \begin{equation}~\label{eqn: pnpic forgetful map trivial action case}
    \mathrm{Pic}^\mathrm{h}(R,\mathrm{id})\rightarrow \mathrm{Pic}^{\mathrm{cl}}(R)[2]
\end{equation}
where the codomain is given by $2$-torsion in the classical Picard group of $R$. 
\begin{proposition}			\label{proposition:split_exact_sequence_for_hermitian_picard_group_with_trivial_action}
    Let $R$ be a discrete commutative ring with trivial action. 
    We have a split exact sequence $$0\rightarrow R^\times/(R^\times)^2\rightarrow\mathrm{Pic}^\mathrm{h}(R,\mathrm{id})\rightarrow \mathrm{Pic}^{\mathrm{cl}}(R)[2]\rightarrow 0$$ where the third map is (\ref{eqn: pnpic forgetful map trivial action case}).
\end{proposition}
Note that the splitting in Proposition \ref{proposition:split_exact_sequence_for_hermitian_picard_group_with_trivial_action} is not necessarily canonical or functorial in $ R $. 
Reyes S\'anchez--Verhaeghe--Verschoren proved an analogous exact sequence for the hermitian Picard group when the involution $ \lambda $ is not necessarily trivial in \cite[Theorem 3.2]{MR1348271}. 
\begin{proof}
    [Proof of Proposition \ref{proposition:split_exact_sequence_for_hermitian_picard_group_with_trivial_action}]
    Suppose given an invertible $ R $-module $ I $ belonging to $ \mathrm{Pic}^{\mathrm{cl}}(R)[2] $. 
    The existence of a pairing $I\otimes I\simeq R$ is equivalent to the existence of an $ R $-linear isomorphism $I\simeq I^\vee$. 
    It follows from Remark \ref{remark:invertible_hermitian_forms_are_locally_fixed_points_of_units} that the sequence is exact at the right, i.e. (\ref{eqn: pnpic forgetful map trivial action case}) is surjective. 
    The kernel of the map (\ref{eqn: pnpic forgetful map trivial action case}) is given by hermitian forms on $R$, up to isomorphism. That is, isomorphisms $\varphi: R\simeq R^\vee$, up to isomorphisms of hermitian forms. Two forms $\varphi$, $\psi$ are isomorphic if there exists an isomorphism $f:R\xrightarrow{\simeq} R$ such that $f^\dagger\circ\psi\circ f=\varphi$. Each of these isomorphisms is given by multiplication by a unit of $R$. Let $a,b,c\in R^\times$ be the units corresponding to $f, \varphi$ and $\psi$. Then $\varphi$ is isomorphic to $\psi$ if and only if $a^2c=b$, i.e. $c$ and $b$ vary by a square of a unit. Thus, the kernel of (\ref{eqn: pnpic forgetful map trivial action case}) is given by $R^\times/(R^\times)^2$. Finally, by Proposition \ref{proposition:hermitian_picard_group_with_trivial_action_is_2-torsion}, this is a short exact sequence of vector spaces over $\mathbb{F}_2$ and thus there exists a splitting.
\end{proof}
		
		\begin{remark}
			\label{remark:hermitian_picard_group_with_trivial_action_is_the_first_fppf_mu_2_cohomology group}
			Recall that we have an exact sequence of fppf sheaves $$0\rightarrow \mu_2\rightarrow \mathbb{G}_m\xrightarrow{\times 2} \mathbb{G}_m\rightarrow 0.$$ over $\mathbb{Z}$ and a canonical isomorphism $$\mathrm{H}_\mathrm{fppf}^1(X,\mathbb{G}_m)\simeq \mathrm{Pic}^{\mathrm{cl}}(X),$$ for $X$ a scheme. 
			Let $R$ be a discrete commutative ring. Proposition \ref{proposition:split_exact_sequence_for_hermitian_picard_group_with_trivial_action} shows that we have the exact sequence $$\mathrm{H}_\mathrm{fppf}^{0}(R,\mathbb{G}_m)\xrightarrow{\times 2} \mathrm{H}_\mathrm{fppf}^{0}(R,\mathbb{G}_m)\rightarrow\mathrm{Pic}^\mathrm{h}(R,\mathrm{id})\rightarrow \mathrm{H}_\mathrm{fppf}^{1}(R,\mathbb{G}_m)\xrightarrow{\times 2} \mathrm{H}_\mathrm{fppf}^{1}(R,\mathbb{G}_m).$$ Hence, in the case of a trivial action on $R$, the hermitian Picard group is isomorphic to the first fppf cohomology group of $R$ with $\mu_2$-coefficients $$\mathrm{Pic}^\mathrm{h}(R,\mathrm{id})\simeq \mathrm{H}_\mathrm{fppf}^{1}(R,\mu_2).$$ If $R$ is moreover smooth over a Noetherian scheme of finite Krull dimension $S$ which is smooth over a Dedekind domain of mixed characteristic or over a field, by \cite[Proposition 10.9]{spitzweck-thesis}, this is isomorphic to motivic cohomology in degree $(1,1)$ with coefficients in $\mathbb{Z}/2$ $$\mathrm{Pic}^\mathrm{h}(R,\mathrm{id})\simeq \mathrm{H}^{1,1}(R,\mathbb{Z}/2).$$
		\end{remark}
		
		The split exact sequence of Proposition \ref{proposition:split_exact_sequence_for_hermitian_picard_group_with_trivial_action} can be generalized to the case of a nontrivial $C_2$-action on $R$. To prove that, let us first establish two basic properties of the hermitian Picard group.
		
		\begin{proposition}
			\label{proposition:hermitian_picard_group_commutes_with_finite_sums}
			The hermitian Picard group $\mathrm{Pic}^\mathrm{h}$ sends  finite products of rings with involution to direct products of abelian groups.
		\end{proposition}
		
		\begin{proof}
			\label{proof:hermitian_picard_group_commutes_with_finite_sums}
			Given $(R,\lambda)$ and $(S,\sigma)$ commutative rings with involution, base change along the projections $ R \times S \to R $, $ R \times S \to S $ induce a homomorphism \[\mathrm{Pic}^{\mathrm{h}}(R\times S,\lambda\times \sigma)\to \mathrm{Pic}^\mathrm{h}(R,\lambda)\times \mathrm{Pic}^\mathrm{h}(S,\sigma)\,.\]
			It suffices to show that this map is an isomorphism. 
			The canonical $ \sigma $-invertible hermitian form $ (S,u_S) $ induces a map $ i_1 :\mathrm{Pic}^\mathrm{h}(R,\lambda) \to \mathrm{Pic}^\mathrm{h}(R \times S) $ given by $ (I,\varphi) \mapsto (I \times S, \varphi \times u_S) $. 
			We may similarly define a map $ i_2 : \mathrm{Pic}^\mathrm{h}(S,\sigma)\to \mathrm{Pic}^\mathrm{h}(R \times S ) $. 
			The map $ i_1 \oplus i_2 $ manifestly defines an inverse to the homomorphism above, hence we are done. 
		\end{proof}
		
		\begin{notation}~\label{notation:sigma-dual_fixed_points_of_picard_group}
			Given a discrete commutative ring $R$ and an isomorphism $\lambda:R\rightarrow R$, the $\lambda$-dual defines a $C_2$-action on the Picard group of $R$: 
			\begin{align*}
				(-)^\dag:\mathrm{Pic}^{\mathrm{cl}}(R) &\rightarrow \mathrm{Pic}^{\mathrm{cl}}(R)\\
				[I] &\mapsto [I^\dagger]=[\lambda^* I^\vee]=-[\lambda^* I].
			\end{align*} 
			Given an invertible $\lambda$-hermitian form $\varphi:I\xrightarrow{\simeq} I^\dagger$, the $R$-module $I$ is a fixed point in the Picard group of $R$ with respect to this action. Let us denote fixed points by this action by the superscript $(-)^{-\lambda}$. We then have that the forgetful map $\mathrm{Pic}^\mathrm{h}(R)\to \mathrm{Pic}^{\mathrm{cl}}(R)$ factors through a forgetful map \begin{equation}~\label{eqn: forgetful map pnpic to pic general case}
				\mathrm{Pic}^\mathrm{h}(R,\lambda)\rightarrow \mathrm{Pic}^{cl}(R)^{-\lambda}.
			\end{equation}
		\end{notation}
		
		An invertible $R$-module $I$ is an element of $\mathrm{Pic}^{\mathrm{cl}}(R)^{-\lambda}$ if and only if $[I]$ is in the kernel of the corresponding norm map: $N([I])=(\lambda +1)[I]=0$. This observation leads to the following generalization of Proposition \ref{proposition:split_exact_sequence_for_hermitian_picard_group_with_trivial_action}.
		
		\begin{proposition}
			\label{proposition:exact_sequence_for_hermitian_picard_group}
			Let $R$ be a discrete commutative ring with few zerodivisors and a $C_2$-action $\lambda:R\rightarrow R$. Then we have the following split exact sequence $$0 \rightarrow (R^\times)^{C_2}/N(R^\times)\rightarrow\mathrm{Pic}^\mathrm{h}(R,\lambda)\rightarrow \mathrm{Pic}^{\mathrm{cl}}(R)^{-\lambda}\rightarrow 0$$ where the second map is~\ref{eqn: forgetful map pnpic to pic general case}. 
		\end{proposition}
		
		\begin{proof}
			\label{proof:exact_sequence_for_hermetian_picard_group}
			It follows from Remark \ref{remark:invertible_hermitian_forms_are_locally_fixed_points_of_units} that the kernel of the forgetful map $\mathrm{Pic}^\mathrm{h}(R,\lambda)\rightarrow \mathrm{Pic}^{\mathrm{cl}}(R)^{-\lambda}$ is given by $(R^\times)^{C_2}/N(R^\times)$. It remains to show that any invertible module $I$ over $R$ has a nondegenerate $\lambda$-hermitian form. Since $R$ has few zerodivisors, the Picard group of $R$ is isomorphic to the ideal class group of $R$, $\mathrm{Cl}(R)\xrightarrow{\sim} \mathrm{Pic}^{\mathrm{cl}}(R)$; see \cite[Exercise 2.5.15]{elliott2019rings}. Thus, we may assume that $I$ is a fractional ideal of $R$, which naturally comes with the pairing $I\otimes_R \lambda^* I\rightarrow R$, $i\otimes_R j\mapsto \lambda(j)i$. This pairing is nondegenerate, by the assumption that $I\in \mathrm{Pic}^{\mathrm{cl}}(R)^{-\lambda}$. The assignment of this pairing provides a section $\mathrm{Pic}^{\mathrm{cl}}(R)^{-\lambda}\rightarrow \mathrm{Pic}^\mathrm{h}(R,\lambda)$ and thus proves the claim.
		\end{proof}
		
\begin{example}
    \label{example:hermitian_picard_group_does_not_map_to_2-torsion}
    Let $R=\mathbb{Z}[\frac{1+\sqrt{-23}}{2}]$ with the action $\lambda:R\rightarrow R$ given by complex conjugation. The fixed points of $R$ are given by integers, thus $(R^\times)^{C_2}/N(R^\times)\simeq \mathbb{Z}/2$. The Picard group $\mathrm{Pic}^{\mathrm{cl}}(R)\simeq \mathbb{Z}/3$ of $R$ is generated by the ideal $I=(2,\frac{1+\sqrt{-23}}{2})$. We have an isomorphism $\lambda^* I\xrightarrow{\sim} (2,\frac{1-\sqrt{-23}}{2})$, $x\mapsto \lambda(x)$. Since $(2,\frac{1+\sqrt{-23}}{2})(2,\frac{1-\sqrt{-23}}{2})=(2)$ is a principal ideal, we know $[\lambda^*I]=-[I]$ in $\mathrm{Pic}^{\mathrm{cl}}(R)$.  Thus the $C_2$-action on the Picard group induced by the $\lambda$-dual is given by the identity. Therefore, $\mathrm{Pic}^{\mathrm{cl}}(R)^{-\lambda}\simeq \mathrm{Pic}^{\mathrm{cl}}(R)\simeq \mathbb{Z}/3$. Thus, by Proposition \ref{proposition:exact_sequence_for_hermitian_picard_group}, we have $$\mathrm{Pic}^\mathrm{h}\left(\mathbb{Z}\left[\frac{1+\sqrt{-23}}{2}\right],\lambda\right)\cong \mathbb{Z}/2\oplus\mathbb{Z}/3.$$ This example shows that the forgetful map $\mathrm{Pic}^\mathrm{h}(R)\rightarrow \mathrm{Pic}^{\mathrm{cl}}(R)$ does generally not land in the $2$-torsion part of the Picard group of $R$.
\end{example}
		
\begin{remark}
    \label{remark:5-term-exact-sequence-for-hermitian_picard_group}
    Let $R$ be a discrete commutative ring with few zerodivisors with a $C_2$-action $\lambda:R\rightarrow R$. It follows from Proposition \ref{proposition:exact_sequence_for_hermitian_picard_group} that we have an exact sequence $$0\rightarrow (R^\times)^{-\lambda}\rightarrow R^\times \xrightarrow{N_\lambda}( R^\times)^{C_2}\rightarrow\mathrm{Pic}^\mathrm{h}(R,\lambda)\rightarrow \mathrm{Pic}^\mathrm{cl}(R)\xrightarrow{N_\lambda} \mathrm{im}(N_\lambda)\rightarrow 0.$$
\end{remark}
		
\subsection{Poincaré Picard group of schemes with involution}
\label{subsection:poincare_picard_group_inv_scheme}
Let $ X $ be a scheme with an involution $ \lambda $; by \S\ref{subsection:Poincare_structures_schemes_involution} and Definition~\ref{definition:poincare_picard_space}, we can take the \emph{Poincaré Picard space} of $ (X,\lambda) $. 
There is a forgetful map from the Poincaré Picard space of $ (X,\lambda) $ to the Picard space of $ X $ (see Definition \ref{defn:pnpic_to_pic_underlying}). 
However, ($ \pi_0 $ of) the Picard space is \emph{not} the ordinary Picard group in the sense of \cite{stacks}; it comprises invertible objects in the \emph{derived category} of $ X $. 
Instead, the inclusion of the heart of the standard t-structure on $ \mathrm{Mod}_{\mathcal{O}_X} $ induces an injective homomorphism from the ordinary Picard group of $ X $ to the Picard group of $ \mathrm{Mod}_{\mathcal{O}_X} $. 
This map is not an isomorphism in general; $ \mathcal{O}_X[\pm 1] $ are invertible in $ \mathrm{Mod}_{\mathcal{O}_X} $ but not discrete for nontrivial $ X $. 
A theorem of Fausk says that these cohomological shifts are essentially the only reason that the inclusion $ \mathrm{Pic}^{\mathrm{cl}}(X) \to \mathrm{Pic}(X) $ fails to be an isomorphism \cite{MR1966659}. 
In this section, we prove an analogous result for the Poincaré Picard group of a scheme with involution. 
In other words, we exhibit an inclusion from the \emph{hermitian Picard group} of $ X $ to the Poincaré Picard group and completely characterize its cokernel (Theorem \ref{thm: Fausk for pnpic and schemes with good quotient}).

\begin{definition}
    \label{definition: involutive picard space}
    Let $(X,\lambda, Y,p)$ be a scheme with good quotient. We then define the \textit{involutive Picard space} $\mathcal{P}\mathrm{ic}(X,\lambda)$ to be the subgroupoid of $\mathcal{P}\mathrm{ic}^\mathrm{p}(X,\lambda, Y, p)$ spanned by objects $(\mathcal{L},q)$ such that the underlying line bundle $\mathcal{L}$ lies in the heart $\mathrm{Mod}_{\mathcal{O}_X}^{\heartsuit}$ of the standard t-structure. 
    Define the \textit{involutive Picard group} to be $\mathrm{Pic}(X,\lambda):=\pi_0(\mathcal{P}\mathrm{ic}(X,\lambda))$. 
\end{definition}
\begin{remark}
    In principle, the space $\mathcal{P}\mathrm{ic}(X,\lambda)$ depends on the choice of good quotient $p:X\to Y$, and so it might appear to be a strange choice to drop these things from the notation. By \cite[Remark 4.20]{azumaya_involution} the scheme $Y$ and map $p$ are uniquely determined up to isomorphism by $(X,\lambda)$ as soon as one exists. 
    The reader interested in extending such definitions to more general objects (cf. \ref{rmk:ringed_topoi_with_involution}) may wish to use different notation. 
\end{remark}
\begin{remark}~\label{remark: categorical description of discrete poincare objects}
    Let $(X,\lambda, Y, p)$ be a scheme with good quotient and associated commutative algebra in Poincar\'e $\infty$-categories $\mathrm{Mod}^\mathrm{p}_{\underline{\mathcal{O}}}$ of Construction~\ref{cons:C2_mod_over_sheaf_of_Green_func}. Define $\mathrm{Pn}(\mathrm{Mod}^{\mathrm{p},\heartsuit}_{\underline{\mathcal{O}}})$ to be the pullback 
    \[
    \begin{tikzcd}
        \mathrm{Pn}(\mathrm{Mod}^{\mathrm{p},\heartsuit}_{\underline{\mathcal{O}}})\ar[d] \ar[r] & \mathrm{Pn}(\mathrm{Mod}^\mathrm{p}_{\underline{\mathcal{O}}})\ar[d] \\
        \mathrm{Mod}_{\mathcal{O}_X}^{\heartsuit} \ar[r] & \mathrm{Mod}_{\mathcal{O}_X}
    \end{tikzcd}
    \] in the $\infty$-category of symmetric monoidal $\infty$-categories. Since the bottom map is fully faithful, the top map will be as well, and so upon taking the Picard space we get that $\mathcal{P}\mathrm{ic}(\mathrm{Pn}(\mathrm{Mod}^{\mathrm{p},\heartsuit}_{\underline{\mathcal{O}}}))$ is a full subgroupoid of $\mathcal{P}\mathrm{ic}^\mathrm{p}(X,\lambda, Y, p)$ spanned by objects $(\mathcal{L},q)$ for which the underlying $\mathcal{O}_X$-module is in the heart. In other words, $\mathcal{P}\mathrm{ic}(\mathrm{Pn}(\mathrm{Mod}^{\mathrm{p},\heartsuit}_{\underline{\mathcal{O}}}))\simeq \mathcal{P}\mathrm{ic}(X,\lambda)$ of Definition~\ref{definition: involutive picard space}. This shows that $\mathcal{P}\mathrm{ic}(X,\lambda)$ naturally admits a lift to $\mathrm{Sp}_{\geq 0}$ which we will denote by $\mathfrak{pic}(X,\lambda)$. Note also that the Picard spectrum functor commutes with limits, and so we get a pullback 
    \[
    \begin{tikzcd}
        \mathfrak{pic}(X,\lambda) \ar[r] \ar[d] & \mathfrak{pic}^\mathrm{p}(X,\lambda, Y, p)\ar[d]\\
        \mathfrak{pic}(\mathrm{Mod}_{\mathcal{O}_X}^\heartsuit) \ar[r] & \mathfrak{pic}(X)
    \end{tikzcd}
    \]
    of connective spectra. Taking $\pi_0$ and noting that the bottom map induces the identity map $\pi_1(\mathfrak{pic}(\mathrm{Mod}_{\mathcal{O}_X}^\heartsuit))\cong\mathrm{H}^0(X;\mathbb{G}_m)\cong \pi_1(\mathfrak{pic}(X))$, we get a pullback
    \[
    \begin{tikzcd}
        \mathrm{Pic}(X,\lambda) \ar[r] \ar[d] & \mathrm{Pic}^\mathrm{p}(X,\lambda, Y, p) \ar[d] \\
        \mathrm{Pic}^{\mathrm{cl}}(X) \ar[r] & \mathrm{Pic}(X)
    \end{tikzcd}
    \]
    in the category of abelian groups.
\end{remark}

\begin{observation}
    \label{obs:pnpic_gs_symmetric_agree_sometimes}
    Let $R$ be a discrete commutative ring with a $C_2$-action $\lambda:R\rightarrow R$, endowed with the genuine symmetric Poincaré structure of Example \ref{ex:fixpt_Mackey_functor}. 
    The canonical map $ \underline{R} \to R^s $ from $ \underline{R} $ to its Borel completion induces a map $ \mathrm{Pic}^\mathrm{p}(\underline{R})\to\mathrm{Pic}^\mathrm{p}(R^\mathrm{s}) $. 
    Since $ \underline{R} \to R^s $ induces an equivalence on Borel $ \EE_\infty $-algebras with $ C_2 $-action, their underlying stable $ \infty $-categories with duality are equivalent, hence (see Definition \ref{defn:pnpic_to_pic_underlying}) there is a commutative diagram
    \begin{equation}\label{eq:pnpic_gs_to_borel}
        \begin{tikzcd}
            \mathrm{Pic}^\mathrm{p}(\underline{R})\ar[r] \ar[d] & \mathrm{Pic}^\mathrm{p}(R^s) \ar[d] \\
            \Pic(R^e) \ar[r,equals] & \Pic(R^e) \,.
        \end{tikzcd}
    \end{equation}
    It follows from \cite[Remark 4.2.21]{CDHHLMNNSI} that given a discrete invertible $ R $-module $ I \in \Pic^{\mathrm{cl}}(R) $, the fiber of the right hand vertical map in (\ref{eq:pnpic_gs_to_borel}) over $ I $ is given by $\mathrm{Hom}_R(I\otimes_R \lambda^*I,R)^{C_2}$. 
    By Remark \ref{remark:hermitian_forms_as_fixed_points}, the fibers of the vertical maps in (\ref{eq:pnpic_gs_to_borel}) over $ I $ are equivalent. Consequently pulling back along the inclusion $\mathrm{Pic}^{\mathrm{cl}}(R^e)\to \mathrm{Pic}(R^e)$ induces an equivalence \[\mathrm{Pic}(R,\lambda)\cong \mathrm{Pic}^{\mathrm{cl}}(R^e)\fiberproduct_{\mathrm{Pic}(R^e)}\mathrm{Pic}^\mathrm{p}(R^s)\] of abelian groups.
\end{observation}
We have in fact already seen the involutive Picard group.
\begin{proposition}\label{prop:hermitian_pic_agrees_with_pnpic_on_heart}
    There is a natural equivalence
    \[\mathrm{Pic}(X,\lambda)\cong \mathrm{Pic}^\mathrm{h}(X,\lambda)\] 
    of functors from $ \mathrm{qSch}^{C_2,\op} $ (Definition \ref{defn:Category of good quotients}) to abelian groups, where the latter is the hermitian Picard group of Definition~\ref{definition:hermitian_picard_group}.
\end{proposition}
\begin{proof}
    Let $\mathrm{Spec}(S)^\bullet\to Y$ be an \'etale hypercover by affine schemes which exists by the assumption that $Y$ is quasicompact and quasiseparated. By Lemma~\ref{lem: c2 etale topology=etale topology on the base} this extends to a $C_2$-\'etale hypercovering $\mathrm{Spec}(R^\bullet)^{gs}\to (X,\lambda, Y,p)$. 
    By construction and since $p$ is affine, each term in this hypercover will be a genuine symmetric Poincar\'e ring. We have that $(\mathrm{Mod}_{\mathcal{O}_X},\Qoppa_{\underline{\mathcal{O}}})$ is a $C_2$-\'etale hypersheaf by Proposition~\ref{prop:Poincare_modules_as_etale_sheaf}, and since $\mathrm{Pn}(-)$ is corepresentable it will also commute with limits. Therefore we have a map \[\Pn(\mathrm{Mod}_{\mathcal{O}_X}^\omega, \Qoppa)\simeq \lim \Pn(\mathrm{Mod}_{R^\bullet}^\omega, \Qoppa^{gs})\to \lim \Pn(\mathrm{Mod}_{R^\bullet}^\omega, \Qoppa^{s})\simeq \lim (\mathrm{Mod}_{R^\bullet}^{\omega, \simeq})^{\mathrm{h}C_2}\simeq (\mathrm{Mod}_{\mathcal{O}_X}^{\omega, \simeq})^{\mathrm{h}C_2}\] where the $C_2$-action on $\mathrm{Mod}_{\mathcal{O}_X}^{\omega, \simeq}$ is induced by the duality associated to $ \Qoppa $. 
    The functors $\mathrm{Mod}_{\mathcal{O}_X}^\omega\to \mathrm{Mod}_{R}$ are $t$-exact for the canonical $t$-structure on $\mathrm{Mod}_{\mathcal{O}_X}$ by definition, so the composite $\mathrm{Pn}\left(\mathrm{Mod}^{\mathrm{p},\heartsuit}_{\underline{\mathcal{O}}}\right)\to \left(\mathrm{Mod}_{\mathcal{O}_X}^{\omega,\heartsuit,\simeq}\right)^{\mathrm{h}C_2}$ is an equivalence. 
    The ordinary line bundles in $\Pn(\mathrm{Mod}_{\mathcal{O}_X}, \Qoppa)$ consist exactly of pairs $ (\mathcal{L},q) $ where $\mathcal{L}$ is a line bundle on $X$ and $ q $ is a nondegenerate $ \lambda $-hermitian pairing $ \mathcal{L} \otimes_{\mathcal{O}_X} \lambda^* \mathcal{L} \to \mathcal{O}_X $, up to isomorphisms which respect the $ \lambda $-hermitian pairing. 
\end{proof}
Let $R$ be a commutative discrete ring with involution $\lambda\colon R\to R$. We will now explain the precise relationship between the three invariants $\Pic^\mathrm{h}(R,\lambda)$, $\mathrm{Pic}^\mathrm{p}(R^s)$, and $\mathrm{Pic}^\mathrm{p}(R)$ where the last invariant considers $R$ as a Poincar{\'e} ring via the genuine symmetric structure of Example~\ref{ex:fixpt_Mackey_functor}. Before we can state the precise relationship, first recall the following theorem of Fausk (which is proved for the Picard spectrum by Antieau-Gepner): 
\begin{theorem}[{\cites{MR1966659}[Corollary 7.10]{MR3190610}}]
    \label{theorem:fausk_for_discrete_rings}
    Let $R$ be a discrete commutative ring. Then we have a product decomposition $$\mathfrak{pic}(R)\simeq \mathfrak{pic}^\mathrm{cl}(R)\times \Gamma(\Spec R;\underline{\mathbb{Z}})$$ of connective spectra, where $\mathfrak{pic}^{cl}(R)$ is $\mathfrak{pic}(\mathrm{Mod}_R^{\heartsuit})$. 
\end{theorem}
The functoriality of this splitting together with Lemma~\ref{lemma:pnpic_of_symmetric_coconnective_alg} produce the desired connection between $\mathrm{Pic}^\mathrm{h}(R)$ and $\mathrm{Pic}^\mathrm{p}(R^s)$:

\begin{corollary}
    \label{cor:fausk_for_Poincare_rings}
    Let $R$ be a discrete commutative ring with a $C_2$-action $\lambda:R\rightarrow R$. Then we have an isomorphism of abelian groups $$\mathrm{Pic}^\mathrm{p}(R^s)\simeq \mathrm{Pic}^\mathrm{h}(R^s)\times \Gamma(R;\underline{\mathbb{Z}})^{-\lambda}$$ where $\mathrm{Pic}^\mathrm{h}$ is the functor defined in Definition~\ref{definition:hermitian_picard_group} and the notation $(-)^{-\lambda}$ denotes fixed-points with respect to the action induced by the action $n\mapsto -\lambda^*n$ on the constant sheaf $\mathbb{Z}$.
\end{corollary}

\begin{proof}
    If we define $\mathrm{Pn}(\mathrm{Mod}_{R^s}^{\mathrm{p},\heartsuit})$ in the same way as in Remark~\ref{remark: categorical description of discrete poincare objects}, we then have that $\mathfrak{pic}(R,\lambda)\simeq \mathfrak{pic}(\mathrm{Pn}(\mathrm{Mod}_{R^s}^{\mathrm{p},\heartsuit}))$ by Observation~\ref{obs:pnpic_gs_symmetric_agree_sometimes}. The functor $\mathrm{Pn}(\mathrm{Mod}_{R^s}^{\mathrm{p},\heartsuit})\to \mathrm{Mod}_{R^e}^{\omega, \heartsuit, \simeq}$ induces an equivalence of infinite loop spaces $ \Omega^\infty \mathfrak{pic}(R,\lambda)\simeq \Omega^\infty \mathfrak{pic}^{\mathrm{cl}}(R^{e})^{\mathrm{h}C_2}$ as can be checked on the level of categories. 
    
    We will now show that the decomposition of Theorem~\ref{theorem:fausk_for_discrete_rings} is equivariant, which then combining this with Proposition~\ref{prop:symmetric_pnpic} will finish the proof. Since we already have that the map $\mathfrak{pic}^{\mathrm{cl}}(R)\to \mathfrak{pic}(R)$ is equivariant by the previous paragraph, it is enough to show that the map $C(\mathrm{Spec}(R),\mathbb{Z})\to \mathfrak{pic}(R)$ is as well. Define $\mathcal{C}^{\simeq}\subseteq \mathrm{Mod}_{R}^{\simeq}$ to be the full subcategory spanned by objects which are finite sums $\bigoplus_{i\in I} e_i R[n_i]$ where $e_i$ form an orthogonal basis of idempotents of $R$. Note that $\mathcal{C}^{\simeq}$ is a discrete category since for any object $\bigoplus_{i\in I}e_iR[n_i]$, for any $j\in J$ we have that \[\mathrm{hom}_{R}(e_jR[n_j],\bigoplus_{i\in I}e_iR[n_i])\simeq \bigoplus_{i\in I}\mathrm{hom}_R(e_jR[n_j],e_iR[e_i])\simeq \mathrm{hom}_R(e_jR[n_j],e_jR[n_j])\] where the last equivalence comes from the assumption that the $e_i$ are orthogonal. Note that $\mathcal{C}$ is closed under the symmetric monoidal structure and taking duals, so we get a $C_2$-equivariant map $\mathfrak{pic}(\mathcal{C})\to \mathfrak{pic}(R)$. This map is exactly an equivariant lift of the map $C(\mathrm{Spec}(R),\mathbb{Z})\to \mathfrak{pic}(R)$.
\end{proof}

We now turn our attention to relating $\mathrm{Pic}^\mathrm{h}(R)$ and $\mathrm{Pic}^\mathrm{p}(R)$.

\begin{theorem}\label{theorem:Poincare_Pic_of_fixpt_Mackey_functor}
    Let $ R $ be a discrete commutative ring with a $ C_2 $-action $ \lambda \colon R \xrightarrow{\sim} R $ via ring maps. 
    Regard $ R $ as a Poincaré ring $ \underline{R}^\lambda $ via Example \ref{ex:fixpt_Mackey_functor}. 
    Then there is a split short exact sequence of abelian groups
    \begin{equation*}
        0 \to \mathrm{Pic}^{\mathrm{h}}(R,\lambda) \to \pi_0\pnpic(\underline{R}^\lambda) \to C_{C_2}(\Spec R, \ZZ^{-}) \to 0
    \end{equation*}
    where $ \ZZ^{-} $ is endowed with the $ C_2 $-action given by multiplication by $ -1 $ and $ C_{C_2} $ denotes continuous functions which are moreover $ C_2 $-equivariant. 
    In particular, when $ R $ has the trivial $ C_2 $-action, then the map $ \mathrm{Pic}^{\mathrm{h}}(R) \to \pi_0\pnpic(\underline{R}^\lambda) $ is an isomorphism. 
    
    Moreover, forgetting the hermitian form (resp. forgetting the $ C_2 $-action) induces a commutative diagram
    \begin{equation*}
        \begin{tikzcd}
            0 \ar[r] & \mathrm{Pic}^\mathrm{h}(R,\lambda) \ar[r]\ar[d] & \pi_0\pnpic(\underline{R}^\lambda) \ar[r] \ar[d] & C_{C_2}(\Spec R, \ZZ^{-}) \ar[r] \ar[d] & 0 \\
            0 \ar[r] & \mathrm{Pic}^{\mathrm{cl}}(R) \ar[r] & \pi_0 \mathrm{Pic}\left(\mathrm{Mod}_R^\omega\right) \ar[r] & C(\Spec R, \ZZ) \ar[r] & 0 
        \end{tikzcd}    
    \end{equation*}
    where the bottom row is that of \cite[Theorem 3.5]{MR1966659}. 
\end{theorem}
\begin{proof} [Proof of Theorem \ref{theorem:Poincare_Pic_of_fixpt_Mackey_functor}]
    An object of $ \pi_0 \pnpic(\underline{R}^\lambda) $ can be represented by a pair $ (I, q) $ where $ I $ is an invertible $ R$-module and $ q $ is a point in $ \pi_0 \Omega^\infty \Qoppa_{R^{gs}}(I) $. 
    By the proof of \cite[Theorem 3.5]{MR1966659}, $ I $ induces a continuous map $ \Psi(I) \colon \Spec R \to \ZZ $. 
    Write $ \lambda $ for the involution on $ R $. 
    Now $ q $ in particular induces an equivalence $ q \colon I \xrightarrow{\sim} I^\dag \simeq (\lambda_*I)^\vee $. 
    For each point $ \mathfrak{p} \in \Spec R $, localizing $ q $ gives an equivalence
    \begin{equation*}
        q_{\mathfrak{p}} \colon I_{\mathfrak{p}} \xrightarrow{\sim} (\lambda_*I)^\vee_{\mathfrak{p}} \simeq \left(\lambda_*(I_{\lambda(\mathfrak{p})})\right)^\vee \,.
    \end{equation*}
    Since $ I_{\mathfrak{p}} $ is an invertible module over a local ring, \cite[Proposition 3.2]{MR1966659} implies that $ q_{\mathfrak{p}} $ induces an equivalence
    \begin{equation*}
        I_{\mathfrak{p}} \simeq R_{\mathfrak{p}}[\phi(\mathfrak{p})] \xrightarrow{\sim} \left(\lambda_*(R_{\lambda(\mathfrak{p})}[\phi(\lambda(\mathfrak{p}))])\right)^\vee \simeq (\lambda_*(R_{\lambda(\mathfrak{p})}))^\vee [-\phi(\lambda(\mathfrak{p}))] \,.
    \end{equation*}
    Since $ R $ is discrete, this implies in particular that $ \Psi(I)(\lambda(\mathfrak{p})) = -\Psi(I)(\mathfrak{p}) $, i.e. that $ \Psi(I) $ is $ C_2 $-equivariant. 
    It follows immediately from \cite[Theorem 3.5]{MR1966659} that $ \Psi $ is a homomorphism and that an element of the kernel of $ \Psi $ lifts to $ \mathrm{Pic}^\mathrm{h}(R,\lambda) $. 
    
    Now consider a $ C_2 $-equivariant map $ g \colon \Spec R \to \ZZ $. 
    As in \emph{loc. cit.}, the image of $ g $ is finite and $ C_2 $-invariant, say $ \{n_1, -n_1, \ldots, n_m, -n_m\} $ or $ \{0, n_1, -n_1, \ldots, n_m, -n_m\} $ for some $ n_i \neq 0 $. 
    As in \emph{loc. cit.}, the disjoint subsets $ U_{\pm n_i} := g^{-1}(\pm n_i) $ correspond to an orthogonal basis of idempotents $ e_{U_{\pm n_i}} $ in $ R $. 
    Since $ g $ is $ C_2 $-equivariant with respect to the sign action on $ \ZZ $, we have $ \lambda(U_{n_i}) = U_{-n_i} $. 
    Moreover, it follows from Lemma 3.4 \emph{ibid.} that $ \lambda(e_{U_{n_i}}) = e_{U_{-n_i}} $. 
    Consider the $ R $-module $ \Phi(g):= \bigoplus_{n \in \mathrm{Im}(g)} e_{g^{-1}(\{n\})} R[n] $. 
    In other words, $ \Phi(g) := \bigoplus_{i=1}^m \left(e_{U_{n_i}}R[n_i] \oplus e_{U_{-n_i}}R[-n_i] \right) $ if $ 0 $ is not in the image of $ g $ and $ \Phi(g) := e_{U_0} \oplus \bigoplus_{i=1}^m \left(e_{U_{n_i}}R[n_i] \oplus e_{U_{-n_i}}R[-n_i] \right) $ otherwise. 
    Observe that $ \left(e_{U_{-n_i}}R[-n_i]\right)^\dag = \hom_R(e_{U_{-n_i}}R[-n_i], \lambda_* R)= \hom_R\left(\lambda_*(e_{U_{-n_i}}R), R\right)[n_i] = \hom_R\left(e_{U_{n_i}}R, R\right)[n_i] $. 
    Finally, we claim that there is a canonical $ \lambda $-hermitian form $ q_g \in \Omega^\infty \Qoppa_{R^{\mathrm{gs}}}(\Phi(g)) $ whose adjoint $ q_g^\dag \colon \Phi(g) \xrightarrow{\sim} \Phi(g)^\dag $ corresponds to the identity. 
    That $ q_g $ defines a point of $ \hom_{R^{\otimes 2}}(\Phi(g)^{\otimes 2}, R)^{\mathrm{h}C_2} $ is evident. 
    Observe that to give a lift of $ q_g $ to $ \Qoppa_{R^{\mathrm{gs}}}(\Phi(g)) = \hom_{N^{C_2}R}\left(N^{C_2}\Phi(g), R\right) $ is equivalent to giving a commutative diagram
    \begin{equation}\label{diagram:lifting_sym_form_to_gen_sym_form}
        \begin{tikzcd}
            \Phi(g) \otimes_R R^{\varphi C_2} \ar[d] \ar[r,dashed,"{\exists ?}"] & R^{\varphi C_2} \ar[d] \\
            \left(\Phi(g)^{\otimes 2}\right)^{\mathrm{t}C_2} \ar[r,"{q_g^{\mathrm{t}C_2}}"] & R^{\mathrm{t}C_2}
        \end{tikzcd} 
    \end{equation} 
    of $ R^{\varphi C_2} $-modules. 
    Let us write $ \eta \colon R \to \pi_0 R^{\varphi C_2} $ for the ring map induced by the structure map. 
    Since $ R $ is a $ C_2 $-$ \EE_\infty $-ring, $ \eta $ is invariant with respect to the given action on $ R $ and the trivial action on $ \pi_0 R^{\varphi C_2} $. 
    Consider $ e_{U_{n_i}} $ an idempotent corresponding to an element of the image of $ g $ so that $ n_i \neq 0 $. Then
    \begin{equation*}
        \begin{split}
            \eta(e_{U_{n_i}}) &= \eta(e_{U_{n_i}})^2 \qquad \text{ ring maps preserve idempotents } \\
            &= \eta(e_{U_{n_i}}) \cdot \eta(e_{U_{-n_i}}) \qquad \text{ $C_2$-invariance of }\eta \\
            &= \eta(e_{U_{n_i}}e_{U_{-n_i}}) \qquad \text{ $ \eta $ is a ring map } \\
            &= 0 \qquad \text{ orthogonality and }n_i \neq 0 \,.  
        \end{split}
    \end{equation*}
    In particular, if $ 0 $ is not in the image of $ g $, $ \Phi(g) \otimes_R R^{\varphi C_2} \simeq 0 $ and (\ref{diagram:lifting_sym_form_to_gen_sym_form}) commutes vacuously. 
    If $ 0 $ is in the image of $ g $, then $ e_{U_0}R $ is an ordinary projective $ e_{U_0}R $-module and $ q_g $ evidently defines a genuine hermitian form on $ e_{U_0}R $ (compare \cite[Remark 4.2.21]{CDHHLMNNSI}). 
    
    Thus, $ g \mapsto (\Phi(g), q_g) $ defines a splitting of $ \Psi $ which agrees with the splitting constructed in \cite[Theorem 3.5]{MR1966659} on underlying objects. 
    It remains to observe that the kernel of $ \Psi $ consists exactly of pairs $ (I,q) $ where $I$ is a discrete invertible $ R $-module and $ q $ is a nondegenerate $ \lambda $-hermitian pairing $ I \otimes_{R} \lambda^* I \to R $, up to isomorphisms which respect the $ \lambda $-hermitian pairing.
\end{proof}

We now turn to the case of schemes with good quotient. 
\begin{theorem}~\label{thm: Fausk for pnpic and schemes with good quotient}
    Let $(X,\lambda, Y, p)$ be a scheme with good quotient in the sense of Recollection \ref{rec:good_quotient}. 
    There is then a split short exact sequence of abelian groups \[0\to \mathrm{Pic}^\mathrm{h}(X,\lambda)\to \pi_0(\pnpic(X,\lambda, Y,p))\to C_{C_2}(X,\ZZ^-)\to 0\] and a map 
    \[
    \begin{tikzcd}
        0 \ar[r] & \mathrm{Pic}^{\mathrm{h}}(X, \lambda) \ar[r]\ar[d] & \pi_0(\pnpic(X,\lambda, Y,p)) \ar[r] \ar[d] & C_{C_2}(X,\ZZ^-) \ar[r] \ar[d] & 0\\
        0 \ar[r] & \mathrm{Pic}^{\mathrm{cl}}(X) \ar[r] & \pi_0(\pic(\mathrm{Mod}_{\mathcal{O}_X})) \ar[r] & C(X,\ZZ)\ar[r] & 0
    \end{tikzcd}
    \] of short exact sequences, where the bottom row is that of \cite[Proposition 4.4]{MR1966659}.
\end{theorem}
\begin{proof}
  Composing the forgetful map of Definition \ref{defn:pnpic_to_pic_underlying} with the projection of Theorem~\ref{theorem:fausk_for_discrete_rings} gives a map $\Pic^\mathrm{p}(X,\lambda, Y,p)\to \Pic(X)\to C(X,\mathbb{Z})$. 
  By the same argument as in Theorem~\ref{theorem:Poincare_Pic_of_fixpt_Mackey_functor} this map will factor through a surjective map $C_{C_2}(X,\mathbb{Z}^-)$.
    The kernel consists exactly of isomorphism classes of invertible Poincaré objects $\left(\mathcal{L},q\right)$ so that $ \mathcal{L} $ is an ordinary line bundle, which we identified $\mathrm{Pic}^{\mathrm{h}}(X,\lambda)$ in Proposition \ref{prop:hermitian_pic_agrees_with_pnpic_on_heart}. 
\end{proof}

\section{The Poincaré Brauer space}~\label{section:the_poincare_brauer_group}

In this section, we introduce a derived version of the classical involutive Brauer group, realizing the goal discussed in \S\ref{subsection:Azumaya_motivation}. 
To motivate the path we will take, consider a scheme with involution $ (X,\lambda) $ and an Azumaya $ \mathcal{O}_X $-algebra $ \mathcal{A} $. 
Note that an equivalence $ \mathcal{A} \simeq \lambda_*\mathcal{A}^\op $ induces an equivalence $ F \colon \Mod_{\mathcal{A}}^\omega \simeq \Mod_{\mathcal{A}}^{\omega,\op} $ so that $ F^\op \circ F \simeq \id $. 
In particular, the functor $ F $ is linear with respect to the involution on $ \Mod_{\mathcal{O}_X} $ induced by $ \lambda $. 
In \S\ref{section:poincare_structures_on_compact_modules}, we saw that the involution on $ X $ gave rise to a canonical symmetric monoidal Poincaré structure on $ \Mod_{\mathcal{O}_X}^\omega $. 
The formalism of \emph{relative} Poincaré $ \infty $-categories (see \S\ref{section:relative_poincare_cats}) will allow us to consider a general $ \mathcal{O}_X $-linear $ \infty $-category $ \mathcal{C} $ with a duality functor $ D \colon \mathcal{C} \to \mathcal{C}^{\mathrm{op}} $ which is \emph{linear with respect to a given involution on $ \Mod_{\mathcal{O}_X} $} (cf. \cite[\S6]{HLAS25}). 
Our relative Poincaré $ \infty $-categories thus offer a natural home for a derived enhancement of the involutive Brauer group.  

We begin in \S\ref{subsection: review of PS} with a review of the classical involutive Brauer group of \cite{Parimala_Srinivas}.  
In \S\ref{subsection:invertible_relative_Poincare_cats}, we define the involutive Brauer group in terms of invertible Poincaré $ \infty $-categories and compute it in a handful of examples. 
In particular, our examples (see Examples \ref{ex:pnbr_closed_pt_unramified} and \ref{ex:pnbr_closed_point_ramified}, especially the latter) show that our involutive Brauer group subsumes the classical involutive Brauer group but also captures information which is genuinely novel.  
In \S\ref{subsection:generalized_Azumaya_alg_with_involution}, we introduce the notion of a genuine Azumaya algebra with involution and show that they give rise to classes in the Poincaré Brauer group. 
We also discuss the relationship between ordinary Azumaya algebras with involution and genuine Azumaya algebras with involution. 
By bootstrapping the relevant results of Toën and Antieau--Gepner and combining with relative enhancements of descriptions of endomorphism algebras contained in \cite[\S3.1]{CDHHLMNNSI}, we show that Poincaré Brauer classes over suitably nice bases are represented by Azumaya algebras with genuine involution. 
		
\subsection{Review of the involutive Brauer group}~\label{subsection: review of PS}

In this section we will review the involutive Brauer group as constructed by Parimala and Srinivas in \cite{Parimala_Srinivas}. For this subsection, let $(X,\lambda, Y, p:X\to Y)$ be a scheme with good quotient such that $\frac{1}{2}\in \mathcal{O}_Y$ and such that either $p $ is {\'e}tale or an isomorphism.\footnote{Unfortunately, if one goes back and looks at \cite{Parimala_Srinivas}, they will see that the roles of $X$ and $Y$ are swapped. This is to match with the notation used in \cite{azumaya_involution}, from which we get the language of schemes with good quotients.} We will handle these two cases separately, beginning with the case of $p $ being an isomorphism. 

In the case where the involution $\lambda$ is the identity, the data of an Azumaya algebra with involution becomes an Azumaya algebra $\mathcal{A}$ on $X$ together with an $\mathcal{O}_X$-linear map $\sigma:\mathcal{A}\to \mathcal{A}^\op$ such that $\sigma^\op\circ \sigma=\mathrm{id}_\mathcal{A}$.

\begin{proposition}[\cite{MR1045736}]~\label{prop: local description of type I Azumaya algebras}
    Let $\mathcal{E}$ be a locally free $\mathcal{O}_X$-module, and let $\mathcal{A}=\mathcal{E}nd_{\mathcal{O}_X}(\mathcal{E})$ be equipped with an involution $\sigma$. There is then some uniquely determined $\epsilon \in \Gamma(X;\mu_2)$ and nondegenerate $\epsilon$-symmetric bilinear form \[B:\mathcal{E}\otimes_{\mathcal{O}_X} \mathcal{E}\to \mathcal{L}\] (uniquely determined up to tensor powers of $\mathcal{L}^{\otimes_{\mathcal{O}_X} 2}$) such that $\sigma$ is given by the composition \[\mathcal{A}\cong \mathcal{E}\otimes \mathcal{E}^\vee\cong \mathcal{E}\otimes \mathcal{L}^{-1}\otimes \mathcal{E}\to \mathcal{E}\otimes \mathcal{L}^{-1}\otimes \mathcal{E}\cong \mathcal{A}^{\op}\] where the second-to-last map is $a\otimes b\otimes c\mapsto \epsilon c\otimes b\otimes a$.
\end{proposition}

In the above situation we write $\sigma=\alpha_B$. The fact that all involutions on these trivial Azumaya algebras come from form data motivates defining the analogue of Morita equivalence in this context via the associated bilinear form. We are now ready to define the involutive Brauer group of type I.

\begin{construction}[Involutive Brauer group of type I, \cite{Parimala_Srinivas}]~\label{const: involutive BG type I}
    Let $X$ be a scheme with trivial $C_2$-action and assume $ 2 \in \Gamma(X, \mathcal{O}_X)^\times $. 
    Let $(\mathcal{A},\sigma)$ and $(\mathcal{B},\tau)$ be Azumaya algebras with involution on $X$. Their tensor product is defined to be $(\mathcal{A}\otimes \mathcal{B},\sigma\otimes \tau)$ where $\sigma\otimes \tau$ is the composition $\mathcal{A}\otimes \mathcal{B}\simeq \mathcal{A}^\op\otimes \mathcal{B}^\op\simeq (\mathcal{A}\otimes \mathcal{B})^\op$. We say that $(\mathcal{A},\sigma)$ and $(\mathcal{B},\tau)$ are equivalent if there exists locally free $\mathcal{O}_X$-modules $\mathcal{E},\mathcal{E}'$ with nondegenerate \textit{symmetric} bilinear forms $B,B'$ and isomorphisms \[(\mathcal{A}\otimes \mathcal{E}nd(\mathcal{E}),\sigma\otimes \alpha_B)\cong (\mathcal{B}\otimes \mathcal{E}nd(\mathcal{E}'),\tau\otimes \alpha_{B'}).\] Let $\sim$ denote this equivalence relation, and define \[\mathrm{Br}(X,\mathrm{id}_X):=(\{(\mathcal{A},\sigma)\mid \mathcal{A}\text{ is an Azumaya algebra with involution} \},\otimes)/\sim\] to be the involutive Brauer group.\footnote{In \cite{Parimala_Srinivas} this is denoted by $\mathrm{Br}^*(X)$. We chose to unify the notation.} This set does in fact form a group by \cite[Lemma 1]{Parimala_Srinivas}
\end{construction}

As shown in \cite[Theorem 1]{Parimala_Srinivas}, there is an injective map $\mathrm{Br}(X,\mathrm{id})\hookrightarrow H^0_{\acute{e}t}(X;\mu_2)\times H^2_{\acute{et}}(X;\mu_2)$ which is an isomorphism if and only if every class in the $2$-torsion in the cohomological Brauer group $\mathrm{Br}'(X)$ is represented by an Azumaya algebra. Additionally, the map $\mathrm{Br}(X,\mathrm{id})\to \mathrm{Br}(X)[2]$ is surjective by \cite{MR1045736}.

We now turn our attention to Azumaya algebras of the second kind. Now $p:X\to Y$ is assumed to be an {\'e}tale cover of degree $2$. In this case, an Azumaya algebra with involution on $X$ is an Azumaya algebra $\mathcal{A}$ on $X$ together with an involution $\sigma:\mathcal{A}\to \mathcal{A}^\op$ of sheaves which restricts to $\lambda$ on the center. Once again, in this setting there is a nice classification of such objects {\'e}tale locally. 

\begin{proposition}~\label{prop: local description of type II Azumaya algebras}
    Let $\mathcal{E}$ be a locally free $\mathcal{O}_X$-module and $\mathcal{A}=\mathcal{E}nd_{\mathcal{O}_X}(\mathcal{E})$. Then for any involution of $\mathcal{A}$ which restricts to $\lambda$ on $\mathcal{O}_X$, there exists some $\mathcal{L}\in \mathrm{Pic}(Y)$ and some nondegenerate $\lambda$-hermitian form \[H:p _*\mathcal{E}\otimes_{\mathcal{O}_Y}p _*\mathcal{E}\to p _*p^*\mathcal{L}\] such that the involution $\sigma$ is given under the equivalence \[p _*(\mathcal{E}\otimes_{\mathcal{O}_X} p ^*\mathcal{L}^{-1}\otimes_{\mathcal{O}_X}\mathcal{E})\cong p _*\mathcal{A}\] by the automorphism of $\mathcal{E}\otimes_{\mathcal{O}_X}p ^*\mathcal{L}^{-1}\otimes_{\mathcal{O}_X}\mathcal{E}$ given by $a\otimes l\otimes b\mapsto b\otimes \lambda(l)\otimes a$.
\end{proposition}

If $\mathcal{E}$ is a locally free $\mathcal{O}_X$-module and $H$ is such a nondegenerate $\lambda$-hermitian form then we let $\alpha_H$ denote the induced involution of $\mathcal{E}nd_{\mathcal{O}_X}(\mathcal{E})$. Unlike in the involutions of the first type case, we will take all of these classes as the Brauer trivial Azumaya algebras with involution. 

\begin{construction}~\label{const: involutive BG of type II}
    We say that two Azumaya algebras with involution of the second kind $(\mathcal{A},\sigma)$ and $(\mathcal{B},\tau)$ are Morita equivalent if there exists some locally free $\mathcal{O}_X$-modules $\mathcal{E},\mathcal{E}'$ and $\lambda$-hermitian forms $H$ and $H'$ on $\mathcal{E}$ and $\mathcal{E}'$, respectively, and an isomorphism \[(\mathcal{A},\sigma)\otimes(\mathcal{E}nd_{\mathcal{O}_X}(\mathcal{E}), \alpha_H)\cong (\mathcal{B},\tau)\otimes(\mathcal{E}nd_{\mathcal{O}_X}(\mathcal{E}'),\alpha_{H'})\] of algebras with involution. Define the involutive Brauer group to be \[\mathrm{Br}(X,\lambda):= (\{(\mathcal{A},\sigma)\mid \mathcal{A}\text{ is an Azumaya algebra with involution of the second kind}\},\otimes)/\sim\] where $\sim$ is the above defined equivalence relation. This is a group by \cite[Page 216]{Parimala_Srinivas}.
\end{construction}

A theorem of Saltman's, in the form reproven in \cite{MR1045736}, shows that there is a surjective map $\mathrm{Br}(X,\lambda)\to \mathrm{Ker}(N_{X/Y}:\mathrm{Br}(X)\to \mathrm{Br}(Y))$. Using the local description of Azumaya algebras with involution of the second kind, \cite{Parimala_Srinivas} also show that there is an injective function \[\mathrm{Br}(X,\lambda) \to H^2_{\acute{e}t}(Y;\mathcal{G})\] where $\mathcal{G}$ is the sheaf of norm $1$ units on $X$, and that this map is an isomorphism if and only if every class in $\mathrm{Ker}(N_{X/Y}:\mathrm{Br}'(X)\to \mathrm{Br}'(Y))$ is representable by an Azumaya algebra. In fact there is an exact sequence \[\mathrm{Pic}(X)\to \mathrm{Pic}(Y)\to \mathrm{Br}(X,\lambda)\to \mathrm{Br}(X)\to \mathrm{Br}(Y)\] compatible with the cohomology long exact sequence for $1\to \mathcal{G}\to p _*\mathbb{G}_{m,X} \to \mathbb{G}_{m,Y}\to 1$.

There are two problems in trying to extend this definition to actions which are not trivial or {\'e}tale, both stemming from the same issue. Unlike in the two cases above, there are schemes with involution and good quotient $(X,\lambda,Y,p)$ and Azumaya algebras with involutions of the second kind $(\mathcal{A},\sigma)$ such that $(\mathcal{A},\sigma)$ is not locally isomorphic to some $(\mathcal{M}_n(\mathcal{O}_X),\tau)$ for any involution $\tau$, even when we allow $\mathcal{A}$ to vary in its Brauer class (see \cite[Remark 1.2]{first2025counterexamplesinvolutionsazumayaalgebras}). We used this local description to motivate what the Brauer trivial classes were. This in principle is surmountable via the classification of types in \cite[Section 5]{azumaya_involution}. The second issue, however, is that this local triviality was used in a key way in producing a comparison of the involutive Brauer group to computable, cohomological invariants. 

To say more about this second point, in \cite{Parimala_Srinivas} the maps $\mathrm{Br}(X,\mathrm{id}_X)\hookrightarrow \mathrm{H}^0_{\Acute{e}t}(X;\mu_2)\times \mathrm{H}^2_{\Acute{e}t}(X;\mu_2)$ and $\mathrm{Br}(X,\lambda)\hookrightarrow \mathrm{H}_{\Acute{e}t}^2(X;\mathcal{G})$ were constructed by identifying isomorphism classes $(A,\sigma)$ of degree $n$ Azumaya algebras $A$ with involution $\sigma$ with torsors for certain group schemes $\mathrm{PO}(n,\epsilon)$. This identification relies on the local identifications of Proposition~\ref{prop: local description of type I Azumaya algebras} and Proposition~\ref{prop: local description of type II Azumaya algebras}. While we still have a torsor description of Azumaya algebras via \cite[Corollary 5.2.19]{azumaya_involution}, the identification of group schemes in question is much more subtle because of the counterexamples constructed in \cite{first2025counterexamplesinvolutionsazumayaalgebras}.

\subsection{Invertible \texorpdfstring{$ R $}{R}-linear Poincaré \texorpdfstring{$\infty$}{∞}-categories}\label{subsection:invertible_relative_Poincare_cats}
Recall that a Poincaré $\infty$-category is called idempotent complete if the underlying stable $\infty$-category is idempotent complete. The full subcategory of $\Catp$ spanned by idempotent complete Poincaré $\infty$-categories is denoted by $\Catpidem$ \cite[Definition 1.3.2]{CDHHLMNNSII}.

\begin{definition}
    \label{definition:poincare_brauer_space}
    Let $R$ be a Poincaré ring spectrum. We define the \emph{Poincaré Brauer space of $R$} as $$\Brp(R):=\mathfrak{pic}\left(\Mod_{\Mod^p_R}\left(\Catpidem\right)\right)\,,$$
    where $ \Mod^p $ is the functor of Theorem \ref{thm:calgp_to_poincare_cat}\ref{thmitem:calgp_to_poincare_cat_with_tensor}. 
    The assignment $ R \mapsto \Brp(R) $ defines a functor
    \begin{equation*}
        \Brp \colon \CAlgp \to \CAlg^{\gp}(\Spaces)
    \end{equation*}
    valued in grouplike $ \Einfty $-spaces. 
    
    Let $ (X, \lambda, Y, p) $ be a scheme with involution and a good quotient. 
    We define the \emph{Poincaré Brauer space of $ X $ with respect to $ p $} as  
    \begin{equation*}
        \Brp(X) := \lim_{ j \colon \Spec R \to Y } \Brp\left(\underline{\Gamma(j^*(X))}^{j^*\sigma} \right)
    \end{equation*}
    where the limit is over all étale $ j $ and $ \Gamma(j^*(X)) $ is a ring with involution by affineness of $ p $, hence $ R \to \Gamma(j^*(X)) $ is a Poincaré ring via Example \ref{ex:fixpt_Mackey_functor}. 
\end{definition}
\begin{remark}
    \emph{A priori}, it is not clear that the two definitions of the Poincaré Brauer space in Definition \ref{definition:poincare_brauer_space} agree when $ X = \Spec A $ is an affine scheme. 
    We show that these definitions do in fact agree, and further we have an identification \[\mathfrak{br}^\mathrm{p}(X,\lambda, Y, p)\simeq \mathfrak{pic}\left(\Mod_{\left(\Mod_{\mathcal{O}_X}^\omega,\Qoppa_{\underline{\mathcal{O}}}\right)}(\Catpidem)\right)\] by Corollary~\ref{cor: two definitions of poincare brauer agree}. 
    The key point is that while the non-derived Brauer group and its involutive counterpart do not satisfy \'etale descent, the Brauer space and the Poincaré Brauer space do. 
    We will therefore focus our attention in the remainder of this section on proving results for Poincar{\'e} rings, and by \'etale descent the same results will hold for schemes with good quotients. 
\end{remark}
\begin{observation}\label{obs:symmetric_relative_poincare_cat}
    Let $ \left(\mathcal{C},\Qoppa\right) $ be a symmetric monoidal Poincaré $ \infty $-category. 
    Taking the bilinear part $ B_\Qoppa $ of $ \Qoppa $, $ \left(\mathcal{C},B_\Qoppa\right) $ is an $ \Einfty $-monoid in the $ \infty $-category of perfect symmetric $ \infty $-categories\footnote{Usual naming conventions suggest the unwieldy moniker ``symmetric monoidal perfect symmetric $ \infty $-category.''} (see Recollection \ref{rec:symmetric_Poincare_cat}). 
    The forgetful functor $ \Catp \to \Cat^{\mathrm{ps}} $ lifts to a symmetric monoidal forgetful functor $ \Mod_{\left(\mathcal{C},\Qoppa\right)} \left(\Catp\right) \to \Mod_{\left(\mathcal{C},B_\Qoppa\right)}\left(\Cat^{\mathrm{ps}}_\infty\right) $. 
    Suppose $ (X,\lambda,Y,p) \in \mathrm{qSch}^{C_2} $ and let $ \left(\mathcal{C},\Qoppa\right) \in \Mod_{\left(\Mod_{\mathcal{O}_X}^\omega,\Qoppa_{\underline{\mathcal{O}}}\right)}(\Catpidem) $. 
    Now the collection of small stable idempotent-complete symmetric monoidal $ \mathcal{O}_X $-linear $ \infty $-categories has a $ C_2 $-action where the generator acts by sending $ \mathcal{D} $ to $ \lambda^* \mathcal{D}^\op $. 
    Then $ \left(\mathcal{C},B_\Qoppa\right) $ defines a $ \lambda^* $-linear equivalence $ D_\Qoppa \colon \mathcal{C} \simeq \mathcal{C} $ exhibiting $ \mathcal{C} $ as a $ C_2 $-homotopy fixed point of this action. 
\end{observation}
\begin{proposition}\label{prop:loops_pnbr_is_pnpic}
    Let $R$ be a Poincaré ring spectrum. Then we have a canonical equivalence $$\Omega \Brp(R) \simeq \Picp(R)$$ of connective spectra.
\end{proposition}
\begin{proof} 
    Since $ \Omega\Brp(R) $ is given by the space of automorphisms of any object in $ \Brp(R) $, it suffices to determine the space of autoequivalences of $ \left(\Mod_{R^e}^\omega, \Qoppa_R \right) $. 
    Note that there is a natural transformation $\mathfrak{pic}^\mathrm{p}(-)\to \Omega \mathfrak{br}^\mathrm{p}(-)$ of connective spectra given by sending a Poincar{\'e} line bundle $(\mathcal{L},q)$ to the autoequivalence $-\otimes(\mathcal{L},q)$. To show that this is an equivalence we may therefore reduce to showing that this map is an equivalence of underlying spaces.
    By Proposition \ref{prop:forms_are_corepresented}, we have that an endomorphism of $(\mathrm{Mod}_{R^e}^\omega,\Qoppa_R)$ is equivalent to an element of $\mathrm{Pn}(\mathrm{Mod}_{R^e}^\omega, \Qoppa_R)$. Chasing through definitions, this equivalence is given by sending an endomorphism to its value on $(R^e,u)$, and so the composition $\mathfrak{pic}^\mathrm{p}(R)\to \Omega \mathfrak{br}^\mathrm{p}(R^e)\to \mathrm{Pn}(\mathrm{Mod}_{R^e}^\omega, \Qoppa_R)$ is the usual inclusion. By the functoriality of $\mathrm{Pn}$ and the fact that any endomorphism of $(\mathrm{Mod}^\omega_{R^e},\Qoppa_R)$ is $(\mathrm{Mod}_{R^e}^\omega,\Qoppa_R)$-linear, we see that the composition of two maps corresponds to the tensor product of the two corresponding elements of $\mathrm{Pn}(\mathrm{Mod}_{R^e}^\omega, \Qoppa_R)$. Thus $\Omega\mathfrak{br}^\mathrm{p}(R)$ is a subspace of $\mathfrak{pic}^\mathrm{p}(R)$, from which we conclude the desired result. \qedhere
\end{proof}

As in the Picard group case, the symmetric monoidal forgetful functor $\theta \colon \Catp_{, R} \to \Catex_{, R^e} $ induces a map of spectra $ \theta \colon \Brp(R)\to \mathfrak{br}(R^e)$. 
When $ R^e $ is endowed with the trivial action, $ \theta $ will factor through the $2$-torsion on $\pi_0$. 
As a consequence of Proposition~\ref{prop:relative_poincare_cats_basic_properties}\ref{propitem:classify_R_lin_hermitian_struct_mod_cat} we can identify the fiber of this map.
\begin{corollary}\label{cor:Poincare_Brauer_to_Brauer_fiber}
    Let $R$ be a Poincar{\'e} ring with underlying genuine $C_2$ spectrum $R^L$ as in Proposition~\ref{prop:relative_poincare_cats_basic_properties}\ref{propitem:classify_R_lin_hermitian_struct_mod_cat}. 
    Write $ \lambda \colon R^e \simeq R^e $ for the $ C_2 $-action on the underlying $ \EE_\infty $-ring associated to $ R $. 
    Then the fiber of the map \[\theta \colon \Brp(R)\to \mathfrak{br}(R^e)\] can be naturally identified with $ \mathfrak{pic} \left(\mathrm{Mod}_{R^L}\left(\mathrm{Sp}^{C_2}\right) \right)  $. 
    Moreover, the connecting map $ \Omega \mathfrak{br}(R^e) \simeq \mathfrak{pic}(R^e) \to \mathrm{fib}(\theta) $ is induced by the norm $ \Mod_{R^e}^\op \to \Mod_{R^L}\left(\Spectra^{C_2}\right) $, $ X \mapsto N^{C_2} (X) \otimes_{N^{C_2}R^e} R^L $, which on underlying spectra is given by $ X \mapsto X \otimes_{R^e} \lambda^* X $. 
\end{corollary}
\begin{proof}[Proof of Corollary \ref{cor:Poincare_Brauer_to_Brauer_fiber}]
    Since $\theta:\mathrm{Mod}_{(\Mod_{R^e}^\omega, \Qoppa_R)}(\Catpidem)\to \mathrm{Mod}_{\mathrm{Mod}_{R^e}^\omega}(\Catex)$ is symmetric monoidal and conservative, it induces a map $ \theta^\simeq \colon \pnbr(R) \to \mathfrak{br}(R^e) $ on the groupoid core of invertible objects. 
    Now observe that $ \theta $ is an isofibration; it follows that $ \theta^\simeq $ is a Kan fibration by \cite[\href{https://kerodon.net/tag/01EZ}{Proposition 01EZ}]{kerodon}. 
    Consequently, to identify the homotopy fiber of $ \theta $, it suffices to identify the fiber of $ \theta $ over a single point. 
    Consider $ \left(\Mod_{R^e}^\omega, \Qoppa\right) $ a point in the fiber of $ \theta $ over $ \Mod_{R^e}^\omega $. 
    By Proposition~\ref{prop:relative_poincare_cats_basic_properties}\ref{propitem:classify_R_lin_hermitian_struct_mod_cat}, $ \Qoppa $ is associated to an $ R $-linear invertible module with involution $ \left(M, N, N \to M^{\mathrm{t}C_2} \right) $. 
    By Proposition~\ref{prop:relative_poincare_cats_basic_properties}\ref{propitem:Rlin_Poincare_cats_tensor_mod_gen_inv}, invertibility of $ \left(\Mod_{R^e}^\omega, \Qoppa\right) $ implies that $ \left(M, N, N \to M^{\mathrm{t}C_2}\right) $ is invertible as a module over $ R^L $. 
    
    Now we give the description of the connecting map. 
    Write $ \left(\Mod_{R^e}^\omega, \Qoppa_R\right) $ for the identity element in the fiber of $ \theta $ over $ \Mod_{R^e}^\omega $, and let $ \gamma \colon S^1 \to \mathfrak{br}(R^e) $ be a loop based at the unit. 
    Write $ \mathcal{L}_\gamma $ for $ \gamma $ regarded as a point in $ \mathfrak{pic}(R^e) \simeq \Omega \mathfrak{br}(R^e) $. 
    Lift $ \gamma $ to a path $ \widetilde{\gamma} $ in $ \pnbr(R) $ starting at $ \left(\Mod_{R^e}^\omega, \Qoppa_R\right) $, and write $ \left(\Mod_{R^e}^\omega, \Phi\right) $ for the other endpoint of $ \widetilde{\gamma} $. 
    By Proposition~\ref{prop:relative_poincare_cats_basic_properties}\ref{propitem:classify_R_lin_hermitian_struct_mod_cat}, $ \Qoppa_R $ is associated to the invertible $ R^L $-module with involution $ R^L $ and $ \Phi $ is associated to some invertible $ R^L $-module with involution $ (M, N, N \to M^{\mathrm{t}C_2}) $. 
    We may regard $ \widetilde{\gamma} $ as an $ R $-linear hermitian equivalence from $ \left(\Mod_{R^e}^\omega, \Qoppa_R\right) $ to $ \left(\Mod_{R^e}^\omega, \Phi\right) $, which by Proposition \ref{prop:relative_poincare_cats_basic_properties}\ref{prop_item:R_linear_poincare_cats_maps} consists of an $ R $-linear hermitian functor $ (F,\eta \colon \Qoppa_R \to \Phi \circ F) $ so that $ F, \eta $ are both equivalences. 
    Since $ \widetilde{\gamma} $ projects to $ \gamma $, we must have that $ F = - \otimes \mathcal{L}_\gamma $. 
    Now the natural equivalence $ \eta $ classifies an equivalence $ R^L \simeq \hom_{R^L}(N^{C_2}_R(\mathcal{L}_\gamma), (M,N, N \to M^{\mathrm{t}C_2})) $ of $ R^L $-modules (Proposition \ref{prop:classify_R_linear_Poincare_functors}). Using the invertability of $\mathcal{L}_\gamma$ and chasing through the definition of the connecting map then gives the result. 
\end{proof}
		
\begin{example}\label{ex: pnbr of sphere}
    Let $\mathbb{S}^u$ denote the universal Poincar{\'e} structure on the sphere spectrum of Example~\ref{example:universal_poincare_ring_spectrum}. 
    By Corollary~\ref{cor:Poincare_Brauer_to_Brauer_fiber}, we have a fiber sequence \[\mathfrak{pic}\left(\mathrm{Sp}^{C_2}\right)\to \pnbr(\mathbb{S}^u)\to \mathfrak{br}(\mathbb{S})\] and by \cite[Corollary 7.17]{MR3190610} we have that $\mathrm{Br}(\mathbb{S})=0$. Therefore we get a long exact sequence \[
    \begin{tikzcd}[row sep=small]
        \ldots \arrow[r] & \pi_1(\pnbr(\mathbb{S}^u)) \arrow[r] \arrow[d, "{\rotatebox{90}{$\sim$}}"] & \pi_1(\mathrm{br}(\mathbb{S})) \arrow[r] \arrow[d, "{\rotatebox{90}{$\sim$}}"] & \Pic\left(\mathrm{Sp}^{C_2}\right) \arrow[r] \arrow[d, "{\rotatebox{90}{$\sim$}}"'] & \pi_0(\pnbr(\mathbb{S}^u)) \ar[r] & 0 \\
        & \mathbb{Z}/2      & \mathbb{Z}                                                  & \mathbb{Z}\times\mathbb{Z}                                  &                            &  
    \end{tikzcd}
    \] where the third isomorphism is given by $ X \mapsto (\dim X^e, \dim X^{\varphi C_2}) $ \cite[Section 8.1]{Krause_picard}. 
    The description of the connecting map in Corollary~\ref{cor:Poincare_Brauer_to_Brauer_fiber} implies that $\mathbb{Z}\to \mathbb{Z}\times \mathbb{Z}$ is the map $1=[\mathbb{S}^1]\mapsto [\mathbb{S}^\rho]=(2,1)$; it follows that $\pi_0(\pnbr(\mathbb{S}^u))\cong \mathbb{Z}$. 
    Consider the generator given by the coset of $ (1,1) $. 
    The map $\Pic\left(\mathrm{Sp}^{C_2}\right) \rightarrow \pi_0(\pnbr(\mathbb{S}^u))$ is induced by the map $\Qoppa_{(-)}$ of Theorem \ref{thm: genuine equivarient modules are hermitian structures}, which is symmetric monoidal and commutes with suspensions: $\Qoppa_{\Sigma(-)}\simeq \Sigma\Qoppa_{(-)}=\Qoppa_{(-)}^{[1]}$. By the above analysis, the elements of $\pi_0(\pnbr(\mathbb{S}^u))$ can be represented by suspensions of $\Qoppa_{\mathbb{S}}=\mathbb{S}^{u}$. 
    Thus, a generator of $\pi_0(\pnbr(\mathbb{S}^u))$ is given by $\left(\Spectra^\omega, \Qoppa^{u,[1]}\right)$.
\end{example}
\begin{example}\label{ex:pnbr_closed_point_ramified}
    Let $ k $ be an algebraically closed field, and regard $ k $ as a Poincar\'e ring $ \underline{k} $ via Example \ref{ex:fixpt_Mackey_functor} with the trivial involution.  
    Then 
    \begin{equation*}
        \pi_0\pnbr(\underline{k}) \simeq \begin{cases}
            \ZZ/4\ZZ & \text{ if }\mathrm{char}\, k \neq 2 \\
            \ZZ & \text{ if }\mathrm{char}\, k = 2 \,.
        \end{cases} 
    \end{equation*}
    To see this, note that by \cite[Proposition 1.9]{MR2957304}, $ \pi_0 \mathfrak{br}(k^e) \simeq 0 $. 
    Thus by Corollary~\ref{cor:Poincare_Brauer_to_Brauer_fiber}, it suffices to understand the fiber sequence \[\mathfrak{pic}(k)\to \mathfrak{pic}\left(\Mod_{\underline{k}}\left(\mathrm{Sp}^{C_2}\right)\right) \to \pnbr(\underline{k}). \] 
    
    Now if char $ k \neq 2 $, then $ \mathfrak{pic}(\underline{k}) \simeq \mathfrak{pic}(k)^{\mathrm{h}C_2} \simeq (\ZZ \times B k^\times)^{\mathrm{h}C_2} $ where the $ C_2 $-action is the base change map of \ref{C2action:basechange}; in particular, the generator of $ C_2 $ acts on $ \pi_1 B k^\times $ trivially. 
    We then have that $H^1(C_2,k^\times)=\mu_2(k)\cong \mathbb{Z}/2\mathbb{Z}$ since $k$ is algebraically closed. 
    We can then deduce that
    \begin{equation*}
        \pi_0 \pic(\Mod_{\underline{k}}(\mathrm{Sp}^{C_2})) \simeq \begin{cases}
            \ZZ \times \ZZ & \text{ if }\mathrm{char}\, k = 2 \\
            \ZZ\times \mu_2(k) & \text{ otherwise }
        \end{cases}
    \end{equation*}
    where the $\mathrm{char}(k)=2$ case follows from an argument similar to \cite[Section 8.1]{Krause_picard}. 
    Let us write $ \pm 1 $ for the elements of $ \mu_2 $. 
    It remains to compute the map $\pic(k)\to \pic(\mathrm{Mod}_{\underline{k}}(\mathrm{Sp}^{C_2}))$ on $ \pi_0 $. 
    This map is induced by the map $[\Sigma k]\mapsto [(N^{C_2}\Sigma k)\otimes_{N^{C_2}k}\underline{k}]\simeq [\Sigma^\rho \underline{k}]$. In the case when $\mathrm{char}\,k\neq 2$ this is the element $(2,-1)$, and the cokernel of this map is $\ZZ/4\ZZ$ generated by the class $(1,-1)$.
    On the other hand, if char $ k = 2 $, then the map $ \pi_0 \pic(k) \to \pi_0\pic(\underline{k}) $ is $ \ZZ \xrightarrow{n \mapsto (2n, n)} \ZZ \times \ZZ $. 
    
    Suppose that char $ k \neq 2 $.  
    Then the Poincaré $ \infty $-category corresponding to the element $ (0, -1) \in \mathbb{Z} \times \mu_2 $ is $ \left(\Mod_k^\omega, \Qoppa_{-k}\right) $, where the notation $ -k $ is that of the discussion between Definition 3.5.12 and Example 3.5.14 of \cite{CDHHLMNNSI} (in particular, see Example 3.5.14(i) of \emph{loc. cit.}).  
    In other words, a sign is introduced into the involution. 
    The presence of this generator (unlike the case when char $ k = 2 $) reflects the distinction between symmetric and skew symmetric forms when $ 2 \neq 0 $. 
    That this class is 2-torsion reflects that $ (-1)^2 = 1 $. 
    The Poincaré $ \infty $-category corresponding to the element $ (1, 1) \in \mathbb{Z} \times \mu_2 $ is $ \left(\Mod_k^\omega, \Qoppa_{\underline{k}}^{[1]}\right) $ (see Example \ref{ex:shifted_Poincare_structure}). 
    By a relative version of \cite[Corollary 5.4.6]{CDHHLMNNSI}, $ \left(\Mod_k^\omega, \Qoppa_{\underline{k}}^{[1]}\right) \otimes_{\mathrm{Mod}_{\underline{k}}^{\mathrm{p}}}  \left(\Mod_k^\omega, \Qoppa_{\underline{k}}^{[1]}\right) \simeq \left(\Mod_k^\omega, \Qoppa_{\underline{k}}^{[2]}\right) $, and the latter is equivalent to $ \left(\Mod_k^\omega, \Qoppa_{-k}\right) $ by \cite[Proposition 3.5.3]{CDHHLMNNSI}. 
\end{example}
In contrast to \cite[Theorem 1]{Parimala_Srinivas}, Example \ref{ex:pnbr_closed_point_ramified} shows that the Poincaré Brauer group is not 2-torsion. 
We explain the reason for this in the following remark. 
\begin{remark}\label{rmk:non_classical_PnBr_classes}
    The 4-torsion class $ \left(\Mod_k^\omega, \Qoppa^{[1]}\right)$ (char $ k \neq 2 $) in Example \ref{ex:pnbr_closed_point_ramified} does not arise as perfect modules over an ordinary Azumaya algebra with involution or central Wall anti-structure. 
    Roughly, the reason is: For any Poincaré object $ (P,q) \in \left(\Mod_k^\omega, \Qoppa^{[1]}\right) $, we have an equivalence $ q^\dagger \colon P \simeq P^\vee [1] $. 
    Therefore, the underlying $ k $-module $ P $ cannot be concentrated in a single degree.  
    It follows that the endomorphism algebra\footnote{We will discuss endomorphisms of Poincaré objects in \S\ref{subsection:generalized_Azumaya_alg_with_involution}.} of any Poincaré object whose underlying object generates $ \Mod_k^\omega $ will not be discrete either. 
    However, the 2-torsion class $ \left(\Mod_k^\omega, \Qoppa^{[1]}\right) \otimes_{\mathrm{Mod}_{\underline{k}}^{\mathrm{p}}}  \left(\Mod_k^\omega, \Qoppa^{[1]}\right)  =  \left(\Mod_k^\omega, \Qoppa_{-k}\right) $ does arise as perfect modules over a classical Azumaya algebra with involution. 
    Therefore, the subgroup of $ \pnbr(\underline{k})\simeq \ZZ/4\ZZ $ which is ``seen'' by classical Azumaya algebras with involution is still 2-torsion! 
    
    Let us now give an indication as to why the argument used to show that $ \left(\Mod_k^\omega, \Qoppa^{[1]}_{\underline{k}} \right) $ is not associated to a discrete Azumaya algebra with involution does not apply to $ \left(\Mod_k^\omega, \Qoppa^{[2]}_{\underline{k}} \right) $, despite the presence of the shift. 
    While a priori a Poincaré object $ (P,q) $ of $ \left(\Mod_k^\omega, \Qoppa^{[2]}_{\underline{k}} \right) $ must have a shifted duality pairing $ q^\dagger: P \simeq P^\vee [2] $, we note that in fact this induces an \emph{un}shifted duality pairing on a different generator: $ q^\dagger[-1] : P[-1] \simeq P^\vee[1] \simeq (P[-1])^\vee $. 
    More details are required to show that $ q^\dagger[-1] $ defines a(n antisymmetric) form on $ P[-1] $; we refer the reader to \cite[\S3.5]{CDHHLMNNSI} for a precise account. 
\end{remark}
\begin{example}\label{ex:pnbr_closed_pt_unramified}
    Let $ k $ be an algebraically closed field, and consider the Poincar\'e ring associated to $ \prod_{C_2} k $ (i.e. $ k \times k $ with the swap action) via Example \ref{ex:fixpt_Mackey_functor}. 
    Similarly to Example \ref{ex:pnbr_closed_point_ramified}, by \cite[Proposition 1.9]{MR2957304} we have $ \pi_0 \mathrm{br}(\Spec k \sqcup \Spec k) \simeq \pi_0 \mathrm{br}(\Spec k)^{\times 2} = 0 $. 
    Thus by Corollary \ref{cor:Poincare_Brauer_to_Brauer_fiber}, it suffices to understand the cokernel of the connecting homomorphism $ \pi_1 \mathrm{br}(k \times k) \to \pi_0 \Pic(\prod_{C_2} k) $. 
    Now since $ \prod_{C_2} k $ is Borel and $(\prod_{C_2}k)^{\mathrm{t}C_2}=0$, 
    \begin{equation*}
        \pic\left(\Mod_{\prod_{C_2} k}\left(\Spectra^{C_2}\right) \right) \simeq \pic \Mod_{k \times k}(\Spectra)^{\mathrm{h}C_2} \simeq \left(\prod_{C_2} (\ZZ \oplus  k^\times[1])\right)^{\mathrm{h}C_2} \simeq \ZZ \oplus k^\times [1] 
    \end{equation*}
    thus $ \pi_0 \pic\left(\Mod_{\prod_{C_2} k}\left(\Spectra^{C_2}\right) \right) \simeq \ZZ $. 
    On the other hand, $ \pi_1 \mathrm{br}(k \times k) \simeq \pi_0 \Pic(k \times k) \simeq \ZZ^{\times 2} $ and the connecting homomorphism is $ (n, m) \mapsto n+m $, whence $ \pi_0\pnbr\left(\prod_{C_2} k\right) = 0 $. 
\end{example}

We end this section by showing that the Poincar\'e Brauer space satisfies $C_2$-\'etale descent in the sense of Definition~\ref{defn:C2_etale_topology}.
\begin{corollary}\label{cor:pnbr_as_etale_sheaf_affine_spectral}
    Let $\mathcal{C}$ denote either the category of schemes with good quotients or the category of Poincar{\'e} rings. Write $ \pnbr $ for the composite functor $ \mathrm{\acute{E}t}_R \xrightarrow{\Mod^\mathrm{p}} \EE_\infty\Alg\left(\Catpidem_{, R}\right) \xrightarrow{\pnbr} \Spectra_{\geq 0} $, where $ \Mod^\mathrm{p} $ is from Notation \ref{notation:poincare_ring_basechange} or Notation~\ref{notation:scheme_involution_basechange} and $ \pnbr $ is Definition \ref{definition:poincare_brauer_space}. 
    Then $ \pnbr $ is a $C_2$-\'etale hypersheaf. 
\end{corollary} 
\begin{proof}
    The case of Poincar{\'e} rings is handled by the fiber sequence of Corollary~\ref{cor:Poincare_Brauer_to_Brauer_fiber}. Note further that if we know that $\pnbr(-)$ is a $C_2$-\'etale sheaf in the case of schemes with good quotients, then we may reduce the hypersheaf question to Poincar{\'e} rings which we already know. Thus we need only check that the Poincar{\'e} Brauer space is an \'etale sheaf. 
    
    Since the Picard space commutes with limits, it is enough to know that the functor \[\mathrm{Mod}_{\mathrm{Mod}^\mathrm{p}}:\mathrm{\acute{E}t}_Y^\op \to \mathrm{LinCat}^\mathrm{p} \] is a $C_2$-\'etale sheaf, where $\mathrm{LinCat}^\mathrm{p}$ is the Grothendieck construction on the functor sending a $C_2$-{\'e}tale extension $X\to X'$ to the category $\mathrm{Mod}_{\mathrm{Mod}_{(X',\Qoppa_{\underline{\mathcal{O}}_{X'}}})}(\Catpidem)$. This follows by the same argument as in \cite[Theorem 5.13]{Lurie2011Descent} (note that the argument is formal once one has that the functors of Notation~\ref{notation:poincare_ring_basechange} and Notation~\ref{notation:scheme_involution_basechange} are \'etale sheaves, which we have already handled.)
\end{proof} 

\begin{corollary}~\label{cor: two definitions of poincare brauer agree}
    Let $(X,\lambda ,Y, p)$ be a scheme with good quotient. Then there is a canonical identification of $\mathfrak{br}^\mathrm{p}(X,\lambda, Y,p)$ with \[\mathfrak{pic}(\mathrm{Mod}_{\mathrm{Mod}_{(\mathcal{O}_X,\Qoppa_{\underline{\mathcal{O}}})}}(\Catpidem))\,.\]
\end{corollary}
\begin{proof}
    Note that there is a natural map \[\mathfrak{pic}(\mathrm{Mod}_{\mathrm{Mod}^\omega_{(\mathcal{O}_X,\Qoppa_{\underline{\mathcal{O}}})}}(\Catpidem))\to \mathfrak{br}^\mathrm{p}(X,\lambda, Y,p)\] given by sending an invertible $(\mathrm{Mod}_{\mathcal{O}_X}^\omega,\Qoppa_{\underline{\mathcal{O}}})$-module to its descent data. Since both sides are \'etale sheaves we may then reduce to showing that this map is an equivalence for $X$ affine, but then this is exactly Proposition~\ref{prop:Poincare_modules_as_etale_sheaf}(1).
\end{proof}
		
\subsection{Azumaya algebras with genuine involution}\label{subsection:generalized_Azumaya_alg_with_involution} 
		
In this section, we introduce the notion of a derived Azumaya algebra with $ \lambda $-involution over a Poincaré ring $ R $. 
We show that such algebras give rise to invertible $ R $-linear Poincaré $ \infty $-categories (Proposition \ref{prop:mod_over_azumaya_geninv_is_invertible}). 
We show that an ordinary Azumaya algebra with $ \lambda $-involution defines a derived Azumaya algebra with $ \lambda $-involution in special cases; see Proposition \ref{prop:classical_inv_Azumaya_to_genuine_inv_Azumaya}. 
By bootstrapping the relevant results of \cites{MR2493386,MR3190610}[\S3.1]{CDHHLMNNSI}, we show that the Poincaré Brauer group can be regarded as Morita equivalence classes of Azumaya algebras with genuine involution in Corollary \ref{cor: Brauer classes represented by Azumaya algebras with genuine involution}.

A derived Azumaya algebra (in the sense of Antieau--Gepner and Toën) has an underlying $ \EE_1 $-algebra, where $ \EE_1 $ denotes a homotopy coherent analogue of associativity. 
To homotopy theorists who study $ C_2 $-actions, there is a `canonical' generalization of involutions on associative algebras to $ \EE_1 $-algebras. 
A derived Azumaya algebra with genuine involution is \emph{not} a derived Azumaya algebra (in the sense of Antieau--Gepner and Toën) with the canonical involution just mentioned -- it is in fact this plus a little more data. 
We recall these notions before introducing the central definition of this section. 

By default, the algebras we will be considering are all in spectra, so we will drop the word `derived' and if we have occasion to refer to an ordinary ring we will emphasize that by calling it `ordinary.' 
\begin{recollection}\label{rec:Azumaya_alg_wo_involution} 
    Let $ R $ be an $ \EE_\infty $-ring spectrum. Recall \cites{MR2927172,MR3190610} that an $ \EE_1 $-$ R $-algebra $ A $ is said to be \emph{Azumaya} if it is a compact generator of $ \Mod_R $ and if the natural $ R $-algebra map giving the bimodule structure on $ A $
    \begin{equation*}
        A \otimes_R A^{\mathrm{op}} \to \mathrm{End}_R(A)
    \end{equation*}
    is an equivalence of $ R $-algebras. 
\end{recollection}
\begin{remark}\label{rmk:azumaya_ness_is_invariant}
    Suppose $ R $ is an $ \EE_\infty $-ring spectrum with an involution $ \lambda \colon R \xrightarrow{\sim} R $. 
    Then an $ \EE_1 $-$ R $-algebra $ A $ is Azumaya in the sense of Recollection \ref{rec:Azumaya_alg_wo_involution} iff $ \lambda_* A $ is Azumaya iff $ \lambda_* A^\op $ is Azumaya iff $ A^\op $ is Azumaya. 
\end{remark}
The `canonical' notion of homotopy-coherent associative algebras with $ \lambda $-involution is recorded in Definition \ref{defn:E1_alg_with_lambda_involution}. 

Informally, an $ \EE_1 $-$ R $-algebra with $ \lambda $-involution is an $ \EE_1 $-$ R $-algebra $ A $ equipped with an equivalence $ \sigma \colon A \to \lambda_* A^\op $, a homotopy $ \lambda^*(\sigma^\op) \circ \sigma \simeq \id_A $, and higher coherences. 
\begin{remark}\label{rmk:alg_with_lambda_involution_as_sheaves_on_quotient}
    Suppose $ (X,\lambda) $ is a scheme with an involution and good quotient $ p \colon X \to Y $. 
    Observe that the functor $ \Mod_{\mathcal{O}_X} \xrightarrow{\sim} \Mod_{p_*\mathcal{O}_X} $ of Lemma \ref{lemma:identify_structure_sheaf_of_Green_func}\ref{lem_item:structured_pushforward_is_equiv} is $ C_2 $-equivariant with respect to the action by $ \lambda_* $ on the source and by $ p_*(\lambda) $ on the target. 
    By Lemma \ref{lemma:identify_structure_sheaf_of_Green_func}\ref{lem_item:structure_pushforward_equiv_is_monoidal}, there is an equivalence $ \left(\EE_1\Alg_{\mathcal{O}_X}\right)^{hC_2} \simeq  \left(\EE_1\Alg_{p_*\mathcal{O}_X}\right)^{hC_2} $. 
\end{remark}
\begin{recollection}\label{rec:E1_alg_involution_as_bimodule_over_itself}
    Let $ R $ be an object of $ \EE_\infty\Alg^{BC_2} $ and let $ A $ be an $ \EE_1 $-$ R $-algebra with $ \lambda $-involution. 
    Now $ A $ can be regarded as a left $ A \otimes_R \lambda^* A $-module via the involution $ \sigma $. 
    Moreover, this lifts to a $ C_2 $-equivariant module structure, where $ A \otimes_R \lambda^* A $ is given the $ C_2 $-action via swap and base change along $ \lambda $, and $\sigma$ endows the underlying spectrum of $ A $ with a $ C_2 $-action.  
    Informally, $ a \otimes b $ acts on $ \alpha $ via $ a \cdot \alpha \cdot \sigma(b) $. 
    Since the Tate construction is lax symmetric monoidal, $ A^{tC_2} $ inherits a module structure over $ (A \otimes_R \lambda^* A)^{tC_2} $. 
\end{recollection}
\begin{definition}~\label{defn:alt definition of Azumaya algebra with involution}
    Let $R$ be a Poincar{\'e} ring spectrum, and write $\lambda:R^e\to R^e$ for the underlying involution. Then an \emph{algebra with genuine involution} over $R$ is a pair $(A, P\to A^{\mathrm{t}C_2})$ such that $A$ is an $\mathbb{E}_1$-$R^e$-algebra with $ \lambda $-involution in the sense of Definition \ref{defn:E1_alg_with_lambda_involution} and $P\to A^{\mathrm{t}C_2}$ is an $A$-linear map of $A^\op \otimes_{R^e} R^{\phi C_2}$-modules, where $A^{\mathrm{t}C_2}$ is regarded as an $ A^\op $-module via the twisted Tate-valued diagonal $A\to (A\otimes \lambda^*A)^{\mathrm{t}C_2}$ and Recollection \ref{rec:E1_alg_involution_as_bimodule_over_itself}. 
    
    We say that such an algebra with genuine involution is an \emph{Azumaya algebra with genuine involution over $ R $} if its underlying $ \EE_1 $-$ R^e $ algebra is Azumaya\footnote{See Remark \ref{rmk:azumaya_ness_is_invariant}.} in the sense of Recollection \ref{rec:Azumaya_alg_wo_involution} and there exists 
    \begin{enumerate}[label=(\roman*)]
        \item a $C_2$-equivariant $(A\otimes_{R^e} \lambda^* A)^{\otimes_{R^e} 2}$-linear equivalence $\mathrm{hom}_{{R^e}\otimes {R^e}}(A\otimes_{R^e} A,R^e)\simeq A\otimes_{R^e} \lambda^* A^\op$
        \item a map of $(A^\op \otimes_{R^e} R^{\phi C_2})$-modules $\overline{P}\to A^{\mathrm{t}C_2}$ 
        \item an equivalence of $(A\otimes_{R^e} \lambda^*A)\otimes_{R^e} R^{\phi C_2}$-modules \[\mathrm{hom}_{{R^e}}(A,R^{\phi C_2})\simeq P\otimes_{R^{\phi C_2}}\overline{P}\]
        and a homotopy making the diagram 
        \begin{equation*}
            \begin{tikzcd}
                \hom_{R^e}(A, R^{\varphi C_2}) \ar[r] \ar[d] & P \otimes_{R^{\varphi C_2}} \overline{P} \ar[d] \\
                \hom_{R^e}(A, R^{\mathrm{t}C_2}) \simeq \hom_{R^e}(A \otimes A^\op, {R^e})^{\mathrm{t}C_2} \ar[r] & \hom_{(A \otimes A^\op)^{\otimes 2}}\left((A \otimes A^\op)^{\otimes 2}, A \otimes A^\op\right)^{\mathrm{t}C_2} \simeq (A \otimes A^\op)^{\mathrm{t}C_2}
            \end{tikzcd}    
        \end{equation*}
        commute.
    \end{enumerate}
\end{definition}
These algebras with genuine involution are not equivalent to the $ \EE_\sigma $-algebras of \S\ref{subsection:c2_operad_preliminaries}. 

While we will eventually show that all elements of the Poincar\'e Brauer groups are represented by Azumaya algebras with genuine involution, it will be helpful to introduce an intermediary definition.
\begin{definition}~\label{defn:Azumaya_genuine_central_Wall_anti-structure}
    Let $R$ be a Poincar{\'e} ring. Then an \textit{Azumaya algebra with central Wall anti-structure}\footnote{In \cite[Example 3.1.13]{CDHHLMNNSI}, the authors consider this structure in the context of invertible modules with genuine involution over the ring $ A $.} over $R$ is the data of 
    \begin{enumerate}[label=(\alph*)]
        \item \label{defnitem:Azumaya_alg_Wall_anti_the_involution} An $\mathbb{E}_1$-$R^e$-algebra with $ \lambda $-involution (Definition \ref{defn:E1_alg_with_lambda_involution}) so that the underlying $ \EE_1 $-$ {R^e} $-algebra $ A $ is an Azumaya $ {R^e} $-algebra in the sense of Recollection~\ref{rec:Azumaya_alg_wo_involution} 
        \item An invertible ${R^e}$ module with involution $M_A$ over $A$ in the sense of \cite[Definition 3.1.4]{CDHHLMNNSI} whose underlying left $A$-module is equivalent to $A$. 
        \item A map of left $A\otimes_{R^e} R^{\phi C_2}$-modules $P\to M_A^{\mathrm{t}C_2}$.    
    \end{enumerate}
    such that there exists
    \begin{enumerate}[label=(\roman*)]
        \item a $C_2$-equivariant $(A\otimes_{R^e} \lambda^* A)^{\otimes_{R^e} 2}$-linear equivalence $\mathrm{hom}_{{R^e}\otimes {R^e}}(A\otimes_{R^e} \lambda^* A, {R^e})\simeq M_A\otimes_{R^e}\lambda^*M_A$
        \item a left $\lambda^*A^\op \otimes_{R^e}R^{\phi C_2}$-linear map $\overline{P}\to (\lambda^* M_A)^{\mathrm{t}C_2}$
        \item an equivalence $\mathrm{hom}_{R^e}(M_A,R^{\phi C_2})\simeq P\otimes_{R^{\phi C_2}}\overline{P}$
        \item a homotopy making the resulting diagram
        \begin{equation*}
            \begin{tikzcd}
                \hom_{R^e}(A, R^{\varphi C_2}) \ar[r] \ar[d] & P \otimes_{R^{\varphi C_2}} \overline{P} \ar[d] \\
                \hom_{R^e}(A \otimes \lambda^*A, {R^e})^{\mathrm{t}C_2} \ar[r] & \hom_{(A \otimes \lambda^*A)^{\otimes 2}}\left((A \otimes \lambda^*A)^{\otimes 2}, M_A \otimes \lambda^*M_A\right)^{\mathrm{t}C_2} \simeq (M_A \otimes \lambda^*M_A)^{\mathrm{t}C_2}
            \end{tikzcd}    
        \end{equation*}
        commute.    
    \end{enumerate}
\end{definition}
\begin{observation}\label{obs:azumaya_geninv_base_change}
    Let $ A $ be an Azumaya algebra with genuine involution over $ R $, and suppose we are given a map $ R \to S $ of Poincaré rings. 
    Then $ (A \otimes_{R^e} S, P \otimes_{R^{\varphi C_2}} S^{\varphi C_2}) $ is an Azumaya algebra with genuine involution over $ S $. 
\end{observation} 
			\begin{remark}\label{rmk:Azumaya_gi_Wall_compare}
				An Azumaya algebra with genuine involution over $ R $ in particular defines an Azumaya algebra with central Wall anti-structure over $ R $ with $ M_A \simeq A $ as $ A \otimes \lambda^*A $-modules. 
				If an Azumaya algebra with central Wall anti-structure is of this form, we say that it \emph{comes from the involution on $ A $} (compare \cite[Example 3.1.9]{CDHHLMNNSI}, which is the special case when $ R = \sphere $). 
				
				If $ A $ is an Azumaya algebra with genuine involution over $ R $, then $ (A, P \to A^{\mathrm{t}C_2}) $ defines a $ N^{C_2}_{R^e} (A) $-module which classifies an $ R $-linear Poincaré structure on $ \Mod_A^\omega $ by Proposition \ref{prop:relative_poincare_cats_basic_properties}\ref{propitem:classify_R_lin_hermitian_struct_mod_cat}.    
			\end{remark}
			
\begin{remark}\label{rmk:azumaya_geninv_gives_module_geninv}
    If $ A $ is an Azumaya algebra with genuine involution over $ R $, then in particular $ M_A = A $, $ N_A = P $ is a module with genuine involution over $ A $ in the sense of \cite[Definition 3.2.3]{CDHHLMNNSI}. 
\end{remark}
\begin{definition}~\label{defn: Azumaya algebra with genuine anti involution scheme case}
    Let $ (X, \lambda) $ be a scheme with an involution and let $ p \colon X \to Y $ exhibit $ Y $ as a good quotient of $ X $. 
    Recall that there is a sheaf of $ C_2 $-$ \EE_\infty $-algebras $ \underline{\mathcal{O}} $ on $ Y $ (Construction \ref{cons:structure_sheaf_of_Green_functors}), and write $ \lambda $ for the involution $ p _* \mathcal{O}_X \xrightarrow{\sim} p _* \mathcal{O}_X $. 
    An \emph{Azumaya algebra with genuine involution} over $ X $ is the data of 
    \begin{enumerate}[label=(\alph*)]
        \item An $ \EE_1 $-$ \mathcal{O}_X $-algebra with $ \lambda $-involution in the sense of Definition \ref{defn:E1_alg_with_lambda_involution} so that, when regarded as a $ p_* \mathcal{O}_X \simeq \underline{\mathcal{O}}^e $ algebra via Remark \ref{rmk:alg_with_lambda_involution_as_sheaves_on_quotient}, the underlying $ \EE_1 $-$ \underline{\mathcal{O}}^e $-algebra $ A $ is a [generalized] Azumaya $ \underline{\mathcal{O}}^e $-algebra in the sense of \cite[Definition 2.11]{MR2957304}. 
        \item A left $ A \otimes_{\underline{\mathcal{O}}^e} \underline{\mathcal{O}}^{\varphi C_2} $-module map $ P\to A^{\mathrm{t}C_2} $ 
    \end{enumerate}
    such that there exists
    \begin{enumerate}[label=(\roman*)]
        \item an $ A^\op \otimes_{\underline{\mathcal{O}}^e} \underline{\mathcal{O}}^{\varphi C_2} $-module $ \overline{P} $
        \item \label{defn_item:global_Azumaya_gen_inv_conn_map} An $ A^\op \otimes_{\underline{\mathcal{O}}^e} {\underline{\mathcal{O}}}^{\varphi C_2} $-linear map $ \overline{P} \to A^{\mathrm{t}C_2} $.  
        \item \label{defnitem:global_Azumaya_alg_gi_underlying} an $ (A \otimes_{\underline{\mathcal{O}}^e} \sigma^* A)^{\otimes_{\underline{\mathcal{O}}^e} 2} $-linear equivalence $ \hom_{\underline{\mathcal{O}}^e}(A \otimes_{\underline{\mathcal{O}}^e} A, R) \simeq A \otimes_{\underline{\mathcal{O}}^e} \sigma^* A^\op $
        \item \label{defn_item:global_Azumaya_gen_inv_geom_fixpt} An equivalence of $ \left(A \otimes \sigma^* A^\op\right) \otimes_{\underline{\mathcal{O}}^e} \underline{\mathcal{O}}^{\varphi C_2} $-modules 
        \begin{equation*}
            \hom_{\underline{\mathcal{O}}^e}(A, \underline{\mathcal{O}}^{\varphi C_2}) \simeq P \otimes_{\underline{\mathcal{O}}^{\varphi C_2}} \overline{P}
        \end{equation*}
        and a homotopy making the diagram 
        \begin{equation*}
            \begin{tikzcd}
                \hom_{\underline{\mathcal{O}}^e}(A, \underline{\mathcal{O}}^{\varphi C_2}) \simeq P \otimes_{\underline{\mathcal{O}}^{\varphi C_2}} \overline{P} \ar[r] \ar[d] & P \otimes_{\underline{\mathcal{O}}^{\varphi C_2}} \overline{P} \ar[d] \\
                \hom_{\underline{\mathcal{O}}^e}(A, {\underline{\mathcal{O}}^e}^{\mathrm{t}C_2}) \simeq \hom_{\underline{\mathcal{O}}^e}(A \otimes A^\op, \underline{\mathcal{O}}^e)^{\mathrm{t}C_2} \ar[r] & \hom_{(A \otimes A^\op)^{\otimes 2}}\left((A \otimes A^\op)^{\otimes 2}, A \otimes A^\op\right)^{\mathrm{t}C_2} \simeq (A \otimes A^\op)^{\mathrm{t}C_2}
            \end{tikzcd}    
        \end{equation*}
        commute, where the lower horizontal arrow is induced by \ref{defnitem:global_Azumaya_alg_gi_underlying} and the right vertical arrow is induced by \ref{defn_item:global_Azumaya_gen_inv_conn_map}. 
    \end{enumerate}
\end{definition}
\begin{example}\label{ex:endomorphisms_of_poincare_object}
    Let $ R $ be a Poincaré ring, and let $ (P,q ) \in \mathrm{Pn}\left(\Mod_{R^e}^\omega, \Qoppa_R \right) $. 
    
    Then $ A:= \mathrm{End}_{R^e}(P) $ admits a canonical lift to an $ \EE_1 $ algebra with genuine involution over $ R^e $ with $ A^{\varphi C_2}:= \hom_{R^e}(P, R^{\varphi C_2}) $. 
    If $ P $ is a generator of $ \Mod_{R^e}^\omega $, then $ A $ is furthermore Azumaya. 
    
    By \cite[Proposition 3.1.16]{CDHHLMNNSI}, $ A $ inherits a canonical involution. 
    To exhibit that this is in fact an Azumaya algebra with genuine involution, observe that $ q^\dag $ induces a canonical $ A \otimes A^\op $-linear equivalence $ A = \mathrm{End}_{R^e}(P) \xrightarrow{D} \mathrm{End}_{\Mod_{R^e}^\mathrm{op}}\left(P^\vee\right) \simeq \mathrm{End}(P^\vee)^\op \xrightarrow{f \mapsto q^{-1} \circ f \circ q} \mathrm{End}_{R^e}\left(P\right)^\op \simeq A^\vee $.  
    If $ P $ is a generator, $ \hom_{R^e}(P,-) $ induces an equivalence $ \Mod_{R^e}^\omega \simeq \Mod_{A}^\omega $, thus we can regard $ \Mod_A^\omega $ as equipped with a Poincaré structure. 
    By the classification of $ R^e $-linear Poincaré structures of Proposition \ref{prop:relative_poincare_cats_basic_properties}\ref{propitem:classify_R_lin_hermitian_struct_mod_cat}, the Poincaré structure on $ \Mod_A^\omega $ is associated to an $ A $-module with genuine involution $ (M_A, N_A, N_A \to M_A^{\mathrm{t}C_2}) $. 
    We claim that $ M_A \simeq A $ with the canonical $ A $-$ A $-bimodule structure: By \cite[Proposition 3.1.6]{CDHHLMNNSI}, as an $ A^\op $-module $ M_A $ is the image of $ A $ under the composite
    \begin{equation*}
        \Mod_A^\omega \xrightarrow{\hom_{R^e}(P,-)^{-1}} \Mod_{R^e}^\omega \xrightarrow{D_R = \hom_{R^e}(-,R^e)} \Mod_{R^e}^{\omega, \op} \xrightarrow{\hom_{R^e}(P,-)} \Mod_A^{\omega,\op} \,.
    \end{equation*}
    Observe that the image of $ A $ in $ \Mod_{R^e}^{\omega, \op} $ is $ D_R(P) $ and $ q^\dag $ induces an equivalence $ D_R(P) \simeq P $, hence $ M_A \simeq A $ as $ A^\op $-modules. 
    
    A similar argument with the linear part of $ \Qoppa $ shows that we have an equivalence $ N_A \simeq \hom_{R^e}(P, R^{\varphi C_2}) $ of $ A $-modules and a commutative square 
    \begin{equation*}
        \begin{tikzcd}
            N_A \ar[r,"\sim"] \ar[d] &  \hom_{R^e}(P, R^{\varphi C_2}) \ar[d] \\
            A^{\mathrm{t}C_2} \ar[r,"\sim"] & \hom_{R^e}(P, R^{\mathrm{t}C_2}) \simeq \hom(P \otimes_{R^e} P, R^e)^{\mathrm{t}C_2}      
        \end{tikzcd}     
    \end{equation*} 
    of $ A $-modules, where $ A $ acts on $ A^{\mathrm{t}C_2} $ via the Tate-valued norm. 
\end{example}
\begin{observation}\label{obs:anti_involutions_shifted_pairing_Skolem_Noether}
    Let $ k $ be an algebraically closed field, and regard $ \underline{k} $ as a Poincar\'e ring spectrum with the trivial involution via Example \ref{ex:fixpt_Mackey_functor}. 
    Let $ (A, A^{\varphi C_2}, A^{\varphi C_2} \to A^{\mathrm{t}C_2}) $ be an Azumaya algebra with genuine involution over $ \underline{k} $. 
    By Proposition \ref{prop:mod_over_azumaya_geninv_is_invertible}, $ \left(\Mod_A^\omega, \Qoppa_A\right) \in \pnbr(\underline{k}) $. 
    By \cite[Corollary 1.15]{MR2957304}, $ A $ is equivalent to $ \mathrm{End}_k(P) $ for some compact $ k $-module $ P $. 
    Observe that there is a canonical identification $ A^\op \simeq \mathrm{End}_k(P^\vee) $. 
    By the derived Skolem--Noether theorem \cite[Theorem 5.1.5]{MR2579390}, there exists a unique $ n \in \ZZ $ so that the involution $ A \simeq A^\op $ is induced by an equivalence $ P \simeq P^\vee [n] $ (which is unique up to multiplication by a unit in $ k $). 
\end{observation}
			
\begin{proposition}\label{prop:mod_over_azumaya_geninv_is_invertible}
    Let $ R $ be a Poincaré ring, and let $ (A, P \to A^{\mathrm{t}C_2}) $ be an Azumaya algebra with genuine involution over $ R $. 
    Then 
    \begin{enumerate}[label=(\arabic*)]
        \item $ \left(\Mod_A^\omega, \Qoppa_A \right) $ defines an $ R $-linear Poincaré $ \infty $-category. 
        \item $ \left(\Mod_A^\omega, \Qoppa_A \right) $ is an invertible object in $ \Mod_{\left(\Mod_R^\omega, \Qoppa_R\right)}\left(\Catpidem\right) $. 
    \end{enumerate}
\end{proposition}
\begin{proof}[Proof of Proposition \ref{prop:mod_over_azumaya_geninv_is_invertible}]
    The first statement follows from Proposition \ref{prop:relative_poincare_cats_basic_properties}\ref{propitem:classify_R_lin_hermitian_struct_mod_cat}; we prove the second statement. 
    First, by \cite[Example 3.2.9]{CDHHLMNNSI}, we see that $ \left(\Mod_A^\omega, \Qoppa_A \right) $ is indeed an $ R$-linear Poincaré $ \infty $-category (and not merely hermitian). 
    To show that the associated Poincaré $ \infty $-category is invertible, we must identify a dual $ \left(\Mod_A^\omega, \Qoppa_A \right)^\vee $ and exhibit an equivalence $ \left(\Mod_A^\omega, \Qoppa_A \right) \otimes_{(\mathrm{Mod}_{R^e}^\omega, \Qoppa_R)} \left(\Mod_A^\omega, \Qoppa_A \right)^\vee \simeq \left(\Mod_{R^e}^\omega, \Qoppa_R \right) $. 
    Since $ \Catpidem_{, R} \to \Catex_R $ is symmetric monoidal, we see that the underlying $ R $-linear $ \infty $-category associated to the dual must be $ \Mod_{A^\op}^\omega $. 
    Moreover, the canonical evaluation map $ \mathrm{ev} \colon \Mod_A^\omega \otimes_{\mathrm{Mod}_{R^e}^\omega} \Mod_{A^\op}^\omega \xrightarrow{\simeq} \Mod_{R^e}^\omega $ sends $ A \otimes_{R^e} A^\op $ to $ A $. 
    Endow $ \Mod_{A^\op}^\omega $ with a Poincaré structure corresponding to the module with genuine involution $ M_{A^\op}:= A^\op $, $ N_{A^\op} := \overline{P} $. 
    It remains to exhibit a natural equivalence 
    \begin{equation}\label{eq:invertible_quadratic_compatibility}
        \eta \colon \left(\Qoppa_A \otimes \Qoppa_{A^\op}\right) \xrightarrow{\simeq} \mathrm{ev}^* \Qoppa_R 
    \end{equation} 
    of quadratic functors $ \left(\Mod_A^\omega \otimes_{R^e} \Mod_{A^\op}^\omega\right)^\op \to \Spectra $. 
    By \cite[Theorem 3.3.1]{CDHHLMNNSI}, it suffices to exhibit equivalences on the bilinear and linear parts of (\ref{eq:invertible_quadratic_compatibility}) which glue compatibly. 
    By Proposition \ref{prop:relative_poincare_cats_basic_properties}\ref{propitem:Rlin_Poincare_cats_tensor_mod_gen_inv}, on linear parts, it suffices to exhibit an $ A \otimes_{R^e} A^\op $-linear equivalence 
    \begin{equation*}
        \hom_{R^e}(A, R^{\varphi C_2}) \simeq N_A \otimes_{R^{\varphi C_2}} N_{A^\op}
    \end{equation*}
    and on bilinear parts, it suffices to exhibit a $ C_2 $-equivariant $ (A \otimes_{R^e} A^\op)^{\otimes_{R^e} 2} $-linear equivalence 
    \begin{equation*}
        \hom_{R^e \otimes R^e}(A \otimes_{R^e} A, R^e) \simeq M_A \otimes_{R^e} M_{A^\op} 
    \end{equation*}
    which glue compatibly. 
    This follows from the definitions, concluding the proof. 
\end{proof}
The next proposition establishes a relationship between classical Azumaya algebras with $ \lambda $-involution and the Azumaya algebras with genuine involution of this paper. 
The proposition hinges on the observation that the additional data beyond an Azumaya algebra with $ \lambda $-involution (Definition \ref{defn:E1_alg_with_lambda_involution}) required to specify an Azumaya algebra with genuine involution becomes unnecessary when either the $ C_2 $-action is Galois or 2 is invertible. 

\begin{proposition}\label{prop:classical_inv_Azumaya_to_genuine_inv_Azumaya}
\begin{enumerate}[label=(\arabic*)]
    \item \label{propitem:classical_inv_Azumaya_to_gen_inv_Azumaya_affine} Let $ R $ be an ordinary ring with a given $ C_2 $-action $ \lambda $, regarded as a Poincaré ring spectrum $ \underline{R}^\lambda $ via Example \ref{ex:fixpt_Mackey_functor}. 
    Let $ A $ be an ordinary Azumaya algebra over $ R $ with a $ \lambda $-involution, i.e. an equivalence of associative $ R $-algebras $ \sigma \colon A \to \lambda^* A^{\op} $ satisfying $ \lambda^*(\sigma) \circ \sigma = \id_A $. 
    Suppose that either:
    \begin{itemize}
        \item the branch locus (Notation \ref{ntn:good_quotient_branch_locus}) in $ \Spec (R)/C_2 = \Spec(R^{C_2}) $ is empty, i.e. $ R^{C_2} \to R^e $ is quadratic étale, and that $R^{C_2}$ is regular Noetherian of finite Krull dimension, or 
        \item $ 2 $ is invertible in $ R $.
    \end{itemize} 
    Then there is a canonical Azumaya algebra with genuine involution over $ \underline{R}^{\lambda} $ whose underlying Azumaya algebra is $ A $. 
    \item \label{propitem:classical_inv_Azumaya_to_gen_inv_Azumaya_scheme} Let $ (X,\lambda, Y, p) $ be a scheme with involution and good quotient in the sense of Definition \ref{defn:Category of good quotients}. 
    Let $ A $ be a classical Azumaya algebra with $ \lambda $-involution on $ X $ in the sense of \cite[Definition 4.10]{azumaya_involution}. 
    Suppose that $ Y $ is regular Noetherian of finite Krull dimension and either:
    \begin{itemize}
        \item the branch locus (Notation \ref{ntn:good_quotient_branch_locus}) in $ Y $ is empty, or 
        \item $ 2 $ is invertible in $ \mathcal{O}_Y $. 
    \end{itemize} 
    Then there is a canonical Azumaya algebra with genuine involution over $ X $ in the sense of Definition \ref{defn: Azumaya algebra with genuine anti involution scheme case} whose underlying Azumaya algebra is $ A $. 
\end{enumerate}
\end{proposition} 
\begin{proof}
    In either case, by Proposition \ref{prop:geomfixpt_C2_structure_sheaf} we have that $R^{\mathrm{t}C_2}\simeq R^{\phi C_2}\simeq 0$ and $ \underline{\mathcal{O}}^{\phi C_2} \simeq 0 $. 
    Therefore, the data of a genuine Azumaya algebra with involution reduces to defining $\mathbb{E}_1$-$R^e$-algebra with $ \lambda $-involution, which is supplied by the data of a classical Azumaya algebra with $\lambda $-involution. 
\end{proof}
Using the previous proposition, we can compare our involutive Brauer group to existing notions of the involutive Brauer group (see Theorem~\ref{thm: PS comparison in text}). 
\begin{definition}\label{defn:classical_inv_Azumaya_to_Poincare_Brauer}
    Let $ (X,\lambda, Y, p) $ be a scheme with involution and good quotient and suppose that $ (X,\lambda,Y,p) $ satisfies the assumptions of Proposition \ref{prop:classical_inv_Azumaya_to_genuine_inv_Azumaya}\ref{propitem:classical_inv_Azumaya_to_gen_inv_Azumaya_scheme}. 
    Let $ (A,\sigma) $ be a classical Azumaya algebra with $ \lambda $-involution on $ X $ in the sense of \cite[Definition 4.10]{azumaya_involution}. 
    Using Propositions~\ref{prop:mod_over_azumaya_geninv_is_invertible} and~\ref{prop:classical_inv_Azumaya_to_genuine_inv_Azumaya}, we can associate to $ (A,\sigma) $ a class in $ \pnbr(X,\lambda,Y,p) $.  
\end{definition}
\begin{proposition}\label{prop:from_involutive_Brauer_to_Poincare_Brauer}
    Let $ (X,\lambda, Y, p) $ be a scheme with involution and good quotient and suppose that $ (X,\lambda,Y,p) $ satisfies the assumptions of Proposition \ref{prop:classical_inv_Azumaya_to_genuine_inv_Azumaya}\ref{propitem:classical_inv_Azumaya_to_gen_inv_Azumaya_scheme}. 
    \begin{enumerate}[label=(\arabic*)]
        \item \label{propitem:pnbr_class_of_Azumaya_descends_typeI} Suppose $ \lambda = \id_X $. 
        Then the assignment of Definition \ref{defn:classical_inv_Azumaya_to_Poincare_Brauer} descends to a group homomorphism \[\mathrm{Br}^*(X)\to \pi_0(\mathfrak{br}^\mathrm{p}(X,\mathrm{id}_X))\,,\] where the former is the involutive Brauer group of Parimala--Srinivas of Construction~\ref{const: involutive BG type I}. 
        \item \label{propitem:pnbr_class_of_Azumaya_descends_typeII} Suppose $p:X\to Y$ is a degree $2$ \'etale cover. 
        Then the assignment of Definition \ref{defn:classical_inv_Azumaya_to_Poincare_Brauer} descends to a group homomorphism \[\mathrm{Br}(X,\lambda)\to \pi_0(\mathfrak{br}^\mathrm{p}(X,\lambda, Y, p))\,,\] where the former is the involutive Brauer group of Parimala--Srinivas of Construction~\ref{const: involutive BG of type II}. 
    \end{enumerate}
\end{proposition}
Proposition \ref{prop:from_involutive_Brauer_to_Poincare_Brauer} applies to rings with involution by Observation \ref{obs:fixpt_Mackey_functor_as_affine_C2_scheme}; we omit the statement to avoid repetition.
\begin{proof} [Proof of Proposition \ref{prop:from_involutive_Brauer_to_Poincare_Brauer}]
    We prove \ref{propitem:pnbr_class_of_Azumaya_descends_typeI}; \ref{propitem:pnbr_class_of_Azumaya_descends_typeII} is proven in exactly the same way. 
    Suppose $ (A,\sigma) $, $ (B,\tau) $ are two classical Azumaya algebras with ($ \lambda $-)involution on $ X $. 
    Let us write $ j(A,\sigma) $, $ j(B,\tau) $ for their respective images in $ \pnbr(X,\lambda,Y,p) $ under the assignment of Definition \ref{defn:classical_inv_Azumaya_to_Poincare_Brauer}. 
    Since $ A $ and $ B $ are locally free as $ \mathcal{O}_X $-modules, their derived tensor product agrees with their classical tensor product. 
    By \'etale descent (\S\ref{subsection: etale descent results}) and Proposition \ref{prop:relative_poincare_cats_basic_properties}\ref{propitem:Rlin_Poincare_cats_tensor_mod_gen_inv}, there is an equivalence $ j\left(A \otimes_{\mathcal{O}_X} B, \sigma \otimes \tau\right) \simeq j(A,\sigma) \otimes j(B,\tau) $ in $ \pnbr(X,\lambda,Y,p) $, where the latter tensor product is taken in $ \Mod_{\Mod^p_{\mathcal{O}_X}}\left(\Catpidem\right) $. 
    
    In order to show that there is a well-defined map which defines a group homomorphism, by the argument in the previous paragraph, it suffices to show that classical Azumaya algebras with $ \lambda $-involution which are Morita trivial are sent to the class of the identity in $ \pi_0 \pnbr(X,\lambda, Y,p) $. 
    Recall that a Morita trivial classical Azumaya algebra with involution is of the form $ \left(\mathrm{End}_{\mathcal{O}_X}(E), \alpha_B\right) $, where $ E $ is a locally free $ \mathcal{O}_X $-module of finite rank, $ \mathrm{End}_{\mathcal{O}_X} $ denotes sheaf hom, and $ \alpha_B $ is the involution on $ \mathrm{End}_{\mathcal{O}_X}(E) $ associated to a symmetric bilinear $ \mathcal{O}_X $-valued form $ B $ on $ E $. 
    We will show that $ (E, B) $ defines a Poincaré object of $ j\left(\mathrm{End}_{\mathcal{O}_X}(E), \alpha_B\right) $ which induces an equivalence $ \Mod^p_X \simeq j\left(\mathrm{End}_{\mathcal{O}_X}(E), \alpha_B\right) $. 

    Recall that our assumptions imply that all $ \underline{\mathcal{O}} $-modules can be taken in Borel $ C_2 $-spectra (see Construction \ref{cons:structure_sheaf_of_Green_functors} and Proposition \ref{prop:classical_inv_Azumaya_to_genuine_inv_Azumaya}). 
    It follows that a lift of $ E $ to a Poincaré object of $ j\left(\mathrm{End}_{\mathcal{O}_X}(E), \alpha_B\right) $ is equivalent to exhibiting a $ C_2 $-equivariant $ \mathrm{End}_{\mathcal{O}_X}(E) $-bilinear map $ E \otimes_{\mathcal{O}_X} E \to \mathrm{End}_{\mathcal{O}_X}(E) $, where $ \mathrm{End}_{\mathcal{O}_X}(E) $ is regarded as a module with involution over itself via $ \alpha_B $. 
    Such a map exists by definition of the involution $ \alpha_B $ on $ \mathrm{End}_{\mathcal{O}_X}(E) $. 
\end{proof}
\begin{recollection}
    [{\cite{MR1803361}}] \label{rec:Frohlich_Wall_equivariant_Brauer}
    Let $ R $ be a discrete commutative ring with an involution $ \lambda $. 
    In the notation of \cite{MR1803361}, take $ \Gamma = C_2 $ and $ w $ to be an isomorphism. 
    In \S~3 \emph{ibid.}, Fröhlich--Wall suggest a definition of an \emph{equivariant Brauer category} $ \mathcal{B}(R,C_2) $ as follows. 
    \begin{enumerate}[label=(\alph*)]
        \item     An object of $ \mathcal{B}(R,C_2) $ is given by a pair $ (A, \sigma) $ where $ A $ is an ordinary Azumaya $ R $-algebra and $ \sigma \colon A \to A^\op $ is $ \lambda $-linear and satisfies $ \sigma^\op \circ \sigma = \id_A $. 
        \item \label{rec_item:equivariant_Brauer_morphism} A morphism in $ \mathcal{B}(R,C_2) $ from $ (A,\sigma) $ to $ (B,\tau) $ is given by a quintuplet $ (M,M',\phi,h,h') $ where $ M $ is an invertible $ A $-$ B $-bimodule, $ M' $ is an invertible $ B $-$ A $-bimodule, and $ h \colon M\to M' $ (resp. $ h'\colon M'\to M $) is a $ \sigma \otimes \tau $-(resp. $ \tau \otimes \sigma $-)linear equivalence satisfying $ h' \circ h = \id_M $ (resp. $ h\circ h'= \id_{M'} $) (using the canonical identification of $ A $-$ B $-bimodules with $ B^\op $-$ A^\op $-bimodules). 
        Finally, $ \phi \colon M \otimes_B M'\to A $ is a perfect pairing which is $ A $-$ A $-bilinear and $ C_2 $-equivariant (as a morphism of abelian groups), where we endow the source with the action by $ h \otimes h' $ and the target with the action by $ \sigma $.  
        \item If $ (M,M', \phi, h, h') $ is a morphism from $ (A,\sigma) $ to $ (B,\tau) $ and $ (N,N', \psi,j, j') $ is a morphism from $ (B,\tau) $ to $ (C,\pi) $, their composite is given by $ \left(M \otimes_B N , N' \otimes_B M', \phi \circ (\mathrm{id}_M \otimes \psi \otimes \mathrm{id}_{M'}) , h \otimes j, j' \otimes h'\right) $. 
    \end{enumerate}
    Proposition 3.1 \emph{ibid} shows that $ \mathcal{B}(R,C_2) $ is a strictly coherent group-like category; in particular, tensor product $ \otimes_R $ induces a symmetric monoidal structure on $ \mathcal{B}(R,C_2) $ for which each object has a $ \otimes $-inverse. 
    In the language of the current work, $ \mathcal{B}(R,C_2) $ is (equivalent to) a grouplike commutative monoid in (1-)groupoids. 
    The \emph{equivariant Brauer group} $ B(R,C_2) $ is defined to be isomorphism classes of objects in $ \mathcal{B}(R,C_2) $, equipped with the group structure induced by $ \otimes_R $. 
    Furthermore, in \emph{loc. cit.} Fröhlich--Wall identify the unit group $ \mathrm{Aut}_{\mathcal{B}(R,C_2)}(e) \simeq \pi_1 \mathcal{B}(R,C_2) $ with $ C(R,C_2) $, their notation for the \emph{equivariant class group} (see p.63 \emph{ibid.}). 
    Unraveling definitions, we see that $ C(R,C_2) \simeq \Pic^{\mathrm{h}}(R,\lambda) $, where $ \Pic^\mathrm{h} $ is the hermitian Picard group of Definition \ref{definition:hermitian_picard_group}. 
\end{recollection}
\begin{observation}\label{obs:equivariant_Brauer_morphism_simplifies}
    In Recollection \ref{rec:Frohlich_Wall_equivariant_Brauer}, observe that we may simplify \ref{rec_item:equivariant_Brauer_morphism} as follows: Using the inverse equivalences $ h, h' $, we may replace $ M' $ by $ (\sigma \otimes \tau)_*M $. 
    Then we may rewrite $ \phi $ as a perfect $A$-$A$-bilinear pairing $ \phi : M \otimes_B \left((\sigma \otimes \tau)_* M\right) \to A $ which is $ C_2 $-equivariant with respect to the action on the source given by composing $ M \otimes_B \left((\sigma \otimes \tau)_* M\right) \simeq \left((\sigma \otimes \tau)_* M\right) \otimes_B M$ with base change along $ \sigma $ and $ \tau $ and the action on the target by $ \sigma $. 
    Equivalently, the adjoint $ \phi^\dag \colon M \to \hom_A\left((\sigma \otimes \tau)_*M, A\right) $ is an $ A $-$ B $-bilinear equivalence so that the composite $ \hom_A\left(((\sigma \otimes \tau)_*\phi^\dag)^{-1}, A\right) \circ \phi^\dag $ is homotopic to the canonical equivalence. 
\end{observation}
\begin{notation}
    Let $ R $ be a Poincaré ring. 
    Write $ \tau_{\leq 1} \Brp(R) $ for the 1-truncation of $ \Brp(R) $; in other words, $ \tau_{\leq 1} \Brp(R) $ is obtained by taking the sub (1-)groupoid consisting of invertible objects and equivalences in the \emph{homotopy category} of $ R $-linear Poincaré $ \infty $-categories. 
    Observe that $ \tau_{\leq 1} \Brp(R) $ is a commutative monoid in 1-truncated groupoids. 
\end{notation}
\begin{variant}\label{var:symmetric_pnbr}
    Let $ R $ be a Poincaré ring spectrum, and recall the $ \infty $-category of perfect symmetric $ \infty $-categories $ \Cat^{\mathrm{ps}}_\infty $ of Recollection \ref{rec:symmetric_Poincare_cat}. 
    Define the \emph{symmetric Poincaré Brauer space of $ R $} to be the connective spectrum $ \mathfrak{br}^{\mathrm{ps}}(R) := \pic\left(\Mod_{\left(\Mod^\omega_{R^e},B_R\right)}\left(\Cat^{\mathrm{ps}}_\infty\right) \right) $, where $ \Mod^p $ is the functor of Theorem \ref{thm:calgp_to_poincare_cat}\ref{thmitem:calgp_to_poincare_cat_with_tensor} and $ \left(\Mod^\omega_{R^e},B_R\right) $ is the image of $ \Mod^p_R $ in $ \Einfty\Mon\left(\Cat^{\mathrm{ps}}_\infty\right) $. 
    (In fact, the assignment $ R \mapsto \mathfrak{br}^{\mathrm{ps}}(R) $ factors through the forgetful functor $ \CAlgp \to \EE_\infty\Alg^{BC_2} $.) 
    There is a canonical map $ \Brp(R) \to \mathfrak{br}^{\mathrm{ps}}(R) $ which arises as a component of a natural transformation of functors $ \CAlgp \to \Spectra_{\geq 0} $.  
\end{variant}

It is unclear when the map $ \Brp(R^s) \to \mathfrak{br}^{\mathrm{ps}}(R) $ is an equivalence. We know that this is true, for example, when $R^{\phi C_2}\simeq 0$, but we do not know if this map is an equivalence outside of this case.

\begin{proposition}\label{prop:Frohlich_Wall_eqvt_Brauer_comparison}
    Let $ R $ be a discrete ring with an involution, and regard $ R^s $ as a Poincaré ring via Example \ref{example:symmetric_poincare_structure}. 
    Then there is a canonical map $ \psi \colon \mathcal{B}(R,C_2) \to \tau_{\leq 1} \mathfrak{br}^{\mathrm{ps}}(R^s) $ of commutative monoids in 1-groupoids, where $ \mathcal{B}(R,C_2) $ is Fröhlich--Wall's equivariant Brauer category of Recollection \ref{rec:Frohlich_Wall_equivariant_Brauer} and the latter is defined in Variant \ref{var:symmetric_pnbr}. 
    Under the identification $ C(R,C_2) \simeq \Pic^\mathrm{h}(R,\lambda) $, the map $ \pi_1 \psi $ induces the canonical inclusion (see Theorem \ref{theorem:fausk_for_discrete_rings}). 
\end{proposition}
\begin{proof}
    Let $ (A,\sigma) $ be an object of $ \mathcal{B}(R,C_2) $. 
    Observe that $ \pic\left(\Cat^{\mathrm{ps}}_\infty\right) \simeq \pic\left(\Catex^{,\mathrm{h}C_2}\right)\simeq \pic\left(\Catex\right)^{\mathrm{h}C_2} $ and likewise for invertible $ R $-linear symmetric Poincaré $ \infty $-categories. 
    It therefore suffices to show that the involution $ \sigma $ induces a functor $ \mathrm{Mod}_A^\omega \xrightarrow{\sim} \mathrm{Mod}_{A^\op}^\omega \xleftarrow[ \hom_A(M,A) \mapsfrom M]{\sim} \mathrm{Mod}_A^{\omega, \op} $ exhibiting $ \mathrm{Mod}_A^\omega $ as a $ C_2 $-homotopy fixed point of the action $ \mathcal{C} \mapsto \sigma_*\mathcal{C}^\op $ on $ \pic\left(\Catex_{,R}\right) $. 
    This follows from modifying the constructions of \S\ref{appendix:calgp_to_catp} by replacing $ \EE_\sigma $-algebras with `naïve' $ \EE_\sigma $-algebras and $ \Fun^q $ by $ \Fun^{s} $ in Construction \ref{cons:precompose_with_norm_in_families}. 
    Alternatively, we may apply (a relative variant of) \cite[Example 3.1.9]{CDHHLMNNSI}. 

    Next, we show that the assignment of the previous paragraph lifts to a functor of $ 1 $-groupoids $ \mathcal{B}(R,C_2) \to \mathrm{ho}\, \pic\left(\Cat^{\mathrm{ps}}_{\infty,R}\right) $. 
    Regarding $ R $ as a symmetric Poincaré ring, let us restrict Corollary \ref{cor:poincare_fourier_mukai} to those Poincaré functors whose underlying functor is an $ R $-linear equivalence. 
    It follows that the underlying $ \infty $-category comprises invertible $ A $-$ B $-bimodules $ P $. 
    The data of an $ R $-linear symmetric functor covering the functor $ P \otimes_B - \colon \mathrm{Mod}_B^\omega \to \mathrm{Mod}_A^\omega $ can be encoded by a map $ (P \otimes P) \otimes_{B\otimes B} B \to A $ in $ \Mod_{A \otimes_R A}^{\mathrm{h}C_2} $. 
    This functor is duality-preserving iff the adjoint $ P \to \hom_A(P, A) $ is an $ A $-$ B $-bilinear equivalence. 
    The existence of such a functor follows from Observation \ref{obs:equivariant_Brauer_morphism_simplifies}. 
    That the functor respects the symmetric monoidal structure follows from combining the previous constructions with an argument similar to the proof of Proposition \ref{prop:from_involutive_Brauer_to_Poincare_Brauer}. 
\end{proof}
\begin{observation}
    Let $ k $ be an algebraically closed field, and regard $ \underline{k} $ as a Poincar\'e ring spectrum with the trivial involution via Example \ref{ex:fixpt_Mackey_functor}. 
    Let $ (A, A^{\varphi C_2}, A^{\varphi C_2} \to A^{\mathrm{t}C_2}) $ be an Azumaya algebra with genuine involution over $ \underline{k} $. 
    By Proposition \ref{prop:mod_over_azumaya_geninv_is_invertible}, $ \left(\Mod_A^\omega, \Qoppa_A\right) \in \pnbr(\underline{k}) $. 
    By \cite[Corollary 1.15]{MR2957304}, $ A $ is equivalent to $ \mathrm{End}_k(P) $ for some compact $ k $-module $ P $. 
    Observe that there is a canonical identification $ A^\op \simeq \mathrm{End}_k(P^\vee) $. 
    By the derived Skolem--Noether theorem \cite[Theorem 5.1.5]{MR2579390}, there exists a unique $ n \in \ZZ $ so that the involution $ A \simeq A^\op $ is induced by an equivalence $ P \simeq P^\vee [n] $ (which is unique up to multiplication by a unit in $ k $). 
\end{observation}
\begin{proposition}\label{prop:Poincare_cat_as_module_cat} 
    Let $ R $ be a Poincaré ring. 
    Let $ \left(\mathcal{C}, \Qoppa\right) $ be an invertible idempotent-complete $ R $-linear Poincaré $ \infty $-category. 
    Suppose that we are given a Poincaré object $ (P, q) $ of $ \left(\mathcal{C}, \Qoppa\right) $ so that $ P $ is a generator for $ \mathcal{C} $, and write $ q^\dagger \colon P \xrightarrow{\sim} D_{\Qoppa_{\mathcal{C}}}(P) $ for the adjoint to the bilinear part of $ q $. Then
    \begin{enumerate}[label=(\arabic*)]
        \item \label{propitem:endomorphisms_Poincare_object_involution} $ \mathrm{End}_{\mathcal{C}}(P) $ inherits the structure of a $\EE_1$-$R^e$-algebra with $ \lambda $-involution in the sense of Definition \ref{defn:E1_alg_with_lambda_involution}. 
        Informally, the involution on $ A:= \mathrm{End}_{\mathcal{C}}(P) $ is of the form $ a \mapsto (q^\dagger)^{-1} \circ D_{\Qoppa_{\mathcal{C}}}(a) \circ q^{\dagger} $. 
        \item \label{propitem:endomorphisms_Poincare_object_central_Wall_anti-struc} There exists a lift of $ A:= \mathrm{End}_{\mathcal{C}}(P) $ to a genuine Azumaya algebra with central Wall anti-structure over $ R $ in the sense of Definition \ref{defn:Azumaya_genuine_central_Wall_anti-structure}. 
        Informally, the $ R $-linear module with genuine involution over $ A $ is given by $ M_A := \hom_{\mathcal{C}}(P, P) $, the $ A \otimes \lambda^*A $-action on $ M_A $ is given by $ (f \otimes g) \otimes a \mapsto f \circ a \circ (q^\dagger)^{-1} \circ D_{\Qoppa_{\mathcal{C}}}(g) \circ q^{\dagger} $, and the involution on $ M_A $ is given by $ a \mapsto \mathrm{ev}_P^{-1} \circ D_{\Qoppa_{\mathcal{C}}}(q^\dagger)^{-1} \circ D_{\Qoppa_{\mathcal{C}}}(a) \circ q^\dagger $. 
        \item \label{propitem:endomorphisms_Poincare_object_Azumaya_gi} Suppose that there exists a $ C_2 $-equivariant map $ \sphere \to M_A $ sending $ 1 \mapsto \mathrm{id}_P $ (where $ \sphere $ is endowed with the trivial $ C_2 $-action). 
        Then $ R $-linear module with genuine involution $ M_A $ of part \ref{propitem:endomorphisms_Poincare_object_central_Wall_anti-struc} comes from the involution on $ A$ in the sense of Remark \ref{rmk:Azumaya_gi_Wall_compare}. 
        In particular, $ \left(\mathcal{C}, \Qoppa\right) $ is of the form $ \left(\Mod^\omega_A, \Qoppa_A \right) $ for some Azumaya algebra over $ R $ with genuine involution.  
        \item \label{propitem:endomorphisms_hyperbolic_Poincare_object} Suppose $ (P,q) $ is a Poincaré object, and write $ A:= \mathrm{End}_{\mathcal{C}}(P) $. 
        Then there exists a $ C_2 $-equivariant map $ \sphere \to M_A $ sending $ 1 \mapsto \mathrm{id}_P $, i.e. $ M_A $ satisfies the hypotheses of \ref{propitem:endomorphisms_Poincare_object_Azumaya_gi}. 
    \end{enumerate} 
\end{proposition} 
\begin{corollary}~\label{cor: brauer classes representable by azumaya algebras}
    Let $ R $ be a Poincaré ring.  
    Let $ \left(\mathcal{C}, \Qoppa\right) $ be an invertible idempotent-complete $ R $-linear Poincaré $ \infty $-category. 
    Then $ \left(\mathcal{C}, \Qoppa\right) $ is of the form $ \left(\Mod^\omega_A, \Qoppa_A \right) $ for some Azumaya algebra over $ R $ with genuine involution.      
\end{corollary}
\begin{proof}
    By Proposition~\ref{prop:Poincare_cat_as_module_cat}, it suffices to exhibit a hyperbolic Poincaré object $ (P, q) $ whose underlying object $ P $ is a generator of $ \mathcal{C} $. 
    Since the underlying $R$-linear stable $ \infty $-category $\mathcal{C}$ is dualizable, we have by \cite[Example 4.11]{MR4743054} that $\mathcal{C}$ admits a compact generator $G$. 
    Now $ G \oplus D_{\Qoppa}G $ promotes canonically to a hyperbolic Poincaré object of $ \left(\mathcal{C}, \Qoppa\right) $ by \cite[Proposition 2.2.5]{CDHHLMNNSI}. 
\end{proof}
\begin{proof}[Proof of Proposition \ref{prop:Poincare_cat_as_module_cat}]
    The proof of \ref{propitem:endomorphisms_Poincare_object_involution} is similar to \cite[Proposition 3.1.16]{CDHHLMNNSI}. 
    By \cite[Corollary 3.1.3]{CHN2024}, the forgetful functor $ \Catp \to \Cat^{\mathrm{ps}} $ of \cite[(179)]{CDHHLMNNSI} admits a symmetric monoidal refinement. 
    It follows that the image $ \left(\Mod_R^{\omega},B_{\Qoppa_R}\right) $ of $ \Mod^\mathrm{p}_R $ in $ \Cat^{\mathrm{ps}} $ is a symmetric monoidal $ \infty $-category with perfect duality, and $ \left(\mathcal{C}, B_{\Qoppa}\right) $ is a left module over $ \left(\Mod_{R^e}^{\omega},B_{\Qoppa_R}\right) $. 
    In particular, by \cite[\S6]{HLAS25}, the perfect duality $ D_{\Qoppa} \colon \mathcal{C}^{\op} \to \mathcal{C} $ promotes $ \mathcal{C} $ to a $ C_2 $-homotopy fixed point of $ \Mod_{R^e} $-enriched $ \infty $-categories. 
    By \cite[Theorem 6.3.2(iii)]{MR3345192}, there is an adjunction between pointed $ \Mod_{R^e} $-enriched $ \infty $-categories and $ \EE_1 $-$ R^e $-algebras. 
    This is $ C_2 $-equivariant with respect to taking opposites $ (-)^\op $ by the argument in \cite[Proposition 3.1.16]{CDHHLMNNSI}. 
    Finally, the $ C_2 $-action $ \lambda $ on $ R^e $ induces actions on both $ \EE_1 $-$ R^e $-algebras and the category of $ \Mod_{R^e} $-enriched $ \infty $-categories, and the adjunction is equivariant with respect to these actions. 
    It suffices to show that $ P $ promotes to a $ C_2 $-fixed point in $ \mathcal{C}^{\simeq} $, which is true by \cite[Corollary 2.2.10]{CDHHLMNNSI}. 
    
    Since $P$ generates $\mathcal{C}$, we have an equivalence $\mathcal{C}\simeq \mathrm{Mod}_A^\omega $ of $ R^e $-linear small idempotent-complete stable $ \infty $-categories. 
    Part \ref{propitem:endomorphisms_Poincare_object_central_Wall_anti-struc} follows from Proposition \ref{prop:relative_poincare_cats_basic_properties}\ref{propitem:classify_R_lin_hermitian_struct_mod_cat} and unraveling definitions. 
    Part \ref{propitem:endomorphisms_Poincare_object_Azumaya_gi} follows from \cite[Proposition 3.1.14]{CDHHLMNNSI}. 
    
    Part \ref{propitem:endomorphisms_hyperbolic_Poincare_object} follows from a similar argument to Example~\ref{ex:endomorphisms_of_poincare_object}. 
\end{proof}

A key consequence of Corollary~\ref{cor: brauer classes representable by azumaya algebras} is that all classes in $\pi_0(\mathfrak{br}^\mathrm{p}(R))$ can be represented by Azumaya $R$ algebras with genuine involution. 
We close this section by showing that this extends to the case of schemes with good quotient.
			
\begin{corollary}~\label{cor: Brauer classes represented by Azumaya algebras with genuine involution}
    Let $(X,\lambda, Y, p)$ be a scheme with good quotient. Let $(A, P\to A^{\mathrm{t}C_2})$ be an Azumaya algebra with genuine involution over $X$ in the sense of Definition~\ref{defn: Azumaya algebra with genuine anti involution scheme case}. Then $(\mathrm{Mod}_A,\Qoppa_A)$ defines an invertible module over $\mathrm{Mod}^\mathrm{p}_{\underline{\mathcal{O}}}$ in the sense of Construction~\ref{cons:C2_mod_over_sheaf_of_Green_func}. Furthermore, any invertible module over $\mathrm{Mod}^\mathrm{p}_{\underline{\mathcal{O}}}$ comes from an Azumaya algebra with genuine involution over $X$ in this way.
\end{corollary}
\begin{proof}
    For the first claim, note that we have a candidate dual given by the internal mapping object $\mathrm{Fun}_{(X,\lambda,Y,p)}((\mathrm{Mod}_A^\omega, \Qoppa_A), \mathrm{Mod}^\mathrm{p}_{\underline{\mathcal{O}}})$ and an evaluation map \[((\mathrm{Mod}_A^\omega, \Qoppa_A)\otimes_{\mathrm{Mod}^\mathrm{p}_{\underline{\mathcal{O}}}}\mathrm{Fun}_{(X,\lambda,Y,p)}((\mathrm{Mod}_A^\omega, \Qoppa_A), \mathrm{Mod}^\mathrm{p}_{\underline{\mathcal{O}}})\to \mathrm{Mod}^\mathrm{p}_{\underline{\mathcal{O}}}\] so it is enough to check invertibility Zariski locally on $Y$. This then follows by Proposition~\ref{prop:mod_over_azumaya_geninv_is_invertible}.
    
    For the second statement, let $(\mathcal{C},\Qoppa)$ be an invertible $\mathrm{Mod}^\mathrm{p}_{\underline{\mathcal{O}}}$-module. 
    Then $ \mathcal{C} $ is an small stable $ \mathcal{O}_X $-linear invertible $ \infty $-category. 
    By \cite[Proposition 3.9]{MR2957304}, there exists a compact generator of $\mathcal{C}$, call it $G$. 
    Then $G\oplus D_{\Qoppa}G$ will be a Poincar\'e object in $\mathcal{C}$. Zariski locally on $Y$ these will remain Poincar\'e generators, and so on affine open subschemes of $Y$ the endomorphism rings of these objects will be Azumaya algebras with genuine involutions by Proposition~\ref{prop:Poincare_cat_as_module_cat}. 
    We therefore get that the endomorphism sheaf $\mathcal{E}\mathrm{nd}(G\oplus D_{\Qoppa}G)$ Zariski-locally admits the structure of an Azumaya algebra with genuine involution. Thus Zariski-locally we have $3$ additional pieces of data, an involution, a module $P$, and a map from $P$ to the Tate construction. All three of these pieces of data glue along Zariski covers, and so the ring $\mathcal{E}\mathrm{nd}(G\oplus D_\Qoppa G)$ promotes to a ring spectrum with genuine involution. The data of the dual will also glue together along Zariski covers, so $\mathcal{E}\mathrm{nd}(G\oplus D_\Qoppa G)$ admits the structure of an Azumaya algebra with genuine involution as desired.
\end{proof}
			
\section{Construction of the involutive Brauer group and comparisons to previous work}\label{section:inv_Brauer_comparison}

In this section, we will define the involutive Brauer group for schemes with good quotient $(X,\lambda, Y, p)$. In order to do this, 
we use \'etale descent to construct a map $\mathrm{Br}^\mathrm{p}(X,\lambda, Y,p)\to \mathrm{H}^1(\mathrm{R}\Gamma_{\Acute{e}t}(X;\mathbb{Z})^{\mathrm{h}C_2})$ and define the cohomological involutive Brauer group to be the kernel of this map. Finally we show that the involutive Brauer group, defined to be the subgroup of the cohomological Brauer group of elements whose underlying Brauer class are representable by classical Azumaya algebras, agrees with the  involutive Brauer group constructed by Parimala and Srinivas in \cite{Parimala_Srinivas}.
			
\subsection{The norm fiber sequence} 
The goal of this section is to show that, under suitable hypotheses on our Poincar\'e ring $R$, there exists a norm map $ \mathfrak{br}(R^e)\to \mathfrak{br}(R^{C_2}) $ and a fiber sequence of the form $ \Brp(R)\to \mathfrak{br}(R^e)\to \mathfrak{br}(R^{C_2}) $. 
In particular, this will extend the fiber sequence of Corollary~\ref{cor:Poincare_Brauer_to_Brauer_fiber} to the right. 
These results will allow us to use Poincar{\'e} Brauer groups to compute ordinary Brauer groups and to obtain a comparison to the involutive Brauer group of \cite{Parimala_Srinivas}. We will first show that the fiber sequence we have already constructed in Corollary~\ref{cor:Poincare_Brauer_to_Brauer_fiber} does in fact recover the first half of the long exact sequence constructed in \cite[Theorem 2]{Parimala_Srinivas}.

\begin{remark}~\label{remark: Borel poincare rings have vanishing tate}
    Let $R$ be a Poincar\'e ring such that $R^{\phi C_2}\simeq 0$. Part of the data of a Poincar\'e ring as defined in Definition~\ref{definition:poincare_ring_spectrum} is a ring map $R^{\phi C_2}\to R^{\mathrm{t}C_2}$, and so in this case $R^{\mathrm{t}C_2}\simeq 0$ as well. In particular, $R$ is given by the symmetric Poincar\'e ring structure $(R^e)^\mathrm{s}$ of Example \ref{example:tate_poincare_structure}. Consequently, a Poincar\'e ring $R$ satisfies $R^{\phi C_2}\simeq 0$ if and only if $R$ arises as the symmetric Poincar\'e ring of an $\Einfty$-ring spectrum $R^e$ with $C_2$-action such that $(R^e)^{t C_2}\simeq 0$.
\end{remark}

One of the key examples of the situation of Remark~\ref{remark: Borel poincare rings have vanishing tate} is when $R$ is a connective ring with involution and with $\frac{1}{2}\in \pi_0(R)$, thought of as a Poincar\'e ring via Example~\ref{example:genuine_symmetric_poincare_structure}. This is not the only example, however, see Example~\ref{ex: Atiyah real K-theory}.

\begin{lemma}~\label{lem: pic of genuine category is pic of fixed points when galois}
    Let $R$ be a Poincar{\'e} ring with $R^{\phi C_2}\simeq 0$.  
    Then the map \[\mathfrak{pic}\left(\mathrm{Mod}_{R^L}\left(\mathrm{Sp}^{C_2}\right)\right)\to \mathfrak{pic}\left(\mathrm{Mod}_{R^e}^{\mathrm{h}C_2}\right)\simeq \mathfrak{pic}(R)^{\mathrm{h}C_2}\] induced by sending a $ C_2 $-spectrum to its underlying spectrum with $ C_2 $ action is an equivalence of connective spectra.
    If furthermore we have that $R^{\mathrm{h}C_2}\to R^e$ is faithful Galois in the sense of \cite{Rognes_galois}, then the map \[\mathfrak{pic}(R^{\mathrm{h}C_2})\to \mathfrak{pic}(R^e)^{\mathrm{h}C_2}\] is also an equivalence.
\end{lemma}
\begin{proof}
    By Remark~\ref{remark: Borel poincare rings have vanishing tate} we have that $R^{\mathrm{t}C_2}=0$ as well. Consequently, any module $M$ over $R^L$ has $M^{\phi C_2}=M^{\mathrm{t}C_2}=0$. 
    Thus $\mathrm{Mod}_{R^L}\left(\mathrm{Sp}^{C_2}\right) \to \mathrm{Mod}_{R^e}\left(\mathrm{Sp}^{BC_2}\right)$ is a symmetric monoidal equivalence of stable $ \infty $-categories; in particular they have equivalent Picard spaces. 
    
    It remains to show that the functor $\mathrm{Mod}_{R^e}(\mathrm{Sp}^{BC_2})\xrightarrow{(-)^{\mathrm{h}C_2}}\mathrm{Mod}_{R^{\mathrm{h}C_2}}$ induces an equivalence on the Picard space when the map $R^{\mathrm{h}C_2}\to R^e$ is faithful Galois. 
    This follows by \cite[Proposition 9.4]{Mathew_galois}.
\end{proof}

We will now consider the case where $R$ is discrete (meaning that $R^e$ is discrete, although it will follow from our other assumptions that $R^{C_2}$ is also discrete), $2\in (R^e)^\times$, and the map $R^{C_2}\to R^e$ is Galois. In this case, the fiber sequence of Corollary~\ref{cor:Poincare_Brauer_to_Brauer_fiber} and the identification of Lemma~\ref{lem: pic of genuine category is pic of fixed points when galois} combine to give a fiber sequence \[\pic(R^e)\xrightarrow{\mathrm{Cores}_{R^e/R^{C_2}}}\pic(R^{C_2})\to \pnbr(R)\] where we have that $\mathrm{Cores}_{R^e/R^{C_2}}$ is given by the map $\mathcal{L}\mapsto (\mathcal{L}\otimes_{R^e} \lambda^*\mathcal{L})^{C_2}$. For an explanation of the name `corestriction,' see for example \cite[Definition 6.2.1]{azumaya_involution}. 
In order to extend this fiber sequence to the right we will need to categorify the corestriction map.
			
\begin{construction}\label{cons:categorified_corestriction}
    Let $R$ be a Poincar{\'e} ring and let $\lambda: R^e\to R^e$ denote the involution. 
    Consider the functor \[\mathrm{Mod}_{\mathrm{Mod}_{R^{e}}^\omega}(\mathrm{Cat}_{\infty, \mathrm{idem}}^{\mathrm{ex}})\xrightarrow{\left(-\otimes_{\mathrm{Mod}_{R^e}^\omega} \lambda^*-\right)^{\mathrm{h}C_2}}\mathrm{Mod}_{(\mathrm{Mod}_{R^e}^\omega)^{\mathrm{h}C_2}}(\mathrm{Cat}_{\infty, \mathrm{idem}}^{\mathrm{ex}})\] 
    which we will denote by $\mathrm{Cores}_{R^e/R^{C_2}}$. Note that this functor is symmetric monoidal when restricted to the full subcategory spanned by invertible objects, since by \cite[Theorem 3.15]{MR3190610} we can reduce this question to one about equivariant objects in module categories where the result follows from the natural $C_2$-equivariant isomorphism \[(A\otimes_{R^e}\lambda^* A)\otimes_{R^e}(B\otimes_{R^e} \lambda^* B)\simeq (A\otimes_{R^e} B)\otimes_{R^e} \lambda^*(A\otimes_{R^e} B)\] for any $ \EE_1 $-$R^e$-algebras $A$ and $B$.
\end{construction}

\begin{lemma}~\label{lem: poincare br to br to c2 fixed points in null}
    Let $ R $ be a Poincaré ring, and write $ \lambda $ for the involution on $ R^e $. 
    The composite $$\pnbr(R)\to \mathfrak{br}(R^e)\to \mathfrak{pic}\left(\mathrm{Mod}_{\mathrm{Mod}_{R^e}^{\omega, \mathrm{h}C_2}}(\Catex_{, \mathrm{idem}})\right)$$ is nullhomotopic, where the right-hand map is induced by the corestriction of Construction \ref{cons:categorified_corestriction}. 	
\end{lemma}
\begin{proof}
    Suppose given an invertible $ R $-linear Poincar{\'e} $ \infty $-category $ \left(\mathcal{C},\Qoppa\right) $.
    Note that we have a factorization of $\mathrm{Cores}_{R^e/R^{C_2}}$ by definition as the composite 
    \[
    \begin{tikzcd}
        \mathrm{Mod}_{\mathrm{Mod}_{R^e}^\omega}(\Catex_{, \mathrm{idem}})\ar[r,"{(\mathrm{id},\lambda^*)}"] \ar[rdd]&  (\mathrm{Mod}_{\mathrm{Mod}_{R^e}^\omega}(\Catex_{, \mathrm{idem}})\times \mathrm{Mod}_{\mathrm{Mod}_{R^e}^\omega}(\Catex_{, \mathrm{idem}}))^{\mathrm{h}C_2}\ar[d,"{(-\otimes_{\mathrm{Mod}_{R^e}^\omega}-)}"]\\
        & \mathrm{Mod}_{\mathrm{Mod}_{R^e}^\omega}(\Catex_{, \mathrm{idem}})^{\mathrm{h}C_2}\ar[d] \\ & \mathrm{Mod}_{\mathrm{Mod}_{R^e}^{\omega,\mathrm{h}C_2}}(\Catex_{, \mathrm{idem}})
    \end{tikzcd}\,.
    \]
     where the $C_2$-action on the upper right is given by swapping the factors and applying $\lambda^*$ diagonally. Let $\theta:\Catpidem_{, R}\to \mathrm{Mod}_{\mathrm{Mod}_{R^e}^\omega}(\Catex_{, \mathrm{idem}})$ be the forgetful functor. Then the duality functor $\lambda^*D_{\Qoppa}$ induces a natural equivalence between $\lambda^*\circ\theta(-)\simeq D_{R}\circ\theta (-)$ as functors $\Catpidem_{, R}^{\simeq}\to \mathrm{Mod}_{\mathrm{Mod}_{R^e}^\omega}(\Catex_{, \mathrm{idem}})^\simeq$, where $D_{R}:\mathrm{Mod}_{\mathrm{Mod}_{R^e}^\omega}(\Catex_{, \mathrm{idem}})\to \mathrm{Mod}_{\mathrm{Mod}_{R^e}^\omega}(\Catex_{, \mathrm{idem}})^{\op}$ is the functor sending $\mathcal{C}\mapsto \mathrm{Fun}_{R^e}^{\ex}(\mathcal{C},\mathrm{Mod}_{R^e}^\omega)$. 
     Applying the functor $\mathfrak{pic}(-)$ to the diagram above and using this natural equivalence we see that the map $\mathfrak{br}^\mathrm{p}(R)\to \mathfrak{pic}\left(\mathrm{Mod}_{\mathrm{Mod}_{R^e}^{\omega, \mathrm{h}C_2}}(\Catex_{, \mathrm{idem}})\right)$ factors through the map \[\mathfrak{br}(R^e)\to (\mathfrak{br}(R^e)\times \mathfrak{br}(R^e))^{\mathrm{h}C_2}\to \mathfrak{br}(R^e)^{\mathrm{h}C_2}\] where the action on $\mathfrak{br}(R^e)\times \mathfrak{br}(R^e)$ is $(x,y)\mapsto (-y,-x)$, the action on $\mathfrak{br}(R^e)$ is multiplication by $(-1)$ and the first map is induced by the map $\mathfrak{br}(R^e)\xrightarrow{x\mapsto (x,-x)}\mathfrak{br}(R^e)\times \mathfrak{br}(R^e)$. This composition is nullhomotopic, completing the proof.
\end{proof}
			
\begin{lemma}~\label{lem: surjection onto kernel PS sequence}
    Let $ R $ be a Poincaré ring and suppose that $R$ is Borel. 
    Let $\mathcal{C}\in \mathrm{Mod}_{\mathrm{Mod}_{R^e}^\omega}\left(\mathrm{Cat}^{st}_{\infty,\mathrm{idem}}\right)$ be an invertible $ R^e $-linear category such that $\mathrm{Cores}_{R^e/R^{C_2}}(\mathcal{C})\simeq (\mathrm{Mod}_{R^e}^\omega)^{\mathrm{h}C_2}$ as $(\mathrm{Mod}_{R^e}^\omega)^{\mathrm{h}C_2}$-modules, where $\mathrm{Cores}_{R^e/R^{C_2}}$ is the functor of Construction~\ref{cons:categorified_corestriction}. 
    Then $\mathcal{C}$ admits an invertible $R$-linear Poincar{\'e} $\infty $-category structure.
\end{lemma}
\begin{proof}
    By \cite[Theorem 3.15]{MR3190610} there is some Azumaya $ R^e $-algebra $A$ such that $\mathcal{C}\simeq \mathrm{Mod}_A^\omega$, and there is an $ R^e $-linear $ C_2 $-equivariant equivalence $ \mathrm{Mod}_{A\otimes_{R^e}\lambda^*A}^{\omega}\simeq \mathcal{C} \otimes_{R^e} \lambda^* \mathcal{C} $.
    By assumption, we have an equivalence $\mathrm{Mod}_{A\otimes_{R^e}\lambda^*A}^{\omega, \mathrm{h}C_2} \simeq \mathrm{Mod}_{R^e}^{\omega, \mathrm{h}C_2}$. Here the $C_2$ action is being recorded by the equivalence $\mathrm{Mod}_{A\otimes_{R^e}\lambda^*A}^{\omega}\xrightarrow{\mathrm{swap}}\mathrm{Mod}_{\lambda^*A\otimes_{R^e}A}^\omega \simeq \lambda^*\mathrm{Mod}^\omega_{A\otimes_{R^e}\lambda^*A}$. The element $A\otimes_{R^e}\lambda^* A$ promotes to an element of $\mathrm{Mod}_{A\otimes_{R^e}\lambda^* A}^{\omega, \mathrm{h}C_2}$ via the swap isomorphism. 
    The image of $A\otimes_{R^e} \lambda^*A $ under this equivalence is some $ R^e $-module $M$ with $C_2$-action. 
    The module $ M $ acquires the structure of a left module over $A\otimes_{R^e}\lambda^* A$ with $ C_2 $-action, where the $C_2$ action on $M$ is $A\otimes_{R^e}\lambda^*A\to \lambda^*A\otimes_{R^e}A$ linear.

    The fact that the functor $\mathrm{Mod}_{A\otimes_{R^e}\lambda^*A}^{\omega, \mathrm{h}C_2}\to \mathrm{Mod}_{R^e}^{\omega \mathrm{h}C_2}$ is an equivalence then forces $M$ to be invertible as an $A\otimes_{R^e}\lambda^*A $-module, and so the hermitian structure $\Qoppa_M^s$ on $\mathrm{Mod}_{A}^{\omega}$ is Poincar{\'e}, and is by construction an $(\mathrm{Mod}_{R^e}^\omega,\Qoppa_R)$-module. The same analysis on $\mathrm{D}_R\mathcal{C}$ shows that $\mathrm{D}_R\mathcal{C}$ admits a canonical Poincar{\'e} structure as well determined by some invertible left $(A^\op\otimes_{R^e}\lambda^*A^\op)$-module $N$, and if one choses the equivalence $(\mathrm{D}_R\mathcal{C}\otimes_{\mathrm{Mod}_{R^e}^\omega}\lambda^*\mathrm{D}_R\mathcal{C})^{\mathrm{h}C_2}\simeq \mathrm{Mod}_{R^e}^{\omega, \mathrm{h}C_2}$ which fits into the following commutative diagram:
    \[
    \begin{tikzcd}
        \mathrm{Cores}_{R^e/R^{C_2}}\mathcal{C}\otimes_{\mathrm{Mod}_{R^e}^{\omega, \mathrm{h}C_2}} \mathrm{Cores}_{R^{e}/R^{C_2}}\mathrm{D}_R\mathcal{C} \ar[r,"\simeq"] \ar[d, "\simeq"] & \mathrm{Mod}_{R^e}^{\omega, \mathrm{h}C_2} \otimes_{\mathrm{Mod}_{R^e}^{\omega,\mathrm{h}C_2}}(\mathrm{D}_R\mathrm{Mod}_{R^e}^{\omega})^{\mathrm{h}C_2} \ar[d,"\simeq"]\\
        \left((\mathcal{C}\otimes_{\mathrm{Mod}_{R^e}^{\omega}}\mathrm{D}_R\mathcal{C})\otimes_{\mathrm{Mod}_{R^e}^\omega}\lambda^*(\mathcal{C}\otimes_{\mathrm{Mod}_{R^e}^{\omega}}\mathrm{D}_R\mathcal{C})\right)^{\mathrm{h}C_2}\ar[r, "\ev\otimes \lambda^*\ev"] & \mathrm{Mod}_{R^e}^{\mathrm{h}C_2}
    \end{tikzcd}
    \] then $N$ is a tensor-inverse of $M$. Hence $(\mathcal{C},\Qoppa_M^s)$ and $(\mathrm{D}_R\mathcal{C},\Qoppa_N^s)$ are tensor-inverses of each other and both define elements in the Poincar{\'e} Brauer group.
\end{proof}			
We finally arrive at the goal of this section. 
\begin{theorem}~\label{thm: extended fiber sequence of pnbr to br}
    Let $R$ be a Poincar{\'e} ring which is Borel. 
    Then there is a fiber sequence \[\pnbr(R)\to \mathfrak{br}(R^e)\to \mathfrak{pic}\left(\mathrm{Mod}_{\mathrm{Mod}_{R^e}^{\omega, \mathrm{h}C_2}}(\Catex_{, \mathrm{idem}})\right) \] of connective spectra, where the right-hand map is induced by Construction~\ref{cons:categorified_corestriction}.  
    
    If $R$ is furthermore such that $R^{\mathrm{h}C_2}\to R^e$ is faithful Galois, then the map \[\mathrm{Br}(R)\to \Pic(\mathrm{Mod}_{\mathrm{Mod}_{R^e}^{\mathrm{h}C_2}}(\Catex_{, \mathrm{idem}}))\simeq \mathrm{Br}(R^{\mathrm{h}C_2})\] is the {\'e}tale cohomological transfer map. 
\end{theorem}
\begin{proof}
    Let us write $\mathcal{F}(-)$ for the fiber of the map induced by Construction~\ref{cons:categorified_corestriction}. 
    From Lemma~\ref{lem: poincare br to br to c2 fixed points in null}, we get that there is a map $\Brp(R)\to \mathcal{F}(-)$.
    Looping both fiber sequences twice we see that we get a map of fiber sequences \[
    \begin{tikzcd}
        \mathfrak{pic}(R^e)\ar[r] \ar[d] & \mathfrak{pic}(\mathrm{Mod}_{R^L}(\mathrm{Sp}^{C_2})) \ar[r] \ar[d] & \pnbr(R) \ar[d]\\
        \mathfrak{pic}(R^e)\ar[r,"\mathrm{Cores}_{R^e/R^{C_2}}"] &\mathfrak{pic}(\mathrm{Mod}_{R^e}^{\mathrm{h}C_2}) \ar[r] & \mathcal{F}(R) 
    \end{tikzcd}
    \]
    and when $R$ is Borel we see that the left and middle vertical maps must be equivalences. 
    This together with Lemma~\ref{lem: surjection onto kernel PS sequence} gives that the right hand vertical map is in fact an equivalence of connective spectra.

    It remains to identify $\mathrm{Mod}_{(\mathrm{Mod}_{R^e}^\omega)^{\mathrm{h}C_2}}(\mathrm{Cat}^{\mathrm{ex}}_{\infty,\mathrm{idem}})\simeq \mathrm{Mod}_{\mathrm{Mod}_{R^{\mathrm{h}C_2}}^\omega}(\mathrm{Cat}^{\mathrm{ex}}_{\infty,\mathrm{idem}})$, which follows from the monoidal equivalence $(\mathrm{Mod}_{R^e}^{\omega})^{\mathrm{h}C_2}\simeq \mathrm{Mod}_{R^{\mathrm{h}C_2}}^\omega$, and the fact that under this equivalence an Azumaya algebra over $A$ is sent to $(A\otimes_{R^e}\lambda^*A)^{C_2}$ which extends the cohomological transfer map when $R^e$ is discrete (see \cite[Example 6.2.3]{azumaya_involution}). 
\end{proof}

\begin{corollary}\label{cor: extended fiber sequence of pnbr to br good quotient}
    Let $ (X,\lambda,Y,p) $ be a scheme with good quotient so that either the branch locus (Notation \ref{ntn:good_quotient_branch_locus}) in $ Y $ is empty (i.e. $ p $ is quadratic étale), or $ 2 $ is invertible in $ \mathcal{O}_Y $. 
    Then there is a fiber sequence \[\pnbr(X,\lambda,Y, p)\to \mathfrak{br}(X)\to \mathfrak{pic}\left(\mathrm{Mod}_{\mathrm{Mod}_{\mathcal{O}_X}^{\mathrm{h}C_2}}(\Catex_{, \mathrm{idem}})\right)\,.\]
\end{corollary}
\begin{proof}
    The same arguments as in the Poincar\'e ring case will apply in this case as well. We will outline the arguments and leave the details to the reader.  The arguments in Construction~\ref{cons:categorified_corestriction}, Lemma~\ref{lem: poincare br to br to c2 fixed points in null}, and Lemma~\ref{lem: surjection onto kernel PS sequence} apply verbatim. We have a fiber sequence \[\mathfrak{pic}(X)^{\mathrm{h}C_2}\to \mathfrak{br}^p(X,\lambda, Y, p)\to \mathfrak{br}(X)\] by taking the sheafification of the fiber sequence of Corollary~\ref{cor:Poincare_Brauer_to_Brauer_fiber} on the \'etale site of $Y$, via Corollary~\ref{cor:pnbr_as_etale_sheaf_affine_spectral}.
\end{proof}
\begin{example}~\label{ex: Atiyah real K-theory}
    Consider Atiyah's real $K$-theory $\kr$ from \cite{K-theory_and_reality}. This is the genuine $C_2$-spectrum with $\kr^e=\mathrm{KU}$ with $C_2$-action given by complex conjugation and with $\kr^{C_2}\simeq \mathrm{KU}^{\mathrm{h}C_2}=\mathrm{KO}$. Further, as pointed out to us by Greenlees, we have that $\kr^{\phi C_2}\simeq \mathrm{KU}^{\mathrm{t}C_2}\simeq 0$. We therefore have that $\kr$ in fact defines a Poincar{\'e} ring and furthermore that $\kr$ agrees with $\mathrm{KU}^s$. 
    
    Note that we also have that the argument used to prove Theorem \ref{thm: extended fiber sequence of pnbr to br} applies to Atiyah's real $K$-theory, and therefore we get a fiber sequence \[\pnbr(\kr)\to \mathfrak{br}(\mathrm{KU})\to \mathfrak{br}(\mathrm{KO})\] of spaces. This induces a long-exact sequence on homotopy groups of the form 
    \[
    \begin{tikzcd}
        \ldots \ar[r] & \pi_0(\gmq(\kr))\ar[r] & \pi_0(\mathrm{KU})^\times \ar[r, "N^{C_2}"] & \pi_0(\mathrm{KO})^\times \ar[lld, out=-45]\\
        & \pi_0(\Picp(\kr)) \ar[r] & \pi_0(\pic(\mathrm{KU})) \ar[r] & \pi_0(\pic(\mathrm{KO})) \ar[lld, out=-45]\\
        & \pi_0(\pnbr(\kr)) \ar[r] & \mathrm{Br}(\mathrm{KU}) \ar[r] &\mathrm{Br}(\mathrm{KO})
    \end{tikzcd}
    \] and we know that $\pi_0(\mathrm{KU})^\times \cong \pi_0(\mathrm{KO})^\times\cong \pi_0(\pic(\mathrm{KU}))\cong\ZZ/2\ZZ$ and $\pi_0(\pic(\mathrm{KO}))\cong \ZZ/8\ZZ$. 
    The map $N^{C_2}:\ZZ/2\ZZ\to \ZZ/2\ZZ$ on units is the zero map because we can identify the complex conjugation action with the Adams operation $\psi^{-1}$.  
    The map $\pi_0(\pic(\mathrm{KU}))\to \pi_0(\pic(\mathrm{KO}))$ is also the zero map.  
    This follows from the computation \[N^{C_2}([\Sigma \mathrm{KU}])=[((N^{C_2}\Sigma \mathrm{KU})\otimes_{N^{C_2}\mathrm{KU}}\kr)^{C_2}]\simeq [(\Sigma^{\rho}\kr)^{C_2}]\] together with the identification $(\Sigma^{\rho}\kr)^{C_2}\simeq \mathrm{KO}$ via Wood's Theorem. More specifically, we know that $(\Sigma^\rho \kr)^{C_2}$ must be a line bundle on $\mathrm{KO}$, of which there are only $8$ up to isomorphism: $\Sigma^i\mathrm{KO}$ for $0\leq i\leq 7$. We have a fiber sequence $\Sigma\kr\to \Sigma^{\rho}\kr$ given by taking $-\otimes \Sigma\kr$ with the inclusion of north and south poles map $S^0\hookrightarrow S^\sigma$. This then induces a fiber sequence \[\Sigma \kr \to \Sigma^{\rho}\kr\to \Sigma^2\mathrm{KU}[C_2]\] and so applying fixed points we get a fiber sequence \[(\Sigma^{\rho}\kr)^{C_2}\to \Sigma^2\mathrm{KU}\to \Sigma^2\mathrm{KO}\,.\] By Wood's theorem we have that $\mathrm{KU}\simeq \mathrm{cofib}(\Sigma\mathrm{KO}\xrightarrow{\times \eta}\mathrm{KO})$, which then gives a fiber sequence $\mathrm{KO}\to \Sigma^2\mathrm{KU}\to \Sigma^2\mathrm{KO}$ by taking the Bott isomorphism $\Sigma^2\mathrm{KU}\simeq \mathrm{KU}$. Out of the line bundles $\Sigma^i\mathrm{KO}$, $0\leq i\leq 7$, only $\mathrm{KO}$ will be the fiber of a $\mathrm{KO}$-module map $\mathrm{KU}\to \mathrm{KO}$ since the group of such maps is $\mathbb{Z}$ and therefore all other maps are multiples of the map giving $\mathrm{KO}$ as the fiber.
    Note that this same computation identifies $\Qoppa_{\kr}(\Sigma\mathrm{KU})\simeq \mathrm{KO}$ and that we get a splitting map $\pi_0(\pic(\mathrm{KU}))\to \pi_0(\Picp(\kr))$ given by sending $[\Sigma \mathrm{KU}]$ to $[(\Sigma\mathrm{KU},1)]$. To see that this is a group homomorphism is equivalent to showing that the square of this class is trivial in $\mathrm{Pic}^\mathrm{p}(\kr)$. Since there is a symmetric monoidal equivalence $\mathrm{Pn}(\mathrm{Mod}^\mathrm{p}_{\kr})\simeq (\mathrm{Mod}_{\mathrm{KU}}^{\omega, \simeq})^{\mathrm{h}C_2}$ by \cite[Proposition 2.2.11]{CDHHLMNNSI}, it is enough to show that the equivalence $\Sigma\mathrm{KU}\xrightarrow{\psi^{-1}(\beta^{-1}\cdot -)}D_{\Qoppa_{\kr}}(\Sigma \mathrm{KU})$ tensored with itself fits into a commutative diagram
    \[
    \begin{tikzcd}
        \Sigma^2\mathrm{KU} \ar[rr, "\psi^{-1}(\beta^{-1}\times -)^{\otimes 2} "] \ar[d,"\beta^{-2}\times-"] & & D_{\Qoppa_{\kr}}(\Sigma^2\mathrm{KU}) \ar[d, "\beta^{2}\times -"]\\
        \mathrm{KU} \ar[rr, "\psi^{-1}(-)"] & & D_{\Qoppa_{\kr}}(\mathrm{KU})
    \end{tikzcd}
    \]
    which follows from the fact that $\psi^{-1}(-)$ is a map of rings.
    We therefore have that \[\pi_0(\pnpic(\kr))\cong \ZZ/2\ZZ\times \ZZ/2\ZZ\] and that there is an injective function $\ZZ/8\ZZ\hookrightarrow \pi_0(\pnbr(\kr))$. Under the further assumption that $\mathrm{Br}(\mathrm{KU})=0$ we then have that $\pi_0(\pnbr(\kr))=\ZZ/8\ZZ$.
\end{example}

\subsection{The involutive and cohomological involutive Brauer groups}
Now that we have \'etale descent we are able to construct the involutive and cohomological involutive Brauer groups.
			
\begin{construction}~\label{const: map from poincare brauer to homotopy fixed points}
    Let $(X,\lambda, Y, p)$ be a scheme with good quotient. 
    The derived global sections $\mathrm{R}\Gamma_{\acute{e}t}(X;\mathbb{Z})$ has a $ C_2 $-action induced by $n\mapsto -\lambda^*(n)$. We can then consider the homotopy fixed points $\mathrm{R}\Gamma_{\acute{e}t}(X;\mathbb{Z})^{\mathrm{h}C_2}$ of this action. 
    We will define a homomorphism $\mathrm{Br}^\mathrm{p}(X,\lambda, Y, p)\to \mathrm{H}^1(\mathrm{R}\Gamma_{\Acute{e}t}(X, \mathbb{Z})^{\mathrm{h}C_2})$. 
    To this end, note that the composition \[\pnbr(X,\lambda, Y, p)\to \mathfrak{br}(X)\to \mathfrak{pic}\left(\mathrm{Mod}_{\mathrm{Mod}_X^{\omega, \mathrm{h}C_2}}(\Catex_{, \mathrm{idem}})\right)\] is nullhomotopic; where the latter map is Construction~\ref{cons:categorified_corestriction}. 
    By {\'e}tale descent (Corollary \ref{cor:pnbr_as_etale_sheaf_affine_spectral}), it suffices to show that there is a functorial nullhomotopy when $ Y $ (and therefore also $ X $) is affine. 
    This is proven in Lemma~\ref{lem: poincare br to br to c2 fixed points in null}. We further have a map $\mathfrak{pic}\left(\mathrm{Mod}_{\mathrm{Mod}_X^{\omega, \mathrm{h}C_2}}(\Catex_{, \mathrm{idem}})\right) \to \mathfrak{pic}\left(\mathrm{Mod}_{\mathrm{Mod}_X^{\omega}}(\Catex_{, \mathrm{idem}})^{\mathrm{h}C_2}\right)\simeq \mathfrak{br}(X)^{\mathrm{h}C_2}$ by the limit-interchange map, and so we get a map \[\mathfrak{br}^\mathrm{p}(X,\lambda, Y, p)\to \mathrm{fib}(\mathfrak{br}(X)\to \mathfrak{br}(X)^{\mathrm{h}C_2})\] and the target can be identified with $\mathfrak{br}(X)^{\mathrm{h}(-\lambda)}$ by Remark~\ref{remark:Poincare_Picard_group_fiber_sequence}. By \cite[Corollary 7.10]{MR3190610} we have a map $\mathfrak{br}(-)\simeq B\mathfrak{pic}(-)\to B\mathbb{Z}_X$ of \'etale sheaves on $X$, and so we obtain a map \[\mathfrak{br}^\mathrm{p}(-)\to (p_*B\mathbb{Z}_X)^{\mathrm{h}(-\lambda)}\] of \'etale sheaves on $Y$. Taking global sections gives a map $\Brp(X,\lambda, Y, p)\to \tau_{\geq 0}(\mathrm{R}\Gamma_{\Acute{e}t}(X,\mathbb{Z})[1])^{\mathrm{h}(-\lambda)}$, which on $\pi_0$ gives the desired map. 
\end{construction}
			
\begin{definition}\label{defn:cohomological_involutive_Brauer_group}
    Let $(X,\lambda, Y, p)$ be a scheme with good quotient. 
    Define the \textit{cohomological involutive Brauer group} to be \[\mathrm{Br}'(X,\lambda):= \mathrm{Ker}\left(\mathrm{Br}^\mathrm{p}(X,\lambda, Y, p)\to \mathrm{H}^1(\mathrm{R}\Gamma_{\Acute{e}t}(X,\mathbb{Z})^{\mathrm{h}C_2})\right)\] where we take the kernel of the map defined in Construction~\ref{const: map from poincare brauer to homotopy fixed points}. 
    By construction, we have a commutative diagram
    \[
    \begin{tikzcd}
        \mathrm{Br}^\mathrm{p}(X,\lambda, Y, p) \ar[r] \ar[d] & \pi_0(\mathfrak{br}(X)) \ar[d]\\
        \mathrm{H}^1(\mathrm{R}\Gamma_{\Acute{e}t}(X,\mathbb{Z})^{\mathrm{h}C_2}) \ar[r] & \mathrm{H}^1_{\Acute{e}t}(X,\mathbb{Z})
    \end{tikzcd} ;
    \]
    and upon taking vertical kernels there is a homomorphism $\mathrm{Br}'(X,\lambda, Y, p)\to \mathrm{H}^2_{\Acute{e}t}(X;\mathbb{G}_m)=\mathrm{Br}'(X)$. 
    Define the \textit{involutive Brauer group} $\mathrm{Br}(X,\lambda)$ as the pullback
    \[
    \begin{tikzcd}
        \mathrm{Br}(X,\lambda) \ar[r] \ar[d] & \mathrm{Br}'(X,\lambda) \ar[d]\\
        \mathrm{Br}(X) \ar[r] & \mathrm{Br}'(X) 
    \end{tikzcd} 
    \]
    in the category of abelian groups.
\end{definition}
\begin{remark}\label{rmk:pnbr_to_H1_fxpt_interpretation}
    By Corollary~\ref{cor: brauer classes representable by azumaya algebras}, every element of $ \mathrm{Br}^\mathrm{p}(X,\lambda, Y,p) $ can be represented by an Azumaya algebra with genuine involution. 
    There is a homomorphism $\mathrm{Br}^\mathrm{p}(X,\lambda, Y,p)\to \mathrm{H}^1(\mathrm{R}\Gamma_{\Acute{e}t}(X,\mathbb{Z})^{{\mathrm{h}(-\lambda)}})$ which detects (a) whether the underlying Brauer class $ \overline{\alpha} $ of $ \alpha $ maps to zero under the projection $ \pi_0(\mathfrak{br}(X)) \to \mathrm{H}^1_{\mathrm{\acute{e}t}}(X,\mathbb{Z}) $ and (b) given any $C_2$-étale cover $ \mathcal{U} = \{p^*U \to U\} $ so that the underlying Brauer class is trivial: $ \left[\overline{\alpha}|_{p^*U}\right] \simeq \mathcal{O}_{p^*U} $, the obstruction to exhibiting the involution on $ \alpha $ as one arising from an involution of classical algebras (see Lemma \ref{lemma:classical_inv_brauer_gp_as_kernel} and Example \ref{ex:br and br' for alg closed field}). 
    Taking the kernel of the map $\mathrm{Br}^\mathrm{p}(X,\lambda, Y,p)\to \mathrm{H}^1(\mathrm{R}\Gamma_{\Acute{e}t}(X,\mathbb{Z})^{{\mathrm{h}(-\lambda)}})$ excludes the truly exotic Azumaya algebras with involution. 
    
    Since the map $\mathrm{Br}(X)\hookrightarrow \mathrm{Br}'(X)$ is always injective, it follows that the map $\mathrm{Br}(X,\lambda)\hookrightarrow \mathrm{Br}'(X,\lambda)$ is also always injective. 
    We can identify the classes in $\mathrm{Br}(X,\lambda)$ as those represented by Azumaya algebras with genuine involution $[(\mathcal{A},\sigma)]\in \mathrm{Br}'(X,\lambda)$ such that the underlying Azumaya algebra with involution $\mathcal{A}^e$ is Morita equivalent to an ordinary Azumaya algebra. 
\end{remark}
\begin{lemma}\label{lemma:classical_inv_brauer_gp_as_kernel}
    Let $ (X, \lambda, Y,p) $ be a scheme with good quotient such that either $ \lambda $ is the identity or $ p $ is quadratic étale. 
    Assume further that $ \frac{1}{2} \in \mathcal{O}_Y $.  
    The composite homomorphism 
    \begin{equation}\label{eq:classical_inv_brauer_gp_as_kernel}
        \mathrm{Br}_{\mathrm{PS}}(X,\lambda) \to \mathrm{Br}^\mathrm{p}(X,\lambda, Y, p)\to \mathrm{H}^1(\mathrm{R}\Gamma_{\Acute{e}t}(X,\mathbb{Z})^{\mathrm{h}C_2})     
    \end{equation} 
    is zero, where the former map is that of Proposition \ref{prop:from_involutive_Brauer_to_Poincare_Brauer} and the latter is that of Construction \ref{const: map from poincare brauer to homotopy fixed points}. 
\end{lemma}
\begin{proof}
    Let $ (W,\ell, Z,\pi) \in \mathrm{qSch}^{C_2} $. 
    Throughout this proof, we will use the notation $ \left(\mathrm{Mod}_{\mathcal{O}_W}^\omega,\Qoppa_\ell\right) $ to denote the image of $ (W,\ell,Z,\pi) $ in $ \Einfty\Mon\left(\Catp\right) $ under \ref{cons:C2_mod_over_sheaf_of_Green_func}. 
    
    Suppose that $ p $ is quadratic étale. 
    Let $ \alpha \in \mathrm{Br}(X,\lambda) $ be an involutive Brauer class and let $ (\mathcal{A},\sigma) $ be an Azumaya algebra with $ \lambda $-involution representing $ \alpha $. 
    By \cite[p.215]{Parimala_Srinivas}, there exists an étale cover $ \left\{ q_i \colon U_i \to Y \right\} $ and a trivialization of the base change of $ (\mathcal{A},\sigma) $ along this cover. 
    In other words, if we write $ X_i := U_i \times_Y X $, then there exist isomorphisms $ q_i^*\mathcal{A} \simeq \mathrm{Mat}_{n}\left(\mathcal{O}_{X_i}\right) $ under which $ q_i^*(\sigma) $ may be identified with the $ q_i^*(\lambda) $-involution $ a \mapsto Q_i^{-1} a^* Q_i $, where $ a^* := \lambda(a^\vee) $ denotes the $ \lambda $-conjugate transpose and $ Q_i \in \Gamma \left(X_i, \mathrm{Mat}_{n}\left(\mathcal{O}_{X_i}\right)\right) $ is invertible and satisfies $ Q_i^* = Q_i $. 
    Its image in $ \mathrm{H}^1(\mathrm{R}\Gamma_{\Acute{e}t}(X,\mathbb{Z})^{\mathrm{h}C_2}) $ may be constructed as follows: Let $ \left(\mathcal{C}_{\mathcal{A}}, \Qoppa_\sigma \right) $ denote the image of $ (\mathcal{A},\sigma) $ in $ \mathrm{Br}^\mathrm{p}(X,\lambda, Y, p) $. 
    By Remark \ref{remark:restriction_of_schemes_with_involution} and Proposition \ref{prop:sym_mon_catp_good_quotient_lax_monoidal} (cf. Corollary \ref{cor:from_sch_good_quotient_to_Z_linear_Poincare_cat}), $ q_i $ induces a symmetric monoidal Poincaré functor $ \left(\mathrm{Mod}_{\mathcal{O}_X}^\omega,\Qoppa_{\lambda} \right) \to \left(\mathrm{Mod}_{\mathcal{O}_{X_i}}^\omega,\Qoppa_{q_i^*\lambda} \right) $; we will write $ q_i^* $ for base change along this functor. 
    The aforementioned trivializations induce equivalences $ \varphi_i \colon q_i^*\left(\mathcal{C}_{\mathcal{A}}, \Qoppa_\sigma \right) \simeq \left(\mathrm{Mod}_{\mathcal{O}_{X_i}}^\omega,\Qoppa_{q_i^*\lambda} \right) $ of $ \left(\mathrm{Mod}_{\mathcal{O}_{X_i}},\Qoppa_{q_i^*\lambda} \right) $-linear Poincaré $ \infty $-categories. 
    Writing $ U_{ij} := U_i \times_Y U_j $ and $ q_{ij} $ for the canonical map $ U_{ij} \to Y $ and $ X_{ij}:= U_{ij} \times_Y X $, we have $ \left(\mathrm{Mod}_{\mathcal{O}_{X_{ij}}}^\omega,\Qoppa_{q_{ij}^*\lambda} \right) $-linear equivalences
    \begin{equation*}
        \left(\mathrm{Mod}_{\mathcal{O}_{X_{ij}}}^\omega,\Qoppa_{q_{ij}^*\lambda} \right) \xleftarrow[\sim]{q_i^*(\varphi_j)} q_{ij}^*\left(\mathcal{C}_{\mathcal{A}}, \Qoppa_\sigma \right) \xrightarrow[\sim]{q_j^*(\varphi_i)} \left(\mathrm{Mod}_{\mathcal{O}_{X_{ij}}}^\omega,\Qoppa_{q_{ij}^*\lambda} \right) \,.
    \end{equation*}
    In particular, $ \left(\mathrm{Mod}_{\mathcal{O}_{X_{ij}}}^\omega,\Qoppa_{q_{ij}^*\lambda} \right) $-linearity and Proposition \ref{prop:loops_pnbr_is_pnpic} imply that each equivalence $ {T_{ji} := q_i^*(\varphi_j) \circ q_j^*(\varphi_i)^{-1}} $ corresponds to a class $ \left(\mathcal{L}_{ij},h_{ij}\right) \in \pnpic\left(X_{ij}, q_{ij}^*(\lambda)\right) $. 
    By Theorem \ref{thm: Fausk for pnpic and schemes with good quotient}, $ \left(\mathcal{L}_{ij},h_{ij}\right) $ determines a class in $ C_{C_2}\left(X_{ij};\ZZ^{-}\right) $, or equivalently a section $ s_{ij} $ of the sheaf $ \left(p_*\ZZ_X\right)^{hC_2} $ over $ U_{ij} $. 
    By construction, there is a canonical equivalence $ q_i^*T_{kj} \circ q_k^* T_{ji} \simeq q_j^*T_{ki} $ for each $ i,j,k $ (and likewise for $ (n+1) $-tuples $ i_0,i_1,\ldots, i_n $). 
    These equivalences induce canonical homotopies from $ q_i^*(s_{kj}) + q_k^*(s_{ij}) $ to $ q_j^*(s_{ik}) $. 
    These assemble to give a class $ s \in \mathrm{H}^1(\mathrm{R}\Gamma_{\Acute{e}t}(X,\mathbb{Z})^{\mathrm{h}C_2}) \simeq \mathrm{H}^1\left(\mathrm{R}\Gamma_{\Acute{e}t}(Y,p_*\mathbb{Z}^{\mathrm{h}C_2})\right) $, where $ \mathrm{R}\Gamma_{\Acute{e}t}(X,\mathbb{Z}) $ is equipped with the $ C_2 $-action described in Construction \ref{const: map from poincare brauer to homotopy fixed points}. 
    Unraveling definitions, we see that $ s $ agrees with the image of $ \alpha $ under the composite map in  (\ref{eq:classical_inv_brauer_gp_as_kernel}). 
    Now it is evident that $ \mathcal{L}_{ij} $ is equivalent to a line bundle concentrated in degree zero for all $ i, j $, hence $ s_{ij} = 0 $ and therefore $ s = 0 $. 

    When $ \lambda $ is the identity, the result follows from a nearly identical argument which we leave to the reader.  
    The relevant étale-local trivialization is described in \cite[\S1(i), also see beginning of (iii)]{Parimala_Srinivas}.  
\end{proof}
We end this subsection by exploring some computations of this theory. For a scheme with good quotient $(X,\lambda, Y, p)$, let $\mathrm{Tr}_{X/Y}:p _*\mathbb{Z}_X\to \mathbb{Z}_Y$ denote the map of sheaves on $Y$ given by $n\mapsto n+\lambda^*(n)$.
\begin{lemma}~\label{lem: some computations of the extra junk in brp}
    Let $(X,\lambda,Y,p)$ denote a scheme with good quotient. Then
    \begin{enumerate}[label=(\arabic*)]
        \item if $\lambda=\mathrm{id}_X$ then $\mathrm{Tr}_{X/Y}$ is multiplication by $2$ and $\mathrm{H}^{1}(\mathrm{R}\Gamma_{\Acute{e}t}(X,\mathbb{Z})^{\mathrm{h}C_2})\simeq H^0_{\Acute{e}t}(X,\mathbb{Z}/2)$;
        \item if $p $ is a degree $2$ \'etale extension then $\mathrm{H}^1(\mathrm{R}\Gamma_{\Acute{e}t}(X,\mathbb{Z})^{\mathrm{h}C_2})\cong \mathrm{H}^0_{\Acute{e}t}(Y,\mathrm{cofib}(\mathrm{Tr}_{X/Y}))$.
        \item if $X$ is reduced then $\mathrm{H}^1(\mathrm{R}\Gamma_{\Acute{e}t}(X,\mathbb{Z})^{\mathrm{h}C_2})\cong C_{C_2}(X,\mathbb{Z})/\{f+f\circ\lambda\mid f\in C(X,\mathbb{Z})\}$.
    \end{enumerate}
\end{lemma}
\begin{proof}
    We have a fiber sequence \[\mathrm{R}\Gamma_{\Acute{e}t}(X,\mathbb{Z})^{\mathrm{h}(-\lambda)}\to \mathrm{R}\Gamma_{\Acute{e}t}(X,\mathbb{Z})\to \mathrm{R}\Gamma_{\Acute{e}t}(X,\mathbb{Z})^{\mathrm{h}C_2}\] where the $C_2$ action on the left hand term is the one we are interested in, and the $C_2$ action on the right is the one induced by $\lambda$, by Remark~\ref{rmk:C2_spectra_twists}. In case 1 above we then have that the $C_2$ action on the right is trivial and there is then an exact sequence \[0\to \mathrm{H}^0_{\Acute{e}t}(X,\mathbb{Z})\xrightarrow{\times 2}\mathrm{H}^0_{\Acute{e}t}(X,\mathbb{Z})\to \mathrm{H}^1(\mathrm{R}\Gamma_{\Acute{e}t}(X,\mathbb{Z})^{\mathrm{h}C_2})\to H^1_{\Acute{e}t}(X,\mathbb{Z})\to\ldots\] from which we see by the 5-lemma the desired result.
    
    If $p $ is a degree $2$ \'etale extension then we get that $\mathrm{R}\Gamma_{\Acute{e}t}(X,\mathbb{Z})^{\mathrm{h}C_2}\simeq \mathrm{R}\Gamma_{\Acute{e}t}(Y,\mathbb{Z})$ by \'etale descent when the $C_2$ action is induced by $\lambda$. The result for 2 then follows.
    
    Finally suppose that $X$ is reduced, so that $H^1_{\Acute{e}t}(X,\mathbb{Z})\cong 0$ by \cite[Lemma 7.15]{MR3190610}. Thus by analyzing the homotopy fixed point spectral sequence for $\mathrm{R}\Gamma_{\Acute{e}t}(X,\mathbb{Z})^{\mathrm{h}(-\lambda)}$ we see that $\mathrm{H}^1(\mathrm{R}\Gamma_{\acute{e}t}(X;\mathbb{Z})^{\mathrm{h}(-\lambda)})\cong \mathrm{H}^1(C_2;C(X;\mathbb{Z}))\cong C_{C_2}(X;\mathbb{Z})/\{f+f\circ \lambda |f\in C(X;\mathbb{Z})\}$ as desired.
\end{proof}

\begin{example}~\label{ex:br and br' for alg closed field}
    Let $k$ be an algebraically closed field with trivial involution. If the characteristic of $k$ is not $2$, then we have already seen that $\mathrm{Br}^\mathrm{p}(k)\cong \mathbb{Z}/4\mathbb{Z}$, and by Lemma~\ref{lem: some computations of the extra junk in brp} we have that $\mathrm{H}^1(\mathrm{R}\Gamma_{\Acute{e}t}(k,\mathbb{Z})^{\mathrm{h}C_2})\cong \mathbb{Z}/2\mathbb{Z}$. Tracing through definitions we see that the map $\mathrm{Br}^\mathrm{p}(k)\to \mathbb{Z}/2$ is the unique surjective map and the kernel \[\mathrm{Br}'(k,\mathrm{id}_k)\cong \mu_2(k)\] generated by the Poincar{\'e} $\infty$-category $(\mathrm{Mod}_k^\omega, \Qoppa_{-k})$. This Poincar\'e category has a generator $k\oplus D_{\Qoppa_{-k}}k=k\oplus k$ as a hyperbolic form, and tracing through definitions the equivalence \[k\oplus k\simeq k\oplus D_{\Qoppa_{-k}}k\to D_{\Qoppa_{-k}}(k\oplus D_{\Qoppa_{-k}}k)\simeq D_{\Qoppa_{-k}}k\oplus D_{\Qoppa_{-k}}D_{\Qoppa_{-k}}k\simeq D_{\Qoppa_{-k}}k\oplus k \simeq k\oplus D_{\Qoppa_{-k}}k\simeq k\oplus k\] is given by sending $(a,b)\mapsto (-b,a)$.
    As an Azumaya algebra with genuine involution, one can take $M_{2\times 2}(k)=\mathrm{End}(k\oplus D_{\Qoppa_{-k}k})$. The involution on this Azumaya algebra is given by conjugating by the equivalence $(a,b)\mapsto (-b,a)$, in other words \[\sigma\begin{bmatrix} a & b \\ c & d\end{bmatrix} = \begin{bmatrix}
        d & -c\\ -b & a
    \end{bmatrix}\] or, in other words, an Azumaya algebra with involution of symplectic type. 
    Hence $\mathrm{Br}(k,\mathrm{id}_k)\cong \mathrm{Br}'(X,\mathrm{id}_k)$.
    
    If the characteristic of $k$ is $2$, then we have that $\mathrm{Br}^\mathrm{p}(k)\cong \mathbb{Z}$ by Example \ref{ex:pnbr_closed_point_ramified}. 
    We still have that $H^1(\mathrm{R}\Gamma_{\Acute{e}t}(k,\mathbb{Z})^{\mathrm{h}(-\lambda)})\cong \mathbb{Z}/2$ and that the map $\mathrm{Br}^\mathrm{p}(k,\mathrm{id}_k)\to \mathbb{Z}/2$ is surjective, so that \[\mathrm{Br}'(k,\mathrm{id}_k)\cong \mathbb{Z}\] generated by the class $\left(\mathrm{Mod}_{k}^\omega,\Qoppa_k^{[2]}\right)$. This is equivalent to the Poincar{\'e} structure of genuine even forms on $\mathrm{Mod}_k^\omega$ by \cite[Corollary 3.5.17 and Proposition 4.2.22]{CDHHLMNNSI}. Thus to give a Poincar{\'e} object we need only give an even form on a discrete $k$-module, which amounts to giving an alternating form. 
    An example of a Poincar{\'e} object in this Poincar{\'e} category is $k\oplus k$ with the form \[S=\begin{bmatrix}
        0 & 1\\
        1 & 0
    \end{bmatrix}\] which has bilinear form $\langle (x_1,y_1),(x_2,y_2)\rangle = x_1y_2+y_1x_2$. Tracing through definitions, the induced equivalence $k\oplus k\to k\oplus k$ is given by sending $(a,b)\to (b,a)$.
    The associated Azumaya algebra with genuine involution then has underlying ring spectrum $M_{2\times 2}(k)$ with the involution
    \[
    \sigma
    \begin{bmatrix}
        a & b\\
        c& d
    \end{bmatrix} = 
    \begin{bmatrix}
        d & c\\
        b & a
    \end{bmatrix}\,.
    \]
    This is an ordinary Azumaya algebra, hence \[\mathrm{Br}(k,\mathrm{id}_k)\cong \mathrm{Br}'(k, \mathrm{id}_k)\,.\]
\end{example}

This equivalence of $\mathrm{Br}(X,\lambda)\cong \mathrm{Br}'(X,\lambda)$ holds in many cases, including all affine $X$ in fact.
\begin{corollary}~\label{cor: br=br' implies involutive version}
    Let $X$ be a scheme satisfying $\mathrm{Br}(X)\cong \mathrm{Br}'(X)$; for instance, this holds for $X$ affine or when $ X $ is quasi-compact and separated and admits an ample line bundle (see \cite{MR611868} or \cite{deJong_Gabber}). 
    Then for any involution $\lambda$ on $X$, \[\mathrm{Br}(X,\lambda)\cong \mathrm{Br}'(X,\lambda).\]
\end{corollary}
\begin{proof}
    This follows from the non-involutive case since pulling back along an isomorphism is an isomorphism. 
\end{proof}

One might hope that the involution in fact cancels out the elements in $\mathrm{Br}'(X)$ which are not representable by ordinary Azumaya algebras, so that $\mathrm{Br}(X,\lambda)$ would always be equal to $\mathrm{Br}'(X,\lambda)$. Said another way, there are classes in $\pi_0(\mathfrak{br}(X))$ which are not representable by an ordinary Azumaya algebra, but an optimistic reader might hope that such classes are sent to nontrivial classes in $\pi_0\left(\mathfrak{pic}\left(\mathrm{Mod}_{\mathrm{Mod}_{X}^{\omega, \mathrm{h}C_2}}(\Catex_{, \mathrm{idem}})\right)\right)$. Such miraculous cancellation happens, for example, in the computations of Grothendieck-Witt theory of number rings in \cite{CDHHLMNNSIII} where all of the unknown $K$-groups do not lift. 
This unfortunately is not the case, as the next example shows.
\begin{example}
    Let $X$ be the union of two copies of $\mathrm{Spec}(\CC[x,y,z]/(xy-z^2))$ along their nonsingular locus. 
    Endow $X$ with the trivial action, considered as a scheme with involution and good quotient via Example~\ref{ex:schemes_and_involutions}\ref{ex_item:scheme_with_trivial_involution}. 
    We then have that {\'e}tale locally $X$ has trivial involution as well, and for affine {\'e}tale maps $\mathrm{Spec}(R)\to X$ we have functorially a fiber sequence \[\pnbr(R)\to \mathfrak{br}(R)\to \pic(\mathrm{Mod}_{\mathrm{Mod}_R(\mathrm{Sp}^{BC_2})^\omega})\] of spaces by Theorem~\ref{thm: extended fiber sequence of pnbr to br}. 
    We then find that $\pnbr(X)$ fits into a similar fiber sequence and in particular the map \[\mathrm{Br}^\mathrm{p}(X)\to \mathrm{Br}'(X)[2]\] to the $2$-torsion subgroup of the cohomological Brauer group is surjective. 
    The square \[
    \begin{tikzcd}
        \mathrm{Br}(X,\mathrm{id}) \ar[r] \ar[d] & \mathrm{Br}'(X)[2] \ar[d,"="]\\
        \mathrm{Br}^\mathrm{p}(X) \ar[r] & \mathrm{Br}'(X)[2]
    \end{tikzcd}
    \]
    commutes, and we know that the top map fails to be surjective by \cite[Corollary 3.11]{br_neq_br_prime} (this class is $2$-torsion by \cite[Theorem 3.6(iv)]{br_neq_br_prime} together with the fact that the stack they construct in \cite[Example 2.21]{br_neq_br_prime} is a $\mu_2$-torsor.). 
    In particular whatever Azumaya algebra with genuine involution corresponds to this class under Corollary \ref{cor: Brauer classes represented by Azumaya algebras with genuine involution} cannot be Morita equivalent to an ordinary Azumaya algebra with involution. 
\end{example}

\subsection{\texorpdfstring{Comparison to \cite{Parimala_Srinivas}}{Comparison to Parimala-Srinivas}}
We will now show that in the cases where \cite{Parimala_Srinivas} and Srinivas were able to define an involutive Brauer group, their construction agrees with ours; moreover, the natural cohomological invariants that they construct live in our cohomological involutive Brauer group. For this subsection, we will use a subscript $(-)_{\mathrm{PS}}$ to denote those invariants which were defined by Parimala and Srinivas.
			
\begin{theorem}~\label{thm: PS comparison in text}
    Let $(X,\lambda,Y,p)$ be a scheme with good quotient such that $2\in \mathbb{G}_m(Y)$. Then 
    \begin{enumerate}[label=(\arabic*)]
        \item If $p:X\to Y$ is \'etale, then the map $\mathrm{Br}_{\mathrm{PS}}(X,\lambda)\to \mathrm{Br}(X,\lambda)$ of Proposition~\ref{prop:from_involutive_Brauer_to_Poincare_Brauer} is an equivalence, where $\mathrm{Br}_{\mathrm{PS}}$ is the involutive Brauer group of Parimala-Srinivas (see Construction~\ref{const: involutive BG of type II} for a review of this functor). Furthermore, in this case we also have that the map $\mathrm{Br}^\mathrm{p}(X,\lambda, Y,p)\to \mathrm{H}^1(\mathrm{R}\Gamma_{\acute{e}t})$ is split.
        \item If $\lambda = p =\mathrm{id}_X$, then the map $\mathrm{Br}_{\mathrm{PS}}(X,\mathrm{id}_X)\to \mathrm{Br}(X,\lambda)$ is an equivalence where the map is the one constructed in Proposition~\ref{prop:from_involutive_Brauer_to_Poincare_Brauer} and $\mathrm{Br}_{\mathrm{PS}}$ is the involutive Brauer group of Parimala-Srinivas (see Construction~\ref{const: involutive BG type I} for a review of this object). Furthermore, there is also an equivalence $\mathrm{Br}'(X,\lambda)\cong H^0_{\acute{e}t}(X,\mu_2)\times H^2_{\Acute{e}t}(X;\mu_2)$. 
    \end{enumerate}
\end{theorem}
\begin{proof}
    We will first handle the case when $p $ is a degree $2$ \'etale extension. Note that in this case we have that $\mathrm{H}^1(\mathrm{R}\Gamma_{\acute{e}t}(X;\mathbb{Z})^{\mathrm{h}(-\lambda)})\cong \mathrm{H}^{0}_{\acute{e}t}(Y;C)$ where $C:=\mathrm{Cofib}(\mathrm{Tr}_{X/Y}:p_*\mathbb{Z}\to \mathbb{Z})$ by Lemma~\ref{lem: some computations of the extra junk in brp}. 
    By Corollary \ref{cor: extended fiber sequence of pnbr to br good quotient}, we have that $\mathrm{Br}^{\mathrm{p}}(X,\lambda,Y,p)$ fits into an exact sequence \[\mathrm{Pic}(X)\to \mathrm{Pic}(Y)\to \mathrm{Br}^\mathrm{p}(X,\lambda,Y,p)\to \mathrm{Br}(X)\to \mathrm{Br}(Y)\] of abelian groups. Note that the Eilenberg-MacLane functor sends Galois extensions $A\to B$ of rings to faithful Galois extensions of ring spectra $HA\to HB$ since both are conditions on $B\otimes_A B$ and $HB\otimes_{HA}HB$, and the tensor product of faithfully flat rings is the derived tensor product of faithfully flat rings.  What is more, by the identification of the connecting maps in Theorem~\ref{cor:Poincare_Brauer_to_Brauer_fiber} we see that the map $\mathrm{Br}_{\mathrm{PS}}\to \mathrm{Br}^{\mathrm{p}}$ extends to a map of exact sequences from the exact sequence of \cite[Theorem 1.2]{Parimala_Srinivas}. 
    Note that we have a commutative diagram 
    \[\begin{tikzcd}
        p_*\mathbb{Z}_X\ar[d] \ar[r,"\mathrm{Tr}_{X/Y}"] & \mathbb{Z}_Y\ar[d] & \\
        p_*\mathfrak{pic}_{X} \ar[r] & \mathfrak{pic}_Y \ar[r] & \pnbr(X,\lambda, Y ,p) 
    \end{tikzcd}\]
    where the vertical maps come from the identification of Theorem~\ref{theorem:fausk_for_discrete_rings} and $\mathfrak{pic}_X$ and $\mathfrak{pic}_Y$ denote the \'etale sheaves $U\mapsto \mathfrak{pic}(U)$ on $X$ and $Y$, respectively. 
    This diagram then induces a map $\mathrm{coker}(p _*\mathbb{Z}_X\to \mathbb{Z}_Y)\to \mathfrak{br}^\mathrm{p}(X,\lambda, Y, p)$ of \'etale sheaves on $Y$, and upon taking $\pi_0$ we get the desired splitting, even as sheaves. 
    To show that the involutive Brauer group we have constructed in Definition~\ref{defn:cohomological_involutive_Brauer_group} does indeed match the one constructed by Parimala and Srinivas, note that we have a map of long exact sequences of groups 
    \[
    \begin{tikzcd}
        \Pic^{\mathrm{cl}}(X) \ar[r] \ar[d,"\cong"] & \Pic^{\mathrm{cl}}(Y) \ar[r] \ar[d,"\cong"] & \mathrm{Br}_{\mathrm{PS}}(X,\lambda) \ar[r] \ar[d] & \mathrm{Br}(X)\ar[r] \ar[d,"\cong"] & \mathrm{Br}(Y)\ar[d,"\cong"]\\
        \Pic^{\mathrm{cl}}(X) \ar[r] & \Pic^{\mathrm{cl}}(Y)\ar[r] & \mathrm{Br}(X,\lambda) \ar[r] & \br(X) \ar[r] & \br(Y)\,.
    \end{tikzcd}
    \]
    Hence by the 5-lemma we have an equivalence as desired.
    
    Now consider the case of the trivial action. It is enough to show the second assertion that $\mathrm{Br}'(X,\mathrm{id}_X)\cong H^0_{\Acute{e}t}(X;\mu_2)\times H^2_{\Acute{e}t}(X;\mu_2)$. 
    The fiber sequence of Corollary \ref{cor: extended fiber sequence of pnbr to br good quotient} induces a long exact sequence \[\ldots \to \pi_0(\mathfrak{pic}(X))\to \pi_0(\mathfrak{pic}(X)^{\mathrm{h}C_2})\to \mathrm{Br}^\mathrm{p}(X,\mathrm{id}_X)\to \mathrm{Ker}\left(\mathrm{Br}(X)\to \Pic\left(\mathrm{Mod}_{\mathrm{Mod}_{\mathcal{O}_X}^{\omega, \mathrm{h}C_2}}\right)\right)\to 0\,.\] 
    To identify this last term, note that the $C_2$ action on $\mathrm{Mod}_{\mathcal{O}_X}$ here is trivial by assumption, and so we have that $\mathrm{Mod}_{\mathcal{O}_X}^{\omega, \mathrm{h}C_2}\simeq \mathrm{Mod}_{\mathcal{O}_X}^{\omega, \mathrm{B}C_2}$. 
    A class $\mathcal{C}\in \mathrm{Br}(X)$ is then sent to $(\mathcal{C}\otimes_{\mathrm{Mod}_{\mathcal{O}_X}^\omega}\mathcal{C})^{\mathrm{h}C_2}$ in $\mathrm{Mod}_{\mathrm{Mod}_{\mathcal{O}_X}^\omega}$. 
    By \cite{MR2957304}, there is an $ \mathcal{O}_X $-linear equivalence $\mathcal{C}\simeq\mathrm{Mod}_{\mathcal{A}}^\omega$ for some sheaf of (generalized) Azumaya $ \mathcal{O}_X $-algebras $ \mathcal{A} $. 
    By assumption, $(\mathcal{C}\otimes_{\mathrm{Mod}_{\mathcal{O}_X}^\omega}\mathcal{C})^{\mathrm{h}C_2}$ represents the trivial class; hence we have an $\mathrm{Mod}_{\mathcal{O}_X}^{\omega, \mathrm{B}C_2}$-linear equivalence $G \colon \mathrm{Mod}_{\mathcal{A}\otimes\mathcal{A}}^{\mathrm{h}C_2}\simeq \mathrm{Mod}_{\mathcal{O}_X}^{\omega, BC_2}$. 
    Now let us observe that $ \mathcal{A}\otimes \mathcal{A} \in \mathrm{Mod}_{\mathcal{A}\otimes\mathcal{A}} $ lifts canonically to an object of the homotopy fixed points $\mathrm{Mod}_{\mathcal{A}\otimes\mathcal{A}}^{\mathrm{h}C_2}$; informally, this is given by endowing $ \mathcal{A}\otimes \mathcal{A} $ with the swap action. 
    Since $ G $ is $ \mathrm{Mod}_{\mathcal{O}_X}^{\omega, \mathrm{B}C_2} $-linear, $\mathcal{A}\otimes \mathcal{A}$ must be sent to some object $\mathcal{M}\in \mathrm{Mod}_{\mathcal{O}_X}^{\omega, \mathrm{BC_2}}$ with internal endomorphism object given by $\mathcal{A}\otimes \mathcal{A}$. 
    Since $ G $ is an equivalence, $\mathcal{O}_X$ with the trivial action must be in the thick subcategory spanned by the image of $\mathcal{A}\otimes \mathcal{A}$, and so in particular $\mathrm{Mod}_{\mathcal{A}\otimes \mathcal{A}}^\omega \simeq \mathrm{Mod}_{\mathcal{O}_X}^\omega$. Consequently every class in the kernel of the map $\mathrm{Br}(X)\to \mathrm{Pic}(\mathrm{Mod}_{\mathrm{Mod}_{\mathcal{O}_X}^\omega})$ is $2$-torsion, and conversely every $2$-torsion class is in the kernel automatically. Thus $\mathrm{Ker}(\mathrm{Br}(X)\to\mathrm{Pic}(\mathrm{Mod}_{\mathrm{Mod}_{\mathcal{O}_X}^\omega}) )\cong \mathrm{Br}(X)[2]$, where by $[2]$ we mean the $2$-torsion subgroup (as opposed to the shift). 
    Additionally we have a map of long exact sequences 
    \[
    \begin{tikzcd}
        \mathrm{Pic}(X)\ar[r,"{(\times 2,0)}"] \ar[d] & \mathrm{Pic}(X)\times \mathrm{H}_{\acute{e}t}^0(X,\mu_2) \ar[r] \ar[d] & \mathrm{Br}^p(X,\mathrm{id}_X) \ar[r] \ar[d] & \mathrm{Br}(X)[2] \ar[d] \\
        \mathrm{H}^0_{\acute{e}t}(X;\mathbb{Z}) \ar[r,"\times 2"] & \mathrm{H}^0_{\acute{e}t}(X;\mathbb{Z}) \ar[r] & \mathrm{H}^1(\mathrm{R}\Gamma_{\acute{e}t}(X;\mathbb{Z})^{\mathrm{h}C_2})\ar[r]& \mathrm{H}^1_{\acute{e}t}(X;\mathbb{Z})[2]
    \end{tikzcd}
    \]
    by the construction of the second from the right vertical map in Construction~\ref{const: map from poincare brauer to homotopy fixed points} and Lemma~\ref{lem: some computations of the extra junk in brp}. Taking vertical kernels and noting the all of the vertical maps are surjective gives a long exact sequence \[\ldots \to \Pic^{\mathrm{cl}}(X)\xrightarrow{(\times 2,0)}\Pic^{\mathrm{cl}}(X)\times \mathrm{H}^0_{\Acute{e}t}(X;\mu_2)\to \mathrm{Br}'(X,\mathrm{id}_X)\to \mathrm{Br}'(X)[2]\to 0\,.\] 
    We then get a short exact sequence \[0\to \Pic^{\mathrm{cl}}(X)/2\times \mathrm{H}^0(X;\mu_2)\to \mathrm{Br}'(X,\lambda)\to \mathrm{Br}'(X)[2]\to 0\] which splits since the terms involved are all $\mathbb{F}_2$ vector spaces. We therefore get a (non-canonical) identification $\mathrm{Br'(X,\mathrm{id}_X)}\cong H^0_{\Acute{e}t}(X,\mu_2)\times H^2_{\Acute{e}t}(X;\mu_2)$ as desired. 
\end{proof}
		
As a consequence of the above proof we also get an extension of Saltman's theorem in \cite[Theorem 3.1]{MR495234} to the spectral case.

\begin{theorem}~\label{thm: Saltman theorem in text}
    Let $R$ be a Poincar{\'e} ring such that $R^{\phi C_2}\simeq 0$. Let $A$ be an Azumaya algebra over $R^e$. Then
    \begin{enumerate}
        \item if the $C_2$-action on $R$ is trivial, then $A$ is Morita equivalent to an Azumaya algebra admitting a lift to an Azumaya algebra with genuine involution if and only if $2[A]=0$ in $\mathrm{Br}(R^e)$, and
        \item if the $C_2$-action on $R$ is faithful Galois, then $A$ is Morita equivalent to an Azumaya algebra admitting an Azumaya algebra with genuine $\lambda$-involution if and only if $[(A\otimes_R \lambda^*A)^{\mathrm{h}C_2}]=0$ in $\mathrm{Br}(R^{\mathrm{h}C_2})$.
    \end{enumerate}
\end{theorem}
\begin{proof}
    This follows from the fiber sequence \[\mathfrak{br}^\mathrm{p}(R)\to \mathfrak{br}(R^e)\to \mathfrak{br}(R^L)\] of Theorem~\ref{thm: extended fiber sequence of pnbr to br} and the same analysis as in Theorem~\ref{thm: PS comparison in text}.
\end{proof}

\begin{remark}\label{remark:tC2_vanishing_for_trivial_action_is_invertibility_of_2}
    By Remark \ref{remark: Borel poincare rings have vanishing tate}, the conditions of Theorem \ref{thm: Saltman theorem in text} hold if and only if $R$ arises as the symmetric Poincar\'e ring of an $\Einfty$-ring spectrum $R^e$ with a $C_2$-action such that $(R^e)^{tC_2}\simeq 0$. If the action on $R^e$ is trivial, then the Tate construction vanishes $(R^e)^{t C_2}\simeq 0$ if and only if $2$ acts invertibly on $R^e$; see the last remark of \cite{burghardt_poincare}. The first case of Theorem \ref{thm: Saltman theorem in text} is an analogue of Saltman's theorem \cite[Theorem 3.1]{MR495234} for $\Einfty$-ring spectra. Saltman's result does not require $2$ to be invertible, and at present we have no evidence that this requirement is strictly necessary in Theorem \ref{thm: Saltman theorem in text}. We hope to remove it in future work.
\end{remark}
		
\newpage
\appendix 

\section{Poincaré structures on module categories}\label{appendix:calgp_to_catp}
The assignment $ A \mapsto \Mod_A $ sending an associative algebra to its category of left modules is functorial in $ A $: Given a map of algebras $ A \to B $, $ B \otimes_A (-) $ defines a functor $ \Mod_A \to \Mod_B $. 
This observation is simultaneously basic and fundamental: it allows one to define quasi-coherent sheaves on a scheme $ X $ as a limit over affine open covers, and to formulate effective descent for modules. 
In this appendix, we consider what it means to extend this functoriality to module categories with Poincaré structure. 

Note that a lift of $ \Mod_A^\omega $ to $ \Catp $ or $ \Cath $ is additional data: by \cite[Theorem 3.3.1]{CDHHLMNNSI}, this equivalent to the data of a module with genuine involution over $ A $. 
We begin in \S\ref{subsection:c2_operad_preliminaries} by defining $ \infty $-categories of algebras with the required additional structure to equip their module categories with this extra structure; our constructions come in two versions: a (homotopy coherently) associative one which we call $ \EE_\sigma $-algebras, and a (homotopy coherently) commutative one, which we call $ \EE_p $-algebras. 
In order to handle the coherences involved, we use the language of parametrized $ \infty $-operads and higher algebra developed by Nardin--Shah in \cite{NS22} for $ \mathcal{T} = O_{C_2} $ the orbit category of $ C_2 $. 
In \S\ref{subsection:Ep_alg_to_Catp}, we construct a functor from $ \EE_\sigma $-algebras to $ \Cath $ and show that it factors through $ \Catp $. 
We show that the restriction of this functor to $ \EE_p $-algebras admits a symmetric monoidal refinement.  
Finally, in \S\ref{subsection:calgp_operadic} we show that $ \CAlgp $ is equivalent to $ \EE_p $-algebras. 

The third author asserted (rather cavalierly) in \cite[\S5.2]{LYang_normedrings} that there is a functorial assignment from $ C_2 $-$ \EE_\infty $-algebras to Poincaré $ \infty $-categories; Theorem \ref{thm:calgp_to_poincare_cat} should be regarded as a more correct/precise demonstration of that assertion. 

\subsection{Homotopy coherent algebras with involution}\label{subsection:c2_operad_preliminaries}
In this section, we introduce various $ \infty $-categories of algebras equipped with the requisite data to endow their module categories with Poincaré structures. 
This work is necessarily of a technical nature and the reader is encouraged to either (a) skip this section on a first read and refer back to it as necessary and/or (b) skim this section, referring to the informal description of Remark \ref{rmk:informal_description_C2_operad_alg} and Theorem \ref{thm:operadic_diagrammatic_calgp_agree}, which makes the corresponding informal description of $ \EE_p $-algebras precise. 
For the interested reader, convincing oneself that the informal description contained in Remark \ref{rmk:informal_description_C2_operad_alg} can be an enlightening exercise. 
We will find it useful to work with a combinatorial model of the $ \EE_\sigma $-operad, whose properties were elucidated in \cites{MR4132956,MR4389764}. 

\begin{notation}\label{ntn:t_order_equivariant_map}
	Let $ S, T $ be finite $ C_2 $-sets, and let $ f \colon S\to T $ be a $ C_2 $-equivariant map. 
	A \emph{t-ordering} on $ f $ is the data of, for all free orbits $ V \subset S $, $ V \simeq C_2 $ so that $ f(V) = \{*\} $, a choice of an ordering $ \le $ on $ V $. 
	A \emph{t-ordered map $ S \to T $} is the data of $ f $ and a t-ordering on $ f $.  
	
	Note that if $ f_i \colon S_i \to T_i $ are all t-ordered, there is a canonical t-ordering on $ \displaystyle \bigsqcup_i f_i $. 
	If $ f \colon S \to T $ is t-ordered and $ g \colon T \to U $ is t-ordered, there is a canonical t-ordering on $ g \circ f $. 
	Let us spell this out: Suppose we are given an orbit $ V \subset S $ on which $ C_2 $ acts freely so that $ g\circ f(V) $ is trivial. 
	There are two cases
	\begin{itemize}
		\item Suppose $ f $ restricts to an isomorphism $ f|_V \colon V \xrightarrow{\sim} f(V) $ and $ g $ sends $ f(V) $ to $ * $. 
		Then the t-ordering on $ g $ means we have an ordering on $ f(V) $, which we transport to an ordering on $ V $ via $ f|_V $. 
		\item Suppose $ f(V) = * $. 
		Then the t-ordering on $ f $ endows $ V $ with a canonical ordering.  
	\end{itemize} 
	Finally, we define a $ t $-ordering on pullbacks as follows: Suppose we are given a pullback square of $ C_2 $-sets. 
	\begin{equation*}
		\begin{tikzcd}
			Z \times_X W \ar[r,"{g^*f}"] \ar[d,"{\pi_1}"] &  W \ar[d,"{g}"] \\
			Z \ar[r,"f"] & X        
		\end{tikzcd}     
	\end{equation*} 
	and a t-ordering on $ f $. 
	Suppose we are given a free orbit $ C_2 \simeq U \subseteq Z \times_X W $ so that $ g^*(f)(U) $ is a singleton (with trivial $ C_2 $-action). 
	Because the square above is a pullback, $ \pi_1(U) \subset Z $ is acted on by $ C_2 $ freely, and $ \pi_1|_U $ defines an isomorphism $ U \simeq \pi_1(U) $. 
	By the $ C_2 $-equivariance of $ g $, $ g(g^*(f)(U)) = f(\pi_1(U)) $ is a singleton. 
	The given t-ordering on $ f $ means $ \pi_1(U) $ has a given ordering, which we use to endow $ U $ with an ordering using the isomorphism $ \pi_1|_U $. 
\end{notation}
\begin{remark}\label{rmk:t_order_underlying}
	In the following pullback diagram in $ \Fin_{C_2} $, both t-orderings on $ f $ induce the same t-ordering on $ g^*(f) $ 
	\begin{equation*}
		\begin{tikzcd}
			C_2 \times C_2 \ar[r] \ar[d] & C_2 \ar[d,"{g}"] \\
			C_2 \ar[r,"{f}"] & C_2/C_2        
		\end{tikzcd}     
	\end{equation*}     
	(in fact the set of t-orderings on $ g^*(f) $ is a singleton.)
\end{remark}
Let $ \Ar^{\mathrm{cart}}\left(\Fin_{C_2}\right) $ denote the wide subcategory of the arrow category whose morphisms are cartesian squares. 
\begin{definition}\label{defn:gen_alg_inv_simplicial_operad}
	Define a colored $ \mathcal{O}_{C_2} $-operad $ \mathbf{\Assoc}_{\sigma} \to \underline{\mathrm{Fin}}_{C_2, *} $ as having 
	\begin{enumerate}[label=(\arabic*)]
		\item $ \mathrm{ob}\mathbf{\Assoc}_{\sigma} \to \mathrm{Fin}_{C_2} $ is classified by a functor $ \mathrm{Fin}_{C_2}^\op \to \mathrm{Set} $ which sends $ C_2 $ and $ C_2/C_2 $ to a singleton  
		\item Multimorphisms are given by 
		\begin{equation}
			\begin{split}
				\mathrm{Mul}_{\mathbf{\Assoc}_{\sigma}} \colon \Ar^{\mathrm{cart}}\left(\Fin_{C_2}\right) & \to \mathrm{Set} \\
				(f : U \to V) & \mapsto \left\{ \left( \leq_v \text{ ordering on }f^{-1}(\{v\})\right)_{v \in V}\right\}^{C_2}
			\end{split}
		\end{equation}
		where $ (-)^{C_2} $ denotes fixed points with respect to the following involution: The generator $ \sigma \in C_2 $ acts on $ \left\{ \left( \leq_v \text{ ordering on }f^{-1}(\{v\})\right)_{v \in V}\right\} $ by $ \left(\leq_v\right)_{v \in V} \mapsto \left(\geq_{\sigma(v)} \right)_{v \in V} $ (here $ \geq_v $ denotes the reverse of the ordering $ \leq_v $ and we use the isomorphism $\sigma \colon f^{-1}( \{v\}) \simeq f^{-1}( \{\sigma(v)\})$ to transport an ordering on the latter to the former).  
		\item The identity $ \ast\to \mathrm{Mul}_{\mathbf{\Assoc}_{\sigma}} $ is determined by observing that $ \left\{ \left( \leq_v \text{ ordering on }\id^{-1}(\{v\})\right)_{v \in V}\right\}^{C_2} \simeq * $ 
		\item The composition law $ \circ : \mathrm{Mul}_{\mathbf{\Assoc}_{\sigma}} \times \mathrm{Mul}_{\mathbf{\Assoc}_{\sigma}} \to \mathrm{Mul}_{\mathbf{\Assoc}_{\sigma}} $ is defined by composing maps in $ \Fin_{C_2} $ and taking the lexicographic order on compositions (see \cite[Remark 4.1.1.4]{LurHA}).
	\end{enumerate}
	Postcomposing $ \mathrm{Mul}_{\mathbf{\Assoc}_{\sigma}} $ with the constant simplicial set functor gives a fibrant simplicial $ C_2 $-operad in the sense of \cite[Definition 2.5.4]{NS22}, which we also denote by $ \mathbf{\Assoc}_{\sigma} $. 
\end{definition}
\begin{definition}\label{defn:calgp_simplicial_operad}
	Define a colored $ \mathcal{O}_{C_2} $-operad $ \mathbf{E}_{p} \to \underline{\mathrm{Fin}}_{C_2, *} $ as having 
	\begin{enumerate}[label=(\arabic*)]
		\item $ \mathrm{ob}\mathbf{E}_{p} \to \mathrm{Fin}_{C_2} $ is classified by a functor $ \mathrm{Fin}_{C_2}^\op \to \mathrm{Set} $ which sends $ C_2 $ and $ C_2/C_2 $ to a singleton  
		\item Multimorphisms are given by 
		\begin{equation*}
			\begin{split}
				\mathrm{Mul}_{\mathbf{E}_{p}} \colon \Ar^{\mathrm{cart}}\left(\Fin_{C_2}\right) & \to \mathrm{Set} \\
				(f : U \to V) & \mapsto \left\{\text{ t-ordering on }f\right\} 
			\end{split}
		\end{equation*}
        where t-orderings were defined in Notation \ref{ntn:t_order_equivariant_map}. 
		\item The identity $ \ast\to \mathrm{Mul}_{\mathbf{E}_{p}} $ is determined by observing that $ \left\{\text{ t-ordering on }\id_U\right\} \simeq * $ 
		\item The composition law $ \circ : \mathrm{Mul}_{\mathbf{E}_{p}} \times \mathrm{Mul}_{\mathbf{E}_{p}} \to \mathrm{Mul}_{\mathbf{E}_{p}} $ is defined by composing maps in $ \Fin_{C_2} $ and endowing composites with t-orders as described in Notation \ref{ntn:t_order_equivariant_map}.
	\end{enumerate}
	Postcomposing $ \mathrm{Mul}_{\mathbf{E}_{p}} $ with the constant simplicial set functor gives a fibrant simplicial $ C_2 $-operad in the sense of \cite[Definition 2.5.4]{NS22}, which we also denote by $ \mathbf{E}_{p} $. 
	
	Before shifting to the $\infty$-categorical setting, we will recall some notation. Specifically, in what follows, we will refer to the following diagram
	\begin{equation}\label{diagram:C2_operad_generic_span}
		\begin{tikzcd}
			U \ar[d] & \ar[l,"g"'] Z \ar[r,"f"]  \ar[d] & X \ar[d] \\
			V &  \ar[l] Y \ar[r,equals] & Y
		\end{tikzcd}
	\end{equation} 
	in $ \underline{\Fin}_{C_2} $ or $ \underline{\Fin}_{C_2,*} $ as a representative of a general span in $\underline{\Fin}_{C_2,*}$ (see \cite[\S2.1]{NS22}). 
	In particular, recall that in $ \underline{\Fin}_{C_2,*} $, the induced map $ Z \to U \times_V Y $ is always assumed to be a summand inclusion.
\end{definition}
\begin{definition}\label{defn:gen_alg_inv_param_operad}
	Define $ \underline{\Assoc}_{\sigma}^\otimes \to \underline{\mathrm{Fin}}_{C_2, *} $ to be the $ C_2 $-$ \infty $-operad associated to $ \mathbf{\Assoc}_{\sigma} $ under \cite[Construction 2.5.5 \& Proposition 2.5.6]{NS22}. 
	In particular, it has 
	\begin{itemize}
		\item objects consist of maps $ (h : U \to V) $ where $ h: U \to V $ is an object of $ \underline{\mathrm{Fin}}_{C_2, *} $.  
		\item a morphism from $ (h : U \to V) $ to $ (h' : X \to Y) $ consists of 
		\begin{itemize}
			\item a span (\ref{diagram:C2_operad_generic_span}) in $ \underline{\mathrm{Fin}}_{C_2, *} $
			\item for each $ x \in X $, an ordering $ \leq_x $ on $ f^{-1}(\{x\}) $ satisfying the condition (write $ \sigma \in C_2 $ for the generator): under the isomorphism $ f^{-1}(\{x\}) \simeq f^{-1}(\{\sigma(x)\}) $ induced by the $ C_2 $-action, the ordering $ \leq_x $ is sent to the reverse of the ordering $ \leq_{\sigma(x)} $.  
		\end{itemize} 
	\end{itemize}
	there is a canonical map to $ \underline{\mathrm{Fin}}_{C_2, *} $ which is an inclusion on objects and on morphisms forgets the data of the orderings. 
\end{definition}

\begin{definition}\label{defn:Poincare_ring_param_operad}
	Define $ \mathbb{E}_{\mathrm{p}}^\otimes \to \underline{\mathrm{Fin}}_{C_2, *} $ to be the $ C_2 $-$ \infty $-operad associated to $ \mathbf{E}_{p} $ under \cite[Construction 2.5.5 \& Proposition 2.5.6]{NS22}. 
	In particular, 
	\begin{itemize}
		\item objects agree with objects of $ \underline{\Fin}_{C_2,*} $ 
		\item morphisms consist of spans (\ref{diagram:C2_operad_generic_span}) in $ \underline{\mathrm{Fin}}_{C_2, *} $, plus the data of a t-ordering on $ f $ in the sense of Notation \ref{ntn:t_order_equivariant_map}.  
		\item composition is defined to agree with that in $ \underline{\mathrm{Fin}}_{C_2, *} $, with t-orderings on composites given by Notation \ref{ntn:t_order_equivariant_map}. 
	\end{itemize}
	There is a canonical map to $ \underline{\mathrm{Fin}}_{C_2, *} $ which is an equivalence on objects and on morphisms forgets t-orderings. 
\end{definition} 
\begin{remark}\label{rmk:surprise_lower_category}
	By \cite[Construction 2.5.5 \& Proposition 2.5.6]{NS22}, $ \underline{\Assoc}_{\sigma}^\otimes $ and $ \EE_p^\otimes $ are equivalent to nerves of 1-categories. 
\end{remark}
\begin{remark}\label{rmk:informal_description_C2_operad_alg}
    Suppose $ \mathcal{C}^\otimes $ is a $ C_2 $-symmetric monoidal $ C_2 $-$ \infty $-category in the sense of \cite[Definition 2.2.3]{NS22}. 
    Write $ \mathcal{C}^{C_2} $ and $ \mathcal{C}^e $ for the value of the underlying $ C_2 $-$ \infty $-category $ O_{C_2}^\op \to \Cat $ on $ C_2/C_2 $ and $ C_2/e $, resp. 
    In particular, $ \mathcal{C}^{C_2} $ is a symmetric monoidal $ \infty $-category, $ \mathcal{C}^e $ is a symmetric monoidal $ \infty $-category with an involution $ \lambda \colon \mathcal{C}^e \simeq \mathcal{C}^e $ (via symmetric monoidal functors), there are symmetric monoidal restriction $ (-)^e \colon \mathcal{C}^{C_2} \to \mathcal{C}^{e} $ and norm $ N \colon \mathcal{C}^{e} \to \mathcal{C}^{C_2} $ maps, and these satisfy higher coherences... 
    Informally, an $ \underline{\Assoc}_{\sigma}^\otimes $-algebra in $ \mathcal{C}^\otimes $ consists of a object $ A $ in $ \mathcal{C}^{C_2} $ whose underlying object $ A^e $ has the structure of an $ \EE_1 $-algebra in $ \mathcal{C}^e $ and an anti-involution $ A^e \simeq \lambda(A^e)^\op $. 
    Furthermore, $ A $ is equipped with the structure of a module over $ N(A^e) $ and a morphism $ N(A^e) \to A $ of $ N(A^e) $-modules which, upon applying the functor $ (-)^e $, is required to be compatible with the given anti-involution on $ A^e $. 
\end{remark}
\begin{recollection} [Little disks operad]
	Let $ V=\mathbb{R}^\sigma $ denote the sign $ C_2 $-representation on $ \mathbb{R} $. 
	Let $ D_V^\otimes $ denote the simplicial $ C_2 $-operad obtained by taking $ \mathrm{Sing}_\bullet $ of the little disks operad associated to $ V $ \cite[Definition 1.1]{GuillouMay}. 
	Write $ \EE_\sigma^\otimes := N\left(D_V^\otimes\right) $ where $ N $ is the genuine operadic nerve \cites[\S4.1]{MR4011810}[Example 3.9.4]{Horev19}; $ \EE_\sigma^\otimes $ is a $ C_2 $-$ \infty $-operad. 
\end{recollection}
The following lemma shows that Definition \ref{defn:gen_alg_inv_param_operad} is a combinatorial model for the little disks operad associated to the sign representation of $ C_2 $. 
We make no claim to originality for the construction of $\mathbb{E}_\sigma$.
\begin{lemma}\label{lemma:little_disks_combinatorial_model}
    There is an equivalence $ \EE_\sigma^\otimes \simeq \underline{\Assoc}_\sigma^\otimes $ of $ C_2 $-$ \infty $-operads. 
\end{lemma}
\begin{proof} 
    Fix a $ C_2 $-equivariant map $ f \colon U \to V $ and consider $ \varphi_f \in \mathrm{Emb}^{C_2,\mathscr{S}}_f(U \times D(\mathbb{R}^\sigma), V \times D(\mathbb{R}^\sigma)) $ in the notation of \cite[\S2.1.1]{Stewart_Dunn}, where $ \mathscr{S} $ denotes the `equidiameter' embeddings of Remark 2.3 \emph{loc. cit.} 
    For any $ v \in V $, such $ \varphi_f $ determines a linear ordering $ \leq_v $ on $ \pi_0 \varphi_f^{-1}( \{v\} \times D(\mathbb{R}^\sigma)) \simeq f^{-1}(v) $. 
    Moreover, since $ \varphi_f $ is $ C_2 $-equivariant, under the isomorphism $ \pi_0 \varphi_f^{-1}( \{v\} \times D(\mathbb{R}^\sigma)) \simeq \pi_0 \varphi_f^{-1}( \{\sigma(v)\} \times D(\mathbb{R}^\sigma)) $, the linear ordering $ \leq_v $ is sent to the linear ordering $ \geq_{\sigma(v)} $. 
    We also note that each connected component of $ \mathrm{Emb}^{C_2,\mathscr{S}}_f(U \times D(\mathbb{R}^\sigma), V \times D(\mathbb{R}^\sigma)) $ is homeomorphic to a nonempty convex subset of Euclidean space and is therefore contractible. 
    Thus taking connected components on multimorphism spaces induces a weak equivalence $ D_V^\otimes \to \mathbf{Assoc}^\otimes_\sigma $ of simplicial $ C_2 $-operads and that there is an equivalence $ \EE_\sigma^\otimes \simeq \underline{\Assoc}_\sigma^\otimes $ of $ C_2 $-$ \infty $-operads. 
\end{proof}
\begin{remark}\label{rmk:E_sigma_operad_unital}
    The $ C_2 $-$ \infty $-operads $ \underline{\Assoc}_\sigma^\otimes $ and $ \EE_p^\otimes $ are unital in the sense of \cite[Definition 5.2.1]{NS22}. 
    Let $ \mathcal{C}^\otimes $ be a $ C_2 $-symmetric monoidal $ C_2 $-$ \infty $-category. 
    It follows from Proposition 5.2.8 and Theorem 5.2.11 \emph{ibid.} that the unit object $ 1 $ of $ \mathcal{C} $ has a canonical structure of an $ \EE_\sigma $ (resp. $ \EE_p $-)algebra, and moreover with this structure $ {1} $ is the initial object in $ \EE_\sigma\Alg\left(\mathcal{C}\right) $, (resp. $ \EE_p\Alg\left(\mathcal{C}\right) $). 
\end{remark}
\begin{observation}\label{obs:E_sigma_param_operad_to_calgp_param_operad}
    Suppose we are given a finite $ C_2 $-set $ Z $ and an ordering $ \leq_Z $ so that the $ C_2 $-action on $ Z $ sends $ \leq_Z $ to its reverse. 
    We may associate to this a canonical t-ordering on the map $ f \colon Z \to C_2/C_2 $: For each free orbit $ U \subseteq Z $, take the restriction of $ \leq_Z $ to $ U $. 
    This assignment induces a canonical forgetful map of $ C_2 $-$ \infty $-operads $ \underline{\Assoc}_\sigma^\otimes \to \EE_p^\otimes $.  
\end{observation}
The following observation will allow us to write down functors between the $ \infty $-operads we care about without worrying about higher coherences. 
\begin{observation}\label{obs:operad_maps_to_infty_operad_maps}
    Let $ \mathcal{T} $ be an atomic orbital category with a final object, and let $ T $ be a 1-categorical model for $ \mathcal{T} $ (\cite[Proposition 2.5.1]{NS22}). 
    Fix an object $ U \in \mathcal{T} $, and suppose we are given a simplicial colored $ T $-operad $ O $ in the sense of Definition 2.5.4 \emph{loc. cit.} 
    We may then associate to this data a colored operad $O_U$ with the following features:
    \begin{enumerate}[label=(\arabic*)]
        \item the set of colors is the value of $ \mathbf{F}^\op_T \to \mathrm{Set} $ on $ U $
        \item for each finite set $ I $ and $ I $-indexed collection of objects $ \{X_i\} $ and $ Y \in \mathrm{ob}O $, the multimorphisms of $O_U$ are those given by evaluating $ \mathrm{Mul}_O $ on $ (\nabla \colon U^I \to U, (X_i)_{i \in I}, Y) $, where $ \nabla $ denotes a fold map.  
        \item the composition law is defined to be the one compatible with $\mathbf{F}_T$ and with the composition on $\mathrm{Mul}_O$. Similarly the unit $\ast\to \mathrm{Mul}_{O_U}$ is determined by the map $\ast\to \mathrm{Mul}_O$,
    \end{enumerate}
    Now suppose we are given a simplicial colored operad $ P $ and a map of simplicial colored operads $ P \to O_U $. 
    Then there are maps of $ \infty $-categories 
    \begin{equation*}
        N^\otimes(P) \to N^\otimes\left(O_U\right) \to N^\otimes (O) 
    \end{equation*} 
    where the left two instances of $ N^\otimes $ denote the operadic nerve of \cite[Definition 2.1.1.23]{LurHA} and the latter denotes the genuine operadic nerve \cites[\S4.1]{MR4011810}. 
    Moreover, the composite $ N^\otimes(P) \to N^\otimes(O) $ sends inert maps in the $ \infty $-operad $ N^\otimes(P) $ to inert maps in the $ \mathcal{T} $-$ \infty $-operad $ N^\otimes(O) $ (cf. proof of Theorem I \emph{ibid.}). 
\end{observation}

\subsection{From Poincaré rings to Poincaré \texorpdfstring{$ \infty $}{∞}-categories} \label{subsection:Ep_alg_to_Catp}
This subsection is devoted to constructing a functor which takes an $\EE_\sigma$ algebra in genuine $C_2$ spectra (or an $\EE_p$ algebra) and produces the Poincar{\'e} $\infty$-category associated to it (with its symmetric monoidal structure, respectively). Before we can do that, we first recall why it is even possible to define $\EE_\sigma$ and $\EE_p$ algebras in genuine $C_2$-spectra. 
\begin{recollection}
    [$ C_2 $-symmetric monoidal structure on $ C_2 $-spectra] 
    The symmetric monoidal structure on genuine $C_2$-spectra extends naturally to a $C_2$-symmetric monoidal structure \cites[\S9]{Bachmann_Hoyois_norms}[Definition 2.2.3 \& Example 2.4.2]{NS22}. 
    This essentially amounts to the existence of the Hill-Hopkins-Ravenel norm functor $N^{C_2}:\mathrm{Sp}\to \mathrm{Sp}^{C_2}$ and the geometric fixed-point functor $(-)^{\phi C_2}:\mathrm{Sp}^{C_2}\to \mathrm{Sp}$ as symmetric monoidal functors together with coherences. 
    We will write $ \left(\underline{\Spectra}^{C_2}\right)^\otimes $ for this $ C_2 $-symmetric monoidal $ \infty $-category. 
    
    Since this will be our main example of interest, we will use the shorthand notation \[\EE_\sigma\Alg := \EE_\sigma\Alg(\mathrm{Sp}^{C_2 ,\otimes})\] and \[\EE_p\Alg := \EE_p\Alg(\mathrm{Sp}^{C_2,\otimes})\] when there is no danger of confusion.
\end{recollection}

\begin{construction}\label{cons:Einfty_norm_in_families}
    Choose a $ C_2 $-set $ U = \{u, v \} $ on which $ C_2 $ acts freely and transitively, and fix an ordering $ u \leq v $ on $ U $; we will write $ p \colon U \to C_2/C_2 $. 
    Using Observation \ref{obs:operad_maps_to_infty_operad_maps}, we define a functor 
    \begin{equation*}
        \begin{split}
            \iota \colon \Delta^1 \times \Fin_* & \to \left(\EE_p^\otimes\right)_{C_2/C_2} \\
            \left(0, \langle n \rangle \right) & \mapsto \left(\langle n \rangle^\circ \times U_+ \to C_2/C_2\right) \\
            \left(1, \langle n \rangle \right) & \mapsto \left(\langle n \rangle \to C_2/C_2\right) \\
            \left(\id_0,\alpha \colon \langle n \rangle \to \langle m \rangle \right) & \mapsto \left(\alpha \times \id_U\right)  \\
            \left(\id_1,\alpha \colon \langle n \rangle \to \langle m \rangle \right) & \mapsto \alpha \\
            \left(0<1,\id_{\langle n \rangle}\right) & \mapsto \left(\id_{\langle n \rangle} \times p \right) \,.
        \end{split}
    \end{equation*} 
    Composing $ \iota $ with the canonical maps $ \left(\EE_p^\otimes\right)_{C_2/C_2} \to \EE_p^\otimes $ and noting that the identity map gives a canonical section of $ Ar^{\mathrm{act}}(\underline{\Fin}_{C_2,*}) \to \underline{\Fin}_{C_2,*} $, we obtain a map $ \iota' \colon \Delta^1 \times \Fin_* \to \mathrm{Env}_{\underline{\Fin}_{C_2,*}}\left(\EE_p^\otimes\right) \simeq \EE_p^\otimes \times_{\underline{\Fin}_{C_2,*}} Ar^{\mathrm{act}}(\underline{\Fin}_{C_2,*})  $. 
    Now there is a commutative diagram
    \begin{equation*}
        \begin{tikzcd}
            \{0\} \times \Delta^1 \times \Fin_* \ar[d] \ar[r,"{\iota'}"] & \mathrm{Env}_{\underline{\Fin}_{C_2,*}}\left(\EE_p^\otimes\right) \ar[d,"{\pi}"] \\
            \Delta^1 \times \Delta^1 \times \Fin_* \ar[r,"{\overline{r}}"] & \underline{\Fin}_{C_2,*}
        \end{tikzcd}
    \end{equation*}
    where $ \left.\overline{r} \right|_{\{1\} \times \Delta^1 \times \Fin_*} $ factors as a composite of the projection $ \Delta^1 \times \Fin_* \to \Fin_* $ composed with the inclusion $ \Fin_* \to \underline{\Fin}_{C_2,*} $. 
    Since $ p $ is a cocartesian fibration \cite[Proposition 2.8.6]{NS22} and $ \{0\} \to \Delta^1 $ is left anodyne, there exists an essentially unique extension 
    \begin{equation*}
        r \colon \Delta^1 \times \Delta^1 \times \Fin_* \to \mathrm{Env}_{\underline{\Fin}_{C_2,*}}\left(\EE_p^\otimes\right)
    \end{equation*}
    making the above square commute (in particular, $ r|_{\{0\} \times \Delta^1 \times \Fin_*} $ agrees with $ \iota' $) and so that $ r $ sends $ (0<1,\id,\id) $ to $ p $-cocartesian arrows. 
    Now restriction along $ r_1 := r|_{\{1\} \times \Delta^1 \times \Fin_*} $ and $ r_{10} := r|_{\{1\} \times \{0\} \times \Fin_*} $ induce maps
    \[
    \begin{tikzcd}
        \EE_p\Alg\left(\underline{\Spectra}^{C_2,\otimes}\right) \simeq \Fun^\otimes_{C_2}\left(\mathrm{Env}_{\underline{\Fin}_{C_2,*}}\left(\EE_p^\otimes\right), \underline{\Spectra}^{C_2,\otimes}\right) \ar[r,dashed] \ar[d] & \Alg_{\Fin_*} \left(r_{10}^*\left(\underline{\Spectra}^{C_2,\otimes}\right)\right) \simeq\EE_\infty\Alg\left(\Spectra^{C_2}\right) \\ \Alg_{\Delta^1 \times \Fin_*} \left(r_1^*\left(\underline{\Spectra}^{C_2,\otimes}\right)\right) \ar[r] & \Alg_{\Fin_*} \left(r_{10}^*\left(\underline{\Spectra}^{C_2,\otimes}\right)\right)\ar[u] & \,.
    \end{tikzcd}
    \]
\end{construction}
\begin{construction}\label{cons:assoc_operad_to_E_sigma_operad}
    We will construct a map $ \Assoc^\otimes \to \underline{\Assoc}^\otimes_{\sigma, C_2/e} $. 
    Here the subscript $ C_2/e $ means we consider the (non-parametrized) fiber over $ C_2/e $. 
    
    Observe that if $ \nabla \colon (C_2/e)^{\sqcup n} \simeq C_2/e \times \langle n \rangle \to C_2/e $ is the fold map, then a lift of $ \nabla $ to a morphism in $ \underline{\Assoc}^\otimes_{\sigma, C_2/e} $ is equivalent to the choice of an ordering on $ \{1,2, \ldots, n\} \simeq \nabla^{-1}(\{e\}) $; the ordering on $ \nabla^{-1}(\{\sigma\}) $ is uniquely determined by the ordering on $ \nabla^{-1}(\{e\}) $. 
    Using Observation \ref{obs:operad_maps_to_infty_operad_maps}, we may define functors
    \begin{equation*}
        \begin{split}
            \iota_e, \iota_\sigma \colon \Assoc^\otimes &\to \underline{\Assoc}^\otimes_{\sigma, C_2/e} \\
            \langle n \rangle &\mapsto C_2/e \times \langle n \rangle \\
            \iota_e \colon \left(f \colon \langle n \rangle \to \langle m \rangle, (\leq_i)_{i \in \langle m \rangle^\circ}\right) & \mapsto \left(\id_{C_2} \times f \colon \langle n \rangle \to \langle m \rangle, (\leq_i \text{ ordering on }f^{-1}(i) \times \{e\})_{i \in \langle m \rangle^\circ}\right) \\
            \iota_\sigma \colon \left(f \colon \langle n \rangle \to \langle m \rangle, (\leq_i)_{i \in \langle m \rangle^\circ}\right) & \mapsto \left(\id_{C_2} \times f \colon \langle n \rangle \to \langle m \rangle, (\leq_i \text{ ordering on }f^{-1}(i) \times \{\sigma\})_{i \in \langle m \rangle^\circ}\right) \,.
        \end{split}
    \end{equation*} 
    Furthermore, there is a commutative diagram
    \begin{equation}
        \begin{tikzcd}
            \Assoc^\otimes \ar[d,"{\mathrm{rev}}"'] \ar[r,"{\iota_e}"] & \underline{\Assoc}^\otimes_{\sigma, C_2/e} \ar[d,"{\sigma^*}"] \\
            \Assoc^\otimes \ar[r,"{\iota_e}"] \ar[ru,"{\iota_\sigma}"] & \underline{\Assoc}^\otimes_{\sigma, C_2/e}
        \end{tikzcd}
    \end{equation} 
    where $ \mathrm{rev} $ is the automorphism considered in \cite[Remark 4.1.1.7]{LurHA} and $ \sigma^* $ is the cocartesian pushforward along the nontrivial automorphism of $ C_2/e $ in $ \mathcal{O}_{C_2}^\op $. 
    Now if $ \mathcal{C}^\otimes $ is any $ C_2 $-symmetric monoidal $ C_2 $-$ \infty $-category, let us write $ \mathcal{C}^{e, \otimes} $ for its underlying symmetric monoidal $ \infty $-category with naive $ C_2 $-action. 
    Then the aforementioned diagram induces a commutative diagram
    \begin{equation*}
        \begin{tikzcd}[row sep=tiny]
            & \EE_1 \Alg\left(\mathcal{C}^{e,\otimes}\right) \ar[dd,"{\sigma_{\mathcal{C}} \circ \mathrm{rev}}","\sim"'] \\
            \Alg_{\underline{\Assoc}^\otimes_{\sigma}}\left(\mathcal{C}^\otimes\right) \ar[ru] \ar[rd] & \\
            & \EE_1 \Alg\left(\mathcal{C}^{e,\otimes}\right) \,. 
        \end{tikzcd}
    \end{equation*}
    Informally, this implies that the underlying object of an $ \EE_\sigma $-algebra consists of an $ \EE_1 $-algebra $ A $ in $ \mathcal{C}^e $ equipped with an equivalence $ A \simeq \sigma_{\mathcal{C}}(A)^\op $, where $ (-)^\op $ denotes the opposite algebra and $ \sigma_\mathcal{C} $ denotes the automorphism on $ \mathcal{C}^e $. 
    If $ \sigma_{\mathcal{C}} $ is the identity (as is the case with $ C_2 $-spectra), then $ A $ has an involution. 
\end{construction}

\begin{construction}\label{cons:lmod_operad_to_E_sigma_operad}
    We will construct a map 
    \begin{equation*}
        q_m \colon \mathcal{LM}^\otimes \to \underline{\Assoc}^\otimes_{\sigma, C_2/C_2} 
    \end{equation*} 
    which is compatible with the map $\iota_e$ defined in Construction~\ref{cons:assoc_operad_to_E_sigma_operad}, where as before the subscript $(-)_{C_2/C_2}$ denotes the non-parameterized fiber. 
    By Observation \ref{obs:operad_maps_to_infty_operad_maps}, it suffices to define a map of simplicial operads. 
    On objects, this functor will be given by \[q_m(\langle n\rangle, S)=((\langle n\rangle^\circ \setminus S)\times C_2)
    \sqcup S\to C_2/C_2\] and on maps by sending $(\alpha: (\langle n\rangle, S)\to (\langle m\rangle, T),(\leq_i)_{i\in \langle m\rangle^\circ})$ to the map $(\alpha\mid_{\langle n\rangle^\circ\setminus S}\times C_2)\sqcup \alpha\mid_S$. 
    We must equip this map with the extra data to be a map in $\underline{\mathrm{Assoc}}^\otimes_{\sigma}$. 
    By definition of $ \underline{\Assoc}_{\sigma} $, for each element $ t $ in the codomain of $ q_m(\alpha) $, we require an ordering $ \preceq_t $ on $ q_m(\alpha)^{-1}(\{t\}) $ so that the $ C_2 $-action sends $ \leq_t $ to $ \geq_{\sigma(t)} $. 
    There are three cases: 
    \begin{itemize}
        \item If $ t = (i, e) $, then equip $ q_m(\alpha)^{-1}(\{t\}) \simeq \alpha^{-1}(\{i\}) \times \{e\} $ with the ordering $ \leq_i $
        \item If $ t = (i, \sigma) $, then equip $ q_m(\alpha)^{-1}(\{t\}) \simeq \alpha^{-1}(\{i\}) \times \{e\} $ with the ordering $ \geq_i $ 
        \item If $ t \in T $, then notice that $ q_m(\alpha)^{-1}(\{t\}) \simeq \left(\left(\alpha^{-1}(\{t\}) \cap \langle n \rangle^\circ \setminus S \right) \times C_2\right) \cup \left(\alpha^{-1}(\{t\}) \cap S \right) $ where $ \alpha^{-1}(\{t\}) \cap S $ is a singleton (see \cite[Notation 4.1.2.6]{LurHA}); let us write $q_m(\alpha)^{-1}(\{x\})=\{i_1,i_2,\ldots, i_k\}\times C_2 \sqcup\{s\}$ where $i_1\leq_t i_2\leq_t\ldots\leq_t i_k \leq_t s$.  
        Now define the linear order $ \preceq_t $ on $ q_m(\alpha)^{-1}(\{t\}) $ to be \[(i_1,e)\preceq_t (i_2,e)\preceq_t\ldots \preceq_t (i_k,e)\preceq_t s \preceq_t (i_k,\sigma)\preceq_t\ldots \preceq_t (i_2,\sigma)\preceq_t(i_1,\sigma)\,. \] 
    \end{itemize}
    
    We now explain what it means for the map $ q_m $ to be compatible with the map $\iota_e$ defined in Construction~\ref{cons:assoc_operad_to_E_sigma_operad}. 
    Note that there is a functor
    \begin{equation*}
        \Assoc^\otimes \xrightarrow{i_{\mathfrak{a}}} \mathcal{LM}^\otimes \xrightarrow{\widetilde{q}_m}  \mathrm{Env}_{\underline{\Fin}_{C_2,*}}\left(\underline{\Assoc}_\sigma^\otimes\right)
    \end{equation*}
    where $ \widetilde{q}_m $ is induced by $ q_m $ and observing that isomorphisms in $ \underline{\Fin}_{C_2,*} $ are active arrows. 
    Now note that there is a functor $ \Assoc^\otimes \times \Delta^1 \to \underline{\Fin}_{C_2,*} $ sending $ \left(\id_{\langle 1 \rangle}, 0 < 1\right) $ to the collapse map $ [C_2 = C_2] \to [C_2/C_2 = C_2/C_2] $ and making the square in the diagram 
    \begin{equation*}
        \begin{tikzcd}
            \Assoc^\otimes \times \{0\} \ar[r,"{\widetilde{i}_e}"] \ar[d] & \mathrm{Env}_{\underline{\Fin}_{C_2,*}}\left(\underline{\Assoc}_\sigma^\otimes\right) \ar[d,"p"] \\
            \Assoc^\otimes \times \Delta^1 \ar[r] \ar[ru,dashed] & \underline{\Fin}_{C_2,*}
        \end{tikzcd}
    \end{equation*} 
    commute. 
    Since $ \{0\} \to \Delta^1 $ is left anodyne and $ p $ is a cocartesian fibration, there is an essentially unique dotted map making the diagram commute; let us write $ n\widetilde{i}_e $ for the restriction of the diagonal map to $ \Assoc^\otimes \times \{ 1\} $. 
    Then by compatibility, we mean that there is an equivalence $ \widetilde{q}_m \circ i_{\mathfrak{a}} \simeq n\widetilde{i}_e $ of functors $ \Assoc^\otimes \to \mathrm{Env}_{\underline{\Fin}_{C_2,*}}\left(\underline{\Assoc}_\sigma^\otimes\right)  $.
\end{construction}
\begin{construction}\label{cons:E_sigma_algebra_to_underlying_bimod}
    Using Observation \ref{obs:operad_maps_to_infty_operad_maps}, define a functor
    \begin{equation*}
        \begin{split}
            q_m^e \colon \mathcal{LM}^\otimes & \to \left(\underline{\Assoc}_\sigma^\otimes\right)_{C_2/e} \\
            \left(\langle n \rangle, S \right) & \mapsto \left(\langle n + k \rangle \times C_2 \to C_2/e \right) \quad \text{ where } n-|S|=k \\
            \left(\alpha \colon \left(\langle n \rangle, S\right) \to \left(\langle m \rangle, T\right) \right) & \mapsto \begin{pmatrix} f(\alpha) \colon \langle n+k \rangle \times C_2 \to \langle m + \ell \rangle \times C_2 \\
                (p,\gamma) \mapsto \begin{cases}
                    (\alpha(p),\gamma) & \text{ if } p \leq n \\
                    (\alpha (i),\gamma) & \text{ if } p = n+i, \alpha(i) \in T \\
                    (m+i',\sigma\gamma) & \text{ if } p = n+i, \alpha(i) = i'_{j'}
            \end{cases}\end{pmatrix} 
        \end{split}
    \end{equation*} 
    where we write $ \gamma $ for a generic element of $ C_2 $ and $ \sigma \in C_2 $ for the generator. 
    To define $ f(\alpha) $, we must give, for each $ (j,e) \in \langle m + \ell \rangle \times C_2 $, an ordering $ \leq_{j,e} $ on $ f(\alpha)^{-1}(j,e) $ so that the $ C_2 $-action sends $ \leq_{j,e} $ to the reverse of $ \leq_{j,\sigma} $. 
    We endow $ f(\alpha)^{-1}(j,e) $ with an ordering via the equivalence $ f(\alpha)^{-1}(j,e) \subseteq \alpha^{-1}(j) \times C_2 \xrightarrow{\pi_1} \alpha^{-1}(j) $. 
    
    Restriction along $ q_m^e $ defines a functor $ \EE_\sigma\Alg\left(\Spectra^{C_2}\right) \to \LMod\left(\Spectra\right) $; informally, this sends an $ \EE_\sigma $-algebra $ A $ to its underlying spectrum $ A^e $, regarded as an $ A^e \otimes A^{e,\op} $-bimodule (compare \cite[Construction 4.6.3.7]{LurHA}). 
\end{construction}
\begin{lemma}\label{lemma:E_sigma_fixpt_to_underlying}
    Constructions \ref{cons:lmod_operad_to_E_sigma_operad} and \ref{cons:E_sigma_algebra_to_underlying_bimod} are compatible with each other in the sense that if $ \overline{q}_m $ is the (essentially) unique map making the diagram
    \begin{equation*}
        \begin{tikzcd}
            {\{0\} \times \mathcal{LM}^\otimes} \ar[r,"{q_m}"] \ar[d] & \underline{\Assoc}^\otimes_{\sigma, C_2/C_2} \ar[r,phantom,"\subseteq"] & { \underline{\Assoc}^\otimes_\sigma} \ar[d] \\
            {\Delta^1 \times \mathcal{LM}^\otimes} \ar[r,"{\pi_1}"'] \ar[rru,dashed,"{\overline{q}_m}", near start] & {\Delta^1} \ar[r,"{C_2/C_2 \leftarrow C_2/e}"'] & {\mathcal{O}^{\mathrm{op}}_{C_2}}
        \end{tikzcd}
    \end{equation*}
    commute, then there is a natural equivalence of functors $ \left.\overline{q}_m\right|_{\{1\} \times \mathcal{LM}^\otimes} \simeq q_m^e $. 
\end{lemma}
\begin{proof}
    Since all $ \infty $-categories in the diagram are equivalent to nerves of 1-categories, it suffices to exhibit a natural equivalence of functors of 1-categories. 
    Note that there is a canonical isomorphism $ q_m^e \left(\langle n \rangle, S\right) \simeq \left(\langle n \rangle \sqcup (\langle n \rangle^\circ \setminus S) \right) \times C_2 \simeq \left(S \sqcup \left(\bigsqcup_{\{\varepsilon,\tau\}}\left(\langle n \rangle^\circ \setminus S\right)\right) \right) \times C_2 $ so that $ \{\varepsilon \} \to \{\varepsilon,\tau\} $ induces the canonical inclusion $ \langle n \rangle \to \langle n+k \rangle $. 
    Recall that $ q_m\left(\langle n \rangle, S\right) \times C_2 \simeq \left[(\langle n\rangle^\circ \setminus S)\times C_2) \sqcup S\right] \times C_2 $. 
    The isomorphism of sets $ \{\varepsilon, \tau\} \simeq C_2 $ sending $ \varepsilon \mapsto e $ induces $ C_2 $-equivariant isomorphisms $ \zeta_{\left(\langle n \rangle, S\right)} \colon q_m^e \left(\langle n \rangle, S\right) \simeq q_m \left(\langle n \rangle, S\right) \times C_2 = \overline{q}_m\left((\langle n \rangle, S), 1\right) $. 
    Let us make cocartesian transport of morphisms explicit. 
    Given $ \alpha \colon \left(\langle n \rangle, S\right) \to \left(\langle m \rangle, T\right) $, for each $ i \in q_m \left(\langle m \rangle, T\right) $, the ordering $ \leq_i $ on $ q_m(\alpha)^{-1}\left(\{i\}\right) $ induces orderings $ \leq_i $ and $ \geq_i $ on $ \left(q_m(\alpha) \times \id_{\{e\}}\right)^{-1} $ and $ \left(q_m(\alpha) \times \id_{\{\sigma\}}\right)^{-1} $, respectively. 
    It is now straightforward to see that $ \zeta $ refines to the desired natural isomorphism. 
\end{proof}
\begin{lemma}\label{lemma:E_p_E_sigma_bimod_compatible}
    There is a commutative square 
    \begin{equation*}
        \begin{tikzcd}[cramped]
            &\mathcal{LM}^\otimes \ar[r, "r"] \ar[d] & \underline{\Assoc}_\sigma^\otimes \ar[d] \\
            \Fin_* \times \Delta^1 \ar[r] & \mathrm{Assem}\left(\Fin_* \times \Delta^1\right) \ar[r] & \EE_p^\otimes         
        \end{tikzcd}     
    \end{equation*} 
    where $ r $ is induced by the map from Construction \ref{cons:lmod_operad_to_E_sigma_operad} and $ \mathrm{Assem} $ denotes \emph{assembly} in the sense of \cite[Definition 2.3.3.1]{LurHA}. 
    The map from the lower left to the lower right is that of Construction \ref{cons:norm_in_families}, and $ \Fin_* \times \Delta^1 \to \mathrm{Assem}\left(\Fin_* \times \Delta^1\right) $ is the unit of the adjunction. 
\end{lemma}
\begin{remark}
    By construction, restriction along $ \mathcal{LM}^\otimes \to \mathrm{Assem}\left(\Fin_* \times \Delta^1\right) $ sends a map of $ \EE_\infty $-rings $ A \to B $ to the object $ B $, regarded as a left $ A $-module. 
\end{remark}
\begin{proof}[Proof of Lemma \ref{lemma:E_p_E_sigma_bimod_compatible}] 
    We will exhibit an explicit model for the assembly. 
    Define an ordinary operad $ \mathbf{M} $ as follows: Let objects be given by pairs $ \left(\langle n \rangle, S\right) $ where $ S \subseteq \langle n \rangle^\circ $ and morphisms $ \left(\langle n\rangle, S\right) \to \left(\langle m \rangle,T\right) $ be given by maps of finite sets $ f \colon \langle n \rangle \to \langle m \rangle $ with the property that $ f(S) \subseteq T $. 
    There is a canonical map $ f \colon \Fin_* \times \Delta^1 \to \mathbf{M}^\otimes $ so that $ f \left(\langle n \rangle, 0\right) = \left(\langle n \rangle, \varnothing\right) $ and $ f \left(\langle n \rangle, 1\right) = \left(\langle n \rangle, \langle n \rangle^\circ \right) $.
    Take the $ \infty $-operad $ M^\otimes $ associated to this operad. 
    The map $ f $ induces a map (which we also denote by $ f $) of generalized $ \infty $-operads which exhibits $ M^\otimes $ as the assembly of $ \Fin_* \times \Delta^1 $. 
    To see this, note that by \cite[Theorem 2.3.3.23]{LurHA}, it suffices to show that $ f $ is a weak approximation which moreover induces a homotopy equivalence on fibers over $ \langle 1 \rangle $. 
    The latter holds by construction. 
    That $ f $ is a weak approximation follows from \cite[Proposition 2.3.3.11]{LurHA}. 
    It follows that $ M^\otimes $ is our desired $ \mathrm{Assem}\left(\Fin_* \times \Delta^1\right) $. 
    
    Now there is a canonical map $ \mathcal{LM}^\otimes \to M^\otimes $ given by the identity on objects and forgetting the ordering on morphisms. 
    The commutativity of the square now follows by inspection. 
\end{proof}

\begin{construction}\label{cons:norm_in_families}
    Let $ \mathcal{O}^\otimes $ be a $ C_2 $-$ \infty $-operad, and suppose we are given a map of ordinary $ \infty $-operads $ q_{-1} \colon \Assoc^\otimes \to \mathcal{O}^\otimes_{C_2/e} $ where $ (-)_{C_2/e} $ denotes the non-parametrized fiber.\footnote{Given a map $ \Assoc^\otimes \to \mathcal{O}^\otimes_{C_2/C_2} $ instead, we can use cocartesian transport along $ C_2/C_2 \to C_2/e $ to obtain a map $ \Assoc^\otimes \to \mathcal{O}^\otimes_{C_2/e} $.} 
    Recall that by \cite[Proposition 2.8.7(1)]{NS22}, there is an equivalence $ \underline{\Alg}_{\mathcal{O}}\left(\underline{\Spectra}^{C_2}\right) \simeq \Fun^\otimes_{C_2}\left(\mathrm{Env}_{\underline{\Fin}_{C_2,*}}\left(\mathcal{O}\right), \left(\underline{\Spectra}^{C_2}\right)^\otimes\right) $. 
    Now recall that $ \mathrm{Env}_{\underline{\Fin}_{C_2,*}}\left(\mathcal{O}\right) \simeq \mathcal{O}^\otimes \times_{\underline{\Fin}_{C_2,*}} Ar^{\mathrm{act}}(\underline{\Fin}_{C_2,*}) $ \cite[Definition 2.8.4]{NS22}. 
    Since the identity map on any object in $ \underline{\Fin}_{C_2,*} $ is an active arrow, $ q_{-1} $ induces a map
    \begin{equation}\label{eq:assoc_opd_to_envelope_param_assoc}
        q_{0} \colon \Assoc^\otimes \to \mathrm{Env}_{\underline{\Fin}_{C_2,*}}\left(\mathcal{O}\right) 
    \end{equation}
    Observe that $ \Assoc^\otimes \times \{0\} \to \Assoc^\otimes \times \Delta^1 $ is left anodyne by (the dual to) \cite[Corollary 2.1.2.7]{HTT}. 
    It follows from (the dual to) \cites[Corollary 2.4.2.5]{HTT}[\href{https://kerodon.net/tag/01VF}{Tag 01VF} Theorem 5.2.1.1(1')]{kerodon} that there exists an essentially unique extension 
    \begin{equation}\label{eq:assoc_opd_interval_to_envelope_param_assoc}
        q\colon \Delta^1 \times \Assoc^\otimes \to \mathrm{Env}_{\underline{\Fin}_{C_2,*}}\left(\mathcal{O}\right) 
    \end{equation}
    of (\ref{eq:assoc_opd_to_envelope_param_assoc}) which sends $ (0<1, \id_A) $ to a morphism which is cocartesian with respect to the structure map $ \mathrm{Env}_{\underline{\Fin}_{C_2,*}}\left(\mathcal{O}\right) \to \underline{\Fin}_{C_2,*} $. 
    Write $ p $ for the composite $ \Delta^1 \times \Assoc^\otimes \to \mathrm{Env}_{\underline{\Fin}_{C_2,*}}\left(\mathcal{O}\right) \to \underline{\Fin}_{C_2,*} $. 
    Restriction along the functor of (\ref{eq:assoc_opd_interval_to_envelope_param_assoc}) defines 
    \begin{equation}\label{eq:assoc_alg_norm_in_families}
        \Fun^\otimes_{C_2}\left(\mathrm{Env}_{\underline{\Fin}_{C_2,*}}\left(\mathcal{O}\right), \left(\underline{\Spectra}^{C_2}\right)^\otimes\right) \to \Fun_{\Delta^1 \times \Assoc^\otimes}\left(\Delta^1 \times \Assoc^\otimes, p^*\left(\underline{\Spectra}^{C_2}\right)\right) \,.
    \end{equation}
    Observe that $ p^*\left(\underline{\Spectra}^{C_2}\right) \to \Delta^1 \times \Assoc^\otimes $ is a $ \Delta^1 $-cocartesian family of monoidal $ \infty $-categories in the sense of \cite[Definition 4.8.3.1]{LurHA}, and the image of the aforementioned restriction functor consists of $ \Delta^1 $-cocartesian families of associative algebra objects in $ p^*\left(\underline{\Spectra}^{C_2}\right) $. 
    In particular, by \cite[Remark 4.8.3.8]{LurHA}, the aforementioned restriction functor refines to 
    \begin{equation}\label{eq:assoc_alg_norm_in_families_straightened}
        G_q \colon \Fun^\otimes_{C_2}\left(\mathrm{Env}_{\underline{\Fin}_{C_2,*}}\left(\mathcal{O}\right), \left(\underline{\Spectra}^{C_2}\right)^\otimes\right) \to \Fun\left(\Delta^1, \CAT[\Alg] \right) \,, 
    \end{equation} 
    where $ \CAT[\Alg] $ is \cite[Definition 4.8.3.7]{LurHA}. 
    By construction, its image in $ \Fun\left(\Delta^1,\Alg\left(\CAT\right) \right) $ consists of the single arrow $ N^{C_2} \colon \Spectra \to \Spectra^{C_2} $. 
\end{construction}
\begin{remark}
    In fact, we can take $ \CAT[\Alg]\left(\mathcal{K}^{\mathrm{sift}}\right) $ where $ \mathcal{K}^{\mathrm{sift}} $ denotes the collection of sifted simplicial sets. 
    However, we cannot take $ \mathcal{K} $ to be a collection containing not necessarily sifted diagrams, since the norm $ N^{C_2} $ does not take coproducts in $ \Spectra $ to coproducts in $ \Spectra^{C_2} $.  
\end{remark}
\begin{construction}\label{cons:module_with_gen_involution_in_families}   
    Let $ \mathcal{O}^\otimes $ be a $ C_2 $-$ \infty $-operad, and suppose we are given a map of $ \infty $-operads $ q \colon \mathcal{LM}^\otimes \to \mathcal{O}^\otimes_{C_2/e} $, where $ (-)_{C_2/C_2} $ denotes the non-parametrized fiber.  
    Recall that by \cite[Proposition 2.8.7(1)]{NS22}, there is an equivalence $ \Alg_{\mathcal{O}}\left(\underline{\Spectra}^{C_2}\right) \simeq \Fun^\otimes_{C_2}\left(\mathrm{Env}_{\underline{\Fin}_{C_2,*}}(\mathcal{O}^\otimes), \left(\underline{\Spectra}^{C_2}\right)^\otimes\right) $. 
    Now recall that $ \mathrm{Env}_{\underline{\Fin}_{C_2,*}}(\mathcal{O}^\otimes) \simeq \mathcal{O}^\otimes \times_{\underline{\Fin}_{C_2,*}} Ar^{\mathrm{act}}(\underline{\Fin}_{C_2,*}) $ \cite[Definition 2.8.4]{NS22}. 
    Since the identity map on any object in $ \underline{\Fin}_{C_2,*} $ is an active arrow, $ q_m $ admits a canonical lift to a map 
    \begin{equation}\label{lm_operad_to_genuine_fiber_hm_param_operad}
        r \colon \mathcal{LM}^\otimes \to \mathrm{Env}_{\underline{\Fin}_{C_2,*}}(\mathcal{O}^\otimes) \simeq \mathcal{O}^\otimes \times_{\underline{\Fin}_{C_2,*}} Ar^{\mathrm{act}}(\underline{\Fin}_{C_2,*}) \,.
    \end{equation}
    Restriction along the functor (\ref{lm_operad_to_genuine_fiber_hm_param_operad}) sends maps of $ C_2 $-$ \infty $-operads to maps of ordinary $ \infty $-operads, hence it defines 
    \begin{equation}
        \Fun^\otimes_{C_2}\left(\mathrm{Env}_{\underline{\Fin}_{C_2,*}}(\mathcal{O}^\otimes), \left(\underline{\Spectra}^{C_2}\right)^\otimes\right) \to \Alg_{\mathcal{LM}/\mathrm{Assoc}}\left(p_1^*\left(\underline{\Spectra}^{C_2,\otimes}\right)\right) \to \Alg_{/\mathrm{Assoc}}\left(p_1^*\left(\underline{\Spectra}^{C_2,\otimes}\right)\right)
    \end{equation}
    in the notation of \cite[Definition 2.1.3.1]{LurHA}. 
    Now observe that $ p_1^*\left(\underline{\Spectra}^{C_2,\otimes}\right) $ is the ordinary $ \infty $-category of $ C_2 $-spectra, regarded as a monoidal $ \infty $-category via smash product, and there is a commutative diagram
    \begin{equation}
        \begin{tikzcd}
            \Fun^\otimes_{C_2}\left(\mathrm{Env}_{\underline{\Fin}_{C_2,*}}(\mathcal{O}^\otimes), \left(\underline{\Spectra}^{C_2}\right)^\otimes\right) \ar[r] \ar[d] & \LMod\left(\Spectra^{C_2,\otimes}\right) \ar[r] \ar[d] & \widetilde{\CAT[\mathrm{ex}]} \ar[d] \\
            \Alg_{\mathbb{E}_1}\left(\Spectra^{C_2}\right)  \ar[r,equals] &  \Alg_{\mathbb{E}_1}\left(\Spectra^{C_2}\right) \ar[r] & \CAT[\mathrm{ex}] 
        \end{tikzcd}
    \end{equation}
    where $ \widetilde{\CAT[\mathrm{ex}]} \to \CAT[\mathrm{ex}] $ is the restriction of the universal cocartesian fibration to $ \CAT[\mathrm{ex}] \subseteq \CAT $, the left vertical arrow is restriction along the composite $ \mathrm{Assoc}^\otimes \to \mathcal{LM}^\otimes \xrightarrow{q_m} \mathcal{O}^\otimes_{C_2/C_2} $, and the left horizontal arrow is (\ref{eq:assoc_alg_norm_in_families}) composed with restriction along $ \Assoc^\otimes \times \{1\} \hookrightarrow \Assoc^\otimes \times \Delta^1 $. 
    The right hand square is a pullback by \cite[Corollary 4.2.3.7(3)]{LurHA} and the straightening-unstraightening equivalence.  
\end{construction}
			
\begin{observation}   
    There is a forgetful functor $ U \colon \CAT[\Mod] \to \CAT $, where $ \CAT[\Mod] = \mathrm{Mon}_{\mathcal{LM}}\left(\CAT\right) $ is defined in \cite[immediately preceding Remark 4.8.3.20]{LurHA}.  
    Combining Constructions \ref{cons:norm_in_families} (using $ \iota_e $ of Construction \ref{cons:assoc_operad_to_E_sigma_operad}) and \ref{cons:module_with_gen_involution_in_families} (using Construction \ref{cons:lmod_operad_to_E_sigma_operad}) and composing with the forgetful functor $ U $, we obtain a commutative diagram 
    \begin{equation}\label{diagram:calgp_to_functor_cat_with_pointed_target}
        \begin{tikzcd}[cramped]
            & \EE_\sigma\Alg \ar[rr] \ar[ld] \ar[d,"{U \circ \Theta \circ G}"] & & \widetilde{\CAT[\mathrm{ex}]} \ar[d] \\ 
            \EE_1\Alg\left(\Spectra\right) \ar[r] & \Fun\left(\Delta^1, \CAT \right) \ar[r,"{\ev_1}"] & \CAT & \CAT[\mathrm{ex}] \ar[l] 
        \end{tikzcd}
    \end{equation}
    where $ G $ is (\ref{eq:assoc_alg_norm_in_families_straightened}) from Construction \ref{cons:norm_in_families} and $ \Theta $ is from \cite[Construction 4.8.3.24]{LurHA} (using \cite[Variant 4.8.3.19]{LurHA}, since we consider left modules instead of right modules). 
    Commutativity of the diagram follows from compatibility of maps of $ \infty $-operads (see Lemma \ref{lemma:E_sigma_fixpt_to_underlying}). 
    Informally, this diagram can be regarded as sending an $ \EE_\sigma $-algebra $ A $ to the tuple $ \left(N^{C_2} \colon \Mod_{A^e}(\Spectra) \to \Mod_{N^{C_2}A}\left(\Spectra^{C_2}\right), A \in \Mod_{N^{C_2}A} \right) $. 
\end{observation}
			
\begin{lemma}\label{lemma:stable_cat_functorial_Yoneda}
    Write $ \widetilde{\CAT[\mathrm{ex}]} \to \CAT[\mathrm{ex}] $ for the restriction of the universal cocartesian fibration to $ \CAT[\mathrm{ex}] \subseteq \CAT $, where $ \CAT[\mathrm{ex}] $ denotes stable $ \infty $-categories and exact functors between them. 
    Consider the cocartesian fibration $ C\textsc{at}^{\mathrm{ex}}_{\infty,(-)^\op/\Spectra} \to \CAT $ classified by the functor $ \Fun^{\mathrm{ex}}\left((-)^\op, \Spectra\right) $, with cocartesian transport given by left Kan extension. 
    Then there is a functor $ \widetilde{\CAT[\mathrm{ex}]} \to C\textsc{at}^{\mathrm{ex}}_{\infty,(-)^\op/\Spectra} $ over $ \CAT[\mathrm{ex}] $ which preserves cocartesian edges. 
\end{lemma}
\begin{proof}
    Write $ \mathrm{RFib} \subset \mathrm{Ar}\left(\CAT\right) \times_{\CAT} \CAT[\mathrm{ex}] $ be the full subcategory of the arrow category spanned by the right fibrations with target a stable $ \infty $-category. 
    Write $ \mathrm{RFib}_* \subseteq \mathrm{RFib} $, resp. $ \mathrm{RFib}^{\omega\text{-fil}} $ for the full subcategory on those arrows $ \mathcal{C} \to \mathcal{D} $ so that $ \mathcal{C} $ has a final object, resp. is $ \omega $-filtered (i.e., for every diagram $ K \to \mathcal{C} $ where $ K $ is finite, there exists an extension $ K^{\vartriangleright} \to \mathcal{C} $). 
    Observe that there is a commutative diagram 
    \begin{equation}\label{diagram:stable_functorial_Yoneda_unstraightened}
        \begin{tikzcd}[column sep=small]
            \mathrm{RFib}_* \ar[rr,"i"] \ar[rd,"{t_0}"'] & & \mathrm{RFib}^{\omega\text{-fil}} \ar[ld,"{t_1}"] \\
            & \CAT[\mathrm{ex}] & 
        \end{tikzcd} 
    \end{equation} 
    of inner fibrations.  
    By \cite[Corollary 6.5]{MR4545922} (cf. Corollary A.32 of \cite{MR3690268}), the left hand arrow in (\ref{diagram:stable_functorial_Yoneda_unstraightened}) is equivalent to the \emph{universal} cocartesian fibration $ \mathrm{RFib}_* \simeq \widetilde{\CAT[\mathrm{ex}]} $.  
    By \cite[Corollary 5.3.5.4]{HTT}, the fiber of $ \mathrm{RFib}^{\omega\text{-fil}} \to \CAT[\mathrm{ex}] $ over $ \mathcal{D} $ may be identified with the $ \infty $-category $ \Fun^{\mathrm{lex}}\left(\mathcal{D}^\op,\Spaces\right) $ of functors which preserve finite limits. 
    By \cite[Corollary 1.4.2.23]{LurHA}, the functor $ \Omega^\infty $ induces an equivalence $ \Fun^{\mathrm{lex}}\left(\mathcal{D}^\op,\Spaces\right) \simeq \Fun^{\mathrm{lex}}\left(\mathcal{D}^\op,\Spectra\right) $.  
    It follows from \cite[Lemma 1.4.1(i)]{CDHHLMNNSI} that $ \mathrm{RFib}^{\omega\text{-fil}} \to \CAT[\mathrm{ex}] $ is a cocartesian fibration classified by the functor $ \Fun^{\mathrm{lex}}\left((-)^\op,\Spectra\right) $, and $ i $ takes $ t_0 $-cocartesian morphisms to $ t_1 $-cocartesian morphisms.     
    Thus, we have identified $ i $ of (\ref{diagram:stable_functorial_Yoneda_unstraightened}) with the desired functor.  
\end{proof}
We are almost ready to produce the functor which sends an $\mathbb{E}_\sigma$ ring to its associated Poincar{\'e} $\infty$-category. The above, as we will show, allows us to produce the associated quadratic functor $\Qoppa_A$ on the category of \textit{all} left modules $\mathrm{LMod}_{A^e}$. This will not be Poincar{\'e} since we need to restrict our attention to only the compact $A^e$-modules. The following observation lets us do that.
\begin{observation}\label{obs:cpct_modules_to_all_modules}
    Consider the cocartesian fibration $ p \colon \LMod\left(\Spectra\right) \to \EE_1\Alg\left(\Spectra\right) $ \cite[Corollary 4.2.3.7(3)]{LurHA}. 
    
    Write $ \LMod^\omega\left(\Spectra\right) $ for the full subcategory of $ \LMod\left(\Spectra\right) $ on those pairs $ (A, M) $ so that $ M $ is $ \omega $-compact as an $ A $-module. 
    Observe that given any $ (A, M) \in \LMod^\omega $ and any map $ f \colon A \to B $ in $ \EE_1 \Alg\left(\Spectra\right) $, then $ p $-cocartesian pushforward of $ (A,M) $ along $ f $ is equivalent to $ (B,B \otimes_A M) \in \LMod^\omega $. 
    In other words, $ p $-cocartesian pushforward preserves the subcategory $ \LMod^\omega\left(\Spectra\right) $ because the functor $ B \otimes_A (-) $ admits a right adjoint which preserves $ \omega $-filtered colimits. 
    It follows that the restriction $ \overline{p} $ of $ p $ to $ \LMod^\omega\left(\Spectra\right) $ exhibits $ \LMod^\omega\left(\Spectra\right) \to \EE_1\Alg\left(\Spectra\right) $ as a cocartesian fibration, and the inclusion $ \LMod^\omega\left(\Spectra\right) \to \LMod\left(\Spectra\right) $ takes $ \overline{p} $-cocartesian edges to $ p $-cocartesian edges. 
    Unstraightening the inclusion $ \LMod^\omega\left(\Spectra\right) \to \LMod\left(\Spectra\right) $, we obtain a functor 
    \begin{equation*}
        \iota^\omega \colon \EE_1\Alg\left(\Spectra\right) \to \Fun\left(\Delta^1,\CAT[\mathrm{ex}]\right) \qquad \text{ or equivalently } \qquad \EE_1\Alg\left(\Spectra\right) \times \Delta^1 \to \CAT[\mathrm{ex}]
    \end{equation*}
    so that the composite $ \ev_0 \circ \iota^\omega $ factors through $ \Catex \subseteq \CAT[\mathrm{ex}] $. 
    Informally, $ \iota^\omega $ sends an $ \EE_1 $-algebra $ A $ to the inclusion $ \LMod^\omega_A \to \LMod_A $. 
    
    If $ A $ is an $ \EE_\infty $-algebra, then $ \Mod^\omega_A \subseteq \Mod_A $ is symmetric monoidal. 
    A similar series of observations shows that there is a commutative diagram
    \begin{equation}\label{diagram:cpct_modules_to_all_modules_sym_mon}
        \begin{tikzcd}
            \EE_\infty\Alg \times \Delta^1 \ar[d] \ar[r] & \EE_\infty\Mon\left(\CAT[\mathrm{ex}]\right) \ar[d] \\
            \EE_1\Alg\left(\Spectra\right) \times \Delta^1 \ar[r,"{\iota^\omega}"] & \CAT[\mathrm{ex}] \,.
        \end{tikzcd}
    \end{equation}
\end{observation}
We now construct the main functor in question. 
\begin{construction}\label{cons:precompose_with_norm_in_families}
    Combining (\ref{diagram:calgp_to_functor_cat_with_pointed_target}) with Lemma \ref{lemma:stable_cat_functorial_Yoneda}, we obtain a commutative diagram 
    \begin{equation}\label{diagram:calgp_to_functor_cat_with_presheaf_on_target}
        \begin{tikzcd}
            & & \int \Fun^{\mathrm{ex}}( (-)^\op, \Spectra) \ar[d] \\
            \EE_\sigma\Alg \ar[r] \ar[rru,dashed,bend left=15] & \left(\CAT\right)^{\Delta^1} \times_{\CAT} \CAT[\mathrm{ex}] \ar[r] & \CAT[\mathrm{ex}] \,.
        \end{tikzcd}
    \end{equation}
    There is a functor $ \Fun((-)^\op, \Spectra ) \colon \CAT \to \CAT $. 
    Under the non-full subcategory inclusion $ \CAT[\mathrm{ex}] \to \CAT $, for any $ \mathcal{C} \in \CAT[\mathrm{ex}] $, there is an inclusion $ \Fun^{\mathrm{ex}}(\mathcal{C}^\op,\Spectra) \subseteq \Fun(\mathcal{C}^\op, \Spectra) $, where $ \Fun^{\mathrm{ex}}(\mathcal{C}^\op,\Spectra) $ denotes the full subcategory spanned by exact functors. 
    In particular, there is a map of cartesian fibrations 
    \begin{equation}\label{diagram:extended_fibration_Cat_epoly}
        \begin{tikzcd}
            \int \Fun^{\mathrm{ex}}( (-)^\op, \Spectra) \ar[r] \ar[d] & \int \Fun( (-)^\op, \Spectra) \ar[d] \\
            \CAT[\mathrm{ex}] \ar[r] & \CAT\,.
        \end{tikzcd}
    \end{equation}
    Composing (\ref{diagram:calgp_to_functor_cat_with_presheaf_on_target}) with (\ref{diagram:extended_fibration_Cat_epoly}) and noting that the map $ \left(\CAT\right)^{\Delta^1} \to \CAT $ defining the middle term in (\ref{diagram:calgp_to_functor_cat_with_presheaf_on_target}) is given by evaluation at the target, we obtain a diagram 
    \begin{equation}\label{diagram:E_sigma_alg_to_functor_cat_with_presheaf_on_target_extended}
        \begin{tikzcd}[row sep=small]
            & & \int \Fun( (-)^\op, \Spectra) \ar[dd] \\
            \EE_\sigma\Alg \times \{1\} \ar[r,"M"] \ar[rru,bend left=15,"{M_*}"] \ar[d] & \left(\CAT\right)^{\Delta^1} \times \{1\} \ar[d] &  \\
            \EE_\sigma\Alg \times \Delta^1 \ar[r,"{M \times \mathrm{id}_{\Delta^1}}"] & \left(\CAT\right)^{\Delta^1} \times \Delta^1 \ar[r,"{\mathrm{ev}}"] & \CAT  \,.
        \end{tikzcd}
    \end{equation}
    Observe that in the bottom row of (\ref{diagram:E_sigma_alg_to_functor_cat_with_presheaf_on_target_extended}), given an $ \EE_\sigma $-algebra $ A $, the morphism $ (\mathrm{id}_A, 0\to 1) $ from $ (A,0) $ to $ (A, 1) $ on the lower left is sent to the morphism $ N^{C_2} \colon \Mod_{A^e}(\Spectra) \to \Mod_{N^{C_2}A}\left(\Spectra^{C_2} \right) $ in the lower right. 
    Now observe that $ \{1\} \to \Delta^1 $ and $ \EE_\sigma\Alg \times \{1\} \to \EE_\sigma\Alg \times \Delta^1 $ are right anodyne by \cite[Corollary 2.1.2.7]{HTT}. 
    It follows from \cites[Corollary 2.4.2.5]{HTT}[\href{https://kerodon.net/tag/01VF}{Tag 01VF} Theorem 5.2.1.1(1)]{kerodon} that there exists a functor $ \overline{M}_* \colon \EE_\sigma\Alg \times \Delta^1 \to \int \Fun\left((-)^\op,\Spectra\right) $ in (\ref{diagram:E_sigma_alg_to_functor_cat_with_presheaf_on_target_extended}) extending $ M_* $. 
    Moreover, by \cite[\href{https://kerodon.net/tag/01VF}{Tag 01VF} Theorem 5.2.1.1(2)]{kerodon}, we see that $ \overline{M}(A,0) $ is the cartesian transport of $ M (A, 1) = \hom_{N^{C_2}A}(-,A) $ along the functor $$ \Fun\left(\Mod_{N^{C_2} A}\left(\Spectra^{C_2}\right) ^\op, \Spectra\right) \to \Fun\left(\Mod_{A^e}\left(\Spectra\right),\Spectra\right) $$ classified by $ \mathrm{ev} \circ \left(M \times \mathrm{id}_{\Delta^1}\right) $. 
    Consider the restriction of $ \overline{M}_* $ to $ \EE_\sigma\Alg \times \{0\} $. 
    Since the restriction of $ M \times \mathrm{id}_{\Delta^1} $ to $ \EE_\sigma\Alg \times \{0\} $ agrees with $ \Mod_{(-)^e} $ by Construction \ref{cons:norm_in_families}, we have that the image of $ \left.\left( M \times \mathrm{id}_{\Delta^1} \right)\right|_{\EE_\sigma\Alg \times \{0\}} $ factors through the inclusion $ \CAT[\mathrm{ex}] \subseteq \CAT$.  
    It follows that we may regard $ \overline{M}_*|_{\EE_\sigma\Alg \times \{0\}} $ as having codomain the total space of the cartesian fibration $ \int \Fun\left((-)^\op,\Spectra\right) \to \CAT[\mathrm{ex}] $. 
    
    We claim that $ \overline{M}_*|_{\EE_\sigma\Alg \times \{0\}} $ admits an essentially unique factorization through the cartesian fibration $ \int \Fun^q(-) \to \Catex $. 
    Since the natural transformation $ \Fun^q(-) \subseteq \Fun\left((-)^\op,\Spectra\right) $ of functors $ \CAT[\mathrm{ex}] \subseteq \CAT $ is given pointwise by full inclusions, the induced map $ \int \Fun^q(-) \to \int \Fun\left((-)^\op,\Spectra\right) $ of cartesian fibrations over $ \CAT[\mathrm{ex}] $ is a full inclusion by \cite[Proposition 2.4.4.2]{HTT}. 
    Therefore, to show that $ \overline{M}_*|_{\EE_\sigma\Alg \times \{0\}} $ factors through $ \int \Fun^q(-) $, it suffices to check that $ \overline{M}_*|_{\EE_\sigma\Alg \times \{0\}} $ sends objects of $ \EE_\sigma\Alg $ to objects of $ \int \Fun^q(-) $. 
    This is true because $ \overline{M}_*|_{\EE_\sigma\Alg \times \{0\}}(A,0) = \left(\Mod_{A^e}, \hom_{N^{C_2}A}\left(N^{C_2}(-), A\right)\right) $, $ N^{C_2} \colon \Mod_{A^e} \to \Mod_{N^{C_2}A} $ is quadratic, and a composite of a reduced quadratic functor and an exact functor is reduced and quadratic. 
    
    In sum, we have produced a commutative diagram like so: 
    \begin{equation}
        \begin{tikzcd}[row sep=small]
            & & \int \Fun^{q}(-) \ar[d] \ar[rd] & \\
            \EE_\sigma\Alg \times \{0\} \ar[rr,"{\mathrm{ev}_0 \circ M \simeq \Mod_{(-)^e}}"] \ar[rru,bend left=15,"{\left.\overline{M}_*\right|_{\CAlgp \times \{0\}}}"] \ar[d] & & \CAT[\mathrm{ex}] \ar[rd] & \int \Fun\left((-)^\op,\Spectra\right) \ar[d] \\
            \EE_\sigma\Alg \times \Delta^1 \ar[r,"{M \times \mathrm{id}_{\Delta^1}}"] & \left(\CAT\right)^{\Delta^1} \times \Delta^1 \ar[rr,"{\mathrm{ev}}"] & &\CAT  \,.
        \end{tikzcd}
    \end{equation}
    Let us write $ \mathrm{ev}_0 \circ M =: M_0 $ and $ \left.\overline{M}_*\right|_{\EE_\sigma\Alg \times \{0\}} =: \overline{M}_{*,0} $. 
    Now let us observe that there is an equivalence of functors $ M_0 \simeq \Mod_{(-)^e} \simeq \ev_1 \circ \iota^\omega $, where $ \iota^\omega $ is defined in Observation \ref{obs:cpct_modules_to_all_modules}.
    Now there is a commutative diagram
    \begin{equation*}
        \begin{tikzcd}
            \EE_\sigma\Alg \times \{1\} \simeq \EE_\sigma\Alg \ar[rr,"{\overline{M}_{*,0}}"] \ar[d] & & \int \Fun^{q}(-) \ar[d,"c"] \\
            \EE_\sigma\Alg \times \Delta^1 \ar[r] & \EE_1\Alg(\Spectra) \times \Delta^1 \ar[r,"{\iota^\omega}"] & \CAT[\mathrm{ex}] \,.
        \end{tikzcd}
    \end{equation*}
    By the same argument as before, there is a canonical diagonal map $ \widetilde{M}_{*} $ extending $ \overline{M}_{*,0} $ so that $ \widetilde{M}_* $ sends $ (\id_A, 0<1) $ to a $ c $-cartesian arrow. 
    Observe that the restriction $ \widetilde{M}_*^\omega $ of $ \widetilde{M}_* $ to $ \EE_\sigma\Alg \times \{0\} $ sits in a commutative diagram
    \begin{equation}\label{diagram:E_sigma_alg_to_hermitian_cat}
        \begin{tikzcd}
            \EE_\sigma\Alg \simeq \EE_\sigma\Alg \times \{0\} \ar[d] \ar[r,"{\widetilde{M}^\omega_*}"] & \int \Fun^{q}(-) \simeq \Cath \ar[d] \\
            \EE_1\Alg(\Spectra) \ar[r,"{\ev_0 \circ \iota^\omega}"] & \Catex \,.         
        \end{tikzcd}     
    \end{equation} 
    By definition of $ \widetilde{M}_* $ and $ \iota^\omega $, the functor $ \widetilde{M}_{*}^\omega $ sends an $ \EE_\sigma $-algebra $ A $ to the composite $ \Mod_{A^e}^\omega \subseteq \Mod_{A^e} \xrightarrow{N^{C_2}} \Mod_{N^{C_2}A} \xrightarrow{\hom_{N^{C_2}A}(-,A)} \Spectra $. 
\end{construction}

\begin{definition}\label{defn:E_sigma_alg_to_hermitian_cat}
    Let us write $ \Mod^\mathrm{p} $ for the functor $ \EE_\sigma\Alg \to \Cath $ of (\ref{diagram:E_sigma_alg_to_hermitian_cat}).
\end{definition}
\begin{theorem}\label{thm:Ep_alg_to_poincare_cat}
    \begin{enumerate}[label=(\arabic*)]
        \item \label{thmitem:E_sigma_alg_to_hermitian_factors} The functor of Definition \ref{defn:E_sigma_alg_to_hermitian_cat} factors through the subcategory $ \Catp \subseteq \Cath $. 
        \item \label{thmitem:borel_to_symmetric} The functor $ \Mod^\mathrm{p} $ sends $ \EE_\sigma $-algebras which are Borel to Poincaré $ \infty $-categories which are symmetric in the sense of \cite[Definition 1.2.11]{CDHHLMNNSI}. 
        \item \label{thmitem:unit_to_unit} Recall that $ \EE_\sigma\Alg $ has an initial object given by the sphere spectrum (Remark \ref{rmk:E_sigma_operad_unital}). 
        Then there is an equivalence $ \Mod^\mathrm{p}_{\sphere} \simeq \left(\Spectra^\omega,\Qoppa^u\right) $ where the latter is the universal Poincaré $ \infty $-category of \cite[\S4.1]{CDHHLMNNSI}. 
        \item \label{thmitem:calgp_to_catp_with_tensor} Recall that $ \Cath $ has a symmetric monoidal structure. 
        There is a functor $ \Mod^{p,\otimes} $ making the diagram 
        \begin{equation}\label{diagram:E_p_alg_to_sym_mon_cath}
            \begin{tikzcd}
                \EE_p \Alg \ar[r,"{\Mod^{p,\otimes}}"'] \ar[d] \ar[rr,bend left=15,"{R \mapsto \Mod_{R^e}^\omega\left(\Spectra\right)}"] & {\EE_\infty\Mon\left(\Cath\right)} \ar[d] \ar[r] & \EE_\infty\Mon\left(\Catex\right) \ar[d] \\
                \EE_\sigma \Alg \ar[r,"{\Mod^\mathrm{p}}"] & \Cath \ar[r] & \Catex 
            \end{tikzcd}
        \end{equation}
        commute, where the left vertical functor is induced by Observation \ref{obs:E_sigma_param_operad_to_calgp_param_operad}. 
        \item \label{thmitem:E_p_alg_to_sym_mon_hermitian_factors} The functor $ \Mod^{p,\otimes} $ of \ref{thmitem:calgp_to_catp_with_tensor} factors through $ \EE_\infty\Mon\left(\Catp\right) $.
        \item \label{thmitem:calgp_to_catp_preserves_tensor} The composite $$ \EE_p\Alg \xrightarrow{\Mod^{p,\otimes}} \EE_\infty \Mon\left(\Catp\right) \xrightarrow{(-)^\natural} \EE_\infty \Mon\left(\Catpidem\right) $$ is symmetric monoidal. 
        Here $ (-)^\natural $ denotes the localization functor of Corollary \ref{proposition:idempotent_completion_of_sym_mon_catp}. 
    \end{enumerate}
\end{theorem}
\begin{proof}
    Since $ \Catp $ is a subcategory of $ \Cath $, for part \ref{thmitem:E_sigma_alg_to_hermitian_factors} we must show that $ \Mod^\mathrm{p} $ sends objects of $ \EE_\sigma\Alg $ to Poincaré $ \infty $-categories and morphisms in $ \EE_\sigma \Alg $ to duality-preserving hermitian functors. 
    That $ \Mod^\mathrm{p} $ sends an $ \EE_\sigma $-algebra to a Poincaré $ \infty $-category follows from Lemma \ref{lemma:E_sigma_fixpt_to_underlying}, Construction \ref{cons:precompose_with_norm_in_families}, and Proposition 3.1.6 and Theorem 3.3.1 of \cite{CDHHLMNNSI} (also see Remark 3.3.4 \emph{loc. cit.}). It remains to see that the maps of $\EE_\sigma$ algebras induce Poincar\'e functors, for which we will follow the argument in \cite[Example 3.4.5]{CDHHLMNNSI},
    Unraveling Construction \ref{cons:precompose_with_norm_in_families} further, given a map $ \varphi \colon A \to B $ in $ \EE_\sigma\Alg $, the induced hermitian functor $ \Mod^\mathrm{p}(\varphi) $ covering the induction functor $ - \otimes_A B \colon \Mod^\omega_A \to \Mod^\omega_B $ corresponds under the identification of \cite[Corollary 3.4.2 and preceding discussion]{CDHHLMNNSI} to the map $ \varphi \colon A \to B $ \emph{of $ N^{C_2} A $-modules}, where $ B $ is regarded as a $ N^{C_2} A $-module via $ N^{C_2} A \xrightarrow{N^{C_2}(\varphi)} N^{C_2} B $. 
    This manifestly satisfies the assumptions of \cite[Lemma 3.4.3]{CDHHLMNNSI}, hence $ \Mod^\mathrm{p} $ takes maps in $ \EE_\sigma\Alg $ to duality-preserving hermitian functors.
    
    Part \ref{thmitem:borel_to_symmetric} follows from the characterization of symmetric Poincaré structures on module categories in \cite[Theorem 3.3.1]{CDHHLMNNSI}. 
    
    Part \ref{thmitem:unit_to_unit} follows from unraveling definitions. 
    
    Part \ref{thmitem:calgp_to_catp_with_tensor} follows from a modification of the constructions used to define $ \Mod^\mathrm{p} $ in the first place. 
    First, note that there exists a commutative diagram 
    \begin{equation*}
        \begin{tikzcd}[cramped, column sep=huge]
            \Assoc^\otimes \ar[rr] \ar[d] & & \underline{\Assoc}_{\sigma}^\otimes \ar[d] \\
            \Fin_* \ar[rr,"{\langle 1 \rangle \mapsto [C_2/C_2 = C_2/C_2]}"] & & \EE_p^\otimes         
        \end{tikzcd}
    \end{equation*} 
    where the upper row is from Construction \ref{cons:norm_in_families}. 
    By \cite[Proposition 4.5.1.4 \& Theorem 4.5.3.1]{LurHA}, this induces a commutative diagram 
    \begin{equation*}
        \begin{tikzcd}[cramped]
            \EE_p\Alg \ar[r,"{G^\otimes}"]\ar[d] & {\EE_\infty \Alg\Fun\left(\Delta^1,\CAT[\Mod]\right)} \ar[d] \\
            \EE_\sigma\Alg \ar[r,"{\Theta \circ G}"] & {\Fun\left(\Delta^1,\CAT[\Mod]\right)}
        \end{tikzcd}
    \end{equation*}
    where $ G $ is the functor of (\ref{eq:assoc_alg_norm_in_families_straightened}). 
    That $ \Theta $ is symmetric monoidal is \cite[Theorem 4.8.5.16]{LurHA}. 
    
    By \cite[Corollary 3.4.1.5]{LurHA}, the functor $ \Delta^1 \times \Fin_*  \to \mathrm{Env}_{\underline{\Fin}_{C_2,*}}\left(\EE_p^\otimes\right) $ of Construction \ref{cons:Einfty_norm_in_families} induces a functor $ \EE_p\Alg \to {\CAT[\EE_\infty \Alg]} $ lifting $ \ev_1 \circ G^\otimes $. 
    By Lemma \ref{lemma:E_p_E_sigma_bimod_compatible}, there is a commutative diagram
    \begin{equation*}
        \begin{tikzcd}[cramped]
            \EE_p\Alg \ar[r,"{A_*}"'] \ar[rr, bend left=15,"{\ev_1 \circ G^\otimes}"] \ar[d] & {\CAT[\EE_\infty \Alg, \mathrm{ex}]} \ar[d] \ar[r] & \Mon_{\EE_\infty}\left(\CAT[\mathrm{ex}]\right) \ar[d] \\
            \EE_\sigma\Alg \ar[r,] \ar[rr,bend right=15, "{\ev_1 \circ G}"'] & {\widetilde{\CAT[\mathrm{ex}]}} \ar[r] & {\CAT[\mathrm{ex}]} 
        \end{tikzcd}
    \end{equation*}
    where the lower row is from (\ref{diagram:calgp_to_functor_cat_with_pointed_target}). 
    Thus far, we have constructed a commutative diagram  
    \begin{equation*}
        \begin{tikzcd}[cramped]
            \EE_p\Alg \ar[rrr,"{A_*}"] \ar[d] & & & {\CAT[\EE_\infty\Alg,\mathrm{ex}]} \ar[d,"\phi"] \\ 
            \EE_\infty\Alg\left(\underline{\Spectra}^{C_2}\right) \ar[r,"{\Theta \circ G}"] & \EE_\infty \Alg \Fun\left(\Delta^1, \CAT[\Mod] \right) \ar[r,"{\ev_1 \circ U}"] & \EE_\infty\Alg\left(\CAT\right) & \EE_\infty\Alg\left(\CAT[\mathrm{ex}]\right) \ar[l] 
        \end{tikzcd}
    \end{equation*} 
    lifting the diagram (\ref{diagram:calgp_to_functor_cat_with_pointed_target}). 
    
    Composing $ A_* $ with the functor of Lemma \ref{lemma:stable_cat_functorial_multiplicative_Yoneda}\ref{lemmaitem:mult_yoneda_on_algebras}, we have a commutative diagram
    \begin{equation*}
        \begin{tikzcd}
            \EE_p\Alg \ar[r] \ar[d] & {\CAT[\EE_\infty\Alg,\mathrm{ex}]} \ar[d] \ar[r] & {\Mon_{\EE_\infty}\left(\CAT[\mathrm{ex}]\right)_{(-)^\op/\mathrm{lax}/\Spectra}} \ar[d] \\
            \EE_\sigma\Alg \ar[r] & {\widetilde{\CAT[\mathrm{ex}]}} \ar[r] & \int \Fun^{\mathrm{ex}}\left((-)^\op,\Spectra\right)     
        \end{tikzcd}
    \end{equation*}
    lifting the diagram (\ref{diagram:calgp_to_functor_cat_with_presheaf_on_target}). 
    Using the cartesian fibration $ \EE_\infty\Mon\left(\CAT\right)_{(-)^\op/\mathrm{lax}/\Spectra} \to \EE_\infty\Mon\left(\CAT\right) $ in place of $ \left(\CAT\right)_{(-)^\op/\Spectra} \to \CAT $, the argument of Construction \ref{cons:precompose_with_norm_in_families} shows that there is a commutative diagram 
    \begin{equation*}
        \begin{tikzcd}
            \EE_p \Alg \ar[r] \ar[d] \ar[rr,bend left=15,"{R \mapsto \Mod_{R^e}\left(\Spectra\right)}"] & \EE_\infty\Mon\left(\CAT\right)_{(-)^\op/\mathrm{lax}/\Spectra} \ar[d] \ar[r] & \EE_\infty\Mon\left(\CAT[\mathrm{ex}]\right) \ar[d] \\
            \EE_\sigma\Alg \ar[r] & \left(\CAT\right){(-)^\op/\Spectra} \ar[r] & \CAT[\mathrm{ex}]
        \end{tikzcd}
    \end{equation*}
    where the bottom row is from (\ref{diagram:E_sigma_alg_to_functor_cat_with_presheaf_on_target_extended}). 
    Now we observe that by construction, $ A_* $ factors through $ \int \Fun^{\mathrm{lax},q}\left((-)^\op,\Spectra\right) $ the sub-cartesian fibration of $ \EE_\infty\Mon\left(\CAT[\mathrm{ex}]\right)_{(-)^\op/\mathrm{lax}/\Spectra} $ spanned by those lax monoidal functors $ \mathcal{C}^\op \to \Spectra $ which are reduced and quadratic. 
    Now using (\ref{diagram:cpct_modules_to_all_modules_sym_mon}), we can construct in a similar fashion a functor $ A_*^\omega $ making the diagram 
    \begin{equation*}
        \begin{tikzcd}[cramped, column sep=large]
            \EE_p \Alg \ar[r,"{A_{*}^\omega}"'] \ar[rr,bend left=15,"{R \mapsto \Mod_{R^e}^\omega\left(\Spectra\right)}"] \ar[d] & {\int \Fun^{\mathrm{lax},q}\left((-)^\op,\Spectra\right)} \ar[d] \ar[r] & \EE_\infty\Mon\left(\Catex\right) \ar[d] \\
            \EE_\sigma \Alg \ar[r,"{\widetilde{M}_*^\omega}"] & \int \Fun^q\left((-)^\op,\Spectra\right) \ar[r] & \Catex 
        \end{tikzcd}
    \end{equation*}
    commute, where the bottom row is from (\ref{diagram:E_sigma_alg_to_hermitian_cat}) and the vertical arrows are forgetful functors. 
    Now Lemma \ref{lemma:cath_to_catex_cartesian_on_Einfty_alg}\ref{lemmaitem:cath_to_catex_cartesian_on_Einfty_alg} and \cite[Lemma 5.3.2 \& Corollary 5.3.7]{CDHHLMNNSI} give an identification $ \int \Fun^{\mathrm{lax},q}\left((-)^\op,\Spectra\right) \simeq \Mon_{\EE_\infty}\left(\Cath\right) $. 
    Writing $ \Mod^{p,\otimes} $ for the composite of $ A_*^\omega $ with the aforementioned equivalence, we obtain the desired commutative diagram (\ref{diagram:E_p_alg_to_sym_mon_cath}).  
    
    Next we prove \ref{thmitem:E_p_alg_to_sym_mon_hermitian_factors}; its proof is similar to the proof of part \ref{thmitem:E_sigma_alg_to_hermitian_factors}. 
    Under the identification of \cite[discussion preceding Lemma 5.4.6, also see Corollary 5.4.7]{CDHHLMNNSI} (and essentially by construction/unraveling definitions), given an $ \EE_p $-algebra $ A $, the symmetric monoidal refinement of the hermitian (in fact, Poincaré by \ref{thmitem:E_sigma_alg_to_hermitian_factors}) structure on $ \Mod^\mathrm{p}_A $ comes from regarding $ A $ as an $ \EE_\infty $-$ N^{C_2} A $-algebra. 
    It follows from Corollary 5.4.8 \emph{loc. cit.} that $ \Mod^{p,\otimes} $ sends an $ \EE_p $-algebra to a symmetric monoidal Poincaré $ \infty $-category. 
    Now suppose given a map $ \varphi \colon A \to B $ of $ \EE_p $-algebras. 
    Under the identification of Lemma \ref{lemma:monoidal_herm_func_repping_obj}\ref{lemma_item:monoidal_herm_func_repping_obj_data}, the symmetric monoidal hermitian functor $ \Mod^{p,\otimes}(\varphi) $ corresponds to $ \varphi $, regarded as a map \emph{of $ \EE_\infty $-$ N^{C_2}A $-algebras}. 
    It follows from Lemma \ref{lemma:monoidal_herm_func_repping_obj}\ref{lemma_item:monoidal_herm_func_repping_obj_Poincare_condition} that $ \Mod^{p,\otimes}(\varphi) $ is in fact a symmetric monoidal Poincaré functor. 
    
    Suppose given Poincaré rings $ A, B $. 
    By \ref{thmitem:calgp_to_catp_with_tensor}, the canonical maps $ A \to A \otimes B $, $ B \to A \otimes B $ induce symmetric monoidal hermitian functors $ \Mod^\mathrm{p}_A \to \Mod^\mathrm{p}_{A \otimes B} $ and $ \Mod^\mathrm{p}_B \to \Mod^\mathrm{p}_{A \otimes B} $. 
    These induce a canonical symmetric monoidal hermitian functor $ (F,\eta) \colon \Mod^\mathrm{p}_A \otimes \Mod^\mathrm{p}_B \to \Mod^\mathrm{p}_{A \otimes B} $; to prove \ref{thmitem:calgp_to_catp_preserves_tensor}, it suffices to show that $ (F,\eta)^\natural $ is an equivalence. 
    By \cites[Remark 4.8.5.19]{LurHA}[Proposition 5.3.4.17]{HTT}, $ F $ induces an equivalence $ \left(\Mod^\omega_A \otimes \Mod^\omega_B\right)^\natural \xrightarrow{\sim} \left(\Mod_{A\otimes B}^\omega\right)^\natural \simeq \Mod_{A\otimes B}^\omega $; thus it suffices to show that $ \eta $ induces an equivalence of $ \EE_\infty $-hermitian structures. 
    By \cite[Corollary 5.3.14 \& Corollary 5.4.7]{CDHHLMNNSI}, it suffices to show that $ \eta \colon \Qoppa_A \otimes \Qoppa_B \Rightarrow F^*\Qoppa_{A \otimes B} $ induces equivalences of the corresponding $ \EE_\infty $-$ N^{C_2}\left(A \otimes B\right) $-algebras. 
    The quadratic functor $ F^*\Qoppa_{A \otimes B} $ is associated to the $ \EE_\infty $-$ N^{C_2}(A \otimes B) $-algebra $ A \otimes B $. 
    It follows from \cite[Example 5.1.6]{CDHHLMNNSI} that the natural transformation $ \eta $ is an equivalence. 
\end{proof}
\begin{corollary}\label{cor:unit_poincare_object}
    There are canonical lifts 
    \begin{equation}
        \EE_\sigma\Alg \to \Catp_{\left(\Spectra^\omega,\Qoppa^u\right)/-} \qquad \qquad \EE_p\Alg \to \EE_\infty \Mon\left(\Catp\right)_{\left(\Spectra^\omega,\Qoppa^u\right)/-}
    \end{equation} 
    of the functors of Theorem \ref{thm:Ep_alg_to_poincare_cat}. 
    In particular, if $ A $ is an $ \EE_\sigma $-algebra, then $ A^e $ canonically promotes to a Poincaré object $ (A^e, q_A) $ of $ \Mod^\mathrm{p}_A $. 
\end{corollary}
\begin{proof}
    We prove the result for $ \EE_p $-algebras; the result for $ \EE_\sigma $-algebras is nearly identical. 
    By Remark \ref{rmk:E_sigma_operad_unital}, there is an equivalence $ \EE_p\Alg \simeq \EE_p\Alg_{\sphere/-} $. 
    By Theorem \ref{thm:Ep_alg_to_poincare_cat}, we have a canonical map $ \Mod^\mathrm{p}_{\sphere} \simeq \left(\Spectra^\omega,\Qoppa^u\right) \to \Mod^\mathrm{p}_A $. 
    The result follows from noting that $ \left(\Spectra^\omega,\Qoppa^u\right) $ corepresents $ \mathrm{Pn} $ by \cite[Proposition 4.1.3]{CDHHLMNNSI}. 
\end{proof}

\begin{recollection}
    [{\cite[Definition 4.8.3.1]{LurHA}}] Let $ O^\otimes $ be an $ \infty $-operad and $ S $ a simplicial set. 
    A \emph{cocartesian $ S $-family of $ O $-monoidal $ \infty $-categories} is a cocartesian fibration $ q \colon \mathcal{C}^\otimes \to O^\otimes \times S $ so that for all $ s \in S $, the induced map of fibers $ \mathcal{C}^\otimes_s = \mathcal{C}^\otimes \times_{S} \{s\} \to O^\otimes $ is a cocartesian fibration of $ \infty $-operads. 
    
    When $ O = \EE_\infty $, we will refer to $ q $ as a \emph{cocartesian $ S $-family of symmetric monoidal $ \infty $-categories.}
\end{recollection}

\begin{example}
    Let $ \Mon_{\EE_\infty} \CAT $ denote the $ \infty $-category of symmetric monoidal $ \infty $-categories (i.e. $ \EE_\infty $-monoid objects of $ \CAT $). 
    There is a canonical map $ \Fin_* \times \Mon_{\EE_\infty} \left(\CAT\right) \to \CAT $ which classifies a cocartesian fibration $ q \colon \widetilde{\Mon}_{\EE_\infty}\left(\CAT\right) \to \Fin_* \times \Mon_{\EE_\infty}\left(\CAT\right) $. 
    The cocartesian fibration $ q $ exhibits $ \widetilde{\Mon}_{\EE_\infty}\left(\CAT\right) $ as the universal cocartesian $ \Mon_{\EE_\infty}\left(\CAT\right) $-family of symmetric monoidal $ \infty $-categories: For any simplicial set $ S $, the assignment 
    \begin{equation*}
        \left(\phi \colon S \to \Mon_{\EE_\infty}\left(\CAT\right)\right) \mapsto \left(S \times_{\Mon_{\EE_\infty}\left(\CAT\right)} \widetilde{\Mon}_{\EE_\infty}\left(\CAT\right)\right)
    \end{equation*}
    establishes an equivalence between diagrams $ S \to \Mon_{\EE_\infty}\left(\CAT\right) $ and cocartesian $ S $-families of symmetric monoidal $ \infty $-categories. 
    
    If $ \mathcal{K} $ is a collection of simplicial sets, there is a variant $ \Mon_{\EE_\infty} \CAT\left(\mathcal{K}\right) $ consisting of symmetric monoidal $ \infty $-categories which are compatible with $ \mathcal{K} $-indexed colimits; morphisms consist of symmetric monoidal functors which preserve $ \mathcal{K} $-indexed colimits. 
\end{example}
\begin{definition}
    Let $ \CAT[\EE_\infty\Alg] $ denote the full subcategory of the fiber product
    \begin{equation*}
        \Mon_{\EE_\infty}\left(\CAT\right) \fiberproduct_{\Fun_{\Fin_*}\left(\Fin_*, \Fin_* \times \Mon_{\EE_\infty}\left(\CAT\right)\right)} \Fun_{\Fin_*}\left(\Fin_*,\widetilde{\Mon}_{\EE_\infty}\left(\CAT\right)\right)       
    \end{equation*}   
    spanned by those pairs $ \left(\mathcal{C}^\otimes, A\right) $ where $ \mathcal{C}^\otimes $ is a symmetric monoidal $ \infty $-category and $ A $ is an $ \EE_\infty $-algebra object of the symmetric monoidal $ \infty $-category $ \widetilde{\Mon}_{\EE_\infty}\left(\CAT\right) \fiberproduct_{{\Mon}_{\EE_\infty}\left(\CAT\right)} \left\{\mathcal{C}^\otimes\right\} \simeq \mathcal{C}^\otimes $. 
\end{definition}
\begin{remark}
    The $ \infty $-category $ \CAT[\EE_\infty\Alg] $ is characterized by the universal property that for any simplicial set $ S $, there is a bijection between equivalence classes of diagrams $ S \to \CAT[\EE_\infty\Alg] $ and diagrams 
    \begin{equation*}
        \begin{tikzcd}
            & \mathcal{C}^\otimes \ar[d,"q"] \\
            \Fin_* \times S \ar[ru,"A"] \ar[r,equals] & \Fin_* \times S
        \end{tikzcd}
    \end{equation*}
    where $ q $ exhibits $ \mathcal{C}^\otimes $ as a cocartesian $ S $-family of symmetric monoidal $ \infty $-categories and $ A $ is an $ S $-family of $ \EE_\infty $-algebra objects of $ \mathcal{C}^\otimes $. 
\end{remark}
\begin{lemma}\label{lemma:stable_cat_functorial_multiplicative_Yoneda}
    \begin{enumerate}[label=(\arabic*)]
        \item \label{lemmaitem:pointed_symmon_cat_cocartesian} 
        The canonical map $ \phi \colon \CAT[\EE_\infty\Alg]\left(\mathcal{K}\right) \to \Mon_{\EE_\infty}^{\mathcal{K}}\left(\CAT\right) $ is a cocartesian fibration. 
        \item \label{lemmaitem:mult_Yoneda_operad} There is a commutative diagram of symmetric monoidal $ \infty $-categories
        \begin{equation*}
            \begin{tikzcd}[column sep=tiny]
                \left(\Cat_{\infty,*}\right)^\times \ar[rd,"{\varphi^\times}"'] \ar[rr,"{i^\times}"] & & \left(\Cat_{\infty}\right)^\times_{(-)^\op/\mathrm{lex}/\Spaces} \ar[ld,"{q^\times}"] \\
                & \Cat_\infty^\times & 
            \end{tikzcd}
        \end{equation*} 
        so that the pullback along the inclusion $ \Catex \subseteq \left(\Cat_\infty^\times\right)_{\langle 1 \rangle} $ recovers the commutative diagram of Lemma \ref{lemma:stable_cat_functorial_Yoneda}.  
        \item \label{lemmaitem:mult_yoneda_on_algebras} Let $ \phi \colon \CAT[\EE_\infty\Alg,\mathrm{ex}]\left(\mathcal{K}\right) \to \Mon_{\EE_\infty}^{\mathcal{K}}\left(\CAT[\mathrm{ex}]\right) $ denote the restriction of $ \phi $ to the subcategory of $ \Mon_{\EE_\infty}^{\mathcal{K}}\left(\CAT\right) $ on stable $ \infty $-categories and exact functors between them. 
        Consider the cocartesian fibration \[ \Mon_{\EE_\infty}^{\mathcal{K}}\left(\CAT[\mathrm{ex}]\right)_{(-)^\op / \mathrm{lax} / \Spectra} \to \Mon_{\EE_\infty}^{\mathcal{K}}\left(\CAT[\mathrm{ex}]\right) \] classified by the functor $ \Mon_{\EE_\infty}^{\mathcal{K}}\left(\CAT[\mathrm{ex}]\right) \to \CAT $ sending a symmetric monoidal $ \infty $-category $ \mathcal{C} $ to $ \Fun^{\mathrm{ex}}_{\mathrm{lax}}\left(\mathcal{C}^\op, \Spectra\right) $. 
        Then there is a commutative diagram
        \begin{equation*}
            \begin{tikzcd}[column sep=tiny]
                \CAT[\EE_\infty\Alg,\mathrm{ex}]\left(\mathcal{K}\right) \ar[rr,"i"] \ar[rd,"{\varphi}"'] & & \Mon_{\EE_\infty}^{\mathcal{K}}\left(\CAT[\mathrm{ex}]\right)_{(-)^\op / \mathrm{lax} / \Spectra}  \ar[ld,"q"] \\
                & \Mon_{\EE_\infty}^{\mathcal{K}}\left(\CAT[\mathrm{ex}]\right) &
            \end{tikzcd}    
        \end{equation*}
        so that $ i $ sends $ \varphi $-cocartesian edges to $ q $-cocartesian edges. 
        This is compatible with the functors of Lemma \ref{lemma:stable_cat_functorial_Yoneda} in the sense that the forgetful functor sending a symmetric monoidal stable $ \infty $-category to its underlying stable $ \infty $-category (and corresponding functors sending $ \left(\mathcal{C}^\otimes, A \in \EE_\infty\Alg(\mathcal{C})\right) \mapsto (\mathcal{C}, A \in \mathcal{C}) $) induce a $ \Delta^1 \times \Delta^2 $-shaped commutative diagram of $ \infty $-categories from the triangle here to the triangle found in Lemma \ref{lemma:stable_cat_functorial_Yoneda}. 
    \end{enumerate} 
\end{lemma}
\begin{proof}
    The proof of part \ref{lemmaitem:pointed_symmon_cat_cocartesian} is similar to the proof of \cite[Proposition 4.8.5.1]{LurHA}. 
    
    To prove \ref{lemmaitem:mult_Yoneda_operad}, suppose that $ p_i \colon \mathcal{C}_i \to \mathcal{D}_i $, $ i \in I $ ($ I $ a finite set) is a collection of right fibrations. 
    Then notice that $ \prod_i p_i \colon \prod_i \mathcal{C}_i \to \prod_i \mathcal{D}_i $ is a right fibration. 
    Furthermore, if each $ p_i $ has a final object (resp. is $ \omega $-filtered), then so is $ \prod_i p_i $. 
    Therefore, the cartesian symmetric monoidal structure on $ \infty $-categories induces a commutative diagram 
    \begin{equation*} 
        \begin{tikzcd}[column sep=small]
            \mathrm{RFib}_*^\times \ar[rr,"i"] \ar[rd,"{t_0}"'] & & \left(\mathrm{RFib}^{\omega\text{-fil}}\right)^\times \ar[ld,"{t_1}"] \\
            & \CAT[\times] & 
        \end{tikzcd} 
    \end{equation*} 
    of $ \infty $-operads. 
    That the pullback along the inclusion $ \Catex \subseteq \left(\Cat_\infty^\times\right)_{\langle 1 \rangle} $ recovers the commutative diagram of Lemma \ref{lemma:stable_cat_functorial_Yoneda} follows by inspection. 
    
    To prove \ref{lemmaitem:mult_yoneda_on_algebras}, consider the diagram obtained by taking $ \EE_\infty $-monoid objects in \ref{lemmaitem:mult_Yoneda_operad}: 
    \begin{equation}\label{diagram:unstable_monoidal_yoneda}
        \begin{tikzcd}[column sep=tiny]
            \EE_\infty\Mon \left((\CAT)_{*/}\right)^\times \ar[rd,"{\varphi}"'] \ar[rr,"{i}"] & & \EE_\infty\Mon \left(\CAT\right)^\times_{(-)^\op/\mathrm{lex}/\Spaces} \ar[ld,"{q}"] \\
            & \EE_\infty\Mon\left((\CAT)^\times\right) \,. & 
        \end{tikzcd}
    \end{equation} 
    By \cite[Corollary 3.2.2.3]{LurHA}, $ \varphi $ and $ q $ are cartesian fibrations with cartesian transport given by precomposition, and $ i $ takes $ \varphi $-cartesian arrows to $ q $-cartesian arrows. 
    Write $ (\CAT[\mathrm{ex}])^\otimes $ for the suboperad of $ (\CAT)^\times $ on those stable $ \infty $-categories and functors $ f \colon \left(\mathcal{C}_i\right)_{1 \leq i \leq n} \to \mathcal{C} $ so that the corresponding functor $ \mathcal{C}_1 \times \cdots \times \mathcal{C}_n \to \mathcal{C} $ is exact in each variable. 
    Now consider the diagram obtained by pulling back (\ref{diagram:unstable_monoidal_yoneda}) along the induced functor $ \EE_\infty\Mon\left((\CAT[\mathrm{ex}])^\otimes\right)\subseteq \EE_\infty\Mon\left((\CAT)^\times\right) $, and suppose that $ \mathcal{C}^\otimes $ is a stable symmetric monoidal $ \infty $-category. 
    Using \cite[Remark 6.15]{MR4545922}, an object of $ \EE_\infty\Mon \left(\Cat_{\infty}\right)^\times_{(-)^\op/\Spaces} $ lying over $ \mathcal{C}^\otimes $ is an $ \EE_\infty $-algebra of the presheaf category $ \Fun^{\mathrm{lex}}\left(\mathcal{C}^\op, \Spaces\right) $ with the Day convolution, or in other words a lax monoidal left-exact functor $ \mathcal{C}^\op \to \Spaces $. 
    By \cite[Proposition 6.5]{Nik_16}, $ \Omega^\infty $ induces an equivalence $ \Fun^{\mathrm{lex}, \mathrm{lax}}\left(\mathcal{C}^\op, \Spectra \right) \simeq \Fun^{\mathrm{lex},\mathrm{lax}}\left(\mathcal{C}^\op, \Spaces\right) $. 
\end{proof}

\begin{lemma}\label{lemma:cath_to_catex_cartesian_on_Einfty_alg}
    \begin{enumerate}[label=(\arabic*)]
        \item \label{lemmaitem:spectra_laxslice_to_catex_cartesian_on_Einfty_alg} The functor $ p^\otimes \colon \left(\Cat_\infty\right)^\otimes_{//\Spectra} \to \left(\Cat_\infty\right)^\otimes $ of \cite[\S5.2]{CDHHLMNNSI} induces a cartesian fibration $ \EE_\infty \Mon\left(\Cat_\infty\right)^\otimes_{(-)^\op/\mathrm{lax}/\Spectra} \to \EE_\infty \Mon \left(\Cat_\infty\right)^\otimes $, where cartesian transport is given by precomposition, where $ \left(\Cat_\infty\right)^\otimes_{//\Spectra} $ denotes the lax slice over $ \Spectra $ of \cite[\S5.2]{CDHHLMNNSI}. 
        \item \label{lemmaitem:cath_to_catex_cartesian_on_Einfty_alg} 
        The functor $ p^\otimes \colon \left(\Cath\right)^\otimes \to \left(\Catex\right)^\otimes $ of \cite[\S5.2]{CDHHLMNNSI} induces a cartesian fibration $ \EE_\infty\Mon\left(\Cath\right) \to \EE_\infty\Mon\left(\Catex\right) $, where cartesian transport is given by precomposition. 
    \end{enumerate}
\end{lemma}
\begin{proof}
    We prove \ref{lemmaitem:cath_to_catex_cartesian_on_Einfty_alg}; the proof of \ref{lemmaitem:spectra_laxslice_to_catex_cartesian_on_Einfty_alg} is similar. 
    That the induced map on $ \EE_\infty $ monoids is a cartesian fibration follows from taking $ \mathcal{O}^\otimes = \Fin_* $, $ \mathcal{C}^\otimes = \left(\Cath\right)^\otimes $, and $ \mathcal{D}^\otimes = \left(\Catex\right)^\otimes $ in \cite[Corollary 3.2.2.3(1)]{LurHA}, in view of the dual to \cite[Example 4.3.1.4]{HTT} and the fact that $ \Cath \to \Catex $ is a cartesian fibration. 
    The description of cartesian transport follows from Corollary 3.2.2.3(2) \emph{ibid.}
\end{proof}
The following is a symmetric monoidal enhancement of \cite[Corollary 3.4.2, Lemma 3.4.3]{CDHHLMNNSI}. 
\begin{lemma}\label{lemma:monoidal_herm_func_repping_obj}
    Let $ A, B $ be $ \EE_\infty $-algebras and let $ \varphi \colon A \to B $ be a map of $ \EE_\infty $-algebras. 
    Let $ M $ (resp. $ N $) be an $ \EE_\infty $-$ N^{C_2} A $ (resp. $ N^{C_2} B$-)algebra, which corresponds to a symmetric monoidal hermitian structure on $ \Mod_A^\omega $ (resp. $ \Mod_B^\omega $) by \cite[Corollary 5.4.7]{CDHHLMNNSI}. 
    \begin{enumerate}[label=(\arabic*)]
        \item \label{lemma_item:monoidal_herm_func_repping_obj_data} Then the data of a symmetric monoidal hermitian functor $ \left(\Mod_A^\omega,\Qoppa_M\right) \to \left(\Mod_B^\omega,\Qoppa_N\right) $ covering the induction functor $ p_\varphi:= (-) \otimes_A B $ is equivalent to the data of a map $ M \to N $ of $ \EE_\infty $-$ N^{C_2} A $-algebras, where $ N $ is regarded as an $ \EE_\infty $-$ N^{C_2} A $-algebra via $ N^{C_2}(\varphi) \colon N^{C_2} A \to N^{C_2} B $.  
        \item \label{lemma_item:monoidal_herm_func_repping_obj_Poincare_condition} In the situation of part \ref{lemma_item:monoidal_herm_func_repping_obj_data}, the symmetric monoidal hermitian functor is symmetric monoidal duality-preserving (resp. Poincaré) if its underlying hermitian functor is duality-preserving (resp. Poincaré). 
    \end{enumerate}
\end{lemma}
\begin{proof}
    Note that $ p_\varphi $ is a morphism in $ \EE_\infty\Mon\left(\Catex\right) $. 
    By Lemma \ref{lemma:cath_to_catex_cartesian_on_Einfty_alg}\ref{lemmaitem:cath_to_catex_cartesian_on_Einfty_alg}, the data of a symmetric monoidal hermitian functor $ (p_{\varphi},\eta) \colon \left(\Mod_A^\omega,\Qoppa_M\right) \to \left(\Mod_B^\omega,\Qoppa_N\right) $ covering the induction functor $ p_\varphi:= (-) \otimes_A B $ is equivalent to the data of a map $ \eta' \colon \Qoppa_M \to p_\varphi^* \Qoppa_N $ of $ \EE_\infty $-algebras in $ \Fun^q\left(\Mod_A^\omega\right) \simeq \Mod_{N^{C_2}A} $, where $ p_\varphi^* \Qoppa_N $ denotes cartesian transport of $ \Qoppa_N $ along $ p_{\varphi} $. 
    Using the description of cartesian transport in Lemma \ref{lemma:cath_to_catex_cartesian_on_Einfty_alg}\ref{lemmaitem:cath_to_catex_cartesian_on_Einfty_alg}, under the identification of \cite[discussion immediately preceding Lemma 5.4.6, also see Corollary 5.4.7]{CDHHLMNNSI}, $ p_\varphi^* \Qoppa_N $ corresponds to $ N $, regarded as an $ \EE_\infty $-$ N^{C_2}A $-algebra via $ N^{C_2}(\varphi) $. 
    Part \ref{lemma_item:monoidal_herm_func_repping_obj_data} follows from these observations. 
    
    Part \ref{lemma_item:monoidal_herm_func_repping_obj_Poincare_condition} follows from the definitions. 
\end{proof}

\subsection{An operadic description of Poincaré rings}\label{subsection:calgp_operadic}
In this subsection, we show that algebras over $ \mathbb{E}_p $ are equivalent to Poincaré rings defined diagrammatically. 
The constructions and proofs are very similar to those of \cite{LYang_normedrings}, with the $ C_2 $-$ \infty $-operad $ \EE_p^\otimes $ in place of the commutative/terminal $ C_2 $-$ \infty $-operad. 
We give a terse outline of the argument, including proofs of modified versions of key auxiliary results of \emph{loc. cit.}, but otherwise refer to \emph{loc. cit.} for details.   

Let $ K $ denote the $ \infty $-categorical nerve of the $ 1 $-category
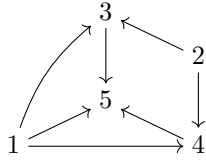
\begin{equation*}
    \begin{tikzcd}[cramped, row sep=tiny]
        & 3 \ar[dd] & \\
        & & 2 \ar[dd] \ar[lu] \\
        & 5 & \\
        1  \ar[rr] \ar[ru] \ar[ruuu, bend left=15] & & 4 \ar[lu]
    \end{tikzcd}
\end{equation*}
in which all triangles and squares commute. 
			
\begin{definition}\label{defn:calgp_alt_diagram}
    Consider the diagram $ \mathcal{P} \colon K \to \Cat_\infty $ 
    \begin{equation*}
        \begin{tikzcd}[column sep=small]
            & \Fun\left(\Delta^1,\EE_\infty\Alg\left(\Spectra\right)^{BC_2}\right) \ar[dd,"{\mathrm{ev}_{0},\mathrm{ev}_{1}}"]  & \\
            & & \Fun\left(\Delta^1,\EE_\infty\Alg\left(\Spectra^{C_2}\right)\right) \ar[dd,"{\mathrm{ev}_{0}, \mathrm{ev}_{1}}"] \ar[lu,"{(-)^e}"'] \\
            &\EE_\infty\Alg(\Spectra)^{BC_2} \times \EE_\infty\Alg(\Spectra)^{BC_2} & \\
            \EE_\infty\Alg\left(\Spectra^{C_2}\right)  \ar[rr,"{N^{C_2}(-^e) \times \id}"'] \ar[ru] \ar[ruuu,"{m \circ (-^e)}", bend left=30] & &\EE_\infty\Alg\left(\Spectra^{C_2}\right) \times \EE_\infty\Alg\left(\Spectra^{C_2}\right) \ar[lu,"{(-)^e \times (-)^e}"]
        \end{tikzcd}
    \end{equation*}
    where $ m $ is the functor of \cite[Construction 3.1]{LYang_normedrings} precomposed with the canonical map $ \Delta^1 \to \mathcal{O}_{C_2} $. 
    Observe that the right-hand trapezoid commutes essentially by definition, and the leftmost triangle commutes because $ \left(N^{C_2} A \right)^e \simeq (A^e)^{\otimes 2} $. 
    Consider the limit of the diagram 
    \begin{equation}
        \CAlgp'\left(\Spectra^{C_2}\right) := \lim_K \mathcal{P} .
    \end{equation} 
    
    There is a canonical forgetful functor $ G' \colon \CAlgp'\left(\Spectra^{C_2}\right) \to \EE_\infty\Alg\left(\Spectra^{C_2}\right)  $ given by the canonical projection to the lower left corner of the diagram. 
\end{definition}
\begin{observation}\label{obs:limits_of_limits}
    Let $ F \colon K \to \mathcal{C} $ be any diagram and let $ \mathcal{C} $ be any $ \infty $-category with finite limits. 
    Then there is a canonical equivalence $ \displaystyle \lim_K F \simeq F(1) \times_{\left(F(3) \times_{F(5)} F(4)\right)} F(2) $. 
\end{observation}
\begin{observation}
    By the recollement decomposition of $ \Spectra^{C_2} $, $ \lim_K \mathcal{P} $ is equivalent to the limit of the diagram\footnote{`res' means restriction.} 
    \begin{equation*}
        \begin{tikzcd}[column sep=small]
            & \Fun\left(\Delta^1,\EE_\infty\Alg\left(\Spectra\right)\right) \ar[dd,"{\mathrm{ev}_{0},\mathrm{ev}_{1}}"]  & \\
            & & \Fun\left(\Delta^1 \times \Delta^1,\EE_\infty\Alg\left(\Spectra\right)\right) \ar[dd,"{(\mathrm{res}_{\Delta^1 \times \{0\}},\mathrm{res}_{\Delta^1 \times \{1\}})}"] \ar[lu,"{\mathrm{res}_{\{1\} \times \Delta^1}}"'] \\
            &\EE_\infty\Alg(\Spectra) \times \EE_\infty\Alg(\Spectra) & \\
            \EE_\infty\Alg\left(\Spectra^{C_2}\right)  \ar[rr,"{N^{C_2}(-^e) \times \id}"'] \ar[ru] \ar[ruuu,"{m^{\mathrm{t}C_2} \circ (-^e)}", bend left=30] & &\EE_\infty\Alg\left(\Spectra^{C_2}\right)^{\Delta^1} \times \EE_\infty\Alg\left(\Spectra^{C_2}\right)^{\Delta^1} \ar[lu,"{\ev_1 \times \ev_1}"]
        \end{tikzcd}
    \end{equation*} 
    By Observation \ref{obs:limits_of_limits} and the fact that all commutative trapezoids below \emph{except} the upper left are pullbacks: 
    \begin{equation}
        \begin{tikzcd}[column sep=small]
            \EE\Alg(\Spectra)^{\Delta^1 \times \Delta^1} \ar[rr,"{\ev_{00\to10\to11}}"] \ar[d,"{(\mathrm{res}_{\Delta^1 \times \{0\}},\mathrm{res}_{\Delta^1 \times \{1\}})}"'] & & \EE\Alg(\Spectra)^{\Delta^2} \ar[d,"{(d_2,\ev_2)}"] \ar[rr,"{d_0}"] & & \EE\Alg(\Spectra)^{\Delta^1} \ar[d,"{(\ev_0,\ev_1)}"] & \\
            \left(\EE\Alg(\Spectra)^{\Delta^1}\right)^{\times 2} \ar[rr,"{(\id,\ev_1)}"] \ar[d,"{\pi_2}"] & &\EE\Alg(\Spectra)^{\Delta^1} \times \EE\Alg(\Spectra) \ar[rd,"{\pi_1}"] \ar[d,"{\pi_2}"] \ar[rr,"{(\ev_1, \id)}"] & & \EE\Alg(\Spectra)^{\times 2} \ar[rd,"{\pi_1}"] & \\
            \EE\Alg(\Spectra)^{\Delta^1} \ar[rr,"{\ev_1}"] & & \EE\Alg(\Spectra) & \EE\Alg(\Spectra)^{\Delta^1} \ar[rr,"{\ev_1}"] & & \EE\Alg(\Spectra)\,,
        \end{tikzcd}
    \end{equation} 
    (Here $ \pi_1, \pi_2 $ denote projections onto the first and second factors, resp.) 
    it follows that 
    \begin{equation*}
        \lim_K \mathcal{P} \simeq \left(\EE_\infty \Alg\left(\Spectra^{C_2}\right)\right) \times_{\left(\EE_\infty\Alg(\Spectra)^{\Delta^2} \times_{\EE_\infty \Alg(\Spectra)} \EE_\infty\Alg(\Spectra)^{\Delta^1}\right)} \EE_\infty\Alg(\Spectra)^{\Delta^1 \times \Delta^1} \,.
    \end{equation*}
    Thus $ \lim_K \mathcal{P} $ is equivalent to $ \CAlgp $ as defined in Definition~\ref{definition:poincare_ring_spectrum}
\end{observation}
\begin{remark}\label{rmk:operadic_Poincare_rings_to_E_infty_alg_in_genuine_spectra}
    There is a map of $ C_2 $-$ \infty $-operads $ \mathrm{Com}_{\mathcal{O}_{C_2}^\simeq} \to \EE_p $. 
    By \cite[Lemma 4.24]{LYang_normedrings}, this induces a functor $ G'\colon \EE_p\Alg\left(\Spectra^{C_2}\right) \to \EE_\infty\Alg\left(\Spectra^{C_2}\right) $, which by an argument similar to Proposition 4.23 \emph{loc.cit.} is monadic. 
\end{remark}
\begin{construction}\label{cons:calgp_operadic_to_diagrammatic}
    There is a canonical functor $ \gamma \colon \EE_p \Alg(\Spectra^{C_2}) \to \CAlgp $. 
    This functor is induced by restricting along different subcategories of $\underline{\mathrm{Fin}}_{C_2,*}$ corresponding to the categories at postions 1, 2, 3, and 4 of $K$ in Definition~\ref{defn:calgp_alt_diagram}. See \cite[Notation 4.1]{LYang_normedrings} for the definition of these subcategories and (after slight modification) the functors to these positions in the diagram. The analogous results of \cite[Proposition 4.6]{LYang_normedrings} continue to hold, and thus we get a comparison functor as claimed. 
    Note that the higher coherences required to show that the individual restriction functors assemble to a functor valued in the limit are handled by Remark \ref{rmk:surprise_lower_category} (cf. Remark 4.8 of \cite{LYang_normedrings}). 
\end{construction}
We will use the Barr-Beck-Lurie theorem to deduce that this map is in fact an equivalence. For this, we will need to understand the free Poincar{\'e} ring spectra, which turn out to be useful independently in the construction of $\mathbb{G}_m^\Qoppa$.
\begin{theorem}\label{thm:free_operadic_calgp_formula}
    Let $ A \in \EE_\infty\Alg\left(\Spectra^{C_2}\right) $ and consider the adjunction $ F \dashv G $ between $ \EE_p $ and $ \EE_\infty $-algebras in $ \Spectra^{C_2} $.  
    \begin{enumerate}[label=(\arabic*)]
        \item The underlying $ C_2 $-spectrum of the free $ \EE_p $ algebra $ F(A) $ on $ A $ is given by 
        \begin{equation}\label{eq:free_operadic_calgp_formula}
            F(A) \simeq
            \begin{tikzcd}
                &  A^{\varphi C_2} \otimes A^e \ar[d,"{s_A \otimes \nu_A}"] \\
                A^e \ar[r,mapsto] & A^{\mathrm{t}C_2}
            \end{tikzcd}
        \end{equation}
        where $ u $ is the unit, $ s_A : A^{\varphi C_2} \to A^{\mathrm{t}C_2} $ is the structure map, and $ \nu_A $ is the twisted Tate-valued Frobenius. 
        \item There is a canonical $ \EE_\infty $ ring map $ \eta_A \colon A \to GF(A) $ given by $ \id_{A^{\varphi C_2}} \otimes (\eta_{A^e}:\mathbb{S}^0 \to A^e) $ on geometric fixed points and the identity on underlying.
    \end{enumerate}
\end{theorem}
\begin{proof}
    The proof is similar to the proof of \cite[Theorem 4.15]{LYang_normedrings}, substituting $ \mathrm{Env}_{\EE_p, \underline{\Fin}_{C_2,*}}\left(\mathrm{Com}_{\mathcal{O}_{C_2}^{\simeq,\op}}\right) $ for $ \mathrm{Com}_{\mathcal{T}^\simeq,act}^\otimes = \mathrm{Env}_{\underline{\Fin}_{C_2,*}, \underline{\Fin}_{C_2,*}}\left(\mathrm{Com}_{\mathcal{O}_{C_2}^{\simeq,\op}}\right) $. 
    Since the indexing diagram is different, one uses Lemma \ref{lemma:freeoperadic_calgp_colimdiagram} instead of Lemma 4.19 \emph{loc. cit}. 
\end{proof}

\begin{lemma}\label{lemma:freeoperadic_calgp_colimdiagram}
    Consider the (non-parametrized) fiber $ I $ of $ \mathrm{Env}_{\EE_p, \underline{\Fin}_{C_2,*}}\left(\mathrm{Com}_{\mathcal{O}_{C_2}^{\simeq,\op}}\right) $ over $ \mathrm{id}_{C_2/C_2} $. 
    An object of $ I $ can be thought of as a pair $ \left([U \to V], \alpha \colon i([U \to V]) \to [\mathrm{id}_{C_2/C_2}]\right)$ where $ [U \to V ] $ is an object in the fiber of $ \mathrm{Com}_{\mathcal{O}_{C_2}^{\simeq,\op}} $ over $ C_2/C_2 $, $ i $ denotes the inclusion $ \mathrm{Com}_{\mathcal{O}_{C_2}^{\simeq,\op}} \to \EE_p $, and $ \alpha $ is an active arrow in $ \EE_p $. 
    
    Choose a $ C_2 $-set $ U $ with free and transitive $ C_2 $-action, and fix an ordering $ \leq $ on $ U $. 
    This determines a t-ordering on the collapse map $ C_2/C_2 \sqcup U \to C_2/C_2 $, hence an active arrow $ \beta $ in $ \EE_p^\otimes $. 
    The inclusion
    \begin{equation*}
        \{*\} \xrightarrow{* \mapsto \beta} I
    \end{equation*}
    is cofinal. 
\end{lemma}
\begin{proof}
    By \cite[Theorem 4.1.3.1]{HTT}, it suffices to show that for all $ X \in I $, the category $ \{*\} \times_{I} I_{X/} $ has an initial object. 
    Informally, an object $ X $ of $ I $ is given by a tuple $ (S, \alpha \colon S \to C_2/C_2) $ where $ S $ is a finite $ C_2 $-set and $ \alpha $ is an active arrow in $ \left(\EE_p^\otimes\right)_{C_2/C_2} $. 
    We observe that $ \alpha $ is given by a $ C_2 $-equivariant map $ f $ (in fact, $ f $ is unique) and a t-ordering on $ f $ in the sense of Notation \ref{ntn:t_order_equivariant_map}. 
    Morphisms are given by tuples; note that a morphism $ \psi \colon (S,\alpha) \to (T,\gamma) $ determines a map $ S \to T $ in $ \left(\mathrm{Com}_{\mathcal{O}_{C_2}^\simeq}\right)_{C_2/C_2} $; in particular, the map $ S \to T $ must be equivalent to a disjoint union of fold maps. 
    Now the t-ordering on $ f $ is equivalent to giving, for each free orbit $ Z \subseteq S $, $ Z \simeq C_2 $, an ordering $ \leq_Z $ on $ Z $. 
    For each such $ Z $, there is a unique order-preserving, $ C_2 $-equivariant isomorphism $ j_Z \colon Z \simeq U $. 
    Write $ S \simeq S^f \sqcup S^t $ for the canonical decomposition where $ C_2 $ acts freely on $ S^f $ and trivially on $ S^t $. 
    Then $ \displaystyle S \simeq S^t \sqcup \left(\bigsqcup_{Z \in \mathrm{Orbit}(S^f)} Z\right) \xrightarrow{\nabla \sqcup \left(\bigsqcup j_Z\right)} C_2/C_2 \sqcup U $ defines a morphism in $ \left(\mathrm{Com}_{\mathcal{O}_{C_2}^{\simeq,\op}}\right)_{C_2/C_2} $ so that the diagram
    \begin{equation*}
        \begin{tikzcd}
            S \ar[rr,"{\nabla \sqcup \left(\bigsqcup j_Z\right)}"] \ar[rd,"\alpha"'] & & C_2/C_2 \sqcup U\ar[ld,"{\beta}"] \\
            & C_2/C_2 &         
        \end{tikzcd}     
    \end{equation*} 
    commutes in $ \EE_p $; this is the desired initial object of $ \{*\} \times_{I} I_{X/} $.  
\end{proof}

Finally, before proving that the map of Construction~\ref{cons:calgp_operadic_to_diagrammatic} is an equivalence, we will need to know how the unit of the adjoint of $G$ behaves.
\begin{proposition}\label{prop:calgp_monadic}
    The forgetful functor $ G' \colon \CAlgp\left(\Spectra^{C_2}\right) \to \EE_\infty\Alg\left(\Spectra^{C_2}\right) $ of Definition \ref{defn:calgp_alt_diagram} is monadic. 
\end{proposition}
\begin{proof}
    By the Barr--Beck--Lurie theorem, it suffices to show that $ G' $ is a right adjoint, is conservative, and any $ G' $-split simplicial object $ V $ of $ \CAlgp $ admits a colimit in $ \CAlgp $ which is preserved by $ G' $ \cite[Theorem 4.7.3.5]{LurHA}. 
    Since $ \CAlgp $ and $ \EE_\infty\Alg\left(\Spectra^{C_2}\right) $ are presentable, by the adjoint functor theorem \cite[Corollary 5.5.2.9]{HTT} to show that $ G' $ admits a left adjoint, it suffices to show that $ G' $ is accessible and preserves small limits. 
    This follows from Theorem \ref{thm:poincare_rings_cat_formal_properties}\ref{thmitem:poincare_ring_to_ring_preserves_lims} \& \ref{thmitem:poincare_ring_to_ring_preserves_sift_colims}. 
    The rest of the proof is similar to that of \cite[Proposition 3.23]{LYang_normedrings} and has been omitted. 
\end{proof}

\begin{proposition}\label{prop:unit_in_diagrammatic_calgp}
    Let $ A $ be an $ \EE_\infty $-ring in $ \Spectra^{C_2} $ and let $ (B, n_B \colon N^{C_2}B \to B) $ be a Poincaré ring in $ \Spectra^{C_2} $. 
    Then precomposition with the $ \EE_\infty $-map $ \eta_A : A \to GF(A) $ of Theorem \ref{thm:free_operadic_calgp_formula} induces an equivalence of morphism spaces
    \begin{equation}\label{eq:unit_in_normedalg}
        \begin{tikzcd}
            {\hom_{\CAlgp}\left(\left(\gamma F(A), n_{F(A)} \colon N^{C_2}\gamma F(A) \to \gamma F(A)\right),\left(B, n_B \colon N^{C_2} B \to B \right)\right)} \ar[d,"{G'}"] \ar[dd, bend left=75,looseness=2,"{\rotatebox{90}{$\sim$}}"] \\
            \hom_{\EE_\infty}(G'\gamma F(A), G'(B)) \ar[d,"{\eta^*}"] \\
            \hom_{\EE_\infty}(A, G'(B)).
        \end{tikzcd}
    \end{equation}
    where $ G' $ is the forgetful functor and $ \gamma $ is the functor of Construction \ref{cons:calgp_operadic_to_diagrammatic}. 
    That is, $ \eta_{(-)} $ is a unit for the functors $ (\gamma \circ F, G') $ in the sense of \cite[Definition 5.2.2.7]{HTT}. 
\end{proposition}
\begin{corollary}\label{cor:free_algebras_agree}
    The natural transformation $ \eta_{(-)} $ exhibits $ \gamma \circ F $ as a left adjoint to $ G' $. 
\end{corollary}
\begin{proof}[Proof of Proposition \ref{prop:unit_in_diagrammatic_calgp}]
    Notice that 
    \begin{equation*}
        \hom_{\EE_\infty\Alg\left(\Spectra^{C_2}\right)}(F(A),B) \simeq \hom_{\EE_\infty\Alg\left(\Spectra^{C_2}\right)}(A,B) \fiberproduct_{\hom\left(A^{\mathrm{t}C_2},B^{\mathrm{t}C_2}\right) } \hom \left( \morphism{A}{\nu_A}{A^{\mathrm{t}C_2}}, \morphism{B^{\varphi C_2}}{\Delta}{B^{\mathrm{t}C_2}} \right)
    \end{equation*}
    and moreover the composite
    \begin{equation*}
        \hom_{\CAlgp}(F(A), B) \xrightarrow{G'} \hom_{\EE_\infty\Alg\left(\Spectra^{C_2}\right)}(F(A),B) \xrightarrow{\pi_1} \hom_{\EE_\infty\Alg\left(\Spectra^{C_2}\right)}(A,B)
    \end{equation*}
    is equivalent to $ \eta^* \circ G' $. 
    Unravelling definitions, we see that given a point $ f \in \hom_{\EE_\infty\Alg\left(\Spectra^{C_2}\right)}(A,B) $, the fiber of $ \eta^* \circ G' $ over $ f $ is given by the space of fillings of the below diagram to a commutative diagram $ (\Delta^1)^{\times 3} \to \EE_\infty\Alg(\Spectra) $:  
    \begin{equation*}
        \begin{tikzcd}[row sep=small] 
            A^e \ar[rr,"{f^e}"] \ar[dd] \ar[rd,equals] & & B^e \ar[dd] \ar[rd,"{n_B}"] & \\
            & A^e  \ar[rr, dotted,crossing over] & & B^{\varphi C_2} \ar[dd] \\
            {(A^{\otimes 2})^{\mathrm{t}C_2}} \ar[rr,"{(f^{\otimes 2})^{\mathrm{t}C_2}}",very near end] \ar[rd] & & {( B^{\otimes 2})^{\mathrm{t}C_2}} \ar[rd] & \\
            & A^{\mathrm{t}C_2} \ar[rr,"{f^{\mathrm{t}C_2}}"] & & B^{\mathrm{t}C_2} \ar[from=2-2,to=4-2,crossing over,"{\nu_{A^e}}", near end]
        \end{tikzcd}
    \end{equation*} 
    wherein all but the top and front face are given. 
    This space is contractible. 
\end{proof}

We are now ready to show that the comparison map between the operadic and diagrammatic descriptions of Poincar{\'e} rings is an equivalence.

\begin{theorem}\label{thm:operadic_diagrammatic_calgp_agree}
    The functor of Construction \ref{cons:calgp_operadic_to_diagrammatic} is an equivalence. 
\end{theorem}
\begin{proof}
    The proof is nearly identical to that of \cite[Theorem 4.21]{LYang_normedrings}, but we include it here for the reader's convenience. 
    Consider the commutative diagram of forgetful functors 
    \begin{equation}\label{diagram:forgetforget_is_forget}
        \begin{tikzcd}[column sep=tiny]
            \EE_p\Alg\left(\Spectra^{C_2}\right) \ar[rr,"\gamma"] \ar[rd,"G"'] & & \CAlgp\left(\Spectra^{C_2}\right) \ar[ld,"G'"] \\
            & \EE_\infty\Alg\left(\Spectra^{C_2}\right) & 
        \end{tikzcd}
    \end{equation} 
    where $ G $ is from Remark \ref{rmk:operadic_Poincare_rings_to_E_infty_alg_in_genuine_spectra} and $ G' $ is from Definition \ref{defn:calgp_alt_diagram}. 
    
    The functor $ G $ is monadic by Remark \ref{rmk:operadic_Poincare_rings_to_E_infty_alg_in_genuine_spectra}. 
    The functor $ G' $ is monadic by Proposition \ref{prop:calgp_monadic}.
    Now for any $ A \in \EE_\infty\Alg\left(\Spectra^{C_2}\right) $, the unit $ A \to \gamma F(A) $ of Theorem \ref{thm:free_operadic_calgp_formula} induces an equivalence $ F'(A) \simeq \gamma F(A) $ by Corollary \ref{cor:free_algebras_agree}.  
    The result follows from \cite[Proposition 4.7.3.16]{LurHA}. 
\end{proof}

\begin{proposition}\label{prop:calgp_modules_poincare_monoidal}
    There is a canonical functor $ \CAlgp \to \EE_p\Mon\left(C_2\Cat\right) $. 
    In other words, for any Poincaré ring $ R $, the $ C_2 $-$ \infty $-category $ \Mod_R\left(\Spectra^{C_2}\right) \to \Mod_{R^e}\left(\Spectra\right) $ admits a canonical $ \EE_p $-monoidal refinement, and for any map of Poincaré rings $ R \to S $, the base change functor $ - \otimes_R S $ admits a canonical $ \EE_p $-monoidal refinement. 
\end{proposition}
\begin{proof}
    In view of Theorem \ref{thm:operadic_diagrammatic_calgp_agree}, $ \CAlgp $ admits a description as a category of algebras over the $ C_2 $-$ \infty $-operad $ \EE_p $. 
    The result follows from the same argument as in \cite[Proposition A.8]{LYang_normedrings}.  
\end{proof}

\addcontentsline{toc}{section}{\protect\numberline{\thesection}References}
\renewcommand*{\bibfont}{\small}
{\printbibliography}
\end{document}